\documentclass[12pt]{amsart}

\usepackage[margin=2.5cm]{geometry}
\usepackage{graphicx}
\usepackage{float}
\usepackage{amsmath,amssymb}
\usepackage{amscd}
\usepackage{verbatim}
\usepackage{enumerate}
\usepackage{subcaption}
\usepackage[colorlinks=true]{hyperref}
\usepackage{bbold}
\usepackage{tikz-cd}
\usepackage{multido}
\usepackage{mathrsfs}

\newtheorem{theorem}{Theorem}[section]
\newtheorem{definition}[theorem]{Definition}
\newtheorem{proposition}[theorem]{Proposition}
\newtheorem{lemma}[theorem]{Lemma}
\newtheorem{remark}[theorem]{Remark}
\newtheorem{example}[theorem]{Example}

\newtheorem{corollary}[theorem]{Corollary}

\newcommand{\g}{\mathfrak{g}}
\newcommand{\hh}{\mathfrak{h}}

\newcommand{\CC}{\mathbb{C}}
\newcommand{\QQ}{\mathbb{Q}}
\newcommand{\ZZ}{\mathbb{Z}}
\newcommand{\TT}{\mathbb{T}}

\newcommand{\Rep}{\mathcal{M}_{\textup{fd}}}
\newcommand{\Tr}{\operatorname{Tr}}

\newcommand{\End}{\operatorname{End}}
\newcommand{\id}{\mathrm{id}}
\newcommand{\mr}{\mathrm}
\newcommand{\ol}{\overline}
\newcommand{\ul}{\underline}

\newcommand{\wts}{\mathrm{wts}}

\newcommand{\str}{\mathrm{str}}

\newcommand{\cF}{\mathcal{F}}
\newcommand{\cG}{\mathcal{G}}
\newcommand{\cO}{\mathcal O}
\newcommand{\cN}{\mathcal N}
\newcommand{\cC}{\mathcal C}
\newcommand{\cD}{\mathcal D}
\newcommand{\cR}{\mathcal R}
\newcommand{\cM}{\mathcal M}

\newcommand{\tens}{\boxtimes}
\newcommand{\Hom}{\mathrm{Hom}}

\newcommand{\mh}{\mathrm{h}}

\newcommand{\bs}{\boldsymbol}

\newcommand{\Mfd}{\mathcal{M}_{\mathrm{fd}}}

\newcommand{\Nadm}{\mathcal{N}_{\mathrm{adm}}}

\newcommand{\Ndyn}{\overline{\mathcal{N}}_{\mathrm{fd}}}
\newcommand{\Rdyn}{R}
\newcommand{\til}{\widetilde}

\numberwithin{equation}{section}

\title[Graphical calculus for quantum vertex operators, II]{Graphical calculus for quantum vertex operators, II:\\ \(q\)-KZB and coordinate Macdonald-Ruijsenaars equations}

\author{Hadewijch De Clercq}
\address{H.D.C.: Department of Electronics and Information Systems, Mathematical Analysis Research Group, Ghent University,
	Belgium.}
\email{hadewijch.declercq@hotmail.com}

\author{Nicolai Reshetikhin}
\address{N.R.: YMSC, Tsinghua University, Beijing, Chine \& BIMSA, Beijing, China \& St. Petersburg University,
Russian Federation \& Department of Mathematics, University of California, Berkeley,
	CA 94720, USA.}
\email{reshetik@math.berkeley.edu}

\author{Jasper Stokman}
\address{J.S.: KdV Institute for Mathematics, University of Amsterdam,
	Science Park 904, 1098 XH Amsterdam, The Netherlands.}
\email{j.v.stokman@uva.nl }

\begin{document}

\begin{abstract}
We extend the graphical calculus developed in the first part of this paper to the parametrising spaces of quantum vertex operators. This involves a graphical implementation of the dynamical twist functor, which is a strict monoidal functor that describes how a morphism acting on the spin space of a quantum vertex operator $\Phi$ is transported to a morphism on the parametrising space of $\Phi$. The monoidal structure of the underlying nonstrict monoidal functor, considered before by Etingof and Varchenko in 1999, is given in terms of dynamical fusion operators, which are operators that describe the fusion of quantum vertex operators on the level of parametrising spaces. 

In the second part of the paper we use the extended graphical calculus to give intuitive, graphical derivations of various systems of difference equations for universal multipoint weighted trace functions. This includes the dual $q$-Knizhnik-Zamolodchikov-Bernard (KZB) and the dual Macdonald-Ruijsenaars (MR) equations, earlier obtained by Etingof and Varchenko in 2000, as well as an extension of the dual MR equations called dual coordinate MR equations. We use a known symmetry property of the universal weighted trace function, involving the exchange of its geometric and spectral parameter, to derive non-dual versions of these equations.
\end{abstract}
\maketitle
\tableofcontents

\section{Introduction}
Let $\mathfrak{g}$ be a semisimple Lie algebra with Cartan subalgebra $\mathfrak{h}$. The Reshetikhin-Turaev theory \cite{Reshetikhin&Turaev-1990} provides a graphical calculus for morphisms in a ribbon category $\mathcal{C}$ based on topological manipulations of $\textup{Ob}(\mathcal{C})$-colored ribbon graphs in $\mathbb{R}^3$. For the ribbon category $\Rep$ of type $1$ finite dimensional  $U_q(\mathfrak{g})$-modules it provides quantum invariants of framed links, such as the colored Jones polynomials.  In the prequel \cite{DeClercq&Reshetikhin&Stokman-2022} to this paper we incorporated quantum vertex operators into this quantum topological framework, in which case colorings by $U_q(\mathfrak{g})$-modules from the BGG category $\mathcal{O}$ are required. Category $\mathcal{O}$ can be embedded into a braided monoidal category $\mathcal{M}_{\textup{adm}}$ with twist, but rigidity is lost. As a consequence, the corresponding graphical calculus is based on braids with full twists.
With this modified graphical calculus we obtained a graphical derivation of $q$-KZ type eigenvalue equations for the dynamical fusion operators, which are the highest weight to highest weight components of quantum vertex operators.
In this paper we extend the graphical calculus to the parametrisation spaces of the quantum vertex operators, and use it to graphically derive known and new properties of the Etingof-Varchenko  \cite{Etingof&Varchenko-2000} weighted traces of quantum vertex operators.

We now give a more detailed description of the results in this paper. 
A $k$-point {\it quantum vertex operator} is a $U_q(\mathfrak{g})$-linear map of the form
\begin{equation}\label{qvoINTRO}
\bigl(\Phi_1\otimes\textup{id}_{V_2\otimes\cdots\otimes V_k}\bigr)\circ\cdots\circ\bigl(\Phi_{k-1}\otimes\textup{id}_{V_k})\circ\Phi_k: M_{\lambda_k}\rightarrow M_{\lambda_0}\otimes
\bigl(V_1\otimes\cdots\otimes V_k\bigr)
\end{equation}
with $V_i\in\textup{Ob}\bigl(\Rep\bigr)$, 
$M_\lambda$ the Verma module of highest weight $\lambda\in\mathfrak{h}^*$, 
and $U_q(\mathfrak{g})$-intertwiners $\Phi_i: M_{\lambda_i}\rightarrow M_{\lambda_{i-1}}\otimes V_i$.
For $\lambda_i$ in a suitable dense subset $\mathfrak{h}_{\textup{reg}}^*\subset\mathfrak{h}^*$ of regular highest weights, the space of $U_q(\mathfrak{g})$-inter\-twi\-ners $\Phi_i: M_{\lambda_i}\rightarrow M_{\lambda_{i-1}}\otimes V_i$ is parametrised by the weight space $V_i[\lambda_{i}-\lambda_{i-1}]$ of $V_i$ of weight $\lambda_{i}-\lambda_{i-1}$, see \cite{Etingof&Varchenko-1999}.
The isomorphism is given by the expectation value map, which picks out the highest weight to highest weight component of $\Phi_i$ in its Verma components. The $k$-point quantum vertex operator \eqref{qvoINTRO} will thus be denoted by $\Phi_{\lambda_k}^{v_1,\ldots,v_k}$, with $v_i$ the expectation value of $\Phi_i$ (in particular, $\Phi_i=\Phi_{\lambda_i}^{v_i}$). The quantum topological framework in \cite{DeClercq&Reshetikhin&Stokman-2022} provides the graphical representation
\begin{figure}[H]
		\centering
		\includegraphics[scale=0.65]{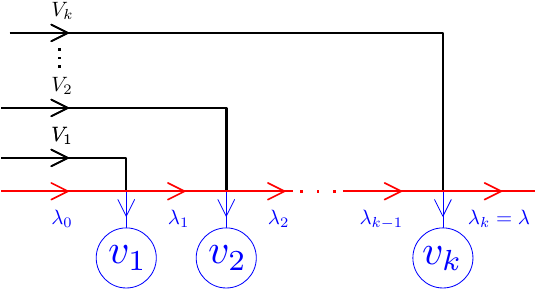}
		\captionof{figure}{}
		\label{vertex operator0}
\end{figure}
\noindent
of $\Phi_\lambda^{v_1,\ldots,v_k}$, with the blue label $\lambda_i$ indicating that the red strand is colored by the Verma module $M_{\lambda_i}$. The Reshetikhin-Turaev graphical calculus for $\Rep$
then applies to the black strands colored by $V_i$, while the restricted graphical calculus for $\mathcal{M}_{\textup{adm}}$
applies to the combined black and red strands. The $k$-point dynamical fusion operator $j_{(V_1,\ldots,V_k)}(\lambda)$ maps $v_1\otimes\cdots\otimes v_k$ to the highest-weight to highest-weight component of $\Phi_\lambda^{v_1,\ldots,v_k}$. It encodes how a $k$-point quantum vertex operator fuses to a $1$-point quantum vertex operator,
\[
 \Phi_\lambda^{v_1,\ldots,v_k}=\Phi_\lambda^{j_{(V_1,\ldots,V_k)}(\lambda)(v_1\otimes\cdots\otimes v_k)}.
 \]
 The highest-weight to highest-weight component of $\Phi_\mu^{v_1,\ldots,v_k}$ is depicted graphically by Fig. \ref{vertex operator0} with added open fat dots on both sides of the Verma strand.

The first part of this paper is devoted to 
extending the graphical calculus to the parametrising spaces of the multipoint quantum vertex operators (the blue strands in Figure \ref{vertex operator0}). For this 
we view the colors of the blue strands as the finite dimensional $\mathfrak{h}^*$-graded vector spaces $\underline{V_i}$ underlying $V_i\in\Rep$. They lie in the category $\overline{\mathcal{N}}_{\textup{fd}}$ of finite dimensional $\mathfrak{h}^*$-graded vector spaces with weights in the weight lattice $\Lambda$, with morphisms $\mathfrak{h}^*$-dependent grade-preserving maps. This category is monoidal 
with respect to the ``dynamical'' tensor product from \cite[\S 2.6]{Etingof&Varchenko-1999}, which involves weight shifts in the arguments (see also Lemma \ref{lemlow}). This monoidal structure gives a (rather limited) graphical calculus on the blue strands based on $\overline{\mathcal{N}}_{\textup{fd}}$-colored forest graphs. We enrich this graphical calculus as follows.

Etingof and Varchenko \cite[Lem. 17]{Etingof&Varchenko-1999} showed that the natural functor 
\[
\mathcal{F}^{\textup{EV}}: \Rep\rightarrow\overline{\mathcal{N}}_{\textup{fd}},
\]
with $\underline{V}:=\mathcal{F}^{\textup{EV}}(V)$ 
 the vector space $V$ endowed with the $\mathfrak{h}^*$-grading induced from its weight decomposition,
is monoidal with monoidal tensor $j=(j_{U,V})_{U,V\in\Rep}$ the $2$-point dynamical fusion operator. 
In Section \ref{Subsection dynamical module category category def} we give the upgrade of $\mathcal{F}^{\textup{EV}}$ to a {\it strict} monoidal functor
\[
\mathcal{F}^{\textup{dt}}: \Rep^\str\rightarrow\overline{\mathcal{N}}_{\textup{fd}}^{\,\str}
\]
called the {\it dynamical twist functor}, where $\mathcal{C}^\str$ is the strictification of a monoidal category $\mathcal{C}$, with objects the $k$-tuples of objects from $\mathcal{C}$ ($k\in\mathbb{Z}_{\geq 0}$). 
The dynamical twist functor $\mathcal{F}^{\textup{dt}}$ provides the functorial description how the action of a $\Rep^\str$-morphism on the spin spaces $(V_1,\ldots,V_k)$ of the $k$-point 
quantum vertex operator $\Phi_\lambda^{v_1,\ldots,v_k}$ can be transported to the local parametrising spaces $(\underline{V_1},\ldots,\underline{V_k})$ of $\Phi_\lambda^{v_1,\ldots,v_k}$. This is graphically depicted by Figure \ref{D2}. In this process the 
underlying morphism in $\Rep$ is twisted by dynamical fusion operators 
(see Theorem \ref{theorem dynamical twist}). For example, the braiding morphisms in $\Rep$, given by the action of the universal $R$-matrix of $U_q(\mathfrak{g})$ on two-fold tensor products $U\otimes V$ left-composed with the flip operator,
turn into dynamical braid morphisms in $\overline{\mathcal{N}}_{\textup{fd}}$, now given by the action of a universal solution $R(\lambda)$ of the {\it quantum dynamical quantum Yang-Baxter equation}, left-composed with the flip operator. The universal $R$-matrix $R(\lambda)$, obtained from the universal $R$-matrix of $U_q(\mathfrak{g})$ by twisting by the fusion operator, is the {\it universal exchange matrix} from
 \cite{Etingof&Varchenko-2000} with the tensor components flipped. 

Let $\textup{Rib}$ be the strict ribbon category of $\Rep^\str$-colored ribbon graphs and $\mathcal{F}^{\textup{RT}}: \textup{Rib}_{\Rep^\str}\rightarrow\Rep^\str$ the Reshetikhin-Turaev functor, which is the functorial encoding of the graphical calculus for $\Rep^\str$ \cite{Reshetikhin&Turaev-1990, DeClercq&Reshetikhin&Stokman-2022}. The composition
\[
\mathcal{F}^{\textup{dt}}\circ\mathcal{F}^{\textup{RT}}: \textup{Rib}_{\Rep^\str}\rightarrow\overline{\mathcal{N}}_{\textup{fd}}^{\,\str}
\]
is a strict monoidal functor that describes how morphisms in $\textup{Rib}_{\Rep^\str}$, attached to the black strands in Figure \ref{vertex operator0}, can be moved to coupons on the blue strands.
This in particular applies to the {\it elementary morphisms} in $\textup{Rib}$, i.e., the colored crossings, cups and caps. 
 We depict the resulting coupons in the graphical calculus of $\overline{\mathcal{N}}_{\textup{fd}}$ again as crossings, cups and caps, but
 decorated with a red dot
 (see Fig. \ref{dynamical R-matrix}--\ref{co-injection dynamical}). The morphism in $\overline{\mathcal{N}}_{\textup{fd}}$ underlying the decorated crossing is the universal $R$-matrix $R(\lambda)$ left-composed with the flip operator.
The morphism in $\overline{\mathcal{N}}_{\textup{fd}}$ underlying the counterclockwise-oriented decorated cap is the standard evaluation morphism for vector spaces, modified by the action of an automorphism $\mathbb{Q}(\lambda)$ in $\overline{\mathcal{N}}_{\textup{fd}}$ (see Subsection \ref{Subsection Q(lambda)}). This automorphism reappears in the self-dual normalisation of the weighted trace function.

Adding these decorated elementary morphisms to the graphical calculus of $\overline{\mathcal{N}}_{\textup{fd}}$ 
enriches the $\overline{\mathcal{N}}_{\textup{fd}}$-colored graphical calculus for the blue strands in Figure \ref{vertex operator0} to a decorated version of the Reshetikhin-Turaev graphical calculus, hence enabling computations based on topological moves on ribbon graphs. In the second part of the paper we show how this expanded graphical calculus can be used to derive various fundamental equations  for the Etingof-Varchenko \cite{Etingof&Varchenko-2000} weighted trace functions of quantum vertex operators by intuitive, graphical computations. Amongst these equations are the dual quantum Knizhnik-Zamolodchikov-Bernard ($q$-KZB) equations \cite[Thm. 1.4]{Etingof&Varchenko-2000}, and an extension of the dual Macdonald-Ruijsenaars (MR) equations \cite[Thm. 1.2]{Etingof&Varchenko-2000} which we call {\it the dual coordinate MR equations}. 

To explain the graphical approach in deriving such equations, we first need to introduce the weighted trace functions and their graphical representations.
Write $S=(V_1,\ldots,V_k)$ and let $\Phi: M_\mu\rightarrow M_\mu\otimes (V_1\otimes\cdots\otimes V_k)$ be a $U_q(\mathfrak{g})$-intertwiner.
The weighted trace function \cite{Etingof&Varchenko-2000} is the partial trace
\[
\xi\mapsto \textup{Tr}_{M_\mu}\bigl(\Phi\,q^\xi\bigr)
\]
for $\xi\in\mathfrak{h}^*$, where we tacitly assume that the real part of $\xi$ lies sufficiently deep in the negative Weyl chamber to ensure convergence. Here $q^\xi$ is viewed as element in the Cartan part of $U_q(\mathfrak{g})$, acting on $M_\mu$. The weighted trace function takes values in $(V_1\otimes\cdots\otimes V_k)[0]$.
We abuse graphical notations and graphically represent the weighted trace function as
\begin{figure}[H]
	\centering
	\includegraphics[scale=0.8]{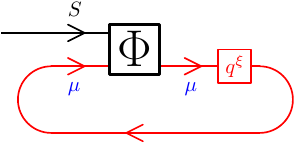}
\end{figure}
\noindent
(care is needed here because the cups and caps in the Verma strand are not part of the graphical calculus of $\mathcal{M}_{\textup{adm}}$ and the action of $q^\xi$ is not a morphism in $\mathcal{M}_{\textup{adm}}$. So $\mathcal{M}_{\textup{adm}}$-graphical calculus can only be applied to the part of the diagram describing $\Phi$). Spin components of the weighted trace function are then obtained by pairing with dual vectors from $(V_k^*\otimes\cdots\otimes V_1^*)[0]$. The space of dual vectors can be parametrised by the space of dual quantum vertex operators
$\Psi: M_\lambda\rightarrow (V_k^*\otimes\cdots\otimes V_1^*)\otimes M_\lambda$ with fixed $\lambda\in\mathfrak{h}^*_{\textup{reg}}$. The identification is again given by taking highest-weight to highest-weight components. The graphical representation of the resulting spin component $\mathfrak{t}_{\lambda,\mu,\xi}^{\Phi,\Psi}$ of the weighted trace function then becomes
\begin{figure}[H]
\centering
	\includegraphics[scale=.8]{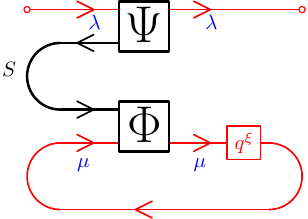}
	\caption{}
	\label{sctrace}
\end{figure}
\noindent

The graphical derivation of the dual $q$-KZB equations is by the following steps. We take $\Phi=\Phi_\mu^{v_1,\ldots,v_k}$ with $v_i\in V_i$ weight vectors whose weights sum up to zero, and substitute its graphical representation (Fig. \ref{vertex operator0}) in the graphical representation (Fig. \ref{sctrace}) of the associated spin-component $\mathfrak{t}_{\lambda,\mu,\xi}^{\Phi,\Psi}$ of the weighted trace function. This will allow us to apply the extended graphical calculus to push diagrams through the lower Verma strand labeled by $\mu$ to the blue strands.
In addition we take $\Psi$ to be a $3$-point dual quantum vertex operator. Acting by the Drinfeld-Lusztig \cite{Lusztig-1994} quantum Casimir and its inverse on the first and second intermediate Verma strand within the $3$-fold composition of $\Psi$ then gives an identity for $\Psi$, which graphically can be represented as
\begin{center}
	\includegraphics[scale = 0.73]{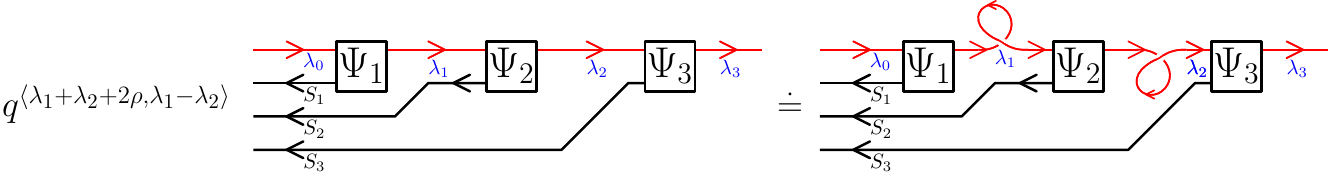}
\end{center}
with $\lambda_0=\lambda=\lambda_3$.
This gives an identity for the $\Psi$-spin-component $\mathfrak{t}_{\lambda,\mu,\xi}^{\Phi,\Psi}$ of the weighted trace function involving Figure \ref{sctrace}
with two full twists inserted in the upper Verma strand. We move these twists towards the parametrising spaces of $\Phi$ using the extended graphical calculus. The highest-weight to highest-weight component in the upper Verma strand is used to resolve crossings between black spin strands and the upper Verma strand. The weighted trace in the lower Verma strand allows to cyclicly reorder the constituents  $\Phi_{\mu_i}^{v_i}$ in $\Phi$, which comes with the cost of shifting $\mu$. We end up with a graphical expression that involves a single unwanted weighted cup in the spin strands attached to the lower Verma strand, which can only be moved to the parametrising space using the dynamical twist functor if the weight is
$\xi=2\lambda+2\rho$, where $\rho$ is the half sum of positive roots. In doing so, the end-result is a coupled system of difference equations in $\mu$ for 
the $\Psi$-spin-components 
$\mathfrak{t}_{\lambda,\mu,2\lambda+2\rho}^{\Phi,\Psi}$ of $k$-point weighted trace functions of the form $\textup{Tr}_{M_\mu}(\Phi\,q^{2\lambda+2\rho})$, where $\Psi$ and $\lambda$ are fixed.
This graphical derivation is given in Subsection \ref{Subsection Twisted trace functions}. In Subsection \ref{Subsection Etingof-Varchenko normalization} we show that these equations are equivalent to 
the dual $q$-KZB equations for the Etingof-Varchenko \cite{Etingof&Varchenko-2000} universal weighted trace function.

For the graphical derivation of dual coordinate MR equations we act by the Drinfeld-Res\-he\-ti\-khin \cite{Drinfeld-1990, Reshetikhin-1990} additive central elements \(\widetilde{\Omega}_W\in Z(U_q(\mathfrak{g}))$ ($W\in\textup{Ob}(\Rep)$) on the intermediate Verma strands of the multipoint dual quantum vertex operator $\Psi$ in the spin-component $\mathfrak{t}_{\lambda,\mu,\xi}^{\Phi,\Psi}$ of the weighted trace function.
The graphical representation of the action of $\widetilde{\Omega}_W$ on $M_\lambda$ is
by the diagram
\begin{figure}[H]
		\centering
		\includegraphics[scale=0.75]{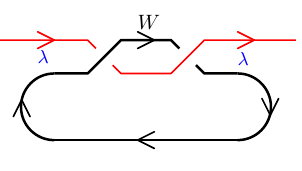}
\end{figure}
\noindent
On the other hand, it acts by scalar multiplication by $\chi_W(q^{-2\lambda-2\rho})$, with $\chi_W$ the character of $W$.
Inserting the resulting graphical identity into the graphical representation of $\mathfrak{t}_{\lambda,\mu,\xi}^{\Phi,\Psi}$ (Fig. \ref{sctrace}) and following a similar approach as for the graphical derivation of the dual $q$-KZB equations, one obtains another coupled system of difference equations in $\mu$ for the $\Psi$-spin-components 
$\mathfrak{t}_{\lambda,\mu,2\lambda+2\rho}^{\Phi,\Psi}$ of the $k$-point weighted trace functions $\textup{Tr}_{M_\mu}(\Phi\,q^{2\lambda+2\rho})$ with fixed $\Psi$ and $\lambda$,
see Subsections \ref{Subsection Topological operator dual MR} \& \ref{Subsection dual MR proof}. In Subsection \ref{Subsection Etingof-Varchenko normalization for MR} we reformulate these equations to a consistent system of difference equations for the Etingof-Varchenko universal trace function, which we call {\it the dual coordinate MR equations}. Here ``coordinate'' refers to the fact that these equations arise from the action of the additive Casimir element on intermediate Verma strands of the multipoint dual quantum vertex operator $\Psi$. The Etingof-Varchenko \cite[Thm. 1.2]{Etingof&Varchenko-2000} dual MR equations correspond to the special case that the additive Casimir elements act on the Verma module $M_\mu$ in the co-domain of $\Psi$. 

We end this introduction by explicitly formulating the new (dual) coordinate MR equations for the universal weighted trace functions.
 For generic $\mu\in\mathfrak{h}^*$ the partially normalised Etingof-Varchenko \cite{Etingof&Varchenko-2000} $k$-point weighted trace function is the $(V_1\otimes\cdots\otimes V_k)[0]$-valued function
\begin{equation}\label{normEVintro}
\mathbb{T}_S^{v_1,\ldots,v_k}(\lambda,\mu):=\delta(q^{2\lambda+2\rho})j_S(-\lambda-2\rho)^{-1}\textup{Tr}_{M_\mu}(\Phi_\mu^{v_1,\ldots,v_k}\,q^{2\lambda+2\rho})
\end{equation}
in $\lambda\in\mathfrak{h}^*$ with
$\delta$ the Weyl denominator, see \eqref{Wd}.

For weight vectors $f_i\in V_i^*$ with the sum of their weights $\textup{wt}(f_i)$ equal to zero we write
\[
\mathbb{T}_S^{v_1,\ldots,v_k\vert f_k,\ldots,f_1}(\lambda,\mu)=(f_1\otimes\cdots\otimes f_k)\bigl(\mathbb{T}_S^{v_1,\ldots,v_k}(\lambda,\mu)\bigr),
\]
so that
\[
\mathbb{T}_S^{v_1,\ldots,v_k}(\lambda,\mu)=\sum_{b_i\in\mathcal{B}_i: \sum_i\textup{wt}(b_i)=0}
\mathbb{T}_S^{v_1,\ldots,v_k\vert b_k^*,\ldots,b_1^*}(\lambda,\mu)b_k^*\otimes\cdots\otimes b_1^*,
\]
with $\mathcal{B}_i=\{b_i\}$ a basis of $V_i$ consisting of weight vectors and corresponding dual basis $\{b_i^*\}$. The {\it normalised universal weighted trace function} is
\[
\mathbb{F}_S(\lambda,\mu):=\sum_{b_i\in\mathcal{B}_i: \sum\textup{wt}(b_i)=0}\mathbb{T}_S^{b_1,\ldots,b_k}(\lambda,\mu)\otimes \mathbb{X}_{\mu,S^*}(b_k^*\otimes\cdots\otimes b_1^*)
\]
where $S^*=(V_k^*,\ldots,V_1^*)$ 
and
\[
\mathbb{X}_{\mu,S^*}:=\mathbb{Q}_{V_k^*}(\mu)^{-1}\otimes\mathbb{Q}_{V_{k-1}^*}(\mu+h_{V_k^*})^{-1}\otimes\cdots\otimes\mathbb{Q}_{V_1^*}(\mu+h_{(V_k^*,\ldots,V_1^*)})^{-1}
\]
(here we use the standard dynamical weight shift notations from, e.g., \cite{Etingof&Varchenko-2000}). The Etingof-Varchenko \cite{Etingof&Varchenko-2000} renormalised universal weighted trace function $F_S(\lambda,\mu)$ defined in \cite[(1.9)]{Etingof&Varchenko-2000} then equals $\mathbb{F}_S(\lambda-\rho,-\mu-\rho)$.

For $i=0,\ldots,k$ and $W\in\textup{Ob}(\Rep)$ the {\it dual coordinate MR operator} $\mathbb{L}_{W,i}^{\vee,S}$ is the difference operator defined by 
\begin{equation*}
\begin{split}
\bigl(\mathbb{L}_{W,i}^{\vee,S}f\bigr)(\mu):=&\sum_{\sigma\in\textup{wts}(W)}\textup{Tr}_{W^*[-\sigma]}\Bigl(
R_{W^*V_1^*}(\mu+h_{(W^*,V_k^*,\ldots,V_1^*)})\cdots R_{W^*V_i^*}(\mu+h_{(W^*,V_k^*,\ldots,V_i^*)})\\
&\qquad\qquad\quad R_{W^*V_{i+1}^*}^{21}(\mu+h_{(W^*,V_k^*,\ldots,V_{i+1}^*)})^{-1}\cdots
R_{W^*V_k^*}^{21}(\mu+h_{(W^*,V_k^*)})^{-1}\Bigr)f(\mu-\sigma)
\end{split}
\end{equation*}
for functions $f: \mathfrak{h}_{\textup{reg}}^*\rightarrow (V_k^*\otimes\cdots\otimes V_1^*)[0]$, where $\textup{wts}(W)$ is the set of weights of $W$. Here $R_{WV}(\lambda)$ stands for the action of the universal dynamical $R$-matrix $R(\lambda)$ on $W\otimes V$, and $R_{WV}^{21}(\lambda)=P_{VW}R_{VW}(\lambda)P_{WV}$ with flip operator $P_{WV}: W\otimes V\rightarrow V\otimes W$. The dynamical $R$-matrices act on the appropriate tensor components of $W^*\otimes V_k^*\otimes\cdots\otimes V_1^*$, and the trace is the partial trace over the weight space $W^*[-\sigma]$ in its first tensor component.
\begin{theorem}[Dual coordinate MR equations]
For $i=0,\ldots,k$ and $W\in\textup{Ob}(\Rep)$ we have
\begin{equation}\label{dualMReqnuniversalintro}
\bigl(\textup{id}_{(V_1\otimes\cdots\otimes V_k)[0]}\otimes\mathbb{L}_{W,i}^{\vee,S}\bigr)\mathbb{F}_S(\lambda,\cdot)=
\bigl(\mathbb{D}_{\lambda,W,i}^{\vee,S}\otimes\textup{id}_{(V_k^*\otimes\cdots\otimes V_1^*)[0]}\bigr)\mathbb{F}_S(\lambda,\cdot)
\end{equation}
where $\mathbb{D}_{\lambda,W,i}^{\vee,S}$ is the linear operator on $V_1\otimes\cdots\otimes V_k$ acting on $V_1[\xi_1]\otimes\cdots\otimes V_k[\xi_k]$ as scalar multiplication by
$\chi_W(q^{-2(\lambda+\xi_{i+1}+\cdots+\xi_k+\rho)})$.
\end{theorem}
This is Corollary \ref{dualqMRthm} in the main text. For $i=0$ the equations \eqref{dualMReqnuniversalintro} are the dual MR equations from \cite[Thm. 1.4]{Etingof&Varchenko-2000}
(note that $\mathbb{D}_{\lambda,W,0}^{\vee,S}$ acts on $(V_1\otimes\cdots\otimes V_k)[0]$ as scalar multiplication by $\chi_W(q^{-2(\lambda+\rho)})$, so the right hand side of \eqref{dualMReqnuniversalintro} simplifies to $\chi_W(q^{-2(\lambda+\rho)})\mathbb{F}_S(\lambda,\cdot)$ for $i=0$). 

The universal weighted trace function satisfies the symmetry property 
\begin{equation}\label{sEV}
\mathbb{F}_S(\lambda,\mu)=P\Bigl(\mathbb{F}_{S^*}(-\mu-2\rho,-\lambda-2\rho)\Bigr)
\end{equation}
(see \cite[Thm. 1.5]{Etingof&Varchenko-2000}) involving the exchange of the geometric and spectral parameter, where $P$ flips the two tensor components $(V_k^*\otimes\cdots\otimes V_1^*)[0]$ and
$(V_1\otimes\cdots\otimes V_k)[0]$. It turns the dual coordinate MR equations into the following system of difference equations for $\mathbb{F}_S(\lambda,\mu)$ in $\lambda$.

For $i=0,\ldots,k$ and $W\in\textup{Ob}(\Rep)$ define the {\it coordinate MR operators} by
\begin{equation*}
\begin{split}
\bigl(\mathbb{L}_{W,i}^Sf\bigr)(\lambda):=&\sum_{\sigma\in\textup{wts}(W)}\textup{Tr}_{W^*[-\sigma]}\Bigl(R_{W^*V_k}(-\lambda-2\rho+h_{(W^*,V_1,\ldots,V_k)})
\cdots\\
\cdots&R_{W^*V_{i+1}}(-\lambda-2\rho+h_{(W^*,V_1,\ldots,V_{i+1})})R_{W^*V_{i}}^{21}(-\lambda-2\rho+h_{(W^*,V_1,\ldots,V_{i})})^{-1}\cdots\\
&\qquad\qquad\qquad\qquad\cdots R_{W^*V_1}^{21}(-\lambda-2\rho+h_{(W^*,V_1)})^{-1}\Bigr)f(\lambda+\sigma)
\end{split}
\end{equation*}
for functions $f\in\mathfrak{h}_{\textup{reg}}^*\rightarrow (V_1\otimes\cdots\otimes V_k)[0]$.

\begin{theorem}[Coordinate MR equations]\label{coordMRthmintro}
For $i=0,\ldots,k$ and $W\in\textup{Ob}(\Rep)$ we have
\begin{equation}\label{MRiintro}
\bigl(\mathbb{L}_{W,i}^S\otimes\textup{id}_{\cF^\str(\ul{S^*})[0]}\bigr)\mathbb{F}_S(\cdot,\mu)=
\bigl(\textup{id}_{\cF^\str(\ul{S})[0]}\otimes \mathbb{D}_{\mu,W,i}^{S}\bigr)\mathbb{F}_S(\cdot,\mu)
\end{equation}
where $\mathbb{D}_{\mu,W,i}^{S}$ is the linear operator on $V_k^*\otimes\cdots\otimes V_1^*$ acting on $V_k^*[\xi_k]\otimes\cdots\otimes V_1^*[\xi_1]$ as scalar multiplication by
$\chi_W(q^{2(\mu-\xi_{1}-\cdots-\xi_i+\rho)})$.
\end{theorem}
This is part of Theorem \ref{coordMRthm} in the main text. For $i=k$ the equations \eqref{MRiintro} are the MR equations from \cite[Thm. 1.1]{Etingof&Varchenko-2000} (note that $\mathbb{D}_{\mu,W,k}^S$ acts on $(V_k^*\otimes\cdots\otimes V_1^*)[0]$ as scalar multiplication by $\chi_W(q^{2(\mu+\rho)})$, so the right hand side of \eqref{MRiintro} reduces to $\chi_W(q^{2(\mu+\rho)})\mathbb{F}_S(\cdot,\mu)$ for $i=k$).

The coordinate MR equations\eqref{MRiintro} imply the following system of difference equations for the partially normalised $k$-point weighted trace function
$\mathbb{T}_S^{v_1,\ldots,v_k}(\lambda,\mu)$ defined by \eqref{normEVintro}.

\begin{corollary}
Fix weight vectors $v_j\in V_j[\nu_j]$ \textup{(}$1\leq j\leq k$\textup{)} such that $\sum_{j=1}^k\nu_j=0$. The corresponding
weighted trace function $\mathbb{T}_S^{v_1,\ldots,v_k}(\cdot,\mu)$ satisfies the system
\begin{equation}\label{coordMRiintro}
\mathbb{L}_{W,i}^S\mathbb{T}_S^{v_1,\ldots,v_k}(\cdot,\mu)=\chi_W\bigl(q^{2(\mu+\nu_1+\cdots+\nu_{i})+2\rho}\bigr)\mathbb{T}_S^{v_1,\ldots,v_k}(\cdot,\mu)
\end{equation}
of coordinate MR equations for $W\in\textup{Ob}(\Rep)$ and $i=0,\ldots,k$.
\end{corollary}
\noindent
{\bf Contents:} In Section \ref{Section preliminaries} we introduce notations and we recall from the first part \cite{DeClercq&Reshetikhin&Stokman-2022} of the paper the graphical calculus for the braided monoidal category $\mathcal{M}_{\textup{adm}}$ of admissible $U_q(\mathfrak{g})$-modules. In Section \ref{Section dynamical module category} we introduce the dynamical twist functor $\mathcal{F}^{\textup{dt}}$ and we show that it is strict monoidal. In Section \ref{Section dual q-KZB} and Section \ref{Section dual MR} we graphically derive dual $q$-KZB type equations and dual coordinate MR type equations for the spin-components $\mathfrak{t}_{\lambda,\mu,2\lambda+2\rho}^{\Phi,\Psi}$ of the weighted trace functions, respectively. In Section
\ref{Section Etingof-Varchenko normalization} we derive from these equations dual $q$-KZB equations and dual coordinate MR equations for the Etingof-Varchenko \cite{Etingof&Varchenko-2000} universal weighted trace function $\mathbb{F}_S(\lambda,\mu)$, and we relate them to the dual $q$-KZB and dual MR equations from \cite[Thm. 1.4 \& Thm. 1.2]{Etingof&Varchenko-2000}. We furthermore use the symmetry property of $\mathbb{F}_S(\lambda,\mu)$ to re-obtain the $q$-KZB equations \cite[Thm. 1.3]{Etingof&Varchenko-2000} for $\mathbb{F}_S(\lambda,\mu)$ and to obtain coordinate MR equations for $\mathbb{F}_S(\lambda,\mu)$. Finally, in Section \ref{App} we include some useful identities for $R(\lambda)$ and $\mathbb{Q}(\lambda)$, derived using graphical calculus.

\vspace{.3cm}
\noindent
{\bf Discussion and outlook:}  The upshot of our graphical derivations of the dual $q$-KZB and dual MR equations for the spin-components $\mathfrak{t}_{\lambda,\mu,2\lambda+2\rho}^{\Phi,\Psi}$ of the weighted trace functions compared to \cite{Etingof&Varchenko-2000} is that it is intuitive and that it requires no intricate algebraic computations. But unfortunately it requires serious algebraic manipulations to obtain from those the dual $q$-KZB equations and dual coordinate MR equations for the Etingof-Varchenko \cite{Etingof&Varchenko-2000} universal trace function $\mathbb{F}_S(\lambda,\mu)$.
 It would be interesting to graphically derive the dual $q$-KZB and coordinate MR equations directly for the universal trace function $\mathbb{F}_S(\lambda,\mu)$, starting from the appropriate graphical realization of $\mathbb{F}_S(\lambda,\mu)$. In addition, it would be interesting to derive the $q$-KZB equations and coordinate MR equations graphically, without resorting to the symmetry \eqref{sEV} of 
 $\mathbb{F}_S(\lambda,\mu)$.

{}From the harmonic analytic viewpoint, $k$-point weighted trace functions may be viewed as the quantum group analogs of elementary vector-valued 
spherical functions on $G\times G/\textup{diag}(G)$ with values in the $k$-fold tensor product of finite dimensional $G$-representations. These in turn may be viewed as elementary spin graph functions on the moduli space of $G$-connections on the cyclic graph with $k$ vertices (see \cite{Reshetikhin&Stokman-2023}). The action of biinvariant differential operators on their edges then produces the classical analogs of Macdonald-Ruijsenaars equations, while the difference of the coordinate actions of the quadratic Casimir on neighboring edges gives KZB equations (see \cite{Stokman&Reshetikhin-2020-A,Reshetikhin&Stokman-2020-B}). The results in these papers are for Riemannian symmetric spaces $G/K$ (split in case of \cite{Stokman&Reshetikhin-2020-A}), in which case the role of $k$-point weighted trace functions is taken over by the so-called {\it $k$-point spherical functions} from \cite{Stokman&Reshetikhin-2020-A}. It will be interesting to generalise the results in this paper to the quantum group analogs of multipoint spherical functions, which will require the theory of quantum symmetric pairs \cite{L,K2} and the extended graphical calculus for braided module categories, which already has partially been developed in \cite{B} and in \cite[Chpt. 6]{dC}. For further discussions how multipoint vertex operators and multipoint weighted trace and spherical functions are rooted in quantum field theory and harmonic analysis, we refer to \cite{ES, DeClercq&Reshetikhin&Stokman-2022, Reshetikhin&Stokman-2020-B, Stokman&Reshetikhin-2020-A, Reshetikhin&Stokman-2023} and references therein.

\vspace{.3cm}
\noindent \textbf{Conventions:} We write $M\in\mathcal{C}$ to denote that $M$ is an object in the category $\mathcal{C}$. For a category $\mathcal{C}$ and $M,N\in\mathcal{C}$ we write 
$\textup{Hom}_{\mathcal{C}}(M,N)$ for the class of morphisms $M\rightarrow N$ in $\mathcal{C}$. We will simply denote this class by $\textup{Hom}(M,N)$ if the category is clear from the context.\newline

\noindent \textbf{Acknowledgments:} H.D.C. is a former PhD fellow of the Research Foundation Flanders (FWO). This paper is based on Chapter 3 of her PhD thesis \cite{dC}. The work of J.S. and N.R. was supported by the Dutch Research Council (NWO), project number 613.009.1260. The work of N.R. was also supported by the Changjiang fund, the international collaboration grant BMSTC and ACZSP, no. Z221100002722017, by grant No. 075-15-2024-631 from the Ministry of Science and Higher Education of the Russian Federation and by the Simons Foundation.

\section{Graphical calculus for quantum group modules and quantum vertex operators}
\label{Section preliminaries}
In this section, we recall from \cite{DeClercq&Reshetikhin&Stokman-2022} the algebraic, diagrammatic and categorical prerequisites for this paper, and we recall the definition of quantum vertex operators.
\subsection{The quantum group \(U_q(\g)\) and the universal R-matrix}
\label{Section U_q}
Let \(\g\) be a semisimple finite-dimensional Lie algebra over \(\CC\), with symmetrizable Cartan matrix \(A = (a_{ij})_{i,j=1,\dots,r}\), and let \(D = \mathrm{diag}(d_i)_{i=1,\dots,r}\) be a diagonal matrix with mutually coprime, positive integer entries such that \(DA\) is symmetric. Choose a Cartan subalgebra \(\mathfrak{h}\)  of \(\g\) and a 
 Borel subalgebra \(\mathfrak{b}\) containing \(\mathfrak{h}\), and write \(\mathfrak{n}^+:=[\mathfrak{b},\mathfrak{b}]\). Let \(\Phi\subseteq \hh^*\) be the associated root system of \(\g\), and \(\{\alpha_i: i= 1,\dots,r\}\) a set of simple roots. Write \(\Lambda\) for the lattice of integral weights and \(Q^+\) for the cone \(\bigoplus_{i=1}^r\ZZ_+\alpha_i \). Write \(\langle\cdot,\cdot\rangle\) for the non-degenerate symmetric bilinear form on \(\hh^\ast\) defined by \(\langle\alpha_i,\alpha_j\rangle = d_ia_{ij} \). The form $\langle\cdot,\cdot\rangle$ is the complexification of a scalar product on the real form $\bigoplus_{i=1}^r\mathbb{R}\alpha_i$
 of $\mathfrak{h}^*$. In particular there exists a basis \(\{x_i: i=1,\dots,r\}\) of \(\hh\) such that \(\langle\mu,\nu\rangle = \sum_{i=1}^r\mu(x_i)\nu(x_i)\) for any \(\mu,\nu\in\hh^\ast\). We will identify $\mathfrak{h}^*\overset{\sim}{\longrightarrow}\mathfrak{h}$ via the resulting isomorphism $\mu\mapsto\sum_{i=1}^r\mu(x_i)x_i$.

For \(0<q<1\), we denote by \(U_q:=U_q(\g)\) the quantized universal enveloping algebra associated to \(\g\) over \(\CC\). We follow here the conventions of this paper's prequel \cite[\S 1.1]{DeClercq&Reshetikhin&Stokman-2022}. In particular, we will write \(E_i, F_i, q^h\), with \(i = 1,\dots,r\) and \(h\in\hh\) for the algebraic generators of \(U_q\), \((\Delta,\epsilon, S)\) for its Hopf algebra structure, and \(U^+:=U_q(\mathfrak{n}^+)\) for the subalgebra of \(U_q\) generated by the elements \(E_i\), \(i = 1,\dots,r\). The square of the antipode is an inner automorphism of $U_q$, in the sense that
\begin{equation}
\label{action of q^2rho}
S^2(X)=q^{2\rho}Xq^{-2\rho}\qquad\quad\forall\, X\in U_q,
\end{equation}
where \(\rho\in\hh^\ast\) is half the sum of the positive roots. 
The quantum universal enveloping algebra \(U_q\) is quasi-triangular. The corresponding 
universal $R$-matrix is of the form
\(\cR = \kappa\ol{\cR}\) with \(\ol{\cR}\) Lusztig's quasi R-matrix and \(\kappa:=q^{\sum_{i=1}^rx_i\otimes x_i}\), interpreted as elements in an appropriate completion of $U_q^{\otimes 2}$. For the precise conventions we refer to \cite[Section 1.4]{DeClercq&Reshetikhin&Stokman-2022}.
\subsection{Categories of \(U_q\)-and \(\hh\)-modules}
\label{Section categories of Uq and h-modules}
In the prequel to this paper \cite[\S 1.3]{DeClercq&Reshetikhin&Stokman-2022}, we considered several full subcategories of the category $\textup{Mod}_{U_q}$ of left \(U_q\)-modules. We recall here only the ones that will be of importance in the present paper. First, let us recall the notation 
\[M[\mu]=\{m\in M \,\, : \,\, q^h\cdot m=q^{\mu(h)}m\quad\,\forall\, h\in\mathfrak{h}\}
\]
for the weight space of weight \(\mu\in\hh^\ast\) in a left \(U_q\)-module \(M\). We denote by
\[
\wts(M) = \{\mu\in\hh^\ast: M[\mu]\neq \{0\}\}
\]
the set of weights of \(M\). The $U_q$-submodule $M^\prime:=\bigoplus_{\mu\in\mathfrak{h}^*}M[\mu]$ inherits an $\mathfrak{h}$-module structure by the formula $h\vert_{M[\mu]}:=\mu(h)\textup{id}_{M[\mu]}$ for $h\in\mathfrak{h}$. 
We say that \(M\) is \(\mathfrak{h}\)-semisimple (sometimes also called a \emph{type $1$ module}) if \(M=M^\prime\). For an $\mathfrak{h}$-semisimple $U_q$-module $M$ we denote by $\mathcal{B}_M$ a homogeneous basis of \(M\), i.e.\ a basis consisting solely of weight vectors.

\begin{definition}
	\label{M_adm def}
	The category \(\cM_{\mr{adm}}\) of admissible $U_q$-modules is the full subcategory of $\textup{Mod}_{U_q}$ consisting of $\mathfrak{h}$-semisimple, locally $U^+$-finite  $U_q$-modules $M$ such that \(\textup{dim}(M[\mu])<\infty\)
for all \(\mu\in\mathfrak{h}^*\) and 
	\[
	\textup{wts}(M)\subseteq \bigcup_{i=1}^k\{\lambda_i-Q^+\}
	\]
	for some finite subset \(\{\lambda_1,\ldots,\lambda_k\}\) of \(\mathfrak{h}^*\).
\end{definition}
We write \(\pi_M\) for the representation map of \(M\in\cM_{\mr{adm}}\). The category \(\cM_{\mr{adm}}\) is braided monoidal.
The monoidal structure is obtained from the coalgebra structure of $U_q$ in the usual way (see e.g.\ \cite[\S 1.3]{DeClercq&Reshetikhin&Stokman-2022}). In particular,
the unit object $\mathbb{1}$ is $\mathbb{C}$, viewed as $U_q$-module with action $X\cdot\lambda:=\epsilon(X)\lambda$ for $X\in U_q$ and $\lambda\in\mathbb{C}$. The 
commutativity constraints are
\[
c_{M,N} = P_{M,N}\cR_{M,N},
\]
where
\begin{equation}
\label{transposition}
P_{M,N}: M\otimes N \to N\otimes M: m\otimes n\mapsto n\otimes m
\end{equation}
and \(\cR_{M,N}\) is the action of the universal R-matrix \(\cR\) on \(M\otimes N\), see \cite[\S 1.4]{DeClercq&Reshetikhin&Stokman-2022}.
In addition, \(\cM_{\mr{adm}}\) is equipped with a twist, given by the action of the ribbon element
\begin{equation}\label{qC}
\vartheta := q^{2\rho}\,m^{\mr{op}}((\id\otimes S^{-1})\cR^{-1}),
\end{equation}
where \(m^{\mr{op}}:U_q\otimes U_q\to U_q: a\otimes b\mapsto ba\) is the opposite multiplication, see \cite[\S 1.5]{DeClercq&Reshetikhin&Stokman-2022}. The resulting $U_q$-linear operators $\vartheta_M$ on $M\in\mathcal{M}_{\textup{adm}}$ are the {\it quantum Casimir operators}.
However, \(\cM_{\mr{adm}}\) is not a ribbon category by the lack of a duality structure.

For any \(\lambda\in\hh^\ast\), we denote by \(M_\lambda\in\cM_{\mr{adm}}\) the Verma module with highest weight \(\lambda\) relative to the Borel subalgebra $\mathfrak{b}$
(see \cite[\S 1.2]{DeClercq&Reshetikhin&Stokman-2022}), and we write \(\pi_\lambda:=\pi_{M_\lambda}\). We fix a highest weight vector $0\not=\mathbf{m}_\lambda\in M_\lambda[\lambda]$ once and for all, and denote by $\mathbf{m}_\lambda^*\in M_\lambda^*$ the unique linear functional such that $\mathbf{m}_\lambda^*(\mathbf{m}_\lambda)=1$
and $\mathbf{m}_\lambda^*\vert_{M_\lambda[\mu]}=0$ for $\mu\not=\lambda$.

Set
\[
\hh_{\mathrm{reg}}^\ast:=\{\lambda\in\hh^\ast: \langle\lambda,\alpha^\vee\rangle\notin\ZZ, \forall\alpha\in\Phi\},
\]
where \(\alpha^\vee = \frac{2\alpha}{\langle \alpha,\alpha\rangle}\). Then \(M_{\lambda+\nu}\) is irreducible for all \(\lambda\in\hh_{\mathrm{reg}}^\ast\) and \(\nu\in\Lambda\). We call \(\hh_{\mathrm{reg}}^\ast\) the set of generic weights. 

\begin{definition}
	We denote by \(\Rep\) the full subcategory of \(\cM_{\mr{adm}}\) consisting of the finite-dimensional \(U_q\)-modules.
\end{definition}
The category \(\Rep\) is a semisimple ribbon category, with braiding and twist inherited from \(\cM_{\mr{adm}}\), and with left and right duality given by \((V^\ast,e_V,\iota_V)\) and \((V^\ast,\widetilde{e}_V,\widetilde{\iota}_V)\), where
\begin{align}
\label{dualities_A}
&e_V: V^\ast\otimes V\to \CC: f\otimes v \mapsto f(v),& \qquad
&\widetilde{e}_V: V\otimes V^\ast \to \CC: v\otimes f\mapsto f\left(q^{2\rho}\cdot v\right),& \\
\label{dualities_B}
&\iota_V: \CC\to V\otimes V^\ast: 1\mapsto \sum_{b\in \mathcal{B}_V} b\otimes b^\ast,&\qquad
&\widetilde{\iota}_V: \CC\to V^\ast\otimes V: 1\mapsto \sum_{b\in\mathcal{B}_V} b^\ast\otimes q^{-2\rho}\cdot b.&
\end{align}

In this paper we consider spaces of intertwiners of $U_q$-modules that are parametrized by $\mathfrak{h}^*$-graded vector spaces. We collect here some convenient notations and terminology.
\begin{definition}
	We denote by \(\cN_{\mr{adm}}\) the category of \(\hh^\ast\)-graded vector spaces with finite-dimensional graded components. 
\end{definition}
The morphisms of \(\cN_{\mr{adm}}\) are the linear maps preserving the grading.
The category \(\cN_{\mr{adm}}\) is naturally isomorphic to the category of semisimple \(\mathfrak{h}\)-modules with finite-dimensional joint eigenspaces for the action of \(\mathfrak{h}\), from which it inherits the structure of a symmetric monoidal category. In particular, the unit object $\mathbb{C}_0$ in $\cN_{\mr{adm}}$ is the trivial one-dimensional $\mathfrak{h}$-module.

The graded components of \(M\in\cN_{\mr{adm}}\) are denoted by \(M[\mu]\) (\(\mu\in\mathfrak{h}^*\)). For the unit object we thus have $\mathbb{C}_0[0]=\mathbb{C}_0$.
We denote by $\mathcal{B}_M$ a homogeneous basis of $M$, i.e. a basis
\(\{b_i\}_i\) of $M$ such that $b_i\in M[\mu_i]$ for some $\mu_i\in\mathfrak{h}^*$. 

The forgetful functor 
\[
\cF^{\mr{frgt}}: \cM_{\mr{adm}}\to\cN_{\mr{adm}}
\]
maps objects \(M\in\cM_{\mr{adm}}\) and morphisms \(A\in\Hom(M,N)\) 
to \(\ul{M}\) and \(\ul{A}\) respectively, where \(\ul{M}\) is the semisimple $\mathfrak{h}$-module $M$ viewed as 
$\mathfrak{h}^*$-graded vector space, 
and \(\ul{A}\) is the linear map \(A\), viewed as morphism of \(\mathfrak{h}^*\)-graded vector spaces.

\begin{definition}
	Let \(\cN_{\mr{fd}}\) be the full subcategory of \(\cN_{\mr{adm}}\) consisting of finite-dimensional \(\hh^\ast\)-graded vector spaces with weights contained in \(\Lambda\). 
\end{definition}
The category \(\cN_{\mr{fd}}\) is a symmetric tensor category, with commutativity constraints inherited from \(\cN_{\mr{adm}}\), left duality \((V^\ast,e_V,\iota_V)\) defined as in (\ref{dualities_A}), and right duality \((V^\ast,\widehat{e}_V,\widehat{\iota}_V)\) given by
\begin{align}
\label{dualities_C}
&\widehat{e}_V:\ V\otimes V^{\ast} \to \CC: v\otimes f \mapsto f(v),& \qquad
&\widehat{\iota}_V:\ \CC\to V^{\ast}\otimes V: 1\mapsto \sum_{b\in\mathcal{B}_V} b^\ast\otimes b.&
\end{align}

\subsection{Strict monoidal categories}\label{SmcSection}

Through a standard construction due to Mac Lane \cite{MacLane-1971}, a monoidal category \(\cC = (\cC, \otimes_{\cC}, \mathbb{1}_{\cC},a^{\cC},\ell^{\cC},r^{\cC})\) with tensor product \(\otimes_{\cC}\), unit object \(\mathbb{1}_{\cC}\), associativity constraint \(a^{\cC}\), and left and right unit constraints \(\ell^{\cC}\) and \(r^{\cC}\), is monoidal equivalent to
a strict monoidal category \((\cC^\str,\tens_{\cC},\emptyset_{\cC})\). We recall the construction of \(\cC^\str\) and the definition of the monoidal equivalences, as these will play an important role in this paper.

The objects of \(\cC^\str\) are \(\{\emptyset_{\cC}\}\cup\{(V_1,\ldots,V_k)\,\,:\,\, k\in\mathbb{Z}_+,\, V_i\in\cC\}\). We say that an object of the form $(V_1,\ldots,V_k)$ has length $k$,
and we declare $\emptyset_{\cC}$ to be the unique object of length zero.
To an object \(S\in\cC^\str\) we associate an object \(\cF^\str_\cC(S)\in\cC\) by
\[
\cF^\str_\cC(\emptyset_\cC):=\mathbb{1}_\cC,\qquad \cF^\str_\cC((V_1,\ldots,V_k)):=V_1\otimes_{\cC}(V_2\otimes_{\cC}(\cdots\otimes_{\cC} (V_{k-1}\otimes_{\cC}V_k)\cdots)).
\]
For \(S,T\in\cC^\str\), the class of morphisms \(\Hom_{\cC^\str}(S,T)\) is defined as \(\Hom_{\cC}(\cF^\str_\cC(S),\cF^\str_\cC(T))\). The composition of $\mathcal{C}$ turns $\mathcal{C}^\str$ into a category. Furthermore, \(\cF^\str_\cC\) canonically extends to a functor
\[
\cF^\str_\cC: \cC^\str\rightarrow\cC
\]
mapping \(A\in\Hom_{\cC^\str}(S,T)\) to $A$, viewed as morphism in \(\textup{Hom}_{\cC}(\cF^\str_\cC(S),\cF^\str_\cC(T))\). It is an equivalence of categories. 

Note that the map \(\cF^\str_\cC: \textup{Ob}(\cC^\str)\rightarrow\textup{Ob}(\cC)\) is surjective but not injective. So a morphism $A$ in $\mathcal{C}^\str$ is only determined
by its underlying morphism $\cF^\str_\cC(A)$ if one also specifies the domain and codomain of $A$. 
\begin{definition}
Let \(B\in\textup{Hom}_{\cC}(V,W)\). We say that the morphism \(A\in\textup{Hom}_{\cC^\str}(S,T)\) is represented by \(B\)
if 
\(\cF^\str_\cC(A)=B\).
\end{definition}

A quasi-inverse of \(\cF^\str_\cC\) is the functor \(\cG^\str_\cC: \cC\rightarrow\cC^\str\) mapping \(V\in\cC\) 
to the sequence \((V)\in\mathcal{C}^\str\) of length one, and mapping the morphism \(A\in\textup{Hom}_{\cC}(V,W)\) to the unique morphism in \(\textup{Hom}_{\cC^\str}((V),(W))\) represented by \(A\). Then
\[
\cF_\cC^\str\circ\cG_\cC^\str=\textup{id}_{\cC},\qquad\qquad\cG_\cC^\str\circ\cF_\cC^\str\simeq\textup{id}_{\cC^\str}.
\]
The natural isomorphism \(J^{\cC}=(J_S^{\cC})_{S\in\cC^\str}: \textup{id}_{\cC^\str}\overset{\sim}{\longrightarrow} \cG^\str_\cC\circ\cF^\str_\cC\) consists of the {\it fusion morphisms} \(J_S^\cC\) in $\cC^\str$, which are the isomorphisms \(J_S^\cC\in\textup{Hom}_{\cC^\str}(S,(\cF^\str_\cC(S)))\) representing \(\id_{\cF^\str_\cC(S)}\in\textup{End}_{\cC}(\cF_\cC^\str(S))\).

The category \(\cC^\str\) can be promoted to a strict monoidal category \((\cC^\str,\tens_{\cC},\emptyset_{\cC})\) with tensor product \(\tens_\cC\) given on objects by concatenation of sequences. For morphisms \(A\in\textup{Hom}_{\cC^\str}(S,T)\) and \(B\in\textup{Hom}_{\cC^\str}(S^\prime,T^\prime)\), the morphism
\(A\tens_\cC B\in\Hom_{\cC^\str}(S\tens_\cC S', T\tens_\cC T')\) is represented by
\[
\sigma_{T,T^\prime}^{\cC}\bigl(\cF_\cC^\str(A)\otimes_{\cC}\cF_\cC^\str(B)\bigr)(\sigma_{S,S^\prime}^{\cC})^{-1}
\]
with 
\begin{equation}
\label{sigma def}
\sigma_{S,S^\prime}^{\cC}: \cF_\cC^\str(S)\otimes_\cC\cF_\cC^\str(S^\prime)\overset{\sim}{\longrightarrow}\cF_\cC^\str(S\tens_\cC S^\prime)
\end{equation}
the isomorphism that moves the closing brackets in \(\cF^\str_\cC(S)\) one by one across \(\cF_\cC^\str(S^\prime)\) using the associativity constraints (in the case that \(S=\emptyset_\cC\) (resp. \(S^\prime=\emptyset_\cC\)), the isomorphism \(\sigma_{S,S^\prime}^\cC\) is \(\ell_{\cF_\cC^\str(S')}^\cC\) (resp. \(r_{\cF_\cC^\str(S)}^\cC\))).
See e.g.\ \cite[\S XI.5]{Kassel-1995} for details. 

The functors $\cF_{\cC}^\str$ extend to an equivalence of monoidal categories,
\[
(\cF_{\cC}^\str,\textup{id}_{\mathbb{1}_\cC},\sigma^\cC): \cC^\str\rightarrow\cC
\]
with monoidal structure $\sigma^{\cC}=(\sigma^{\cC}_{S,T})_{S,T\in\cC^\str}$ (see e.g.\ \cite[\S 2.4]{Etingof&al-2015} or \cite[\S XI.4]{Kassel-1995}
for the definition and basic properties of monoidal functors). The corresponding monoidal extension of the quasi-inverse $\cG_{\cC}^\str$ is
\[
(\cG_{\cC}^\str,J_{\emptyset_\cC}^\cC,\tau^\cC): \cC\rightarrow\cC^\str
\]
with 
\(\tau^\cC=(\tau_{U,V}^{\cC})_{U,V\in\cC}\), where \(\tau_{U,V}^{\cC}: (U)\tens_\cC (V)\overset{\sim}{\longrightarrow}(U\otimes_\cC V)\) is defined by
\[
\tau_{U,V}^{\cC}:=J_{(U,V)}^{\cC}.
\]
Braiding, twists and duality structures naturally extend from \(\cC\) to \(\cC^\str\), cf.\ e.g.\ \cite[\S 2.3]{DeClercq&Reshetikhin&Stokman-2022}. For instance, if \(\cC\) is braided with braiding \((c_{U,V})_{U,V\in\cC}\), then the braiding \((c_{S,T})_{S,T\in\cC^\str}\) on \(\cC^\str\) consists of the isomorphisms
\(c_{S,T}\in\textup{Hom}_{\cC^\str}(S\tens_\cC T,T\tens_\cC S)\) represented by 
\[\sigma_{T,S}^\cC\circ c_{\cF^\str(S),\cF^\str(T)}\circ(\sigma_{S,T}^\cC)^{-1}.
\]

If $\mathcal{C}$ and $\mathcal{D}$ are monoidal categories and $F: \mathcal{C}\rightarrow\mathcal{D}$ is a functor, then we say that $F$ is {\it a strict monoidal functor}
if $(F,\textup{id}_{\mathbb{1}_{\mathcal{D}}},\textup{id}): \mathcal{C}\rightarrow\mathcal{D}$ is a monoidal functor, with $\textup{id}=\textup{id}_{F\circ\otimes_\cC}$ the trivial monoidal structure, $\textup{id}_{V,W}:=\textup{id}_{F(V\otimes_\cC W)}$ (viewed as isomorphism $F(V)\otimes_\cD F(W)\overset{\sim}{\longrightarrow} F(V\otimes_\cC W)$).
In \cite[\S 1]{Joyal&Street-1993} Joyal and Street show that a monoidal functor $\mathcal{C}\rightarrow\mathcal{D}$ naturally lifts to a strict monoidal functor $\mathcal{C}^\str\rightarrow\mathcal{D}^\str$. This is based on a general coherence theorem for monoidal functors (see \cite[Thm. 1.7]{Joyal&Street-1993}). In case one considers a strict monoidal functor $\mathcal{C}\rightarrow\mathcal{D}$, the induced strict monoidal functor $\mathcal{C}^\str\rightarrow\mathcal{D}^\str$ can be explicitly described as follows.
\begin{lemma}\label{stricttensorfunctorlemma}
Let $\cC$ and $\cD$ be two monoidal categories and let
\[
F=(F,\textup{id}_{\mathbb{1}_{\cD}},\textup{id}): \cC\rightarrow\cD
\]
be a strict monoidal functor. Then there exists a unique strict monoidal functor \(\widetilde{F}=(\widetilde{F},\textup{id}_{\emptyset_\cD},\textup{id}): \cC^\str\rightarrow\cD^\str\) such that
\begin{equation}
\label{comm2}
\cG_\cD^\str\circ F=\widetilde{F}\circ\cG_\cC^\str\qquad\mr{and}\qquad F\circ\cF_\cC^\str=\cF_\cD^\str\circ\widetilde{F}.
\end{equation}
\end{lemma}
\begin{proof}
The requirement that the functor \(\widetilde{F}\) extends to a strict monoidal functor satisfying \eqref{comm2} forces \(\widetilde{F}\) to be defined by 
\[
\widetilde{F}(\emptyset_\cC)=\emptyset_\cD,\qquad \widetilde{F}((V_1,\ldots,V_k))=(F(V_1),\ldots,F(V_k))
\]
on objects, while for a morphism \(A\in\textup{Hom}_{\cC^\str}(S,T)\), \(\widetilde{F}(A)\) has to be the unique morphism in \(\textup{Hom}_{\cD^\str}(\widetilde{F}(S),\widetilde{F}(T))\)
represented by
\[
F(\cF_\cC^\str(A))\in  \textup{Hom}_{\cD}(F(\cF_\cC^\str(S)),F(\cF_\cC^\str(T))=
\textup{Hom}_{\cD}(\cF_\cD^\str(\widetilde{F}(S)),\cF_\cD^\str(\widetilde{F}(T))).
\]
It is a straightforward check that these formulas indeed give rise to a strict monoidal functor \(\widetilde{F}: \cC^\str\rightarrow\cD^\str\) satisfying \eqref{comm2}.
\end{proof}
\begin{remark}
The equalities \eqref{comm2} are in fact equalities of monoidal functors, i.e.,
\begin{equation}\label{comm3}
\begin{split}
(\cG_\cD^\str,J_{\emptyset_\cD}^{\cD},\tau^\cD)\circ (F,\textup{id}_{\mathbb{1}_{\cD}},\textup{id})&=(\widetilde{F},\textup{id}_{\emptyset_\cD},\textup{id})\circ (\cG_\cC^\str,J_{\emptyset_\cC}^{\cC},\tau^\cC),\\
(F,\textup{id}_{\cD},\textup{id})\circ(\cF_\cC^\str,\textup{id}_{\mathbb{1}_\cC},\sigma^\cC)&=(\cF_\cD^\str,\textup{id}_{\mathbb{1}_\cD},\sigma^\cD)\circ (\widetilde{F},\textup{id}_{\emptyset_{\cD}},\textup{id}).
\end{split}
\end{equation}
\end{remark}
In the next section we will consider lifts of particular (non-strict) monoidal functors naturally appearing in the theory of multipoint weighted quantum group trace functions.

From now on we omit the super-- or sub-indices \(\cC\) in \(\cF_\cC^\str, \cG_\cC^\str,\otimes_\cC,\tens_\cC,\emptyset_\cC,\sigma^\cC, J^\cC...\) and we write objects \((V)\) of length 1 in \(\cC^\str\) simply as \(V\), when no confusion can arise.

\subsection{Graphical calculus for braided monoidal categories with twist}\label{gcfirst}
In the prequel to this paper \cite{DeClercq&Reshetikhin&Stokman-2022}, we have defined a graphical calculus for braided monoidal categories with twist, as an extension of the well-known graphical calculus for ribbon categories, and established the compatibility between the former and the latter. We briefly recall some of the key definitions here.

Let $\mathcal{D}$ be a braided monoidal category with twist. Let \(\mathbb{B}_{\cD^\str}\) be the strict braided monoidal category with twist whose objects are the tuples \((S_1,\dots,S_k)\) with \(S_i\in\cD^\str\), including the empty tuple \(\emptyset\), and whose morphisms are the isotopy classes of \(\cD^\str\)-colored ribbon-braid graphs, as defined in \cite[\S 2.2 \& \S2.4]{DeClercq&Reshetikhin&Stokman-2022}. Composition of morphisms is tantamount to vertical stacking of the associated diagrams, whereas tensor product of morphisms corresponds to placing the diagrams next to each other. We orient the strands of a ribbon-braid graph by requiring that the orientation is outgoing (resp. incoming) at the top (resp. bottom) endpoints of the strands.

Figure \ref{morphism} depicts a coupon in \(\mathbb{B}_{\cD^\str}\) with strands labeled by objects \(S_i\), \(T_j\in\cD^\str\), colored by
\(A\in\textup{Hom}_{\cD^\str}(S_1\tens\cdots\tens S_k, T_1\tens\cdots\tens T_\ell)\). 
\begin{figure}[H]
		\centering
		\includegraphics[scale=0.8]{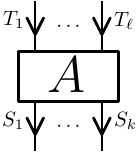}
	\captionof{figure}{}
		\label{morphism}
\end{figure}
\noindent
Figures \ref{braiding} -- \ref{twist} are the \(\cD^\str\)-colored ribbon-braid graphs providing the braiding and twist in \(\mathbb{B}_{\cD^\str}\) in case of single strands.
\begin{figure}[H]
	\begin{minipage}{0.32\textwidth}
		\centering
		\includegraphics[scale=0.8]{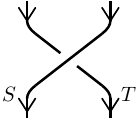}
		\captionof{figure}{\\\(c_{S,T}
		\)}
		\label{braiding}
	\end{minipage}
	\begin{minipage}{0.32\textwidth}
		\centering
		\includegraphics[scale=0.8]{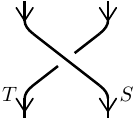}
		\captionof{figure}{\\\(c_{S,T}
		^{-1}\)}
		\label{inverse braiding}
	\end{minipage}
	\begin{minipage}{0.32\textwidth}
		\centering
		\includegraphics[scale=0.8]{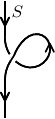}
		\captionof{figure}{\\\(\theta_{S}
		\)}
		\label{twist}
	\end{minipage}
\end{figure}

There exists a strict, twist preserving, braided monoidal functor
\begin{equation}
\label{RT functor extended}
\cF_{\cD^\str}^{\mr{br}}: \mathbb{B}_{\cD^\str}\to \cD^\str
\end{equation}
mapping an object \((S)\) of length one in \(\mathbb{B}_{\cD^\str}\) to \(S\in\cD^\str\) and mapping the coupon in Figure \ref{morphism} to its color \(A\)
(see \cite[Thm. 2.2]{DeClercq&Reshetikhin&Stokman-2022}). Given \(L\) and \(L'\) two (isotopy classes of) \(\cD^\str\)-colored ribbon-braid graphs in \(\mathbb{B}_{\cD^\str}\), we write \(L\dot{=} L'\) to indicate that \(\cF_{\cD^\str}^{\mr{br}}(L) = \cF_{\cD^\str}^{\mr{br}}(L')\).

Note that  Figure \ref{bundling} holds true for \(S = (V_1,\dots,V_k)\in\cD^\str\).  The \((S)\)-colored strand can be related to \(k\) parallel \((V_i)\)-colored strands in the category \(\mathbb{B}_{\cD^\str}\) itself using the coupon colored by the fusion morphism \(J_S\in\Hom_{\cD^\str}(S,\cF^\str(S))\) and the coupon colored by its inverse \(J_S^{-1}\) (see \cite[\S 2.8]{DeClercq&Reshetikhin&Stokman-2022}). We will denote these coupons by 
Figure \ref{fusion} and 
Figure \ref{inverse fusion}, respectively.

\begin{figure}[H]
	\begin{minipage}{0.32\textwidth}
		\centering
		\includegraphics[scale=0.7]{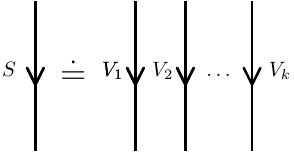}
		\captionof{figure}{}
		\label{bundling}
	\end{minipage}
	\begin{minipage}{0.32\textwidth}
		\centering
		\includegraphics[scale=0.5]{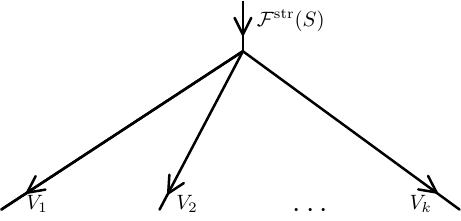}
		\captionof{figure}{}
		\label{fusion}
	\end{minipage}\quad
	\begin{minipage}{0.32\textwidth}
		\centering
		\includegraphics[scale=0.5]{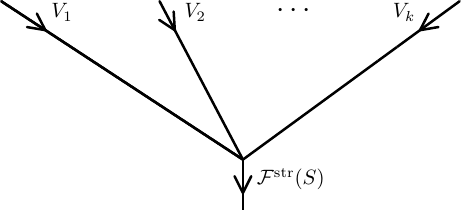}
		\captionof{figure}{}
		\label{inverse fusion}
	\end{minipage}
\end{figure}

If \(\cD\) is a ribbon category, then \(\cF_{\cD^\str}^{\mr{br}}\) factorizes through the Reshetikhin-Turaev functor \(\cF^{\mr{RT}}_{\cD}: \textup{Rib}_{\cD^\str}\rightarrow\cD^\str\),
where \(\textup{Rib}_{\cD^\str}\) is the strict ribbon category of isotopy classes of \(\cD^\str\)-colored ribbon graphs 
(see \cite[\S 2.5]{DeClercq&Reshetikhin&Stokman-2022}). 
Figures \ref{cap A}--\ref{cup B} are the \(\cD^\str\)-colored ribbon graphs providing the 
(co-)evaluation morphisms in \(\textup{Rib}_{\cD^\str}\). 

\begin{figure}[H]
	\begin{minipage}{0.28\textwidth}
		\centering
		\includegraphics[scale=0.8]{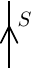}
		\captionof{figure}{\\\(\id_{S^\ast}
		\)}
		\label{dual}
	\end{minipage}\kern-3em
	\begin{minipage}{0.28\textwidth}
		\centering
		\includegraphics[scale=0.8]{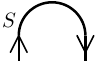}
		\captionof{figure}{\\\(e_{S}
		\)}
		\label{cap A}
	\end{minipage}\kern-3em
	\begin{minipage}{0.28\textwidth}
		\centering
		\includegraphics[scale=0.8]{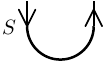}
		\captionof{figure}{\\\(\iota_{S}
		\)}
		\label{cup A}
	\end{minipage}\kern-3em
	\begin{minipage}{0.28\textwidth}
		\centering
		\includegraphics[scale=0.8]{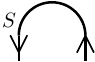}
		\captionof{figure}{\\\(\widetilde{e}_{S}
		\)}
		\label{cap B}
	\end{minipage}\kern-3em
	\begin{minipage}{0.28\textwidth}
		\centering
		\includegraphics[scale=0.8]{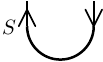}
		\captionof{figure}{\\\(\widetilde{\iota}_{S}
		\)}
		\label{cup B}
	\end{minipage}
\end{figure}

\subsection{Quantum vertex operators}
\label{Subsection Quantum vertex operators}
Throughout this section, let \(\lambda\in\hh_{\mathrm{reg}}^\ast\) and \(V\in\Rep\). Set
\begin{equation}\label{IntV}
\mr{Int}_{\lambda,V}:=\bigoplus_{\mu\in\hh^\ast}\Hom_{\cM_{\mr{adm}}}(M_\lambda,M_{\lambda-\mu}\otimes V).
\end{equation}
We view \(\mr{Int}_{\lambda,V}\) as \(\hh^\ast\)-graded vector space with \(\mu\)-graded component \(\textup{Int}_{\lambda,V}[\mu]\)
the space of intertwiners \(\Hom_{\cM_{\mr{adm}}}(M_\lambda,M_{\lambda-\mu}\otimes V)\). 

Recall the following well-known result \cite{Etingof&Latour-2005, Etingof&Varchenko-1999}.
\begin{proposition}
	\label{prop parametrizing spaces}
	The expectation value map
	\[
	\langle\cdot\rangle_{\lambda,V}: \mr{Int}_{\lambda,V}\rightarrow \underline{V}: (\phi^{(\mu)})_{\mu\in\hh^\ast} \mapsto \sum_{\mu\in\mathfrak{h}^*} (\mathbf{m}_{\lambda-\mu}^{\ast}\otimes\id_V)(\phi^{(\mu)}(\mathbf{m}_{\lambda}))
	\]
	is an isomorphism of \(\hh^\ast\)-graded vector spaces.
\end{proposition}
In particular, \(\textup{Int}_{\lambda,V}\) lies in \(\cN_{\rm{fd}}\) and \(\textup{Int}_{\lambda,V}[\mu]\neq \{0\}\) if and only if \(\mu\in\textup{wts}(V)\). The inverse of the isomorphism \(\langle\cdot\rangle_{\lambda,V}\) is denoted by 
\[
\ul{V}\to \mr{Int}_{\lambda,V}:\quad v\mapsto \phi_\lambda^v.
\] 
Concretely, if \(v\in V[\mu]\) is homogeneous, then 
\[
\phi_\lambda^v:\, M_\lambda\to M_{\lambda-\mu}\otimes V
\]
is the unique \(U_q\)-linear map such that \(\phi_\lambda^v(\mathbf{m}_\lambda)\in\mathbf{m}_{\lambda-\mu}\otimes v+\bigoplus_{\nu<\lambda-\mu}M_{\lambda-\mu}[\nu]\otimes V\), where \(\leq\) is the dominance order on \(\hh^*\). If no confusion can arise, then we omit the subscripts in the notation \(\langle\cdot\rangle_{\lambda,V}\) for the expectation
value.

Throughout this subsection, let \(S = (V_1,\dots,V_k)\in\Rep^\str\) be an object of length $k\in\mathbb{Z}_{>0}$ and let \(v_i\in V_i[\nu_i]\) be homogeneous vectors.
We define \(\phi_\lambda^{v_1,\ldots,v_k}\in\textup{Int}_{\lambda,\cF^\str(S)}[\sum_i\nu_i]\)
	by
	\begin{equation}\label{kEV}
	\phi_\lambda^{v_1,\ldots,v_k}:=
	(\phi_{\lambda_1}^{v_1}\otimes\textup{id}_{\cF^\str(S_2)})\cdots (\phi_{\lambda_{k-1}}^{v_{k-1}}\otimes\textup{id}_{\cF^\str(S_k)})\phi_{\lambda_k}^{v_k},
	\end{equation}
	with \(\lambda_j:=\lambda-\nu_{j+1}-\cdots-\nu_k\) and \(S_j:=(V_j,\dots,V_k)\),  for \(j = 1,\dots,k\) (here $\lambda_k$ should be read as $\lambda$).
	We write \(\Phi_\lambda^{v_1,\dots,v_k}\) for the unique morphism in \(\Hom_{\cM_{\mr{adm}}^\str}(M_\lambda,M_{\lambda-\sum_i\nu_i}\tens S)\) 
	represented by \(\phi_\lambda^{v_1,\ldots,v_k}\).
\begin{definition}
	\label{k-point quantum vertex operator def}
	We call morphisms of the form $\phi_\lambda^{v_1,\ldots,v_k}$ and  \(\Phi_\lambda^{v_1,\dots,v_k}\) \textup{(}$v_i\in V_i[\nu_i]$\textup{)}
	\(k\)-point quantum vertex operators of weight \((\nu_1,\ldots,\nu_k)\) and degree $\lambda$.
\end{definition}
The spaces \(V_i\in\Rep\) will be referred to as the local spin spaces of \(\Phi_\lambda^{v_1,\ldots,v_k}\). 

In the notation of Lemma \ref{stricttensorfunctorlemma}, let us write 
\[
\underline{S}:=\widetilde{\cF^{\rm{frgt}}}(S)\in
\cN_{\rm{fd}}^\str,
\] 
so that \(\underline{S}:=(\underline{V_1},\cdots,\underline{V_k})\) for $S=(V_1,\ldots,V_k)\in\Rep^\str$. Similarly,
we write 
\[
\underline{A}:=\widetilde{\cF^{\rm{frgt}}}(A)\in\textup{Hom}_{\cN_{\mr{fd}}^{\str}}(\underline{S},\underline{T})
\] 
for a morphism \(A\in\textup{Hom}_{\Rep^{\str}}(S,T)\).

For \(v_i\in V_i\) we use the standard multi-tensor notation \(v_1\otimes\cdots\otimes v_k\) to denote the iterated tensor product \(v_1\otimes (v_2\otimes(\cdots\otimes (v_{k-1}\otimes v_k)\cdots))\in\cF^\str(S)\). Similarly, 
for \(A_i\in\Hom_{\Rep}(V_i,W_i)\) we write \(A_1\otimes\cdots\otimes A_k\) to denote \(A_1\otimes(A_2\otimes(\cdots\otimes (A_{k-1}\otimes A_k)\cdots))\). We employ the same notations for multi-tensors in \(\cN_{\mr{fd}}\). The following result is from \cite{Etingof&Varchenko-1999}. 

\begin{proposition}
	The formula
	\[
	j_S(\lambda)(v_1\otimes\cdots\otimes v_k):=\langle\phi_\lambda^{v_1,\ldots,v_k}\rangle_{\lambda,\cF^\str(S)}
	\]
	for homogeneous vectors \(v_i\) in \(V_i\) defines an automorphism \(j_S(\lambda)\in\mr{Aut}_{\cN_{\mr{fd}}}(\cF^\str(\ul{S}))\).
\end{proposition}
For the unit object
\(\emptyset\in\cM_{\rm{fd}}^\str\) we set \(j_\emptyset(\lambda):=\id_{\ul{\mathbb{1}}}\). Note that for $S=(V)$ an object of length one, we have $j_{(V)}(\lambda)=\textup{id}_{\underline{V}}$. We write $\overline{J_S}(\lambda)$ for the unique morphism in $\textup{Hom}_{\cN_{\mr{fd}}^\str}(\underline{S},\mathcal{F}^\str(\underline{S}))$ represented by $j_S(\lambda)$. We will show in Section \ref{Subsection dynamical module category category def} that $\overline{J_S}$ is the dynamicalization of the fusion morphism $J_S$ from Section \ref{SmcSection} (this will also justify its notation).
\begin{definition}\label{dfoper}
We call morphisms $j_S(\lambda)$ and $\overline{J_S}(\lambda)$ dynamical fusion operators.
\end{definition}
Note that the dynamical fusion operator \(j_S(\lambda)\) (\(S\in\cM_{\rm{fd}}^\str\)) for an object $S$ of length $k$ relates
\(k\)-point quantum vertex operators to \(1\)-point quantum vertex operators, in the sense that
\begin{equation}
\label{fused}
\phi_\lambda^{v_1,\ldots,v_k}=\phi_\lambda^{j_S(\lambda)(v_1\otimes\cdots\otimes v_k)}.
\end{equation}
In the following lemma we list two additional properties of the dynamical fusion operators (their compatibility with tensor products will be discussed in Section \ref{Subsection dynamical module category category def}).
\begin{lemma}\label{hulplemma}
Let \(\lambda\in\mathfrak{h}_{\mr{reg}}^*\).
\begin{enumerate}
\item\label{hulplemma 1} For any \(A_i\in\textup{Hom}_{\cM_{\mr{fd}}}(V_i,W_i)\) \(\mr{(}i=1,\dots,k\mr{)}\) we have the identity
\[
(\ul{A_1}\otimes\cdots\otimes\ul{A_k})j_S(\lambda)=j_T(\lambda)(\ul{A_1}\otimes\cdots\otimes\ul{A_k})
\]
in \(\textup{Hom}_{\cN_{\mr{fd}}}(\cF^\str(\ul{S}),\cF^\str(\ul{T}))\),
where \(S:=(V_1,\ldots,V_k)\) and \(T:=(W_1,\ldots,W_k)\).
\item\label{hulplemma 2} For any \(V\in\cM_{\mr{fd}}\) we have \(j_{(\mathbb{1},V)}(\lambda)=\textup{id}_{\ul{\mathbb{1}}\otimes\ul{V}}\) and \(j_{(V,\mathbb{1})}(\lambda)=\textup{id}_{\ul{V}\otimes\ul{\mathbb{1}}}\).
\end{enumerate}
\end{lemma}
\begin{proof}
(\ref{hulplemma 1}) We use the notations as in Definition \ref{k-point quantum vertex operator def}. For \(v_i\in V_i[\nu_i]\) we have
\[
j_T(\lambda)\bigl((\ul{A_1}\otimes\cdots\otimes\ul{A_k})(v_1\otimes\cdots\otimes v_k)\bigr)=\langle\phi_\lambda^{\ul{A_1}(v_1),\ldots,\ul{A_k}(v_k)}\rangle,
\]
hence it suffices to note that
\[
\phi_{\lambda_i}^{\ul{A_i}(v_i)}=(\textup{id}_{M_{\lambda_{i-1}}}\otimes A_i)\phi_{\lambda_i}^{v_i},
\]
which in turn follows by computing the expectation value of both sides.\\
(\ref{hulplemma 2}) The formula \(j_{(V,\mathbb{1})}(\lambda)=\textup{id}_{\ul{V}\otimes\ul{\mathbb{1}}}\) follows from the fact that for \(c\in\mathbb{1}\) the intertwiner $\phi_\lambda^c\in\textup{Hom}_{\cM_{\mr{fd}}}(M_\lambda,M_\lambda\otimes\mathbb{1})$ is explicitly given by \(\phi_\lambda^c(m)=m\otimes c\) for $m\in M_\lambda$.
The second formula follows similarly.
\end{proof}

Recall from \cite[\S 3.1]{DeClercq&Reshetikhin&Stokman-2022} the convention to rotate all graphical calculus diagrams counterclockwise over 90 degrees when working with quantum vertex operators, and to color strands labeled by Verma modules \(M_\lambda\), with \(\lambda\in\hh_{\mathrm{reg}}^\ast\), in red, decorated with a blue label indicating the highest weight \(\lambda\). Consequently, the coupon in \(\mathbb{B}_{\cM_{\mr{adm}}}\) colored by \(\Phi\in\Hom_{\cM_{\mr{adm}}^\str}(M_\lambda,M_{\mu}\tens S)\) will be depicted by Figure \ref{intertwiner}. In the special case of a $k$-point vertex operator \(\Phi = \Phi_\lambda^{v_1,\dots,v_k}\), we will depict the corresponding coupon by Figure \ref{vertex operator}. We will incorporate the dynamical fusion operators $\overline{J_S}$ into this graphical calculus in Section \ref{Section strictified dynamical module category}.

\begin{figure}[H]
	\begin{minipage}{0.32\textwidth}
		\centering
		\includegraphics[scale=0.8]{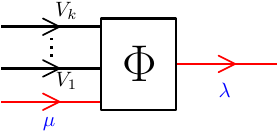}
		\captionof{figure}{}
		\label{intertwiner}
	\end{minipage}\quad
	\begin{minipage}{0.64\textwidth}
		\centering
		\includegraphics[scale=0.65]{diagram_0m_vertex_operator}
		\captionof{figure}{}
		\label{vertex operator}
	\end{minipage}
\end{figure}
\begin{remark}
	\label{remark dual quantum vertex operators}
	We will also make use of dual quantum vertex operators, corresponding to morphisms in \(\Hom_{\cM_{\mr{adm}}}(M_\lambda,V\otimes M_{\lambda-\mu})\) for some \(\mu\in\hh^\ast\) and $V\in\cM_{\mr{fd}}$. In analogy to Proposition \ref{prop parametrizing spaces} one has an isomorphism of \(\mathfrak{h}^*\)-graded spaces
	\[
	 \bigoplus_{\mu\in\hh^\ast}\Hom_{\cM_{\mr{adm}}}(M_\lambda,V\otimes M_{\lambda-\mu})\overset{\sim}{\longrightarrow}\underline{V}:\,\,\, (\phi^{(\mu)})_{\mu\in\hh^\ast} \mapsto \sum_{\mu\in\mathfrak{h}^*} (\id_V\otimes \mathbf{m}_{\lambda-\mu}^{\ast})(\phi^{(\mu)}(\mathbf{m}_{\lambda})),
	\]
with the obvious \(\mathfrak{h}^*\)-grading on the domain of the map. 
	With some abuse of notation, we will denote this isomorphism by \(\langle \cdot \rangle_{V,\lambda}\) and we will omit the subscript \(V,\lambda\) if no confusion can arise.
\end{remark}

\subsection{Graphical calculus for the category \(\cN_{\mr{adm}}\)}\label{GcSection}

The braiding \((P_{S,T})_{S,T\in\cN_{\mr{adm}}^\str}\) in the strictified symmetric monoidal category \(\cN_{\mr{adm}}^\str\) is obtained from the permutation morphisms \((P_{M,N})_{M,N\in\cN_{\mr{adm}}}\) in $\cN_{\mr{adm}}$.
The coupon in \(\mathbb{B}_{\cN_{\mr{adm}}^\str}\) colored by \(P_{S,T}\) will be depicted by Figure \ref{permutation}. If the color of a strand is $\underline{S}\in\cN_{\mr{adm}}^\str$ for some
$S\in\cM_{\mr{adm}}^\str$, then we depict the color graphically by $S$ if no confusion can arise.

For \(\nu\in\mathfrak{h}^*\) denote by \(\CC_\nu=\CC 1_\nu\in\cN_{\mr{fd}}\) the \(1\)-dimensional  \(\hh^\ast\)-graded vector space with \(\CC_\nu[\nu]=\CC_\nu\). 
For any \(M\in\cN_{\mr{adm}}\) and \(m\in M[\nu]\), we have a morphism $\alpha_m\in\Hom_{\cN_{\mr{adm}}}(\CC_\nu,M)$ defined by
\[
\alpha_m(1_\nu):=m.
\]
Similarly, for \(f\in M[\nu]^\ast\) there exists a morphism $\beta_f\in\Hom_{\cN_{\mr{adm}}}(M,\CC_\nu)$ defined by
\begin{equation*}
\beta_f(m):=
\begin{cases}
f(m)1_\nu\qquad &\hbox{ if } m\in M[\nu],\\
0\qquad &\hbox{ otherwise.}
\end{cases}
\end{equation*}
The coupons in \(\mathbb{B}_{\cN_{\rm{adm}}^\str}\) colored by the morphism
$\mathcal{G}^\str(\alpha_m)\in \Hom_{\cN_{\mr{adm}}^\str}((\CC_\nu),(M))$ and the morphism
$\mathcal{G}^\str(\beta_f)\in\Hom_{\cN_{\mr{adm}}^\str}((M),(\CC_\nu))$ will be depicted by Figures \ref{eval_m} and \ref{eval_f} respectively. In the upcoming Section \ref{Section strictified dynamical module category}, we will see that this is in agreement with the notation for the quantum vertex operators $\Phi_\lambda^{v_1,\ldots,v_k}$ as given in Figure \ref{vertex operator}.

\begin{figure}[H]
	\begin{minipage}{0.32\textwidth}
		\centering
		\includegraphics[scale=0.5]{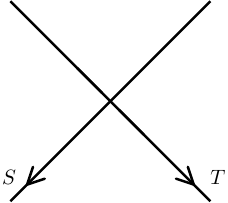}
		\captionof{figure}{}
		\label{permutation}
	\end{minipage}\quad
	\begin{minipage}{0.32\textwidth}
		\centering
		\includegraphics[scale=0.8]{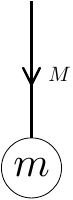}
		\captionof{figure}{}
		\label{eval_m}
	\end{minipage}
	\begin{minipage}{0.32\textwidth}
		\centering
		\includegraphics[scale=0.8]{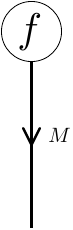}
		\captionof{figure}{}
		\label{eval_f}
	\end{minipage}
\end{figure} 
In case $M$ is of the form \(\ul{M_\lambda}\), the corresponding strand will be colored red, with a blue label \(\lambda\), and the diagram will be rotated over 90 degrees, in agreement with Section \ref{Subsection Quantum vertex operators}. Finally, we will use short-hand notations for the coupons colored by $\mathcal{G}^\str(\alpha_{\mathbf{m}_\lambda})$ and $\mathcal{G}^\str(\beta_{\mathbf{m}_\lambda^*})$, see Figures \ref{eval_lambda} and \ref{eval_lambda_dual}.
\begin{figure}[H]
	\begin{minipage}{0.48\textwidth}
		\centering
		\includegraphics[scale=0.8]{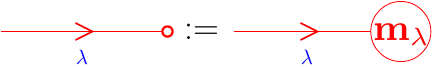}
		\captionof{figure}{}
		\label{eval_lambda}
	\end{minipage}\quad
	\begin{minipage}{0.48\textwidth}
		\centering
		\includegraphics[scale=0.8]{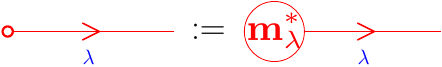}
		\captionof{figure}{}
		\label{eval_lambda_dual}
	\end{minipage}
\end{figure}
For later purposes, we recall the following result from \cite[Lemma 3.10]{DeClercq&Reshetikhin&Stokman-2022}.
\begin{lemma}
	\label{lemma 3.10}
	For any \(\lambda\in\hh_{\mathrm{reg}}^\ast\) and \(S\in\Rep^\str\), we have
	\begin{center}
		\includegraphics[scale = 0.75]{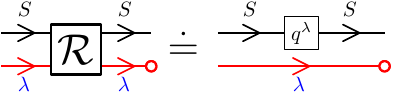} \qquad\qquad\qquad \includegraphics[scale = 0.75]{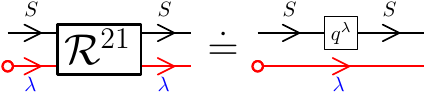}
	\end{center}
in the graphical calculus for \(\mathbb{B}_{\cN_{\mr{adm}}^\str}\).
Here the coupon colors $\mathcal{R}$ and $\mathcal{R}^{21}$ are the endomorphisms of $\underline{M_\lambda}\otimes\underline{S}$ 
representing the $\mathfrak{h}$-linear action of $\mathcal{R}$ and $\mathcal{R}^{21}$ on the two-fold tensor product representation $M_\lambda\otimes\cF^\str(S)$ of $U_q$,
and the color $q^\lambda$ is the endomorphism of $\underline{S}$ representing the $\mathfrak{h}$-linear action of $q^\lambda$ on $\cF^\str(S)$.
\end{lemma}

\section{Dynamical module categories}
\label{Section dynamical module category}
The aim of this section is to extend the graphical calculus for the braided monoidal category  \(\cM_{\mr{adm}}\) 
in such a way that it allows for graphical computations in the parametrizing spaces \(V\in\cN_{\mr{fd}}\) of the intertwiners \(\phi_\lambda^v\). The interplay between the two graphical calculi is described by a strict monoidal functor, which we will call the {\it dynamical twist functor}. It is a strict lift of a nonstrict monoidal functor
\begin{equation}
\label{the EV-functor intro}
\cF^{\textup{EV}}: \cM_{\mr{fd}}\rightarrow\Ndyn,
\end{equation}
where \(\Ndyn\) is a suitable
dynamical analogue of the monoidal category \(\cN_{\mr{fd}}\) introduced by Etingof and Varchenko in \cite[Section 2.6]{Etingof&Varchenko-1999}. The nontrivial monoidal structure of $\mathcal{F}^{\textup{EV}}$ is described by dynamical fusion operators. 

\subsection{The dynamical twist functor}
\label{Subsection dynamical module category category def}
In this subsection we introduce the dynamical module category $\Ndyn$, the functor $\cF^{\textup{EV}}$, and its lift to a strict monoidal functor.

Recall the dynamical fusion operators \(j_S(\lambda)\in\mr{Aut}_{\cN_{\mr{fd}}}(\cF^\str(\ul{S}))\) defined in Section \ref{Subsection Quantum vertex operators}. Their cocycle properties can most conveniently be described in terms of a dynamical version \(\Ndyn\) of the monoidal category \(\cN_{\mr{fd}}\) from \cite{Etingof&Varchenko-1999}. As a category, it is defined as follows.
\begin{definition}
	\label{def Dynamical module category}
	We denote by \(\Ndyn\) the category whose objects are the objects of $\cN_{\mr{fd}}$,
	and the class of morphisms \(\textup{Hom}_{\Ndyn}\bigl(V,W)\) consists of the maps 
	\[
	\hh^\ast_{\textup{reg}}\rightarrow \textup{Hom}_{\cN_{\mr{fd}}}(V,W).
	\]
	The composition is pointwise, 
	\[
	\bigl(A\circ B)(\lambda):=A(\lambda)\circ B(\lambda)
	\]
	for \(A\in\textup{Hom}_{\Ndyn}(V,W)\), \(B\in\textup{Hom}_{\Ndyn}(U,V)\) and \(\lambda\in\hh_{\textup{reg}}^\ast\). 
\end{definition}
We denote by 
\[
\mathcal{F}^{\textup{cst}}: \cN_{\mr{fd}}\rightarrow \Ndyn
\]
the functor mapping an object \(V\in\cN_{\mr{fd}}\) to itself and a morphism \(A\in\textup{Hom}_{\cN_{\mr{fd}}}(V,W)\) to the constant map \(\hh_{\textup{reg}}^\ast\to \textup{Hom}_{\cN_{\mr{fd}}}(V,W): \lambda\mapsto A\). We will denote \(\mathcal{F}^{\textup{cst}}(A)\) simply by \(A\) if no confusion can arise. 

\begin{example}\label{jS}
Recall the dynamical fusion operators $j_S(\lambda)$ from 
Definition \ref{dfoper}.
The resulting map
\[
j_S: \mathfrak{h}_{\textup{reg}}^*\rightarrow\textup{End}_{\cN_{\mr{fd}}}(\cF^\str(\underline{S})),\qquad
\lambda\mapsto j_S(\lambda),
\]
defines an 
endomorphism 
 of $\cF^\str(\ul{S})$ in $\Ndyn$. 

We have also introduced in Definition \ref{dfoper} the lift $\overline{J_S}(\lambda)$ of $j_S(\lambda)$ to a morphism in $\textup{Hom}_{\cN^{\str}_{\mr{fd}}}(\ul{S},\cF^\str(\ul{S}))$. In the present dynamical setup, it corresponds to the unique morphism in $\textup{Hom}_{\Ndyn^{\,\str}}(\ul{S},\cF^\str(\ul{S}))$ represented by $j_S$, which we will denote by $\ol{J_S}$.
\end{example}

Following \cite{Etingof&Varchenko-1999}, we provide \(\Ndyn\) with a nonstandard ``dynamical'' monoidal structure. The tensor product of morphisms in \(\Ndyn\) can be most conveniently defined in terms of the following weight-shifting morphisms in $\mathcal{N}_{\mr{fd}}$.

For \(A\in\textup{Hom}_{\Ndyn}(V_1,W_1)\), \(B\in\textup{Hom}_{\Ndyn}(V_2,W_2)\) and \(\lambda\in\mathfrak{h}_{\mr{reg}}^*\) we write
\begin{equation}\label{dynnot}
A(\lambda\pm\mh^{(2)})\otimes B(\lambda)
\end{equation}
for the morphisms in $\textup{Hom}_{\mathcal{N}_{\mr{fd}}}(V_1\otimes V_2,W_1\otimes W_2)$ which map \(v_1\otimes v_2\) onto \(A(\lambda\pm\mu_2)v_1\otimes v_2\)
for homogeneous vectors \(v_j\in V_j[\mu_j]\). Note that \eqref{dynnot} equals
\begin{equation}\label{rightdecomp}
(\textup{id}\otimes B(\lambda))(A(\lambda\pm\mh^{(2)})\otimes\textup{id}).
\end{equation}
Moreover, since $B(\lambda)$ preserves the grading, \eqref{dynnot} also equals
\begin{equation}\label{leftdecomp}
(A(\lambda\pm\mh^{(2)})\otimes\textup{id})(\textup{id}\otimes B(\lambda)),
\end{equation}
which thus makes the notation \eqref{dynnot} unambiguous. It is also sometimes convenient to use objects of the considered categories as the sublabels of the weight shift operator,
in which case we write $A(\lambda\pm\mh_{V_2})\otimes B(\lambda)$ if we want to think of \eqref{dynnot} as the composition \eqref{rightdecomp}, and $A(\lambda\pm\mh_{W_2})\otimes B(\lambda)$ if we want to think of \eqref{dynnot} as the composition \eqref{leftdecomp}.
In the natural manner we extend these definitions and notations to the other tensor factor, and to multifold tensor products of morphisms in $\Ndyn$.

The following lemma is from \cite{Etingof&Varchenko-1999}. For completeness we sketch its proof.
\begin{lemma}\label{lemlow}
		\(\Ndyn=(\Ndyn,\overline{\otimes},\mathbb{C}_0,a,\ell,r)\) is a monoidal category with tensor product $\overline{\otimes}$ defined on objects by
		\(V\overline{\otimes}\,V':=V\otimes V'\) and on morphisms by
		\begin{equation}
		\label{dynamical tensor product}
		\bigl(A\overline{\otimes}B\bigr)(\lambda):=A(\lambda-\mh^{(2)})\otimes B(\lambda)
		\end{equation}
		for \(A\in\textup{Hom}_{\Ndyn}(V,W)\),  \(B\in\textup{Hom}_{\Ndyn}(V',W')\) and
		\(\lambda\in\hh_{\textup{reg}}^\ast\). The unit object \(\mathbb{C}_0\) and the associativity and unit constraints \(a,\ell,r\) are inherited from the monoidal category \(\cN_{\mr{fd}}\) via \(\cF^{\mr{cst}}\).
\end{lemma}
\begin{proof}
	\(A\overline{\otimes} B\) is a well-defined morphism in \(\Ndyn\) since \(\textup{wts}(V')\subset\Lambda\) and
	\(\hh_{\textup{reg}}^\ast\) is stable under shifts by \(\Lambda\). A direct computation shows that \(\overline{\otimes}\) is a functor
	\(\Ndyn\times\Ndyn\rightarrow\Ndyn\) (the fact that the morphisms in \(\Ndyn\) respect the \(\hh^\ast\)-grading is needed here). It is easy to show that \((\Ndyn,\ol{\otimes})\) 	inherits a unit object, and associativity and unit constraints from \(\cN_{\mr{fd}}\) via \(\cF^{\mr{cst}}\).
\end{proof}
\begin{remark}
The weight condition for modules in \(\Ndyn\), i.e.\ the fact that \(\wts(V)\) must be contained in \(\Lambda\) for \(V\in\Ndyn\) , is needed in order to make sense of the \(\Lambda\)-shifts in the definition of the dynamical tensor product \(\overline{\otimes}\). This weight condition can be avoided by using the space of meromorphic functions \(\mathfrak{h}^*\rightarrow\textup{Hom}_{\cN_{\mr{fd}}}(V,W)\) as morphism spaces instead, as was done in \cite[Section 2.6]{Etingof&Varchenko-1999}.
\end{remark}
The monoidal category \(\Ndyn\) with the dynamical tensor product $\overline{\otimes}$ from Lemma \ref{lemlow} gives rise to the strict monoidal category \((\Ndyn^\str,\overline{\tens},\emptyset)\), see \S\ref{SmcSection}. On the other hand, $\cN_{\textup{fd}}$ is a symmetric monoidal category having the same objects as $\Ndyn$ and having its own (nondynamical) extension $(\cN_{\textup{fd}}^\str,\tens,\emptyset)$ to a strict monoidal category. Also $\Ndyn^\str$ and $\cN_{\textup{fd}}^\str$ have the same objects,
and the tensor product is the same on objects: $S\overline{\tens}T=S\tens T$.
We will now introduce notations relating morphisms in these two strictified monoidal categories, in such a way that it is compatible with \eqref{dynamical tensor product}.

Let \(S,T\in\Ndyn^\str\), \(A\in\Hom_{\Ndyn^\str}(S,T)\) and $\lambda\in\mathfrak{h}_{\textup{reg}}^*$. We will write 
\[
A(\lambda)\in \Hom_{\cN_{\mr{fd}}^\str}(S,T)
\]
for the unique morphism in \(\Hom_{\cN_{\mr{fd}}^\str}(S,T)\) represented by 
\[
\cF^\str_{\Ndyn}(A)(\lambda)\in\Hom_{\cN_{\mr{fd}}}(\cF^\str(S),\cF^\str(T)).
\]

For \(A\in\Hom_{\Ndyn^\str}(S,T)\) and $B\in\Hom_{\Ndyn^\str}(S',T')$ the morphism $A\overline{\tens}B$ in $\Ndyn^\str$
is represented by $\sigma_{T,T^\prime}(\cF^\str(A)\overline{\otimes}\cF^\str(B))\sigma_{S,S^\prime}^{-1}$, and hence the morphism $(A\overline{\tens}B)(\lambda)$
in $\cN_{\mr{fd}}^\str$ 
is represented by 
\[
\sigma_{T,T^\prime}\circ (\cF^\str(A)(\lambda-h^{(2)})\otimes\cF^\str(B)(\lambda))\circ\sigma_{S,S^\prime}^{-1},
\]
which justifies writing $A(\lambda-\mh^{(2)})\tens B(\lambda)$ for the morphism $(A\overline{\tens}B)(\lambda)$ in $\cN_{\mr{fd}}^\str$.
More generally, we write
\begin{equation}
\label{shifted weights}
A(\lambda\pm\mh^{(2)})\tens B(\lambda)
\end{equation}
for the morphisms in \(\Hom_{\cN_{\mr{fd}}^\str}(S\tens S', T\tens T')\) represented by
\[
\sigma_{T,T^\prime}\circ\bigl(\cF^\str(A)(\lambda\pm\mh^{(2)})\otimes\cF^\str(B)(\lambda)\bigr)\circ\sigma_{S,S^\prime}^{-1},
\]
and the morphisms \eqref{shifted weights} are alternatively denoted by $A(\lambda\pm\mh_{S^\prime})\tens B(\lambda)$ and $A(\lambda\pm\mh_{T^\prime})\tens B(\lambda)$.

With the monoidal structure on $\Ndyn$ from Lemma \ref{lemlow}, we can formulate the compatibility of the dynamical fusion operator $j_S$ (see Example \ref{jS}) with the monoidal structure as follows.
\begin{lemma}\label{coco}
For \(S\not=\emptyset\) and \(T\not=\emptyset\) we have
\begin{equation}\label{compassoc2}
\begin{split}
\sigma_{\ul{S},\ul{T}}\circ j_{(\cF^\str(S))\tens T}\circ(j_S\ol{\otimes}\textup{id}_{\cF^\str(\ul{T})})\circ\sigma_{\ul{S},\ul{T}}^{-1}&=j_{S\tens T}\\
&=j_{S\tens (\cF^\str(T))}\circ\sigma_{\ul{S},\ul{T}}\circ (\textup{id}_{\cF^\str(\ul{S})}\ol{\otimes}j_T)\circ\sigma_{\ul{S},\ul{T}}^{-1}.
\end{split}
\end{equation}
\end{lemma}
\begin{proof}
Write $S=(V_1,\ldots,V_k)$ and $T=(W_1,\ldots,W_\ell)$. For $\lambda\in\mathfrak{h}_{\textup{reg}}^*$, $v_i\in V_i[\mu_i]$ and $w_j\in W_j[\nu_j]$, 
the vector 
\[
u:=j_{S\tens T}(\lambda)(v_1\otimes\cdots\otimes v_k\otimes w_1\otimes\cdots\otimes w_\ell)
\]
is the unique weight vector in $\cF^\str(S\tens T)$ of weight $\sum_i\mu_i+\sum_j\nu_j$ such that 
\[
\phi_\lambda^u=\phi_\lambda^{v_1,\ldots,v_k,w_1,\ldots,w_\ell}.
\]
Upon rewriting the $(k+\ell)$-point quantum vertex operator $\phi_\lambda^{v_1,\ldots,v_k,w_\ell,\cdots,w_\ell}$ as a composition of two quantum vertex operators, i.e.
\[
\phi_\lambda^{v_1,\ldots,v_k,w_1,\ldots,w_\ell}=(\phi_{\lambda-\sum_j\nu_j}^{v_1,\ldots,v_k}\otimes\textup{id}_{\cF^\str(T)})\phi_\lambda^{w_1,\ldots,w_\ell},
\]
we can fuse the $(k+\ell)$-point quantum vertex operators in steps: first fuse $\phi_{\lambda-\sum_j\nu_j}^{v_1,\ldots,v_k}$, and then fuse the resulting $(\ell+1)$-point quantum vertex operator, or first fuse $\phi_{\lambda}^{w_1,\ldots,w_\ell}$, and then the resulting $(k+1)$-point quantum vertex operator. This gives
\[
\phi_\lambda^{j_{\cF^\str(S)\tens T}(\lambda)(j_S(\lambda-\sum_j\nu_j)(v_1\otimes\cdots\otimes v_k)\otimes w_1\otimes\cdots\otimes w_\ell)}
=\phi_\lambda^u=
\phi_\lambda^{j_{S\tens\cF^\str(T)}(\lambda)(v_1\otimes\cdots\otimes v_k\otimes j_T(\lambda)(w_1\otimes\cdots\otimes w_\ell))},
\]
where we have omitted the isomorphisms $\sigma$ that straighten the bracket order. Taking the expectation value of this identity immediately leads to \eqref{compassoc2}.
\end{proof}

Note that both \(\cF^{\mr{frgt}}: \Rep\to \cN_{\mr{fd}}\) and \(\cF^{\mr{cst}}: \cN_{\mr{fd}}\rightarrow\Ndyn\) are strict monoidal functors, and hence so is
\[
\cF^{\mr{EV}}:=\cF^{\mr{cst}}\circ\cF^{\mr{frgt}}: \Rep\rightarrow\Ndyn.
\]
In particular,
\begin{equation}\label{formulaalr}
\cF^{\mr{EV}}(a_{U,V,W})=a_{\ul{U},\ul{V},\ul{W}},\quad \cF^{\mr{EV}}(\ell_{U})=\ell_{\ul{U}},\quad
\cF^{\mr{EV}}(r_{U})=r_{\ul{U}}.
\end{equation}
Etingof and Varchenko observed in \cite[Lemma 17]{Etingof&Varchenko-1999} that the dynamical fusion operator produces another, non-trivial monoidal structure on
\(\cF^{\mr{EV}}: \Rep\rightarrow\Ndyn\). The result is as follows.

\begin{proposition}\label{proplow}
Set \(j:=(j_{(U,V)})_{U,V\in\Rep}\), with \(j_{(U,V)}\in\End_{\Ndyn}(\ul{U\otimes V})\) the dynamical fusion operator, viewed as a morphism \(\ul{U}\otimes\ul{V}\rightarrow\ul{U\otimes V}\). Then
\[
(\cF^{\mr{EV}},\textup{id}_{\mathbb{C}_0},j): \Rep\rightarrow \Ndyn
\] 
is a monoidal functor.
\end{proposition}
\begin{proof}
	By \eqref{formulaalr} and Lemma  \ref{hulplemma}(\ref{hulplemma 2})
	all requirements follow immediately, besides
	\begin{equation}\label{compassoc}
	a_{\ul{U},\ul{V},\ul{W}}\circ j_{(U\otimes V,W)}\circ (j_{(U,V)}\overline{\otimes}\textup{id}_{\underline{W}})=j_{(U,V\otimes W)}\circ (\textup{id}_{\ul{U}}\overline{\otimes} j_{(V,W)})
	\circ a_{\ul{U},\ul{V},\ul{W}}.
	\end{equation}
	But formula \eqref{compassoc} is obtained by observing that both sides of \eqref{compassoc} are equal to $j_{(U,V,W)}$. For the left-hand side of
	\eqref{compassoc} this follows from \eqref{compassoc2} by taking $S=(U,V)$ and $T=(W)$, for the right-hand side from \eqref{compassoc2} with $S=(U)$ and
	$T=(V,W)$.
\end{proof}

By Lemma \ref{stricttensorfunctorlemma}, \(\cF^{\mr{EV}}:\Rep\rightarrow\Ndyn\), viewed as a strict monoidal functor, extends to a strict monoidal functor 
\((\widetilde{\cF^{\mr{EV}}},\textup{id}_\emptyset,\textup{id}): \Rep^\str\rightarrow\Ndyn^{\,\str}\) satisfying
\[
(\cF^\str_{\Ndyn},\textup{id}_{\CC_0},\sigma)
\circ(\widetilde{\cF^{\mr{EV}}},\textup{id}_{\emptyset},\textup{id})\circ(\cG^\str_{\Rep},J_\emptyset,\tau)=(\cF^{\mr{EV}},\textup{id}_{\CC_0},\textup{id}).
\]
Our next aim will be to establish a concrete lift of the non-strict monoidal functor 
\[(\cF^{\textup{EV}},\textup{id}_{\CC_0},j): \Rep\rightarrow \Ndyn
\]
to a strict monoidal functor \((\cF^{\mr{dt}},\textup{id}_\emptyset,\textup{id}): \Rep^{\str}\rightarrow\Ndyn^{\,\str}\). 
It will play an important role in describing the graphical calculus for quantum vertex operators.

For $S\in\Rep^\str$, set
\[
\cF^{\mr{dt}}(S):=\underline{S}\in\Ndyn^{\,\str},
\] 
and for \(A\in\Hom_{\Rep^\str}(S,T)\), denote by $\cF^{\mr{dt}}(A)$ the unique morphism in \(\Hom_{\Ndyn^{\,\str}}(\ul{S},\ul{T})\)
represented by 
\begin{equation}\label{olAdown}
j_T^{-1}\circ \cF^{\mr{EV}}(\cF^{\str}(A))\circ j_S\in\Hom_{\Ndyn}(\cF^\str(\ul{S}),\cF^\str(\ul{T})).
\end{equation}
It is clear that this defines a functor $\cF^{\mr{dt}}: \Rep^\str\rightarrow\Ndyn^{\str}$.
\begin{definition}
	\label{dynamical twist functor def}
	We call $\cF^{\mr{dt}}: \Rep^\str\rightarrow\Ndyn^\str$ the dynamical twist functor. 
\end{definition}
We can now formulate the following main result of this subsection. 
\begin{theorem}
	\label{theorem dynamical twist}
	The dynamical twist functor $\cF^{\mr{dt}}$ is strict monoidal and satisfies
	\begin{equation}\label{descendEV}
	(\cF^\str_{\Ndyn},\textup{id}_{\CC_0},\sigma)\circ(\cF^{\mr{dt}},\textup{id}_{\emptyset},\textup{id})\circ
	(\cG^\str_{\Rep},J_\emptyset,\tau)=(\cF^{\mr{EV}},\textup{id}_{\CC_0},j).
	\end{equation}
\end{theorem}
\begin{proof}
We first note that 
\begin{equation}
\label{hulp0}
j_{S\tens S^\prime}=\sigma_{\ul{S},\ul{S^\prime}}\circ j_{(\cF^\str(S),\cF^\str(S^\prime))}\circ (j_S\ol{\otimes} j_{S^\prime})\circ \sigma_{\ul{S},\ul{S^\prime}}^{-1}
\end{equation}
for all \(S,S^\prime\in\Rep^\str\). When \(S=\emptyset\) this follows from the fact that \(\sigma_{\emptyset,\ul{S^\prime}}=\ell_{\cF^\str(\ul{S^\prime})}\), \(j_\emptyset=\textup{id}_{\ul{\mathbb{1}}}\) and Lemma \ref{hulplemma}(\ref{hulplemma 2}). In a similar fashion one verifies \eqref{hulp0} for \(S^\prime=\emptyset\). If \(S\not=\emptyset\) and \(S^\prime\not=\emptyset\), then \eqref{hulp0} follows from applying the first equation in \eqref{compassoc2} to \(j_{S\tens S^\prime}\) and subsequently the second equation in \eqref{compassoc2} to \(j_{(\cF^\str(S)),S^\prime}\), and noting that \(\sigma_{(\cF^\str(\ul{S})),\ul{S^\prime}}=\textup{id}_{\cF^\str(\ul{S})\otimes\cF^\str(\ul{S^\prime})}\).

For the first statement of the theorem, the nontrivial check is showing that
\[
\cF^{\mr{dt}}(A\tens B)=\cF^{\mr{dt}}(A)\ol{\tens}\cF^{\mr{dt}}(B)
\]
in \(\Hom_{\Ndyn^{\,\str}}(\ul{S}\ol{\tens}\ul{S^\prime},\ul{T}\ol{\tens}\ul{T^\prime})\)
for morphisms \(A\in\Hom_{\Rep^\str}(S,T)\) and \(B\in\Hom_{\Rep^\str}(S^\prime,T^\prime)\). To this end it suffices to show that
\[
\cF^\str(\cF^{\mr{dt}}(A\tens B))=\cF^\str(\cF^{\mr{dt}}(A)\ol{\tens}\cF^{\mr{dt}}(B))
\] 
in \(\Hom_{\Ndyn}(\cF^\str(\ul{S}\ol{\tens}\ul{S^\prime}),\cF^\str(\ul{T}\ol{\tens}\ul{T^\prime}))\), i.e.\ that
\begin{equation}\label{step1}
\begin{split}
&j_{T\tens T^\prime}^{-1}\circ\sigma_{\ul{T},\ul{T^\prime}}\circ (\cF^{\mr{EV}}(\cF^\str(A))\ol{\otimes}\cF^{\mr{EV}}(\cF^\str(B)))
\circ\sigma_{\ul{S},\ul{S^\prime}}^{-1}\circ j_{S\tens S^\prime}\\
=\ &\sigma_{\ul{T},\ul{T^\prime}}\circ (j_T^{-1}\ol{\otimes}j_{T^\prime}^{-1})\circ (\cF^{\mr{EV}}(\cF^\str(A))\ol{\otimes}\cF^{\mr{EV}}(\cF^\str(B)))
\circ (j_S\ol{\otimes} j_{S^\prime})\circ\sigma_{\ul{S},\ul{S^\prime}}^{-1}.
\end{split}
\end{equation}
Here we used that \(\cF^{\mr{EV}}(\sigma_{S,S^\prime})=\sigma_{\ul{S},\ul{S^\prime}}\) since $\cF^{\textup{EV}}$ is strict monoidal, cf.
\eqref{formulaalr}. 
By Lemma \ref{hulplemma}(\ref{hulplemma 1}) (or Proposition \ref{proplow}), the factor \(\cF^{\mr{EV}}(\cF^\str(A))\ol{\otimes}\cF^{\mr{EV}}(\cF^\str(B))\) in the right-hand side of \eqref{step1} may be replaced by 
\[
j_{(\cF^\str(T),\cF^\str(T^\prime))}^{-1}\circ(\cF^{\mr{EV}}(\cF^\str(A))\ol{\otimes}\cF^{\mr{EV}}(\cF^\str(B)))\circ j_{(\cF^\str(S),\cF^\str(S^\prime))},
\]
and then the desired equation (\ref{step1}) follows by applying \eqref{hulp0} twice.

In order to prove \eqref{descendEV}, observe first that \(\cF^\str(\cF^{\mr{dt}}(\cG^\str(U)))=\ul{U}=\cF^{\mr{EV}}(U)\) for \(U\in\Rep\), while for \(A\in\textup{Hom}_{\Rep}(U,V)\) we have
\[
\cF^\str(\cF^{\mr{dt}}(\cG^\str(A)))=j_{(V)}^{-1}\circ\cF^{\mr{EV}}(A)\circ j_{(U)}=\cF^{\mr{EV}}(A)
\]
since \(\cF^\str\circ\cG^\str=\textup{id}\) and \(j_{(U)}=\textup{id}_{\ul{U}}\). Hence the left-hand side of \eqref{descendEV} yields the monoidal functor
\((\cF^{\mr{EV}},\eta,\xi)\) with 
\[
\eta=\cF^\str(\cF^{\mr{dt}}(J_\emptyset))\circ\cF^\str(\textup{id}_{\emptyset})\circ\textup{id}_{\ul{\mathbb{1}}}=
(j_{(\mathbb{1})}^{-1}\circ\cF^{\mr{EV}}(\cF^\str(J_\emptyset))\circ j_\emptyset)\circ\textup{id}_{\ul{\mathbb{1}}}=\textup{id}_{\mathbb{C}_0}
\]
since \(\ul{\mathbb{1}} = \CC_0\), while for \(U,V\in\Rep\) we have
\begin{equation*}
\begin{split}
\xi_{U,V}&=\cF^\str(\cF^{\mr{dt}}(\tau_{U,V}))\circ\cF^\str(\textup{id}_{(\ul{U},\ul{V})})\circ\sigma_{\cF^{\mr{dt}}(\cG^\str(U)),\cF^{\mr{dt}}(\cG^\str(V))}\\
&=\cF^\str(\cF^{\mr{dt}}(J_{(U,V)}))\circ\sigma_{(\ul{U}),(\ul{V})}\\
&=j_{(U\otimes V)}^{-1}\circ\cF^{\mr{EV}}(\cF^\str(J_{(U,V)}))\circ j_{(U,V)}=j_{(U,V)}.
\end{split}
\end{equation*}
The last equality follows from the fact that \(j_{(W)}=\textup{id}_{\ul{W}}\) for \(W\in\Rep\) and the fact that
\(\cF^\str(J_{(U,V)})=\textup{id}_{U\otimes V}\).
\end{proof}
\begin{remark}
	The concrete construction of the lift $\cF^{\mr{dt}}$ of $\cF^{\textup{EV}}$ is consistent with the general procedure of constructing strict lifts of monoidal functors developed by Joyal and Street in \cite[Section 1]{Joyal&Street-1993}, which is based on a coherence theorem for monoidal functors, see \cite[Theorem 1.7]{Joyal&Street-1993}.
\end{remark}
We have defined the fusion operators $J_S\in\textup{Hom}_{\Rep^\str}(S,\cF^\str(S))$ and the dynamical
fusion operators $\overline{J_S}\in\textup{Hom}_{\Ndyn^{\,\str}}(\ul{S},\cF^\str(\ul{S}))$ for $S\in\Rep^\str$ in Section \ref{SmcSection}, Definition \ref{dfoper} and Example \ref{jS}. They are related by the dynamical twist functor in the following sense.
\begin{lemma}\label{jcomp}
We have $\cF^{\mr{dt}}(J_S)=\overline{J_S}$.
\end{lemma}
\begin{proof}
By definition the morphism
\[
\cF^{\mr{dt}}(J_S)\in\textup{Hom}_{\Ndyn^{\,\str}}(\ul{S},\cF^\str(\ul{S}))
\]
is represented by 
\[
j_{(\cF^\str(S))}^{-1}\circ \cF^{\mr{EV}}(\cF^\str(J_S))\circ j_S.
\]
This equals $j_S$, since $j_{(\cF^\str(S))}=\textup{id}_{\cF^\str(\ul{S})}$ and \(\cF^\str(J_S)=\textup{id}_{\cF^\str(S)}\).
Hence 
\[
\cF^\str(\cF^{\mr{dt}}(J_S))=j_S=\cF^\str(\ol{J_S}).
\] 
Since the (co)domains of $\cF^{\mr{dt}}(J_S)$ and $\ol{J_S}$ coincide, it follows that $\cF^{\mr{dt}}(J_S)=\ol{J_S}$.
\end{proof}
Lemma \ref{jcomp} prompts the following notation and terminology.
\begin{definition}\label{olA}
For a morphism \(A\in\Hom_{\Rep^\str}(S,T)\) we write \(\ol{A}\) for the morphism
\[
\ol{A}:=\cF^{\mr{dt}}(A)\in\textup{Hom}_{\Ndyn^{\,\str}}(\ul{S},\ul{T}),
\]
which we call the {\it dynamical twist} of \(A\). Furthermore, for \(\lambda\in\mathfrak{h}_{\mr{reg}}^*\) we write \(\ol{A}(\lambda)\) for the morphism in \(\textup{Hom}_{\mathcal{N}_{\mr{fd}}^\str}(\ul{S},\ul{T})\) represented by 
\[
\cF^\str(\ol{A})(\lambda)=j_T(\lambda)^{-1}\circ\cF^{\mr{frgt}}(\cF^\str(A))\circ j_S(\lambda)\in
\textup{Hom}_{\mathcal{N}_{\mr{fd}}}(\cF^\str(\ul{S}),\cF^\str(\ul{T})).
\]
\end{definition}
Note that \(\ol{A}\) is uniquely determined by the morphisms \(\ol{A}(\lambda)\), \(\lambda\in\hh_{\mathrm{reg}}^\ast\) after specifying the domain and codomain of \(\ol{A}\). 
By Lemma \ref{jcomp}, the notations introduced in Definition \ref{olA} are consistent with the notations for the dynamical fusion operators.

The following proposition provides an explicit formula for the dynamical twist $\ol{A}$ of a morphism $A\in\Hom_{\Rep^\str}(S,T)$ in terms of its trivial dynamical lift
$ \widetilde{\mathcal{F}^{\mr{EV}}}(A)$.
\begin{proposition}\label{olAexplicit}
The dynamical twist $\ol{A}\in\textup{Hom}_{\Ndyn^{\,\str}}(\ul{S},\ul{T})$ of $A\in\Hom_{\Rep^\str}(S,T)$ is explicitly given by
\begin{equation}\label{olAformula}
\ol{A}=\ol{J_T}^{\,-1}\circ J_{\ul{T}}\circ \widetilde{\mathcal{F}^{\mr{EV}}}(A)\circ J_{\ul{S}}^{-1}\circ\ol{J_S},
\end{equation}
where $J_{\ul{S}}=J_{\ul{S}}^{\Ndyn}$.
\end{proposition}
\begin{proof}
Recall that $\ol{J_S}\in\Hom_{\Ndyn^{\,\str}}(\ul{S},\cF^\str(\ul{S}))$ is represented by 
$j_S\in\textup{End}_{\Ndyn}(\cF^\str(\ul{S}))$, whence $\overline{A}$ and 
\begin{equation}\label{olAalmost}
\ol{J_T}^{-1}\circ\cG_{\Ndyn}^\str(\cF^{\mr{EV}}(\cF^\str_{\Rep}(A)))\circ\ol{J_S}
\end{equation}
are morphism in $\Hom_{\Ndyn^{\,\str}}(\ul{S},\ul{T})$ which are both
represented by \eqref{olAdown}. Consequently, $\ol{A}$ equals \eqref{olAalmost}. Then \eqref{olAformula} follows from
the fact that
\begin{equation*}
\begin{split}
\cG_{\Ndyn}^\str(\cF^{\mr{EV}}(\cF^\str_{\Rep}(A)))&=
\cG_{\Ndyn}^\str(\cF_{\Ndyn}^\str(\widetilde{\cF^{\mr{EV}}}(A)))\\
&=
J_{\ul{T}}\circ\widetilde{\cF^{\mr{EV}}}(A)\circ J_{\ul{S}}^{-1}.
\end{split}
\end{equation*}
\end{proof}

The importance of the dynamical twist functor for the graphical calculus of quantum vertex operators is captured by the following result.
\begin{theorem}
\label{prop pushed under water}
Let \(S=(V_1,\ldots,V_k), T=(W_1,\ldots,W_\ell)\in\Rep^\str\) and fix weight vectors \(v_i\in V_i[\nu_i]\). For \(\lambda\in\mathfrak{h}_{\mr{reg}}^*\) and
\(A\in\Hom_{\Rep^\str}(S,T)\) we have
\begin{equation}\label{dtintertwiner}
(\textup{id}_{M_{\lambda-\sum_i\nu_i}}\tens A)\Phi_\lambda^{v_1,\ldots,v_k}=\sum_n\Phi_\lambda^{w_1^n,\ldots,w_\ell^n},
\end{equation}
with \(w_j^n\in W_j\) homogeneous vectors such that 
\[
\cF^\str(\ol{A})(\lambda)(v_1\otimes\cdots\otimes v_k)=\sum_nw_1^n\otimes\cdots\otimes w_\ell^n.
\]
\end{theorem}
\begin{proof}
Both sides of \eqref{dtintertwiner} are morphisms in \(\Hom_{\Rep^\str}(M_\lambda,M_{\lambda-\sum_i\nu_i}\tens T)\). The left-hand side represents
\begin{equation*}
\begin{split}
(\textup{id}_{M_{\lambda-\sum_i\nu_i}}\otimes\cF^\str(A))\phi_\lambda^{j_S(\lambda)(v_1\otimes\cdots\otimes v_k)}
&=\phi_\lambda^{\cF^{\mr{frgt}}(\cF^\str(A))j_S(\lambda)(v_1\otimes\cdots\otimes v_k)}\\
&=\phi_\lambda^{j_T(\lambda)\cF^\str(\ol{A})(\lambda)(v_1\otimes\cdots\otimes v_k)},
\end{split}
\end{equation*}
where the first equality follows from the proof of Lemma \ref{hulplemma}(\ref{hulplemma 1}), and the second equality from the definition of \(\ol{A}(\lambda)\), see Definition \ref{olA}. The latter expression represents the right-hand side of \eqref{dtintertwiner}.
\end{proof}
In the following corollary we show that $\mathcal{F}^{\mr{dt}}$ reduces to 
$\widetilde{\mathcal{F}^{\textup{EV}}}$ when acting on tensor products of morphisms in the image of $\mathcal{G}^\str_{\Rep}$. 
\begin{corollary}\label{cor dyn twist}\hfill
\begin{enumerate}
\item\label{cor_item2} As an equality of functors, but not of monoidal functors, we have
\[\cF^{\mr{dt}}\circ\cG^\str_{\Rep}=
\widetilde{\cF^{\mr{EV}}}
\circ\cG^\str_{\Rep}.
\]
\item\label{cor_item3} For \(A_i\in\Hom_{\Rep^\str}((U_i),(V_i))\) with $U_i,V_i\in\Rep$ we have
\[
\cF^{\mr{dt}}\bigl(A_1\tens\cdots\tens A_k\bigr)=\widetilde{\cF^{\mr{EV}}}(A_1\tens\cdots\tens A_k).
\]
\end{enumerate}
\end{corollary}
\begin{proof}
	(\ref{cor_item2}) We have to show that \(\ol{A}=\widetilde{\cF^{\mr{EV}}}(A)\) for morphisms \(A\in\Hom_{\Rep^\str}((U),(V))\) with domain and codomain of length one.
	
	Note that both $\ol{A}$ and $\widetilde{\cF^{\mr{EV}}}(A)$ are morphisms
	in $\textup{Hom}_{\Ndyn^{\,\str}}((\ul{U}),(\ul{V}))$. Hence it suffices to show that their images under $\cF^\str$ coincide.
	By the second equality in \eqref{comm2} applied to $F=\mathcal{F}^{\textup{EV}}$,
	this will be the case when
	\begin{equation}\label{lengthonealt}
	\mathcal{F}^\str(\ol{A})=\cF^{\mr{EV}}(\mathcal{F}^\str(A))
	\end{equation}
	in $\textup{Hom}_{\Ndyn}(\ul{U},\ul{V})$.
	This follows from the fact that 
	\(j_{(W)}=\textup{id}_{\ul{W}}\) for
	all \(W\in\Rep\).\\
	(\ref{cor_item3}) This follows immediately from (\ref{cor_item2}) and the fact that both $\cF^{\mr{dt}}$ and $\widetilde{\cF^{\mr{{EV}}}}$ are strict monoidal functors.
\end{proof}

\subsection{Dynamical braiding, (co-)evaluation and twist}\label{sectiondb}
We can use the dynamical twist functor $\cF^{\mr{dt}}: \Rep^\str\rightarrow\Ndyn^{\,\str}$, which strictly preserves the monoidal structures, to push the braiding, (co-)evaluation and twist of the strict ribbon category $\Rep^\str$ to the category $\Ndyn^{\,\str}$. This does not lead to a ribbon subcategory structure on $\Ndyn^{\,\str}$, but it will allow us to develop a ribbon graph calculus for special classes of morphisms in $\Ndyn^{\,\str}$, as will be performed in Section \ref{sectiongi}.

By Proposition \ref{olAexplicit}, the dynamical counterparts of the braiding $(c_{S,T})_{S,T}$, left (co-)evaluations $(e_S)_{S}$, $(\iota_S)_{S}$,
and twist $(\vartheta_S)_{S}$ of $\Rep^\str$ are as follows:

\begin{enumerate}[\hspace{1pt}a.]
	\itemsep0em 
	\item {\bf Dynamical braiding:} the morphisms 
	\[
	\ol{c_{S,T}}=\ol{J_{T\tens S}}^{\,-1}\circ J_{\ul{T\tens S}}\circ\widetilde{\cF^{\mr{EV}}}(c_{S,T})\circ J_{\ul{S\tens T}}^{-1}\circ\ol{J_{S\tens T}}
	\]
	in $\textup{Hom}_{\Ndyn^{\,\str}}(\ul{S\tens T},\ul{T\tens S})$, for $S,T\in\Rep^\str$.
	\item {\bf Dynamical evaluation:} the morphisms 
	\begin{equation}
	\label{dyn eval expr}
	\ol{e_S}=\widetilde{\cF^{\mr{EV}}}(e_S)\circ J_{\ul{S^*\tens S}}^{-1}\circ \ol{J_{S^*\tens S}},
	\end{equation}
	in $\textup{Hom}_{\Ndyn^{\,\str}}(\ul{S^*\tens S},\emptyset)$, 
	for $S\in\Rep^\str$.
	\item {\bf Dynamical co-evaluation:} the morphisms 
	\begin{equation}
	\label{dyn coeval expr}
	\ol{\iota_S}=\ol{J_{S\tens S^*}}^{\,-1}\circ J_{\ul{S\tens S^*}}\circ\widetilde{\cF^{\mr{EV}}}(\iota_S)
	\end{equation}
	in $\textup{Hom}_{\Ndyn^{\,\str}}(\emptyset,\ul{S\tens S^*})$,
	for $S\in\Rep^\str$.
	\item {\bf Dynamical twist:} the morphisms 
	\[
	\ol{\vartheta_S}=\ol{J_S}^{\,-1}\circ J_{\ul{S}}\circ\widetilde{\cF^{\mr{EV}}}(\vartheta_S)\circ J_{\ul{S}}^{-1}\circ \ol{J_S}
	\]
	in $\textup{End}_{\Ndyn^{\,\str}}(\ul{S})$, for $S\in\Rep^\str$.
\end{enumerate}

The properties of $(c_{S,T})_{S,T}$, $((e_S)_{S},(\iota_S)_S)$ and $(\vartheta_S)_S$ turning them into a braiding, left (co-) evaluation and twist for $\Rep^\str$, have direct dynamical counterparts obtained by pushing these properties through the strict monoidal functor $\cF^{\mr{dt}}$. 

For instance, the properties characterizing the braiding are naturality and the hexagon identities. Their
dynamical counterparts are
\begin{equation}\label{dynnatural}
(\ol{B}\,\ol{\tens}\,\ol{A})\circ\ol{c_{S,T}}=\ol{c_{S',T'}}\circ(\ol{A}\,\ol{\tens}\,\ol{B})
\end{equation}
for $A\in\textup{Hom}_{\Rep^\str}(S,S')$ and $B\in\textup{Hom}_{\Rep^\str}(T,T')$, and 
\begin{equation}\label{dynhexagon}
\overline{c_{S,S^\prime\,\tens\,S^{\prime\prime}}}=(\textup{id}_{\ul{S^\prime}}\,\overline{\tens}\,\overline{c_{S,S^{\prime\prime}}})\circ
(\overline{c_{S,S^\prime}}\,\overline{\tens}\,\textup{id}_{\ul{S^{\prime\prime}}}),\qquad
\overline{c_{S\,\tens\,S^\prime,S^{\prime\prime}}}=(\ol{c_{S,S^{\prime\prime}}}\,\overline{\tens}\,\textup{id}_{\ul{S^\prime}})\circ (\textup{id}_{\ul{S}}\,\overline{\tens}\,\overline{c_{S^\prime,S^{\prime\prime}}}).
\end{equation}

The hexagon relations for $(c_{S,T})_{S,T}$ in $\Rep^\str$ lead to the braid form of the quantum Yang-Baxter equation, whose dynamical counterpart is given by
\begin{equation}\label{dynbraidYB}
\begin{split}
&(\textup{id}_{\ul{S^{\prime\prime}}}\,\overline{\tens}\,\overline{c_{S,S^\prime}})\circ(\overline{c_{S,S^{\prime\prime}}}\,\overline{\tens}\,\textup{id}_{\ul{S^\prime}})\circ(\textup{id}_{\ul{S}}\,\overline{\tens}\,\overline{c_{S^\prime,S^{\prime\prime}}})\\
=\ &(\overline{c_{S^\prime,S^{\prime\prime}}}\,\overline{\tens}\,\textup{id}_{\ul{S}})\circ(\textup{id}_{\ul{S^\prime}}\,\overline{\tens}\,\overline{c_{S,S^{\prime\prime}}})\circ(\overline{c_{S,S^\prime}}\,\overline{\tens}\,\textup{id}_{\ul{S^{\prime\prime}}}).
\end{split}
\end{equation}
\begin{definition}\label{Rdyn}
Write $R_{S,T}\in\textup{End}_{\Ndyn}(\cF^\str(\ul{S}\tens\ul{T}))$ for the endomorphism such that 
\[
P_{\cF^\str(\ul{S}),\cF^\str(\ul{T})}\circ R_{S,T}\in\textup{Hom}_{\Ndyn}(\cF^\str(\ul{S}\tens\ul{T}),\cF^\str(\ul{T}\tens\ul{S}))
\]
represents $\overline{c_{S,T}}$ \textup{(}we omitted here the associators\textup{)}.
\end{definition}
We then have the following result.
\begin{proposition}
	\label{prop hexagon for dynamical R-matrix}
\hfill
\begin{enumerate}
\item	The dynamical $R$-matrices $\Rdyn_{S,T}$ satisfy the dynamical hexagon identities
	\begin{equation*}
	\begin{split}
	&\Rdyn_{S, S'\tens S''}(\lambda) = \Rdyn_{S,S''}(\lambda)\Rdyn_{S,S'}(\lambda-\mh_{\ul{S''}}),\\
	&\Rdyn_{S\tens S',S''}(\lambda) = \Rdyn_{S,S''}(\lambda-\mh_{\ul{S'}})\Rdyn_{S',S''}(\lambda)
	\end{split}
	\end{equation*}
	and the dynamical quantum Yang-Baxter equation
	\begin{equation}\label{dynYBeqn}
	\Rdyn_{S',S''}(\lambda-\mh_{\ul{S}})\Rdyn_{S,S''}(\lambda)\Rdyn_{S,S'}(\lambda-\mh_{\ul{S''}})=
	\Rdyn_{S,S'}(\lambda)\Rdyn_{S,S''}(\lambda-\mh_{\ul{S'}})\Rdyn_{S',S''}(\lambda)
	\end{equation}
	in $\textup{End}_{\mathcal{N}_{\mr{fd}}}(\cF^\str(\ul{S})\otimes\cF^\str(\ul{S'})\otimes\cF^\str(\ul{S''}))$, where all the factors are viewed as endomorphisms of 
	$\cF^\str(\ul{S})\otimes\cF^\str(\ul{S'})\otimes\cF^\str(\ul{S''})$ in the natural manner.
\item For objects $(V)$ and $(W)$ of length one we have
\[
\Rdyn_{V,W}(\lambda)=(j^{21}_{(V,W)})^{-1}(\lambda)\circ\mathcal{R}_{V,W}\circ j_{(V,W)}(\lambda),
\]
where \((j^{21}_{(V,W)})^{-1}(\lambda):=P_{\ul{W},\ul{V}}\circ j_{(W,V)}(\lambda)^{-1}\circ P_{\ul{V},\ul{W}}\).
\item For objects $S$ and $T$ of arbitrary length we have
 \begin{equation}
\label{R_S,T dyn def}
\Rdyn_{S,T}(\lambda)=\bigl(j_S(\lambda)^{-1}\otimes j_T(\lambda-h_{\ul{S}})^{-1}\bigr)\circ R_{\cF^\str(S),\cF^\str(T)}(\lambda)\circ
\bigl(j_S(\lambda-h_{\ul{T}})\otimes j_T(\lambda)\bigr).
\end{equation}
\end{enumerate}
\end{proposition}
\begin{proof}
(1) A direct computation shows that the first two formulas are the identities in $\mathcal{N}_{\mr{fd}}$ underlying the dynamical hexagon identities \eqref{dynhexagon}, while \eqref{dynYBeqn} is the identity in $\mathcal{N}_{\mr{fd}}$ underlying the dynamical braid relation \eqref{dynbraidYB}.\\
(2) The morphism in $\Ndyn$ representing $\ol{c_{V,W}}\in\Hom_{\Ndyn^{\,\str}}((\ul{V},\ul{W}),(\ul{W},\ul{V}))$ is given by
\[
j_{(W,V)}^{-1}\circ\ul{c_{V,W}}\circ j_{(V,W)},
\]
i.e., it is the map $\mathfrak{h}_{\mr{reg}}^*\rightarrow\Hom_{\mathcal{N}_{\mr{fd}}}(\ul{V}\otimes\ul{W},\ul{W}\otimes\ul{V})$ defined by
\[
\lambda\mapsto j_{(W,V)}(\lambda)^{-1}\circ P_{V,W}\circ \mathcal{R}_{V,W}\circ j_{(V,W)}(\lambda).
\]
The result now follows immediately.\\
(3) The morphism in $\textup{Hom}_{\Ndyn}(\cF^\str(\ul{S}\tens\ul{T}),\cF^\str(\ul{T}\tens\ul{S}))$ representing $\ol{c_{S,T}}$ is 
\[
j_{T\tens S}^{-1}\circ P_{\cF^\str(\ul{S}),\cF^\str(\ul{T})}\circ\mathcal{R}_{\cF^\str(S),\cF^\str(T)}\circ j_{S\tens T},
\] 
where we have omitted the associators. In view of \eqref{hulp0}, this can be expressed in terms of the dynamical R-matrix as
\begin{equation}\label{intermediate}
\bigl(j_T^{-1}\ol{\otimes}j_S^{-1}\bigr)\circ P_{\cF^\str(\ul{S}),\cF^\str(\ul{T})}\circ R_{\cF^\str(S),\cF^\str(T)}\circ
\bigl(j_S\ol{\otimes}j_T\bigr).
\end{equation}
Thus composing \eqref{intermediate} on the left with $P_{\cF^\str(\ul{T}),\cF^\str(\ul{S})}$ yields $R_{S,T}$.
Evaluating the resulting expression at $\lambda\in\mathfrak{h}^*_{\textup{reg}}$ yields \eqref{R_S,T dyn def}.
\end{proof}

\begin{remark}\label{smfRemark}
\hfill
\begin{enumerate}
\item $\cF^{\mr{cst}}: \cN_{\mr{fd}}\rightarrow\Ndyn$ is a strict monoidal functor.
\item  
In \cite{Etingof&Varchenko-2000} the {\it exchange matrices} are denoted by $R_{V,W}(\lambda)$. In the present notation it corresponds to 
	\[
	R^{21}_{V,W}(\lambda) := P_{W,V}\Rdyn_{W,V}(\lambda)P_{V,W}.
	\]
We will moreover write for $S,S'\in\Rep^\str$,
\[
R_{S,S^\prime}^{21}(\lambda):=P_{\cF^\str(\ul{S'}),\cF^\str(\ul{T'})}R_{S',S}(\lambda)P_{\cF^\str(\ul{S}),\cF^\str(\ul{S'})}
\]	
\textup{(}the justification for this notation comes from the existence of a universal dynamical $R$-matrix, see \cite{Etingof&Varchenko-2000} and Subsection \ref{Subsection Q(lambda)}\textup{)}.
In \cite[Theorem 14]{Etingof&Varchenko-1999}, direct algebraic computations are used to show that the exchange matrices satisfy dynamical hexagon and dynamical quantum Yang--Baxter equations \textup{(}see also \cite[Chapter 6]{Etingof&Latour-2005}\textup{)}. 
\end{enumerate}
\end{remark}
Applying the dynamical twist functor to the identities 
guaranteeing that $(e_S)_{S\in\Rep^\str}$ and $(\iota_S)_{S\in\Rep^\str}$ define a left duality for $\Rep^\str$, gives
\begin{equation}\label{dynlefteval}
(\textup{id}_{\ul{S}}\ol{\tens}\,\ol{e_S})\circ(\ol{\iota_S}\,\ol{\tens}\textup{id}_{\ul{S}})=\textup{id}_{\ul{S}},\qquad\qquad
(\ol{e_S}\,\ol{\tens}\textup{id}_{\ul{S^*}})\circ(\textup{id}_{\ul{S^*}}\ol{\tens}\,\ol{\iota_S})=\textup{id}_{\ul{S^*}}.
\end{equation}
Similarly, the properties that turn $(\vartheta_S)_{S\in\Rep^\str}$ into a twist for the braided monoidal category $\Rep^\str$ with left duality, have as dynamical analogs
\begin{equation}\label{dynnattwist}
\ol{A}\circ\ol{\vartheta_S}=\ol{\vartheta_T}\circ\ol{A}
\end{equation}
for $A\in\textup{Hom}_{\Rep^\str}(S,T)$, and
\begin{equation}\label{dyntensdualtwist}
\begin{split}
\ol{\vartheta_{S\tens T}}&=(\ol{\vartheta_S}\,\ol{\tens}\,\ol{\vartheta_T})\circ \ol{c_{T,S}}\circ\ol{c_{S,T}},\\
\ol{\vartheta_{S^*}}&=( \ol{e_S}\,\ol{\tens}\textup{id}_{\ul{S^*}})\circ(\textup{id}_{\ul{S^*}}\ol{\tens}\,\ol{\vartheta_S}\,\ol{\tens}\textup{id}_{\ul{S^*}})\circ(\textup{id}_{\ul{S^*}}\ol{\tens}\,\ol{\iota_S}).
\end{split}
\end{equation}
The right (co-)evaluation maps $(\widetilde{e}_S)_{S}$, $(\widetilde{\iota}_S)_{S}$ of the resulting ribbon category $\Rep^\str$ have as dynamical counterparts
\begin{equation}\label{dynrighteval}
\begin{split}
\ol{\widetilde{e}_S}&=\widetilde{\cF^{\mr{EV}}}(\widetilde{e}_S)\circ J_{\ul{S\tens S^*}}^{-1}\circ \ol{J_{S\tens S^*}},\\
\ol{\widetilde{\iota_S}}&=\ol{J_{S^*\tens S}}^{\,-1}\circ J_{\ul{S^*\tens S}}\circ\widetilde{\cF^{\mr{EV}}}(\widetilde{\iota}_S).
\end{split}
\end{equation}
They satisfy the properties
\begin{equation}\label{dynrightstraight}
\bigl(\ol{\widetilde{e_S}}\,\ol{\tens}\textup{id}_{\ul{S}}\bigr)\circ \bigl(\textup{id}_{\ul{S}}\ol{\tens}\,\ol{\widetilde{\iota_S}}\bigr)=\textup{id}_{\ul{S}},\qquad\qquad
\bigl(\textup{id}_{\ul{S^*}}\ol{\tens}\,\ol{\widetilde{e_S}}\bigr)\circ\bigl(\ol{\widetilde{\iota_S}}\,\ol{\tens}\textup{id}_{\ul{S^*}}\bigr)=\textup{id}_{\ul{S^*}}.
\end{equation}
The dynamical counterparts of the identities relating the right (co-)evaluation to the left (co-) evaluation via the twist are given by
\begin{equation}\label{dynevtwist}
\ol{\widetilde{\iota_S}}=(\textup{id}_{\ul{S^*}}\,\ol{\tens}\,\ol{\vartheta_S})\circ\ol{c_{S,S^*}}\circ\ol{\iota_S}, \qquad\qquad
\ol{\widetilde{e_S}}=\ol{e_S}\circ\ol{c_{S,S^*}}\circ(\ol{\vartheta_S}\,\ol{\tens}\textup{id}_{\ul{S^*}}).
\end{equation}
Finally, the identities involving (co-)evaluation and braiding morphisms in \(\Rep^\str\) obtained from \cite[Section 3.5]{Turaev-1994} translate to the following properties under application of the dynamical twist functor:
\begin{align}
\label{from Turaev identities}
\begin{split}
(\id_{\ul{S}}\,\ol{\tens}\, \ol{e_T})(\ol{c_{T^\ast,S}}\,\ol{\tens}\, \id_{\ul{T}}) =\ & (\ol{e_T}\,\ol{\tens}\,\id_{\ul{S}})(\id_{\ul{T^\ast}}\,\ol{\tens}\, \ol{c^{-1}_{T,S}}), \\
(\id_{\ul{S}}\,\ol{\tens}\, \ol{e_T})(\ol{c^{-1}_{S,T^\ast}}\,\ol{\tens}\, \id_{\ul{T}}) =\ & (\ol{e_T}\,\ol{\tens}\,\id_{\ul{S}})(\id_{\ul{T^\ast}}\,\ol{\tens}\, \ol{c_{S,T}}),\\
(\ol{c_{S,T}}\,\ol{\tens}\, \id_{\ul{T^\ast}})(\id_{\ul{S}}\,\ol{\tens}\, \ol{\iota_T}) =\ & (\id_{\ul{T}}\,\ol{\tens}\, \ol{c^{-1}_{S,T^\ast}})(\ol{\iota_T}\,\ol{\tens}\,\id_{\ul{S}}),\\
(\ol{c^{-1}_{T,S}}\,\ol{\tens}\, \id_{\ul{T^\ast}})(\id_{\ul{S}}\,\ol{\tens}\, \ol{\iota_T}) =\ & (\id_{\ul{T}}\,\ol{\tens}\, \ol{c_{T^\ast,S}})(\ol{\iota_T}\,\ol{\tens}\,\id_{\ul{S}}).
\end{split}
\end{align}
We leave it to the reader to write out the explicit identities in $\mathcal{N}_{\mr{fd}}$ underlying \eqref{dynlefteval}, \eqref{dyntensdualtwist},
\eqref{dynrightstraight}, \eqref{dynevtwist} and (\ref{from Turaev identities}).

\subsection{Graphical calculus for the dynamical module category}
\label{Section strictified dynamical module category}
Building on the terminology from Subsection \ref{gcfirst}, we call a ribbon-braid graph a \emph{forest graph} in case its isotopy class has a representative whose ribbon-braid graph diagram does not contain any crossings (and hence, in particular, contains no full twists). A \((k,\ell)\)-forest graph is a forest graph with \(k\) bottom bases and \(\ell\) top bases. 

Let \(\cD = (\cD,\otimes,\mathbb{1},a,\ell,r)\) be a monoidal category. Then we denote by \(\mathbb{G}_{\cD^\str}\) the category whose objects are the tuples \((S_1,\dots,S_k)\) with \(S_i\in\cD^\str\), including \(\emptyset\), and where for \(S_1,\dots,S_k,T_1,\dots,T_\ell\in\cD^\str\), the class of morphisms \(\Hom_{\mathbb{G}_{\cD^\str}}((S_1,\dots,S_k),(T_1,\dots,T_\ell))\) consists of the isotopy classes of \((k,\ell)\)-forest graphs colored by morphisms 
\[
A\in\Hom_{\cD^\str}(S_1\tens\dots\tens S_k,T_1\tens\dots\tens T_\ell),
\]
in the sense of \cite[Section 2.4]{DeClercq&Reshetikhin&Stokman-2022}. It is a strict monoidal category with the monoidal structure defined in the usual manner, cf. Subsection \ref{gcfirst}. 
In analogy with \cite[Theorem 2.2]{DeClercq&Reshetikhin&Stokman-2022}, one can then construct a unique strict monoidal functor
\[
\cF^{\mr{forest}}_{\cD^\str}: \mathbb{G}_{\cD^\str}\to \cD^\str
\]
which maps an object \((S)\in\mathbb{G}_{\cD^\str}\) of length 1 to \(S\in\cD^\str\), and any \(\cD^\str\)-colored coupon to its color. We will write \(L\dot{=} L'\) for 
two morphisms $L$ and $L^\prime$ in \(\mathbb{G}_{\cD^\str}\) if \(\cF_{\cD^\str}^{\mr{forest}}(L) = \cF_{\cD^\str}^{\mr{forest}}(L')\).

When working with \(\cD = \Ndyn\), we will depict \(\cD^\str\)-colored coupons in \(\mathbb{G}_{\cD^\str}\) in blue, such that the coupon colored by \(A\in\Hom_{\Ndyn^\str}(S_1\ol{\tens}\dots\ol{\tens} S_k,T_1\ol{\tens}\dots\ol{\tens} T_\ell)\) is graphically represented by Figure \ref{morphism dynamical}. 
We also introduce Figure \ref{morphism dynamical in lambda} for the coupon colored by 
\[
A(\lambda)\in\Hom_{\cN_{\mr{fd}}^\str}(S_1\tens \dots\tens S_k,T_1\tens\dots\tens T_\ell)
\]
with \(\lambda\in\hh_{\mathrm{reg}}^\ast\), in the graphical calculus of the symmetric monoidal category \(\cN_{\mr{fd}}\). Here the weight $\lambda$ is assigned to the whole region right of the diagram. 
\begin{figure}[H]
	\begin{minipage}{0.48\textwidth}
		\centering
		\includegraphics[scale = 1.2]{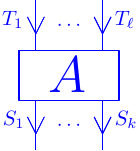}
		\captionof{figure}{}
		\label{morphism dynamical}
	\end{minipage}\quad
	\begin{minipage}{0.48\textwidth}
		\centering
		\includegraphics[scale = 1.2]{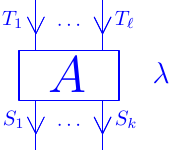}
		\captionof{figure}{}
		\label{morphism dynamical in lambda}
	\end{minipage}	
\end{figure}
In order to make the latter notation compatible with the monoidal structures, we need to impose a rule specifying how to transport the weight $\lambda$ over the vertical (bundles of) blue strands to other regions in between vertical strands and to the left of the diagram. The rule is that crossing a bundle of vertical blue strands labeled by $S$ from right to left shifts the weight $\lambda$ to the weight $\lambda-\mh_{S}$. 
Here $\mh_{S}$ denotes the weight shift operator measuring the grading of the $\mathfrak{h}^*$-graded vector space $\cF^\str(S)$. More precisely yet, if $S=(V_1,\ldots,V_k)$ and if we attach to the bundle of $k$ strands colored by \(V_1,\dots,V_k\) coupons labeled by $\alpha_{v_1},\ldots,\alpha_{v_k}$ with $v_i\in V_i[\nu_i]$, then crossing the bundle from right to left is shifting the weight from $\lambda$ to $\lambda-\sum_{i=1}^k\nu_i$,
\begin{figure}[H]
\centering
	\includegraphics[scale=0.9]{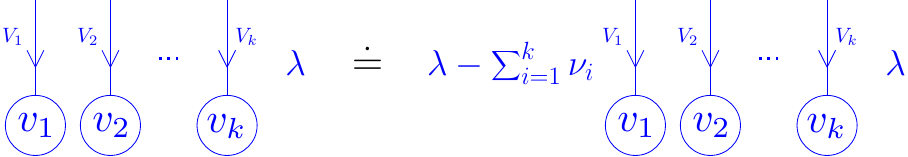}
	\captionof{figure}{}
\label{evgraph}
\end{figure}
With this convention of transporting the weight to other vertical regions, we have for \(A_i\in\Hom_{\Ndyn^\str}(S_i,T_i)\),
\begin{center}
	\includegraphics[scale=1]{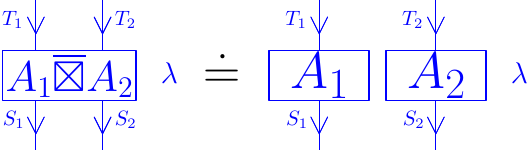}
\end{center}
in the graphical calculus for $\cN_{\mr{fd}}$.

Since \(\Ndyn\) is merely monoidal but not a ribbon category, its graphical calculus is restricted to the use of forest graphs in \(\mathbb{G}_{\Ndyn^\str}\) as shown in Figure \ref{morphism dynamical}, and height-preserving ambient isotopies. In the upcoming Section \ref{sectiongi} we will demonstrate how a nontrivial ribbon graph type of calculus can be obtained for $\Ndyn$, based on the results of Section \ref{sectiondb}.

Note that the diagrams with specialized weight as shown in Figure \ref{morphism dynamical in lambda} can be used inside the Reshetikhin-Turaev graphical calculus for the
symmetric tensor category $\cN_{\mr{fd}}$. 
This graphical calculus includes the oriented $\cN_{\mr{fd}}^\str$-colored ribbon graph diagrams capturing the topological crossings and the maxima and minima of $\cN_{\mr{fd}}^\str$-colored ribbon graphs, as displayed below.
\begin{center}
	\includegraphics[scale=0.8]{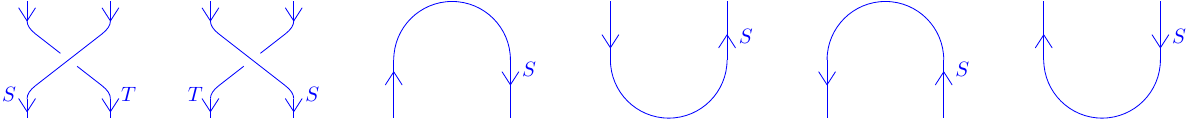}
\end{center}

Let $S=(V_1,\ldots,V_k)\in\Ndyn^{\,\str}$. 
Following the notations from Subsection \ref{gcfirst}, we denote the coupons in \(\mathbb{G}_{\Ndyn^{\,\str}}\) labeled by the fusion morphisms \(J_S\in\textup{Hom}_{\Ndyn^{\,\str}}(S,\cF^\str(S))\) and \(J_S^{-1}\in\Hom_{\Ndyn^{\,\str}}(\cF^\str(S),S)\) by Figures \ref{blue fusion} and \ref{blue antifusion} respectively. 
And following the notations from Subsection \ref{GcSection}, we denote the coupon in $\mathbb{G}_{\Ndyn^{\,\str}}$ labeled by the flip morphism $P_{S,T}\in \Hom_{\Ndyn^{\,\str}}(S\tens T,T\tens S)$
by Figure \ref{PST}.
\begin{figure}[H]
	\begin{minipage}{0.32\textwidth}
		\centering
		\includegraphics[scale = 0.8]{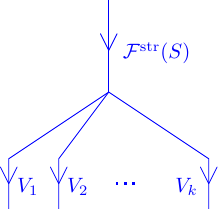}
		\captionof{figure}{}
		\label{blue fusion}
	\end{minipage}
	\begin{minipage}{0.32\textwidth}
		\centering
		\includegraphics[scale = 0.8]{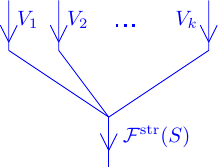}
		\captionof{figure}{}
		\label{blue antifusion}
	\end{minipage}
	\begin{minipage}{0.32\textwidth}
		\centering
		\includegraphics[scale=0.7]{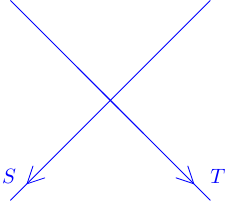}
		\captionof{figure}{}
		\label{PST}
	\end{minipage}
\end{figure}
\subsection{Graphical implementation of the dynamical ribbon structures}\label{sectiongi}
We continue the development of our graphical framework by introducing graphical notations for the dynamical ribbon structures in $\Ndyn^\str$ arising from the dynamical twist functor, as introduced algebraically in
Section \ref{sectiondb}. The basic building blocks of this graphical calculus are forest graphs in $\mathbb{G}_{\Ndyn^{\,\str}}$ with strands colored by objects $\ul{S}\in\Ndyn^\str$ arising from objects $S\in\Rep^\str$. We will label these strands simply by \(S\) instead of \(\ul{S}\) for ease of notation.

Given a morphism \(A\in\Hom_{\Rep^\str}(S,T)\), there are two natural ways to transform \(A\) into a morphism in \(\Ndyn^\str\) that can serve as the color of a coupon in \(\mathbb{G}_{\Ndyn^\str}\), namely by applying either \(\widetilde{\cF^{\mr{EV}}}\) or \(\cF^{\mr{dt}}\) to \(A\). The coupon in \(\mathbb{G}_{\Ndyn^\str}\) colored by \(\widetilde{\cF^\mr{EV}}(A)\in\textup{Hom}_{\Ndyn^{\,\str}}(\ul{S},\ul{T})\) will be denoted by the left-hand side of Figure \ref{D11}. On the other hand, we introduce the left-hand side of Figure \ref{A(lambda)} as a graphical notation for the coupon in \(\mathbb{G}_{\Ndyn^\str}\) colored by the dynamical twist 
\[
\ol{A}=\cF^{\mr{dt}}(A)\in\textup{Hom}_{\Ndyn^{\,\str}}(\ul{S},\ul{T})
\] 
of \(A\). It is hence the red or blue color of the coupon that determines the interpretation of the coloring morphism.

\begin{figure}[H]
	\begin{minipage}{0.48\textwidth}
		\centering
		\includegraphics[scale = 1]{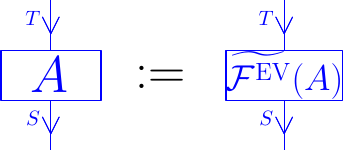}
		\captionof{figure}{}
		\label{D11}
	\end{minipage}\quad
	\begin{minipage}{0.48\textwidth}
		\centering
		\includegraphics[scale=1]{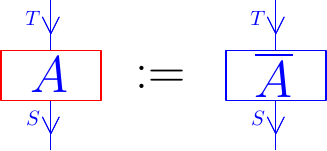}
		\captionof{figure}{}
		\label{A(lambda)}
	\end{minipage}
\end{figure}

The fact that $\cF^{\mr{dt}}$ is a strict monoidal functor is reflected by Figures \ref{D4A}--\ref{D4B}.

\begin{figure}[H]
	\begin{minipage}{0.43\textwidth}
		\centering
		\includegraphics[scale=1]{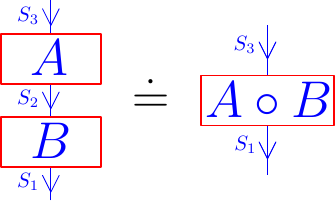}
		\captionof{figure}{}
		\label{D4A}
	\end{minipage}\quad
	\begin{minipage}{0.53\textwidth}
		\centering
		\includegraphics[scale=1]{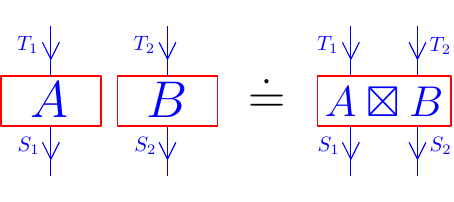}
		\captionof{figure}{}
		\label{D4B}
	\end{minipage}
\end{figure}

In analogy to the special notation for coupons labeled by fusion morphisms $J_{\ul{S}}$ (see Section \ref{Section strictified dynamical module category})
we shrink the coupons in \(\mathbb{G}_{\Ndyn^\str}\) colored by the dynamical fusion morphism \(\ol{J_S}\) and its inverse \(\ol{J_S}^{\,-1}\) to a red dot. 
So the coupons in \(\mathbb{G}_{\Ndyn^\str}\) labeled by the dynamical fusion morphisms \(\ol{J_S}\) and \(\ol{J_S}^{-1}\) are represented by Figures \ref{j dynamical}--\ref{j inverse dynamical}. 
\begin{figure}[H]
	\begin{minipage}{0.3\textwidth}
		\centering
		\includegraphics[scale=0.8]{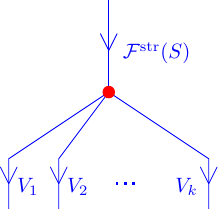}
		\captionof{figure}{}
		\label{j dynamical}
	\end{minipage}\quad
	\begin{minipage}{0.3\textwidth}
		\centering
		\includegraphics[scale=0.8]{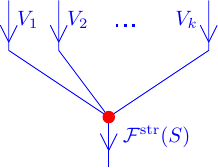}
		\captionof{figure}{}
		\label{j inverse dynamical}
	\end{minipage}
\end{figure}
\noindent
Note that the $\cF_{\mathbb{G}_{\Ndyn^\str}}^{\mr{forest}}$-images of 
these coupons are each other's inverse in view of Figure \ref{D4A}. Through the concept of bundling the strands, as established in \cite[Section 2.8]{DeClercq&Reshetikhin&Stokman-2022}, Figures \ref{j dynamical} and \ref{j inverse dynamical} can alternatively be expressed as Figures \ref{j dynamical bundled} and \ref{j inverse dynamical bundled}, respectively.
\begin{figure}[H]
	\begin{minipage}{0.4\textwidth}
		\centering
		\includegraphics[scale=0.9]{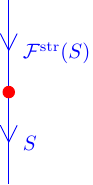}
		\captionof{figure}{}
		\label{j dynamical bundled}
	\end{minipage}\quad
	\begin{minipage}{0.4\textwidth}
		\centering
		\includegraphics[scale=0.9]{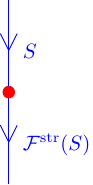}
		\captionof{figure}{}
		\label{j inverse dynamical bundled}
	\end{minipage}
\end{figure}
\noindent
Observe that Lemma \ref{coco} now admits the following graphical reformulation.
\begin{lemma}
	For \(S, T\in\Rep^\str\) of length $>0$ we have
	\begin{center}
		\includegraphics[scale=0.7]{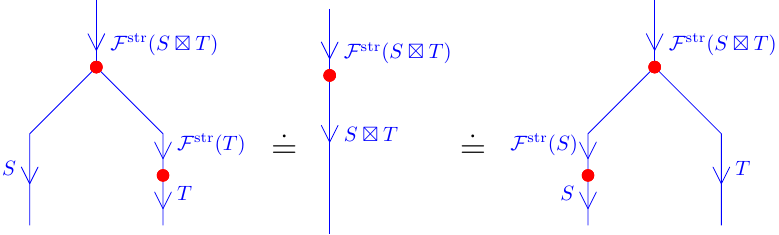}
	\end{center}
\end{lemma}
The explicit expression for $\ol{A}$ in terms of (dynamical) fusion operators obtained in Proposition \ref{olAexplicit}, with \(A\in\Hom_{\Rep^\str}(S,T)\), \(S = (V_1,\dots,V_k)\) and \(T = (W_1,\dots,W_\ell)\), can now be expressed graphically as in Figure \ref{D3}.

\begin{figure}[H]
	\centering
	\includegraphics[scale=0.8]{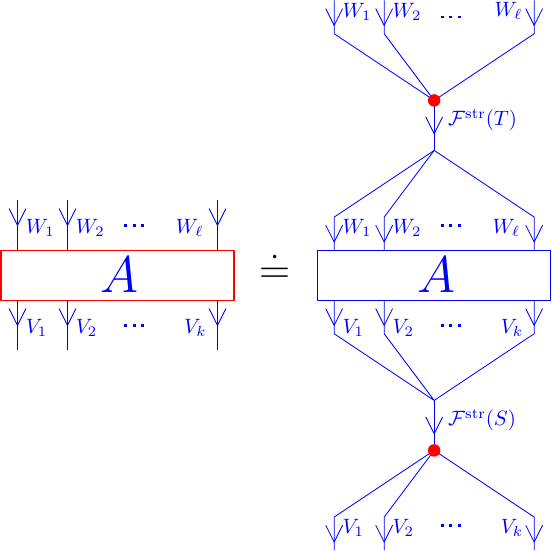}
	\caption{}
	\label{D3}
\end{figure}

As a next step we introduce Figures \ref{dynamical R-matrix}--\ref{co-injection dynamical} for the coupons in $\mathbb{G}_{\Ndyn^{\,\str}}$ colored by the dynamical braiding and dynamical (co-)evaluation morphisms.
\begin{figure}[h]
	\begin{minipage}{0.48\textwidth}
		\centering
		\includegraphics[scale=0.9]{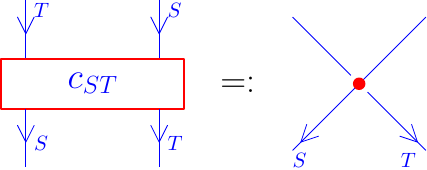}
		\captionof{figure}{}
		\label{dynamical R-matrix}
	\end{minipage}\quad
	\begin{minipage}{0.48\textwidth}
		\centering
		\includegraphics[scale=0.9]{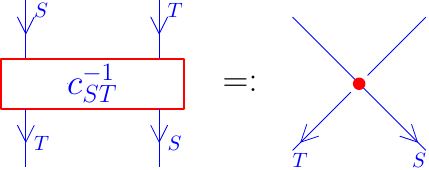}
		\captionof{figure}{}
		\label{dynamical R-matrix inverse}
	\end{minipage}
\end{figure}
\begin{figure}[H]
	\begin{minipage}{0.48\textwidth}
		\centering
		\includegraphics[scale = 0.9]{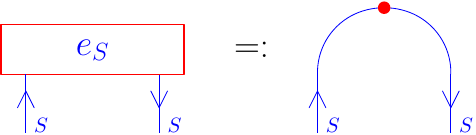}
		\captionof{figure}{}
		\label{evaluation dynamical}
	\end{minipage}\quad
	\begin{minipage}{0.48\textwidth}
		\centering
		\includegraphics[scale = 0.9]{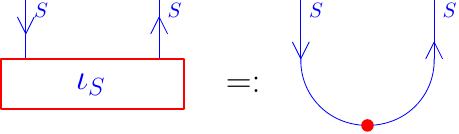}
		\captionof{figure}{}
		\label{injection dynamical}
	\end{minipage}
\end{figure}

\begin{figure}[H]
	\begin{minipage}{0.48\textwidth}
		\centering
		\includegraphics[scale = 0.9]{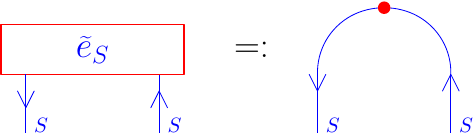}
		\captionof{figure}{}
		\label{co-evaluation dynamical}
	\end{minipage}\quad
	\begin{minipage}{0.48\textwidth}
		\centering
		\includegraphics[scale = 0.9]{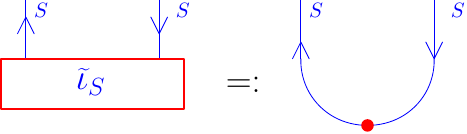}
		\captionof{figure}{}
		\label{co-injection dynamical}
	\end{minipage}
\end{figure}
The dynamical versions of the fundamental properties of the braiding and (co-)evaluation of $\Rep^\str$, as formulated in Section \ref{sectiondb}, now admit natural graphical representations in the graphical calculus for $\Ndyn^{\,\str}$.
For instance, \eqref{dynnatural} is graphically represented by Figure \ref{naturality dynamical braiding}.
\begin{figure}[H]
	\centering
	\includegraphics[scale=0.9]{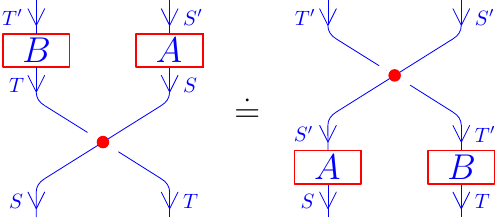}
	\caption{}
	\label{naturality dynamical braiding}
\end{figure}
\noindent
The dynamical hexagon identities \eqref{dynhexagon} are represented by Figures \ref{hexagon dynamical A}--\ref{hexagon dynamical B}.
\begin{figure}[H]
	\begin{minipage}{0.48\textwidth}
		\centering
		\includegraphics[scale=0.65]{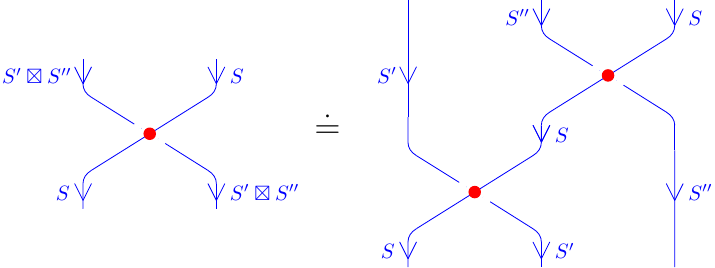}
		\captionof{figure}{}
		\label{hexagon dynamical A}
	\end{minipage}\quad
	\begin{minipage}{0.48\textwidth}
		\centering
		\includegraphics[scale=0.65]{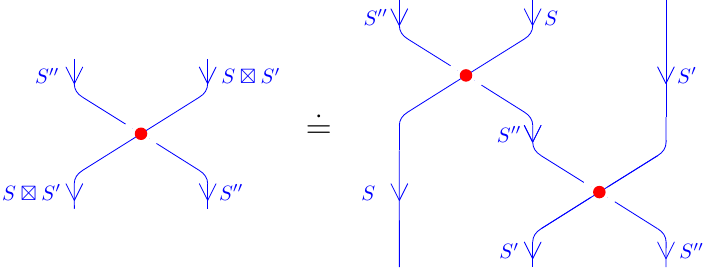}
		\captionof{figure}{}
		\label{hexagon dynamical B}
	\end{minipage}
\end{figure}

The dynamical braid form \eqref{dynbraidYB} of the quantum Yang-Baxter equation is represented by Figure \ref{D5}.

\begin{figure}[H]
	\centering
	\includegraphics[scale=0.8]{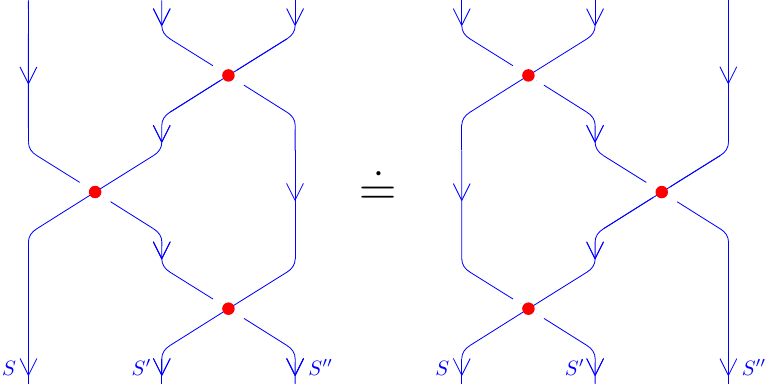}
	\caption{}
	\label{D5}
\end{figure}

The identities \eqref{dynlefteval} and \eqref{dynrightstraight} for the dynamical (co-)evaluations are graphically represented by Figures \ref{duality A}--\ref{duality B}.
\begin{figure}[H]
	\begin{minipage}{0.48\textwidth}
		\centering
		\includegraphics[scale = 0.9]{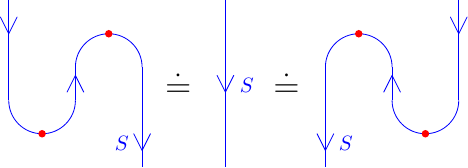}
		\captionof{figure}{}
		\label{duality A}
	\end{minipage}\quad
	\begin{minipage}{0.48\textwidth}
		\centering
		\includegraphics[scale = 0.9]{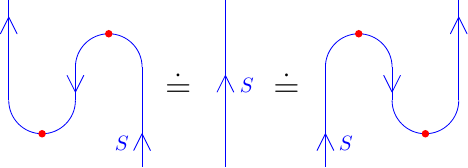}
		\captionof{figure}{}
		\label{duality B}
	\end{minipage}
\end{figure}

Finally, the graphical reformulations of the identities (\ref{from Turaev identities}) relating the dynamical braiding to dynamical (co-)evaluation morphisms are displayed in Figures \ref{braided evaluation A}--\ref{braided evaluation D}.

\begin{figure}[H]
	\begin{minipage}{0.48\textwidth}
		\centering
		\includegraphics[scale = 0.75]{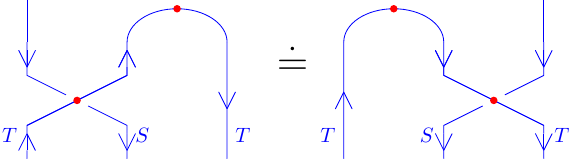}
		\captionof{figure}{}
		\label{braided evaluation A}
	\end{minipage}\quad
	\begin{minipage}{0.48\textwidth}
		\centering
		\includegraphics[scale = 0.75]{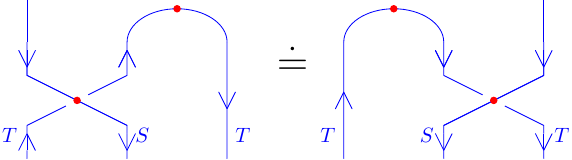}
		\captionof{figure}{}
		\label{braided evaluation B}
	\end{minipage}
\end{figure}

\begin{figure}[H]
	\begin{minipage}{0.48\textwidth}
		\centering
		\includegraphics[scale = 0.75]{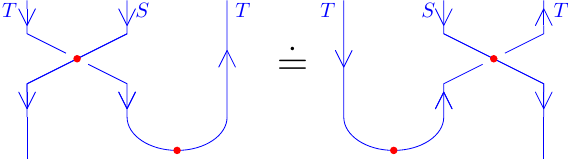}
		\captionof{figure}{}
		\label{braided evaluation C}
	\end{minipage}\quad
	\begin{minipage}{0.48\textwidth}
		\centering
		\includegraphics[scale = 0.75]{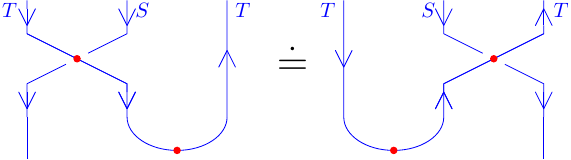}
		\captionof{figure}{}
		\label{braided evaluation D}
	\end{minipage}
\end{figure}

\subsection{Quantum vertex operators as interface between the graphical calculi}
\label{Section Interface}

Theorems \ref{theorem dynamical twist} \& \ref{prop pushed under water} allow to transport computations in the strict ribbon category $\Rep^\str$ to the strict monoidal category $\Ndyn^\str$. In this subsection we incorporate this into the graphical calculus. In this subsection we assume $\lambda\in\mathfrak{h}_{\textup{reg}}^*$ unless stated explicitly otherwise.

Recall that for $S=(V_1,\ldots,V_k)$ and $v_j\in V_j[\nu_j]$ we have introduced Figure \ref{vertex operator} as the notation for the coupon in $\mathbb{B}_{\mathcal{M}_{\textup{adm}}^\str}$ colored by the $k$-point quantum vertex operator
$\Phi_\lambda^{v_1,\ldots,v_k}\in\textup{Hom}_{\mathcal{M}_{\textup{adm}}^\str}(M_\lambda, M_{\lambda_0}\tens S)$, where $\lambda_j:=\lambda-\nu_{j+1}-\cdots-\nu_k$. 
We extend this notation by writing Figure \ref{universal intertwiner bundled} for the coupon with strands both in $\mathbb{B}_{\mathcal{M}_{\textup{adm}}^\str}$
and in $\mathbb{G}_{\Ndyn^\str}$, whose color is represented by the isomorphism
\begin{equation}\label{isoSint}
\mathcal{X}_{\lambda,S}: 
\cF^\str(\ul{S})\overset{\sim}{\longrightarrow} \mr{Int}_{\lambda,\cF^\str(S)}=\bigoplus_{\mu\in\hh^\ast}\textup{Hom}_{\mathcal{M}_{\mr{adm}}}(M_\lambda,M_{\lambda-\mu}\otimes
\cF^\str(S))
\end{equation}
of finite-dimensional $\mathfrak{h}^*$-graded vector spaces, mapping $v_1\otimes\cdots\otimes v_k$ to the $k$-point quantum vertex operator $\phi_\lambda^{v_1,\ldots,v_k}$ 
representing $\Phi_\lambda^{v_1,\ldots,v_k}$. Upon viewing the $\ul{S}$-colored and $S$-colored strands as bundles of \(k\) strands colored by the underlying non-strict monoidal categories $\Rep$ and $\Ndyn$, this coupon will be denoted by Figure \ref{universal intertwiner}. 
Note that these notations are consistent with the graphical notations for the coupon colored by $\Phi_\lambda^{v_1,\ldots,v_k}$ 
(Figure \ref{vertex operator}) and for coupons labeled by tensor products of evaluation morphisms (Figure \ref{evgraph}), since 
\[
\phi_\lambda^{v_1,\ldots,v_k}=\mathcal{X}_{\lambda,S}(v_1\otimes\dots\otimes v_k).
\]

In Figures \ref{universal intertwiner bundled} \& \ref{universal intertwiner}, the weight of the region appearing right of the blue strand(s) has a twofold interpretation. On the one hand it indicates the specialization of the blue diagram to a diagram with coupons colored by morphisms from $\mathcal{N}_{\mr{fd}}^\str$, following the rule of coloring adjacent regions with shifted weights as established in Section \ref{Section strictified dynamical module category}. On the other hand it indicates the highest weight of the Verma module that is being attached as color to the red strand that bounds the above-mentioned vertical region. 
In particular, the omitted weights in Figure \ref{universal intertwiner} are (from left to right) the shifted weights \(\lambda-\mh_{(\ul{V_1},\dots,\ul{V_k})}, \lambda-\mh_{(\ul{V_2},\dots,\ul{V_k})}, \dots, \lambda-\mh_{(\ul{V_k})}\), in the sense of (\ref{shifted weights}). This is consistent with the composition rule 
\[
\Phi_\lambda^{v_1,\ldots,v_k,v_1',\ldots,v_\ell'}=(\Phi_{\lambda-\nu_1'\cdots-\nu_\ell'}^{v_1,\ldots,v_k}\tens\textup{id}_{S'})\Phi_\lambda^{v_1',\ldots,v_\ell'}
\]
for quantum vertex operators, where $S^\prime=(V_1^\prime,\ldots,V_\ell^\prime)\in\Rep^\str$ and $v_i'\in V_i'[\nu_i']$. 

\begin{figure}[H]
	\begin{minipage}{0.48\textwidth}
		\centering
		\includegraphics[scale=0.7]{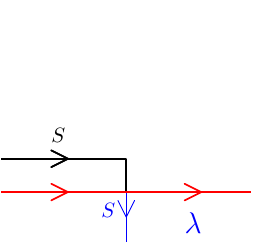}
		\captionof{figure}{}
		\label{universal intertwiner bundled}
	\end{minipage}
	\begin{minipage}{0.48\textwidth}
		\centering
		\includegraphics[scale = 0.7]{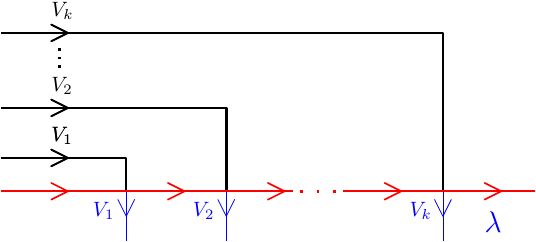}
		\captionof{figure}{}
		\label{universal intertwiner}
	\end{minipage}
\end{figure}
\noindent

It is often convenient to regard the red $\mathcal{M}_{\textup{adm}}^\str$-colored strands in Figures \ref{universal intertwiner bundled} \& \ref{universal intertwiner}
as a fixed border line between the two topological monoidal categories $\textup{Rib}_{\Rep^\str}$ and $\mathbb{G}_{\Ndyn^\str}$.
{}From this viewpoint the open black $S$-colored strand lies in $\textup{Rib}_{\Rep^\str}$, 
the open blue $\underline{S}$-colored strand lies in $\mathbb{G}_{\Ndyn^\str}$, and the isotopies in these two topological categories are supposed to fix the red borderline.
Note that the red borderline and the open black $S$-colored strand may also be viewed as being part of the topological monoidal category $\mathbb{B}_{\mathcal{M}_{\textup{adm}}^\str}$ upon adding evaluation coupons to the blue strands (Figure \ref{evgraph}). This allows for topological interplay in $\mathbb{B}_{\mathcal{M}_{\textup{adm}}^\str}$ between the red border line and the black strands. 

Recall that the dynamical twist functor $\cF^{\mr{dt}}: \Rep^\str\rightarrow\Ndyn^{\,\str}$ is a strict monoidal functor which encodes how the action of morphisms on the spin space $S\in\Rep^\str$ of a quantum vertex operator
$\Phi\in\textup{Hom}_{\mathcal{M}_{\mr{adm}}^\str}(M_\lambda,M_{\lambda-\mu}\tens S)$ can be described on the level of the parametrizing space $\ul{S}$ of $\Phi$.
This leads to the following graphical recipe how coupons in $\textup{Rib}_{\Rep^\str}$ can be moved to $\mathbb{G}_{\Ndyn^\str}$ through the red border line.
Recall the notation defined in Figure \ref{A(lambda)} for the $\ol{A}$-colored coupon in $\mathbb{G}_{\Ndyn^{\,\str}}$, where
$A\in\Hom_{\Rep^\str}(S,T)$.
\begin{proposition}\label{pushdiagram}
	Let $S_1,S_2,S_3,T_2\in\Rep^\str$ and $A\in\textup{Hom}_{\Rep^\str}(S_2,T_2)$. Then we have
	\begin{figure}[H]
		\centering
		\includegraphics[scale=0.8]{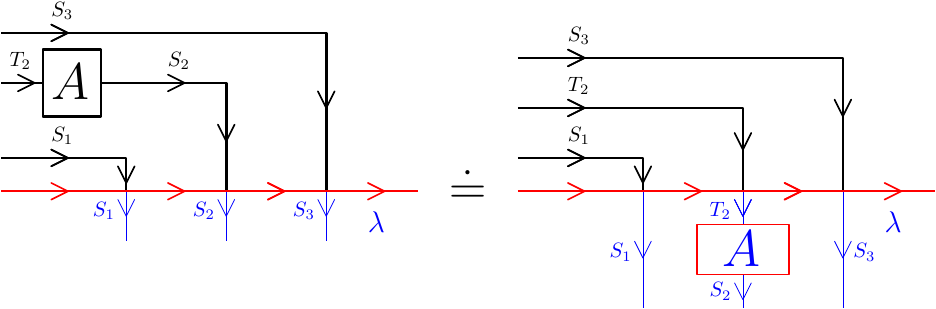}
		\caption{}
		\label{D2}
	\end{figure}
\noindent
meaning that adding to the blue strands in Figure \ref{D2} coupons colored by evaluation maps \textup{(}Figure \ref{evgraph}\textup{)}, the identity holds true in the graphical calculus for $\mathcal{M}_{\textup{adm}}^\str$.
\end{proposition}
\begin{proof}
	For $S_1=\emptyset=S_3$ this is the graphical representation of Theorem \ref{prop pushed under water}. 
	
	In the general case, we may replace in the left-hand side of Figure \ref{D2} the coupon colored by $A$ with the coupon attached to all spin strands and colored by $\textup{id}_{S_1}\tens A\tens\textup{id}_{S_3}$, in which case we have just seen that Figure \ref{D2} applies to push the coupon to the blue strands.
	The result now follows by observing that the resulting coupon in $\mathbb{G}_{\Ndyn^\str}$ is colored by
	\[\mathcal{F}^{\textup{dt}}(\textup{id}_{S_1}\tens A\tens\textup{id}_{S_3})=
	\textup{id}_{\ul{S_1}}\ol{\tens}\,\ol{A}\,\ol{\tens}\textup{id}_{\ul{S_3}}
	\]
	since $\cF^{\mr{dt}}$ is strict monoidal. 
\end{proof}

In particular, the order in which two coupons in $\textup{Rib}_{\Rep^\str}$, involving disjoint spin bundles, are moved to the blue strands, is irrelevant. We also have the following two direct consequences of Proposition \ref{pushdiagram}.

\begin{corollary}
	\label{thm dynamical twist graphical}
	For \(A'\in\Hom_{\Rep^\str}((V),(W))\), with \(V, W\in \Rep\), we have
	\begin{figure}[H]
		\centering
		\includegraphics[scale=0.7]{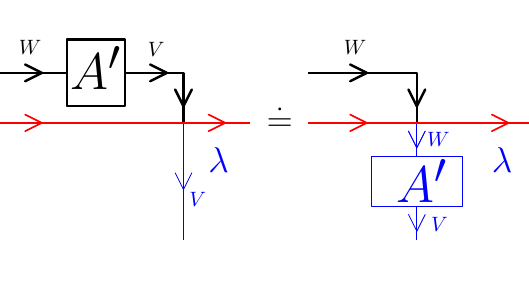}
		\caption{}
		\label{dyn twist 1}
	\end{figure}
	and for \(S= (V_1,\dots,V_k)\in\Rep^\str\) we have 
	\begin{figure}[H]
		\begin{minipage}{0.46\textwidth}
			\centering
			\includegraphics[scale=0.5]{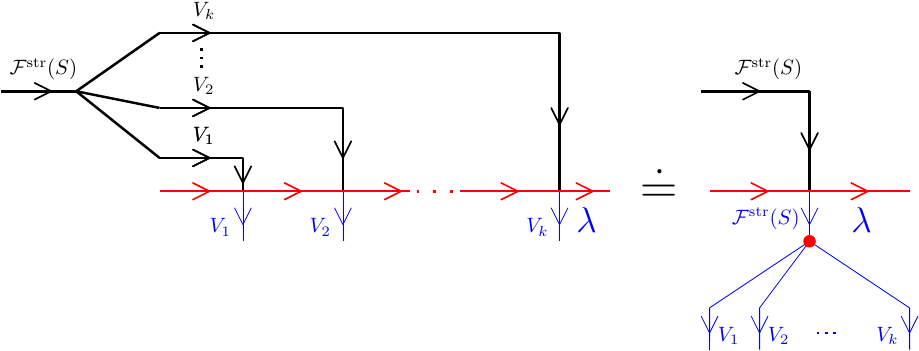}
			\caption{  }
			\label{descent equation J}
		\end{minipage}\quad\quad
		\begin{minipage}{0.46\textwidth}
			\centering
			\includegraphics[scale=0.5]{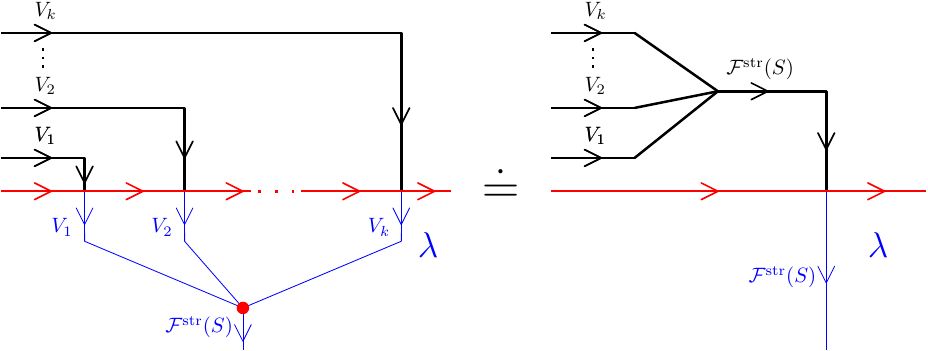}
			\captionof{figure}{}
			\label{descent equation J inverse}
		\end{minipage}
	\end{figure}
\end{corollary}
\begin{proof}
	Figure \ref{dyn twist 1} is a consequence of Corollary \ref{cor dyn twist}(\ref{cor_item3}). Figures \ref{descent equation J} and \ref{descent equation J inverse} follow from Lemma \ref{jcomp}, but they can also be seen as the special cases of Proposition \ref{pushdiagram} when the coupons are colored by the morphisms $J_S$ and $J_S^{-1}$, resepctively.
\end{proof}

Let $S=(V_1,\ldots,V_k)\in\Rep^\str$. Recall that the $k$-point quantum vertex operator $\Phi_\lambda^{v_1,\ldots,v_k}$, depending on a fixed $k$-tuple $(v_1,\ldots,v_k)$ of weight vectors $v_i\in V_i[\nu_i]$, is the unique morphism in $\Hom_{\cM_{\mr{adm}}^\str}(M_\lambda,M_{\lambda-\nu}\tens S)$ representing
$\phi_\lambda^{j_S(\lambda)(v_1\otimes\cdots\otimes v_k)}$.
 It is convenient to extend this notation of $k$-point quantum vertex operator to morphisms in $\Hom_{\cM_{\mr{adm}}^\str}(M_\lambda,M_{\lambda-\nu}\tens S)$ 
 representing $\phi_\lambda^v$ for any weight vector $v$ in $V_1\otimes\cdots\otimes V_k$ .
\begin{definition}\label{reversevertex}
	Let $S=(V_1,\ldots,V_k)\in\Rep^{\str}$ 
	and consider $v\in\cF^\str(S)$ a homogeneous vector of weight $\nu$. 
	We write $\Phi_{\lambda;S}^v$ for the morphism in $\Hom_{\cM_{\mr{adm}}^\str}(M_\lambda,M_{\lambda-\nu}\tens S)$ representing 
$\phi_\lambda^v\in\Hom_{\cM_{\mr{adm}}}(M_\lambda,M_{\lambda-\nu}\otimes\cF^\str(S))$.
\end{definition}
The morphism $\Phi_{\lambda;S}^v$ can then be expressed as sum of $k$-point quantum vertex operators in the following way.
\begin{lemma}\label{reversevertexlemma}
	Let $S=(V_1,\ldots,V_k)\in\Rep^{\str}$ 
	and let $v\in\cF^\str(S)$ be a homogeneous vector of weight $\nu$. Then
	\begin{equation}\label{isk}
	\Phi_{\lambda;S}^v=\sum_n\Phi_\lambda^{w_1^n,\dots,w_k^n}
	\end{equation}
	where \(w_i^n\in V_i\) are homogeneous vectors such that
	\[
	j_S^{-1}(\lambda)(v) = \sum_n w_1^n\otimes\dots\otimes w_k^n.
	\]
\end{lemma}
\begin{proof}
The right hand side of \eqref{isk} is represented by the morphism
\[
\sum_n\phi_\lambda^{w_1^n,\dots,w_k^n}\in\Hom_{\cM_{\mr{adm}}}(M_\lambda,M_{\lambda-\nu}\otimes\cF^\str(S)),
\]
whose expectation value is
\[
\sum_n\langle\phi_\lambda^{w_1^n,\ldots,w_k^n}\rangle=\sum_nj_S(\lambda)(w_1^n\otimes\cdots\otimes w_k^n)=v.
\]
So both sides of \eqref{isk} are morphisms in $\Hom_{\cM_{\mr{adm}}^\str}(M_\lambda,M_{\lambda-\nu}\tens S)$ representing 
$\phi_\lambda^v$.
Hence they are equal.
\end{proof}

A natural further extension of $\Phi_{\lambda;S}^v$ is as follows. For \(S_1,\dots,S_r\in\Rep^\str\) and \(w_i\in\cF^\str(S_i)[\nu_i]\), we set
\begin{align}
\label{reversevertex eq}
\begin{split}
\Phi_{\lambda;S_1,\dots,S_r}^{w_1,\dots,w_r}&\ :=\left(\Phi_{\lambda_1;S_1}^{w_1}\tens\id_{S_2\tens\dots\tens S_r}\right)\circ\dots\circ\left(\Phi_{\lambda_{r-1};S_{r-1}}^{w_{r-1}}\tens\id_{S_r}\right)\circ\Phi_{\lambda_r;S_r}^{w_r}\\&\ \in\Hom_{\cM_{\mr{adm}}^\str}(M_\lambda,M_{\lambda_0}\tens S_1\tens\dots\tens S_r),
\end{split}
\end{align}
where we have written \(\lambda_r:=\lambda\) and \(\lambda_j := \lambda-\nu_{j+1}-\dots-\nu_r\) for any \(j=0,\dots,r-1\). 
In Section \ref{Section dual q-KZB}, we will use these notations and the identity (\ref{large intertwiner}) in the special case where \(r = 3\) and \(\ell_1+\ell_2+\ell_3 = k\), as a way of parametrizing a \(k\)-point quantum vertex operator by only 3 vectors \(w_1\), \(w_2\) and \(w_3\). We will refer to this approach as \emph{reduction to the threefold case}.
\begin{remark}
Note that 
\begin{equation}
\label{large intertwiner}
\Phi_{\lambda;S_1,\dots,S_r}^{w_1,\dots,w_r} = \Phi_\lambda^{v_1^1,\dots,v_{\ell_1}^1,v_1^2,\dots,v_{\ell_2}^2,\dots,v_1^r,\dots,v_{\ell_r}^r}
\quad \hbox{ if }\,\, w_i = j_{S_i}(\lambda_i)(v_1^{i}\otimes\dots\otimes v_{\ell_i}^{i})
\end{equation}
where \(S_i = (V_1^{i},\dots,V_{\ell_i}^{i})\) and \(v_{j}^{i}\in V_j^{i}\) \textup{(}$1\leq j\leq\ell_i$\textup{)} are homogeneous vectors with their weights adding up to $\nu_i$.
\end{remark}

In order to use these extensions of $k$-point quantum vertex operators graphically, we introduce the new notation depicted in Figure \ref{evaluation in S} for $v\in\cF^\str(S)$.
Note that for objects \(S = (V)\) of length 1 we have \(j_{(V)} = \id_{(V)}\) and the red dot may be omitted, in which case it matches with the earlier notation for coupons colored by evaluation morphisms (see Figure \ref{evgraph}). Then Figure \ref{universal intertwiner huge} provides a graphical presentation for \(\Phi_{\lambda;S_1,\dots,S_r}^{w_1,\dots,w_r}\), which is consistent with our prior graphical notation, see Figure \ref{vertex operator}, when all the $S_i$'s are objects of length one.

\begin{figure}[H]
	\begin{minipage}{0.36\textwidth}
		\centering
		\includegraphics[scale = 0.9]{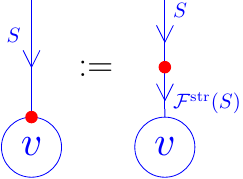}
		\captionof{figure}{}
		\label{evaluation in S}
	\end{minipage}\quad
	\begin{minipage}{0.6\textwidth}
		\centering
		\includegraphics[scale=0.7]{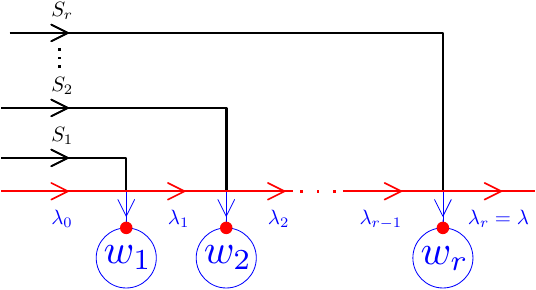}
		\captionof{figure}{}
		\label{universal intertwiner huge}
	\end{minipage}
\end{figure}

\subsection{Dynamical twisting endomorphisms}
\label{Subsection Q(lambda)}
For any \(S\in\Rep^\str\), let us denote by \(\QQ_S\) the endomorphism in \(\End_{\Ndyn^\str}(\ul{S})\) defined by
\begin{equation}
\label{QQ_S lambda def}
\QQ_S(\lambda) := (\ol{\widetilde{e}_S}(\lambda)\tens\id_{\ul{S}})\left(\id_{\ul{S}}\tens\widetilde{\iota}_S\right),\qquad \lambda\in\mathfrak{h}_{\textup{reg}}^*,
\end{equation}
where we slightly abuse notation by writing $\widetilde{\iota}_S$ for $\widetilde{\cF^{\rm{EV}}}(\widetilde{\iota}_S)$ (a convention which we will adopt from now on for any morphism in $\Rep^\str$). We furthermore write 
\[
\mathbb{Q}_V:=\mathcal{F}^\str_{\Ndyn}\bigl(\QQ_{(V)}\bigr)\in\End_{\Ndyn}(\ul{V}).
\]

In view of Figure \ref{co-evaluation dynamical} and \eqref{dualities_B}, definition (\ref{QQ_S lambda def}) can be graphically represented as 
\begin{center}
	\includegraphics[scale=0.75]{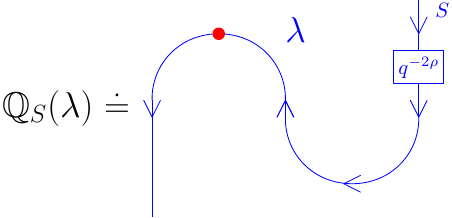}
\end{center}
in the graphical calculus of $\mathcal{N}_{\mr{fd}}^\str$, where the left-hand side stands for the coupon colored by $\mathbb{Q}_S(\lambda)$. 

Note that upon pushing the right duality axiom in $\Rep^\str$ to $\Ndyn^{\,\str}$ through the strict monoidal lift $\widetilde{\cF^{\rm{EV}}}$ of $\cF^{\rm{EV}}$, we obtain the identity 
\begin{equation}
\label{criterion Q}
\ol{\widetilde{e}_S}(\lambda)=(\ol{\widetilde{e}_S}(\lambda)\tens\widetilde{e}_S)(\id_{\ul{S}}\tens\widetilde{\iota}_S\tens\id_{\ul{S^\ast}})=
\widetilde{e}_S(\QQ_S(\lambda)\,\tens\,\id_{\ul{S^\ast}}).
\end{equation} 
This shows that the endomorphism $\mathbb{Q}_S$ encodes the difference between the evaluation morphism $\widetilde{e}_S$ and its dynamical twist $\ol{\widetilde{e}_S}$.  

By \cite[Theorem 5.8]{Etingof&Latour-2005} there exists for $\lambda\in\mathfrak{g}_{\textup{reg}}^*$ a unique solution $\mathbb{J}(\lambda)$ in $\mathcal{U}^{(2)}$ to the Arnaudon-Buffenoir-Ragoucy-Roche \cite{Arnaudon&Buffenoir&Ragoucy&Roche-1998} (ABRR) equation of the form 
\[
\mathbb{J}(\lambda) = \sum_{\beta\in Q^+}\mathbb{J}_\beta(\lambda)\in \mathcal{U}^{(2)}
\]
with \(\mathbb{J}_0(\lambda) = 1\otimes 1\) and \(\mathbb{J}_\beta(\lambda)\in U^-[-\beta]\otimes U^+[\beta]\) (we use here the notations from \cite[Section 1.4]{DeClercq&Reshetikhin&Stokman-2022}). It is invertible and 
\begin{equation}\label{relJj}
(\pi_V\otimes\pi_W)(\mathbb{J}(\lambda))=j_{(V,W)}(\lambda)
\end{equation}
as endomorphisms in $\textup{End}_{\mathcal{N}_{\mr{fd}}}(\ul{V}\otimes\ul{W})\) 
for any \(V,W\in\Rep\).  The element $\mathbb{J}(\lambda)$ is referred to as the {\it universal dynamical fusion matrix}. 

Etingof \& Varchenko \cite{Etingof&Varchenko-2000} introduced the element
\begin{equation}
\label{Q algebraic def}
\QQ(\lambda) := m^{\mr{op}}((\id\otimes S^{-1})\mathbb{J}(\lambda))
\end{equation}
where \(m^{\mr{op}}(a\otimes b) = ba\), which naturally acts on any finite dimensional $\mathcal{U}_q$-representation $V\in\Rep$. The resulting endomorphism
$\mathbb{Q}(\lambda)\vert_V$ is in 
$\textup{End}_{\mathcal{N}_{\textup{fd}}}(\ul{V})$, i.e., it preserves weight spaces. The resulting function $\lambda\mapsto \mathbb{Q}(\lambda)\vert_V$ defines a morphism $\mathbb{Q}(\cdot)\vert_V\in\textup{End}_{\Ndyn}(\ul{V})$.

\begin{proposition}
	\label{prop interpretation of Q} 
	$\mathbb{Q}_V=\mathbb{Q}(\cdot)\vert_V$ for all $V\in\Rep$.
\end{proposition}
\begin{proof}
	By (\ref{QQ_S lambda def}) and (\ref{criterion Q}), in order to prove that \(\QQ_V(\lambda)\) and \(\QQ(\lambda)\vert_V\) coincide, it suffices to show that
	\[
	\ol{\widetilde{e}_V}(\lambda) = \widetilde{e}_V(\QQ(\lambda)\vert_V\otimes \id_{\ul{V^\ast}}),
	\]
	or in other words, that
	\begin{equation}
	\label{to be proven Q}
	\ol{\widetilde{e}_{V}}(\lambda)(v\otimes f) = f(q^{2\rho}\QQ(\lambda)\cdot v)
	\end{equation}
	for any \(v\in V\) and \(f\in V^\ast\).
	Upon writing
	\[
	\mathbb{J}(\lambda) = \sum\nolimits_{\beta\in Q^+} X_\beta\otimes Y_\beta,
	\]
	with \(X_\beta\in U^-[-\beta]\) and \(Y_\beta\in U^+[\beta]\), the right hand side of \eqref{to be proven Q} equals
	\[
\sum\nolimits_{\beta\in Q^+}f\bigl(q^{2\rho}S^{-1}(Y_\beta)X_\beta\cdot v\bigr),
	\]
	whereas by \eqref{dynrighteval} and \eqref{relJj}
	the left-hand side of (\ref{to be proven Q}) becomes
	\[
	\sum\nolimits_{\beta\in Q^+}(Y_\beta\cdot f)(q^{2\rho}X_\beta\cdot v) =
	 \sum\nolimits_{\beta\in Q^+}f\left(S(Y_\beta)q^{2\rho}X_\beta\cdot v\right).
	\]
	The result hence follows straightforwardly from the property \(q^{2\rho}X = S^{2}(X)q^{2\rho}\), valid for any \(X\in U_q(\g)\).
\end{proof}

The concept of the universal dynamical fusion operator \(\mathbb{J}(\lambda)\) naturally leads to the notion of the {\it universal dynamical R-matrix}, defined by
\begin{equation}
\label{universal dyn R-matrix def}
R(\lambda):= (\mathbb{J}^{21}(\lambda))^{-1}\,\cR\,\mathbb{J}(\lambda),\qquad \lambda\in\mathfrak{h}_{\textup{reg}}^*
\end{equation}
and satisfying
\begin{equation}\label{relRr}
\bigl(\pi_V\otimes\pi_W\bigr)(R(\lambda))=\Rdyn_{V,W}(\lambda)
\end{equation}
as endomorphisms in $\textup{End}_{\mathcal{N}_{\textup{fd}}}(\underline{V}\otimes\underline{W})$ for all $V,W\in\Rep$.

\begin{remark}\label{notconv}
The notational conventions in this subsection deviate from those of Etingof and Varchenko. More precisely, in \cite{Etingof&Varchenko-2000} the notation $\mathbb{J}(\lambda)$ is used for the universal dynamical fusion matrix with argument $-\lambda-\rho$, 
and the notation \(\QQ(\lambda)\) is used for the element (\ref{Q algebraic def}) with $\lambda$ replaced by $-\lambda-\rho$.
 \end{remark}

\begin{proposition}
	\label{prop Q inverse} 
	 \(\QQ_S\in\End_{\Ndyn^\str}(\ul{S})\) is an automorphism, and for $V\in\Rep$ and $\lambda\in\mathfrak{h}_{\textup{reg}}^*$ we have
	\[
	\QQ_V(\lambda)^{-1}\vert_{V[\nu]}=
	m^{\mr{op}}\bigl((S^{-1}\otimes\id)\mathbb{J}(\lambda+\nu)^{-1}\bigr)\vert_{V[\nu]}.
	\] 
\end{proposition}
\begin{proof}
	For any \(S\in \Rep^\str\), upon combining Figure \ref{duality A} with the relations of \cite[Rel\(_1\) and Rel\(_4\)]{Reshetikhin&Turaev-1990} within the ribbon graph calculus for \(\cN_{\mr{fd}}^\str\), one finds that
	\begin{center}
		\includegraphics[scale=0.8]{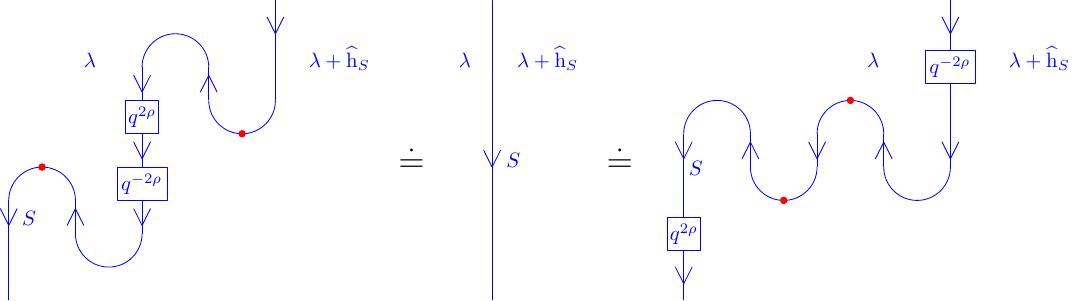}
	\end{center}
	where $\widehat{\mh}_S$ denotes the forward weight shift (i.e.\ if we attach coupons labeled by $\alpha_{v_i}$, for weight vectors $v_i\in V_i[\nu_i]$ ($1\leq i\leq k$), to the bundle of strands colored by $S=(V_1,\ldots,V_k)$, then the weight appearing in the region right of the bundle is $\lambda+\sum_{i=1}^k\nu_i$).
	
	This asserts that \(\QQ_S(\lambda)\) is invertible in \(\End_{\cN_{\mr{fd}}^\str}(\ul{S})\), with inverse
	\begin{center}
		\includegraphics[scale=0.75]{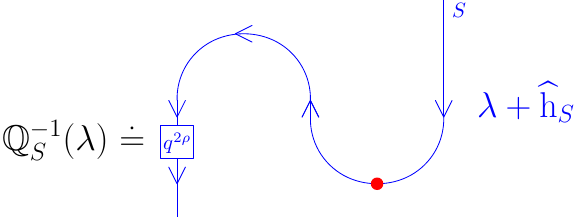}
	\end{center}
	The anticipated formula for \(\QQ_V(\lambda)^{-1}\vert_{V[\nu]}\) then follows from applying the Reshetikhin-Turaev functor \(\cF^{\mr{RT}}_{\cN_{\mr{fd}}^\str}\)
	to both sides 
	with \(S = (V)\) and using \eqref{dynrighteval}, (\ref{action of q^2rho}) and Proposition \ref{prop interpretation of Q}.
\end{proof}

 In the Appendix we will use the graphical calculus to obtain graphical proofs for various other identities involving $\mathbb{J}_S$ and $\mathbb{Q}_S$.

\section{Graphical derivation of dual \(q\)-KZB equations for weighted trace functions}
\label{Section dual q-KZB}

In this section we give a graphical proof of dual $q$-KZB type equations for weighted trace functions of \(k\)-point quantum vertex operators. We will show in
Section \ref{Section Etingof-Varchenko normalization} that these equations imply the dual $q$-KZB equations for weighted trace functions 
derived before by Etingof \& Varchenko \cite{Etingof&Varchenko-2000} through intricate algebraic computations. 

\subsection{Weighted trace functions}\label{Subsection Formal weighted trace functions}

Before introducing the weighted trace functions we need to introduce some more notations.

Suppose that $\phi: U\rightarrow U\otimes V$ is a linear map with $V$ finite dimensional, and fix a basis $\{u_i\}_i$ of $U$ with dual basis $\{u_i^*\}_i$. Suppose that
$\sum_i(u_i^*\otimes\psi)(\phi(u_i))$ absolutely converges for all $\psi\in V^*$. Then there exists a unique vector $\textup{Tr}_U(\phi)\in V$ such that
\[
\psi\bigl(\textup{Tr}_U(\phi)\bigr)=\sum_i(u_i^*\otimes\psi)(\phi(u_i))\qquad \forall\,\psi\in V^*,
\]
and we call $\textup{Tr}_U(\phi)\in V$ the partial trace of $\phi$ over $U$. It does not depend on the choice of basis $\{u_i\}_i$.

The real span $\mathfrak{h}_{\mathbb{R}}^*$  of the roots $\Phi$ is a real form of $\mathfrak{h}^*$. We write $\Re(\xi)\in\mathfrak{h}_{\mathbb{R}}^*$
for the real component of $\xi\in\mathfrak{h}^*$. The negative Weyl chamber is the region in $\mathfrak{h}_{\mathbb{R}}^*$
consisting of $\xi\in\mathfrak{h}_{\mathbb{R}}^*$ satisfying $\langle\xi,\alpha_i\rangle<0$ for $i=1,\ldots,r$. 

Fix \(\mu\in\hh_{\mathrm{reg}}^\ast\), \(S=(V_1,\ldots,V_k)\in\Rep^\str\) and \(\Phi\in\Hom_{\cM_{\mr{adm}}^\str}(M_\mu,M_\mu\tens S)\). From now on we denote the underlying morphism
$\cF^\str(\Phi)\in\Hom_{\cM_{\mr{adm}}}(M_\mu,M_\mu\otimes\cF^\str(S))$ again by $\Phi$ if it does not lead to confusion.
The $\cF^\str(S)[0]$-valued, weight trace function of $\Phi$ (cf. \cite{Etingof&Varchenko-2000}) is defined by
\begin{equation}
\label{trace formally}
H_\mu^\Phi(q^\xi):=\Tr_{M_\mu}(\Phi\circ\pi_\mu(q^\xi))=\sum_{\beta\in Q^+} \Tr_{M_\mu[\mu-\beta]}(\Phi)q^{\langle \mu-\beta,\xi\rangle},
\end{equation}
for $\Re(\xi)$ deep enough in the negative Weyl chamber. The absolute convergence of the series is guaranteed by \cite[Proposition 3.2]{Etingof&Styrkas-1998} (recall here that $0<q<1$).
In this paper we will always assume that \(\xi\) lies deep in the negative Weyl chamber when considering the trace function \(H_\mu^\Phi(q^\xi)\). 
We use Figure \ref{formal trace} as the graphical representation of $H_\mu^\Phi(q^\xi)$.

The formal interpretation of Figure \ref{formal trace} is as a coupon in the graphical calculus of $\mathcal{N}_{\textup{fd}}^\str$ colored by the morphism
$\mathcal{H}_\mu^\Phi(q^\xi)\in\Hom_{\cN_{\mr{fd}}^\str}(\CC_0,\ul{S})$ represented by the co-evaluation map 
\begin{equation}
\label{notation with curly H}
\alpha_{H_\mu^\Phi(q^\xi)}: \CC_0\to \cF^\str(\ul{S}),\qquad 1_0\mapsto H_\mu^\Phi(q^\xi)
\end{equation}
(cf. Subsection \ref{GcSection}).
But the notation Figure \ref{formal trace} purposely includes the structure of the colour of the coupon in terms of $\Phi$, since it 
 is still allowed to manipulate the black spin strands of $\Phi$ according to the graphical calculus of \(\Mfd^\str\) as long as this manipulation only involves the black strands of \(\Phi\). We will encounter an example of such a situation in the proof of the upcoming Proposition \ref{prop boundary qKZB}. 
 
The red strands in Figure \ref{formal trace} remember the fact that the color $H_\mu(q^\xi)$ is defined as weighted trace over $M_\mu$. One does need to be careful in applying graphical calculus to the red strands when it involves the red cups and caps, since they are colored by objects in a monoidal category without duality.

\begin{figure}[H]
	\centering
	\includegraphics[scale=1]{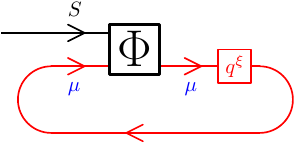}
	\caption{}
	\label{formal trace}
\end{figure}

Spin components of \(\mathcal{H}_\mu^\Phi(q^\xi)\) are obtained by pairing \(\cF^\str(\ul{S})[0]\) to a vector  
 \(f\in\cF^\str(\ul{S^\ast})[0]\).  In turn, $\cF^\str(\ul{S^\ast})[0]$ can be naturally parametrized by expectation values $\langle\Psi\rangle$ of morphisms
\(\Psi\in\Hom_{\cM_{\mr{adm}}^\str}(M_{\lambda},S^\ast\tens M_{\lambda})\) (cf. Remark \ref{remark dual quantum vertex operators}), where $\lambda\in\mathfrak{h}_{\textup{reg}}^*$ is an arbitrary second regular highest weight. We use Figure
\ref{spincomp} as the graphical representation of $\langle\Psi\rangle$.
\begin{figure}[H]
\centering
\includegraphics[scale=1]{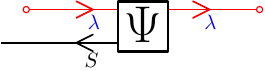}
\caption{}
\label{spincomp}
\end{figure}

\noindent The formal interpretation of Figure \ref{spincomp} is as a coupon in the graphical calculus of $\mathcal{N}_{\textup{fd}}^\str$
colored by the morphism $\mathbb{C}_0\rightarrow\ul{S^*}$ in $\mathcal{N}_{\textup{fd}}^\str$ represented by  
$\alpha_{\langle\Psi\rangle}$. But as is the case for Figure \ref{formal trace}, the notation purposely remembers the structure of the coupon as the expectation value of a dual quantum vertex operator because graphical calculus of $\Rep^\str$ may be applied to the black spin strands. Furthermore, graphical calculus of $\mathcal{M}_{\textup{adm}}^\str$
may be applied to the black and red strands as long as it does not involve the red circles (Figures \ref{eval_lambda} \& \ref{eval_lambda_dual}).
 
The $\langle\Psi\rangle$-spin component of $H_\mu^\Phi(q^\xi)$ is obtained by pairing $H_\mu^\Phi(q^\xi)\in\cF^\str(S)[0]$ to the vector
$\langle\Psi\rangle\in\cF^\str(S^*)[0]$ via the map 
representing the right evaluation morphism
\(\widehat{e}_{\ul{S}}: \ul{S}\tens \ul{S^\ast} \to \emptyset\) in \(\cN_{\mr{fd}}^\str\). 
Lifting it to the strictified category $\mathcal{N}_{\textup{fd}}^\str$ gives the morphism
\begin{equation}
\label{t functions def}
\mathfrak{t}^{\Phi,\Psi}_{\lambda,\mu,\xi}:=\widehat{e}_{\ul{S}}\left(\mathcal{H}_\mu^\Phi(q^\xi)\tens \langle \Psi\rangle\right)
\end{equation}
in $\mathcal{N}_{\textup{fd}}^\str$, where $\langle\Psi\rangle$ now stands for the morphism $\mathbb{C}_0\rightarrow \ul{S^*}$ represented by $\alpha_{\langle\Psi\rangle}$.
Since the map representing $\widehat{e}_{\ul{S}}$ is acting on a vector in $\cF^\str(S)[0]\otimes\cF^\str(S^*)[0]$, we may as well replaced it by the map 
representing the right evaluation morphism $\widetilde{e}_{S}: S\tens S^*\rightarrow\emptyset$ in $\Rep^\str$.
Hence $\mathfrak{t}^{\Phi,\Psi}_{\lambda,\mu,\xi}$ may be graphically represented by
\begin{center}
	\includegraphics[scale=1]{diagram_25b_formal_spin_component}
\end{center}
\noindent
relative to the graphical calculus for $\mathcal{M}_{\mr{adm}}^\str$ and $\Rep^\str$ (the refined graphical calculus for $\Rep^\str$ only applies to the part of the diagram involving solely black strands).

Dual $q$-KZB equations \cite{Etingof&Varchenko-2000} for $\mathfrak{t}^{\Phi,\Psi}_{\lambda,\mu,\xi}$ arise when $\Phi$ and $\Psi$ are (dual) $k$-point quantum vertex operators and 
\begin{equation}
\label{k-point case}
\xi = 2\lambda+2\rho. 
\end{equation}
They then take the form
\begin{equation}\label{dqKZB}
c_{\Psi,\lambda}^{(i)}\,\mathfrak{t}^{\Phi,\Psi}_{\lambda,\mu,2\lambda+2\rho} = \sum_{\sigma\in\wts(V_i)}\mathfrak{t}^{D_{\sigma,i}(\mu)(\Phi),\Psi}_{\lambda,\mu+\sigma,2\lambda+2\rho},\qquad i=1,\ldots,k
\end{equation}
for certain \(D_{\sigma,i}(\mu)\in\End_\CC(\Hom_{\cM_{\mr{adm}}^\str}(M_\mu,M_\mu\tens S)) \) and scalars \(c_{\Psi,\lambda}^{(i)}\). Identifying 
\[
\Hom_{\cM_{\mr{adm}}^\str}(M_\mu,M_\mu\tens S)\simeq\cF^\str(S)[0]
\]
using the expectation value map, the linear operators 
\(D_{\sigma,i}(\mu)\) admit an explicit expression
in terms of dynamical R-matrices. The goal of this section is to derive these dual $q$-KZB equations graphically.

To this end, we will consider any decomposition of \(S=(V_1,\ldots,V_k)\in\Rep^\str\) as three-fold tensor product
\begin{equation}\label{3decomp}
S = S_1\tens S_2\tens S_3
\end{equation}
as well as a corresponding decomposition of the morphism $\Phi$ (resp. $\Psi$) as composition of three morphisms creating the spin space components $S_1,S_2,S_3$ 
(resp. $S_3^*,S_2^*,S_1^*$). We then graphically derive a dual $q$-KZB type equation for $\mathfrak{t}^{\Phi,\Psi}_{\lambda,\mu,2\lambda+2\rho}$ relative to the spin component $S_2$. 
The dual $q$-KZB equations (\ref{dqKZB}) then correspond to the decompositions
\[
S_1 = (V_1,\dots,V_{i-1}), \quad S_2 = (V_i), \quad S_3 = (V_{i+1},\dots,V_k)
\]
of $S$ as three-fold tensor products.

The dual $q$-KZB type equation relative to the decomposition \eqref{3decomp} only requires the morphisms $\Phi\in\textup{Hom}_{\mathcal{M}_{\textup{adm}}^\str}(M_\mu,M_\mu\tens S)$ and $\Psi\in\textup{Hom}_{\mathcal{M}_{\textup{adm}}^\str}(M_\lambda,S^*\tens M_\lambda)$ to be of the form
\begin{equation}\label{Phi3}
\Phi := \Phi_{\mu;S_1,S_2,S_3}^{w_1,w_2,w_3}
\end{equation}
with  \(w_i\in\cF^\str(S_i)[\nu_i]\) such that \(\nu_1+\nu_2+\nu_3 = 0\) (see Definition \ref{reversevertex} and (\ref{reversevertex eq})), and 
 \begin{equation}
\label{bold Psi def}
\Psi:=\Psi^{g_3,g_2,g_1}_{\lambda;S_3^\ast,S_2^\ast,S_1^\ast}
\end{equation}
 with \(g_i\in \cF^\str(S_i^\ast)[\nu_i']\) and \(\nu_1'+\nu_2'+\nu_3'=0\), where
 \[
 \Psi^{g_3,g_2,g_1}_{\lambda;S_3^\ast,S_2^\ast,S_1^\ast}:=(\id_{S_3^\ast\tens S_2^\ast}\tens \Psi_{\lambda_1;S_1^*}^{g_1})(\id_{S_3^\ast}\tens \Psi_{\lambda_2;S_2^*}^{g_2})\Psi_{\lambda_3;S_3^*}^{g_3}
 \]
and
\begin{equation}
\label{choice of Psi_i}
\Psi_{\lambda_i;S_i^*}^{g_i}:= c_{S_i^\ast,M_{\lambda_{i-1}}}^{-1}\circ\Phi_{\lambda_i;S_i^\ast}^{g_i}\in\Hom_{\cM_{\mr{adm}}^\str}(M_{\lambda_i},S_i^\ast\tens M_{\lambda_{i-1}})
\end{equation}
with the highest weight shifts given by
\begin{equation}
\label{lambda_j def}
\lambda_0 = \lambda_3 := \lambda, \qquad \lambda_1 := \lambda-\nu_2'-\nu_3', \qquad \lambda_2:=\lambda-\nu_3'.
\end{equation}

Let us write \(\boldsymbol{S} = (S_1,S_2,S_3)\), \(\boldsymbol{w} = (w_1,w_2,w_3)\) and \(\bs{g} = (g_1, g_2, g_3)\), and introduce the notation
\begin{equation}
\label{Y def}
\mathcal{Y}_{\boldsymbol{S}}^{\bs{w},\bs{g}}(\lambda,\mu, \xi):=\mathfrak{t}^{\Phi,\Psi}_{\lambda,\mu,\xi}
\end{equation}
when $\Phi$ and $\Psi$ are given by \eqref{Phi3} and \eqref{bold Psi def}, respectively.
It is graphically represented by
\begin{figure}[H]
\centering
\includegraphics[scale=0.75]{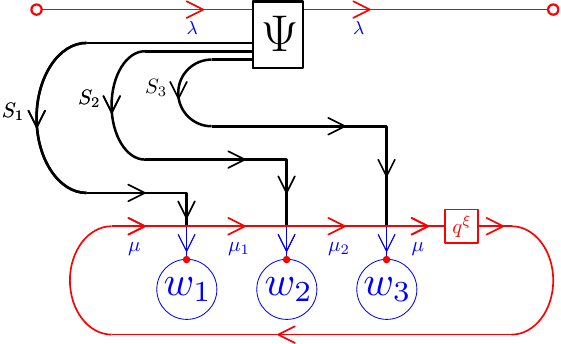}
\caption{}
\label{Ydiagram}
\end{figure}
\noindent relative to the graphical representation for $\cN_{\mr{adm}}^\str$ and $\cN_{\mr{fd}}^\str$,
where $\Psi:= \Psi^{g_3,g_2,g_1}_{\lambda;S_3^\ast,S_2^\ast,S_1^\ast}$ and 
\begin{equation}\label{mu_j def}
\mu_1:=\mu-\nu_2-\nu_3, \qquad \mu_2 := \mu -\nu_3
\end{equation}
(the $\mathbf{g}$-dependence of $\Psi=\Psi^{g_3,g_2,g_1}_{\lambda;S_3^\ast,S_2^\ast,S_1^\ast}$ (see \eqref{bold Psi def}) can also be made visible graphically, but for the moment we do not need it in the diagrams).

As a first step towards the graphical derivation of dual $q$-KZB type equations, we will consider the functions obtained from \(\mathcal{Y}_{\boldsymbol{S}}^{\boldsymbol{w},\bs{g}}(\lambda,\mu, \xi)\) by removing the weighted trace over the lower Verma strand, and the evaluation in a highest weight vector and dual vector in the upper Verma strand. We will refer to the latter operations as removing (or, for the inverse operations, imposing) boundary conditions, as justified by their interpretation as boundary conditions for the corresponding quantum spin chains (cf. \cite{Reshetikhin&Stokman-2020-B}). The resulting functions 
\begin{equation}
\label{T_S def}
\mathcal{T}_{\bs{S}}^{\bs{w},\bs{g}}(\lambda,\mu):= (\id_{M_{\mu}}\tens\widetilde{e}_S\tens\id_{M_{\lambda}})\left(\Phi_{\mu;S_1,S_2,S_3}^{w_1,w_2,w_3}\tens \Psi^{g_3,g_2,g_1}_{\lambda;S_3^\ast,S_2^\ast,S_1^\ast}\right)\in\End_{\cM_{\mr{adm}}^\str}(M_\mu\tens M_{\lambda})
\end{equation}
are graphically represented by
\begin{figure}[H]
\centering
	\includegraphics[scale = 0.75]{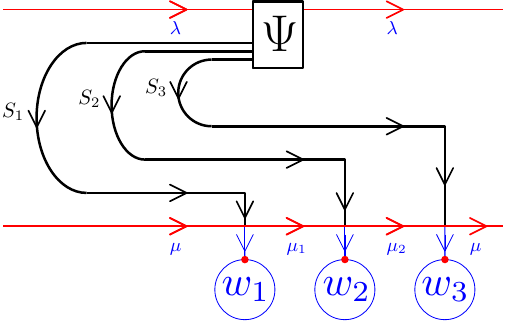}
\caption{}
\label{Tdiagram}
\end{figure}
\noindent
relative to the graphical calculus for $\mathcal{M}_{\mr{adm}}^\str$ and $\Rep^\str$.
The relation to the trace functions is
\begin{equation}
\label{Y in terms of T}
\mathcal{Y}_{\bs{S}}^{\bs{w},\bs{g}}(\lambda,\mu, \xi) =
 \bigl(\Tr_{M_\mu}\otimes\langle\cdot\rangle\bigr)\left(\mathcal{T}_{\bs{S}}^{\bs{w},\bs{g}}(\lambda,\mu)\circ(\pi_\mu(q^\xi)\otimes\id_{M_{\lambda}}) \right)
\end{equation}
on the level of the representing maps.

In the next subsection
we will first obtain operator dual \(q\)-KZB type equations for the functions \(\mathcal{T}_{\bs{S}}^{\bs{w},\bs{g}}(\lambda,\mu)\). We will then use it 
to derive in Section \ref{Subsection Twisted trace functions} dual $q$-KZB type equations in $\mu$ for the scalar trace functions 
\begin{equation}\label{Zdef}
\mathcal{Z}_{\boldsymbol{S}}^{\bs{w},\bs{g}}(\lambda,\mu):=
\mathcal{Y}_{\boldsymbol{S}}^{\boldsymbol{w},\bs{g}}(\lambda,\mu,2\lambda+2\rho)
\end{equation}
upon graphically implementing the boundary conditions (the specialisation of the weight $\xi$ to $2\lambda+2\rho$ will have a natural graphical interpretation).

\subsection{Operator \(q\)-KZB equations}
\label{Subsection Topological operator q-KZB}

Add to the dual quantum vertex operator $\Psi=\Psi_{\lambda; S_3^*,S_2^*,S_1^*}^{g_3,g_2,g_1}$, consisting of the composition of the three intertwiners $\Psi_i=\Psi_{\lambda_i;S_i^*}^{g_i}$ (see \eqref{choice of Psi_i}),
the action of the quantum Casimir \eqref{qC} and its inverse on the Verma modules \(M_{\lambda_1}\) and \(M_{\lambda_2}\), respectively.
Since the quantum Casimir acts as scalars on Verma modules (cf., e.g., \cite[Proposition 1.11]{DeClercq&Reshetikhin&Stokman-2022}), we obtain
\begin{center}
	\includegraphics[scale = 0.73]{diagram_20f_Casimirs}
\end{center}
relative to the graphical calculus for $\mathcal{M}_{\textup{adm}}^\str$. In the following lemma we move the twists in the Verma strands to the spin strands using the graphical calculus for $\mathcal{M}_{\textup{adm}}^\str$.
\begin{lemma}
	\label{lemma before topological qKZB}	
	For \(\lambda_0,\lambda_1,\lambda_2,\lambda_3\in\hh_{\mathrm{reg}}^\ast\), \(S_i\in\Rep^\str\) and \(\Psi_i\in \Hom_{\cM_{\mr{adm}}^\str}(M_{\lambda_i}, S_i^\ast\tens M_{\lambda_{i-1}})\), let us write \(S:=S_1\tens S_2\tens S_3\) and set
	\begin{equation}\label{Psiform}
	\Psi:=(\id_{S_3^\ast\tens S_2^\ast}\tens \Psi_1)(\id_{S_3^\ast}\tens \Psi_2)\Psi_3\in\Hom_{\cM_{\mr{adm}}^\str}(M_{\lambda_3},S^\ast\tens M_{\lambda_0}).
	\end{equation}
	Then we have
	\begin{figure}[H]
	\centering
		\includegraphics[scale = 0.75]{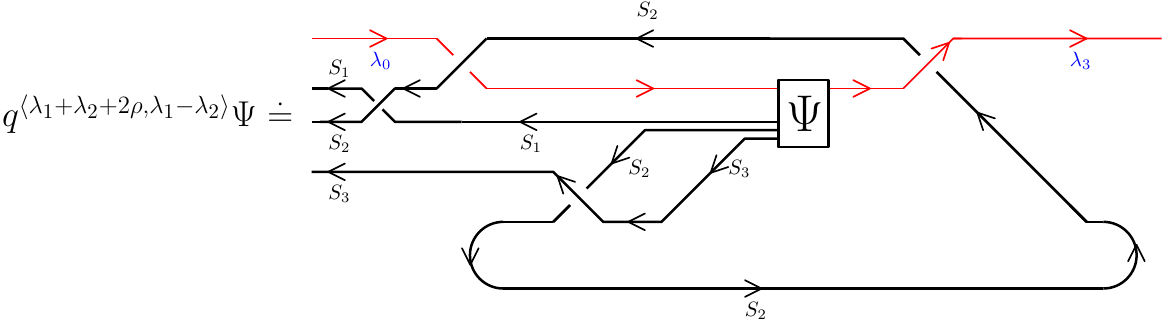}
		\caption{}
		\label{SS1}
	\end{figure}
\noindent
relative to the graphical calculus for $\mathcal{M}_{\mr{adm}}^\str$ and $\Rep^\str$ \textup{(}the left hand side should be considered as the coupon colored by $q^{\langle\lambda_1+\lambda_2+2\rho,\lambda_1-\lambda_2\rangle}\Psi$, and the refined graphical calculus for $\Rep^\str$ applies to the part of the diagram in the right hand side involving solely black strands\textup{)}.
\end{lemma}
\begin{proof}
	This follows in complete analogy to the graphical derivation of the operator \(q\)-KZB equations for 3-point quantum vertex operators given in \cite[Proposition 3.7]{DeClercq&Reshetikhin&Stokman-2022}.
\end{proof}

In particular, this result can be applied to the dual quantum vertex operator  
\(\Psi=\Psi^{g_3,g_2,g_1}_{\lambda;S_3^\ast,S_2^\ast,S_1^\ast}\) in the definition (\ref{T_S def}) of the functions \(\mathcal{T}_{\bs{S}}^{\bs{w},\bs{g}}(\lambda,\mu)\) (recall that $\bs{g}$ is parametrizing the dual quantum vertex operator $\Psi$, see \eqref{bold Psi def}--\eqref{lambda_j def}). The effect of the action of the quantum Casimir on the Verma modules labeled by \(\lambda_1\) and \(\lambda_2\) in $\Psi$, which we partially moved to the spin spaces in Lemma \ref{lemma before topological qKZB},
can then be moved further to the parametrizing space of the quantum vertex operator \(\Phi_{\mu;S_1,S_2,S_3}^{w_1,w_2,w_3}\) using Proposition \ref{pushdiagram}.

\begin{proposition}
	\label{prop bulk q-KZB}
	For \(\Psi=\Psi^{g_3,g_2,g_1}_{\lambda;S_3^\ast,S_2^\ast,S_1^\ast}\) we have 
	\begin{figure}[H]
		\centering
		\includegraphics[scale=0.75]{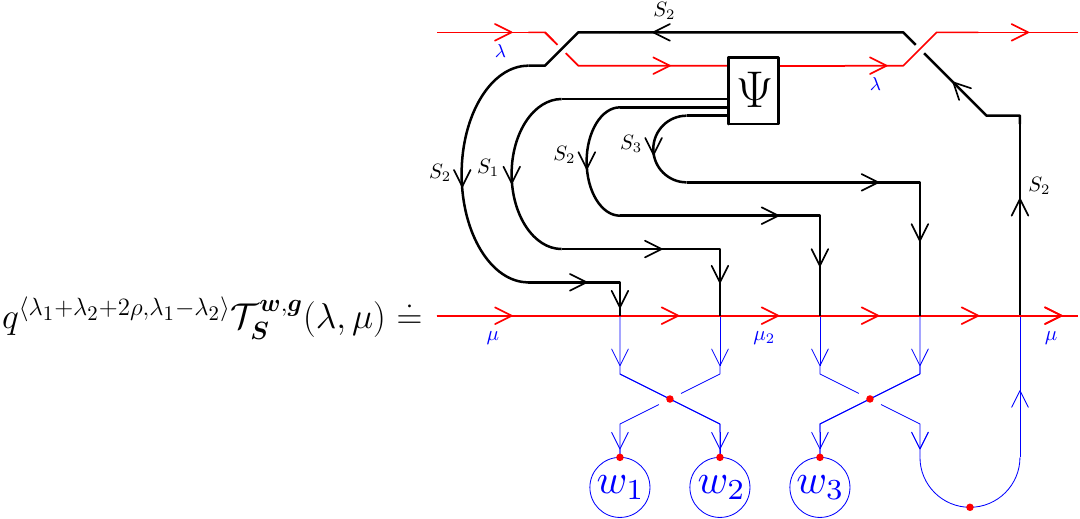}
		\caption{}
		\label{bulk q-KZB diagram}
	\end{figure}
\noindent
relative to the graphical calculi for $\mathcal{M}_{\textup{adm}}^\str$ \textup{(}on black and red strands\textup{)}, $\Rep^\str$ \textup{(}on solely black strands\textup{)}
and $\mathcal{N}_{\textup{fd}}^\str$ \textup{(}on blue strands\textup{)}.
\end{proposition}
\begin{proof}
	Lemma \ref{lemma before topological qKZB} asserts that the left-hand side of Figure \ref{bulk q-KZB diagram} may be replaced by
	\begin{center}
		\includegraphics[scale=0.75]{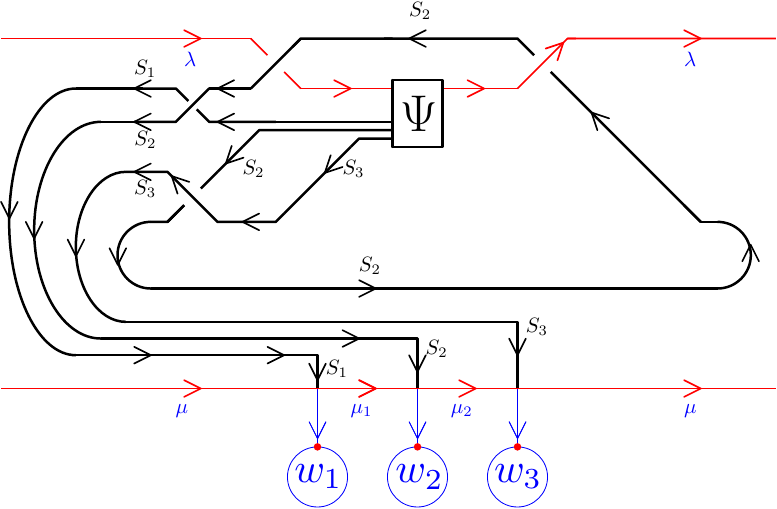}
	\end{center}
	Upon pushing the braidings through the caps representing \(\widetilde{e}_S\) and \(\widetilde{e}_{S_2}\), this may be replaced by
	\begin{center}
		\includegraphics[scale=0.75]{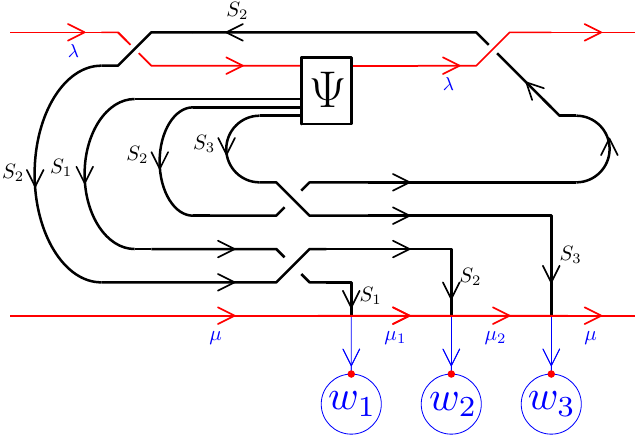}
	\end{center}
	Finally, upon pushing first the cup labeled by \(S_2\) (representing \(\iota_{S_2}\)) and then the two braidings through the red strand using
	Lemma \ref{reversevertexlemma} and Proposition \ref{pushdiagram}, we get to the right-hand side of Figure \ref{bulk q-KZB diagram}.
\end{proof}
The identity in $\textup{End}_{\mathcal{M}_{\textup{adm}}}(M_\mu\otimes M_\lambda)$ underlying Figure \ref{bulk q-KZB diagram} entails operator and dual operator $q$-KZB type equations for $k$-point quantum vertex operators. For instance, the identity in $\textup{End}_{\mathbb{C}}(M_\mu)$ obtained by taking the highest-weight to highest-weight component in $M_\lambda$ of the identity underlying Figure \ref{bulk q-KZB diagram} gives rise to
a dual operator $q$-KZB type equation for $\Phi_{\mu;S_1,S_2,S_3}^{w_1,w_2,w_3}$ (cf. Subsection \ref{Subsection Twisted trace functions}). In this case the dual quantum vertex operator $\Psi=\Psi^{g_3,g_2,g_1}_{\lambda;S_3^\ast,S_2^\ast,S_1^\ast}$ featuring in
Figure \ref{bulk q-KZB diagram} serves as a particular way of parametrizing the spin-component $\mathcal{F}^\str(S)[0]$ of the quantum vertex operator $\Phi_{\mu;S_1,S_2,S_3}^{w_1,w_2,w_3}$.
\begin{remark}
Operator $q$-KZ equations for intertwiners of quantum affine algebra modules were derived in \cite[Thm. 5.1]{Frenkel&Reshetikhin-1992}. 

\end{remark}

\subsection{Applying boundary conditions}
\label{Subsection Boundary conditions}

The diagram in $\mathbb{B}_{\mathcal{M}_{\mr{adm}}^\str}$ given by Figure \ref{Tdiagram}, which is mapped to \(\mathcal{T}_{\bs{S}}^{\bs{w},\bs{g}}(\lambda,\mu)\) by $\mathcal{F}_{\mr{adm}}^{\mr{br}}$, has two incoming and two outgoing Verma strands. In Subsection \ref{Subsection Twisted trace functions} we will graphically impose on these Verma strands the boundary conditions turning \(\mathcal{T}_{\bs{S}}^{\bs{w},\bs{g}}(\lambda,\mu)\)  into the spin components \(\mathcal{Y}_{\bs{S}}^{\bs{w},\bs{g}}(\lambda,\mu,\xi)\) of the weighted trace function. In other words, the lower incoming and outgoing Verma strand, both colored by $M_\mu$, will be coupled via weighted cyclic boundary conditions. The resulting morphism represents a weighted trace function of the intertwiner $\Phi_{\mu;S_1,S_2,S_3}^{w_1,w_2,w_3}$ with the dual intertwiner $\Psi_{\lambda;S_3^*,S_2^*,S_1^*}^{g_3,g_2,g_1}$ attached to its spin strands. In the remaining incoming and outgoing Verma strand, both colored by \(M_\lambda\), we will impose highest-weight-to-highest-weight boundary conditions, which leads to the spin components \(\mathcal{Y}_{\bs{S}}^{\bs{w},\bs{g}}(\lambda,\mu,\xi)\) 
of the weighted trace function. 

When imposing such boundary conditions graphically, the framework of graphical calculus for the category $\mathcal{M}_{\mr{adm}}$ and \(\Rep\) is reduced to the one for 
$\cN_{\mr{adm}}$ and \(\cN_{\mr{fd}}\), cf. \cite[Section 3.3]{DeClercq&Reshetikhin&Stokman-2022}. It will be useful to translate the special coupons from Subsection \ref{sectiongi} that implement the dynamical ribbon type structures for \(\Ndyn^\str\) in terms of the standard graphical calculus for the ribbon category \(\cN_{\mr{fd}}^\str\). Let us start by clarifying this for the dynamical braiding \(\ol{c_{S,T}}\) with \(S,T\in\Rep^\str\).

Denoting the morphism in \(\End_{\Ndyn^\str}(\ul{S}\ol{\tens}\ul{T})\) represented by \(\Rdyn_{S,T}\)
(see Definition \ref{Rdyn})
by \(\Rdyn_{S,T}\) as well, one has the identity
\begin{equation}\label{morform}
\ol{c_{S,T}} = P_{S,T}\Rdyn_{S,T}
\end{equation}
in \(\Hom_{\Ndyn^\str}(\ul{S}\ol{\tens}\ul{T}, \ul{T}\ol{\tens}\ul{S})\). Hence the coupon in $\textup{Rib}_{\cN_{\mr{fd}}^\str}$
colored by \(\ol{c_{S,T}}(\lambda)\), 
which we have denoted by Figure \ref{dynamical R-matrix} with the rightmost vertical region colored by \(\lambda\), and
the standard diagram in $\textup{Rib}_{\cN_{\mr{fd}}^\str}$ given by Figure \ref{dynamical R-matrix boundary}, 
have the same image under 
the Reshetikhin-Turaev functor $\mathcal{F}_{\cN_{\mr{fd}}^\str}^{\mr{RT}}$, i.e.,
\begin{figure}[H]
\centering
\includegraphics[scale=0.75]{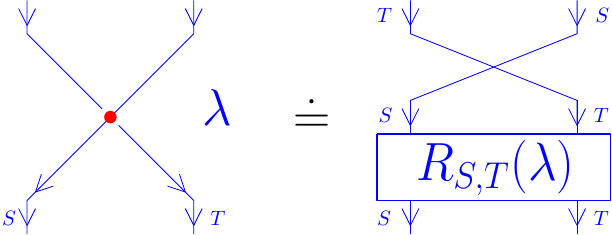}
\caption{}
\end{figure}
Recall that we have introduced Figure \ref{dynamical R-matrix boundary B} as an alternative notation for Figure \ref{dynamical R-matrix boundary}. We will write the color $R_{S,T}$ of these coupons simply by $R$ if no confusion can arise. We will use the same convention for colors obtained from the universal elements $\mathbb{J}(\lambda)$ and $\mathbb{Q}(\lambda)$ introduced in Subsection \ref{Subsection Q(lambda)}.

\begin{figure}[H]
	\begin{minipage}{0.48\textwidth}
		\centering
		\includegraphics[scale=0.75]{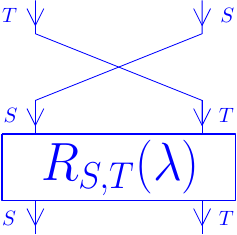}
		\caption{}
		\label{dynamical R-matrix boundary}
	\end{minipage}
	\begin{minipage}{0.48\textwidth}
		\centering
		\includegraphics[scale=0.75]{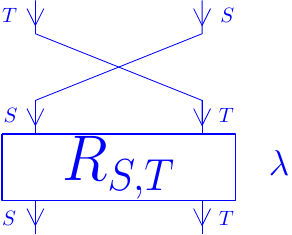}
		\caption{}
		\label{dynamical R-matrix boundary B}
	\end{minipage}
\end{figure}

For the dynamical evaluation and co-evaluation morphisms we only discuss the case of strands colored by an object $S=(V)$ of length one.
By (\ref{dyn eval expr}) the morphism $\ol{e_{(V)}}(\lambda)$ is represented by
$e_V\circ (\pi_{V^\ast}\otimes\pi_V)(\mathbb{J}(\lambda))=e_V\circ j_{(V^*,V)}(\lambda)$.
Then the diagram in $\textup{Rib}_{\cN_{\mr{fd}}^\str}$ given by Figure \ref{dyn eval left boundary} 
and the diagram in $\textup{Rib}_{\cN_{\mr{fd}}^\str}$ given by Figure \ref{evaluation dynamical} with $S=(V)$ and with its rightmost vertical region colored by \(\lambda\) 
have the same image under $\cF_{\cN_{\mr{fd}}^\str}^{\mr{RT}}$.
A similar observation applies to the morphisms \(\ol{\iota_{(V)}}(\lambda)\), \(\ol{\widetilde{e}_{(V)}}(\lambda)\) and \(\ol{\widetilde{\iota}_{(V)}}(\lambda)\), 
leading to Figures \ref{dyn co-eval left boundary}--\ref{dyn co-eval right boundary} as the reformations in terms of the standard graphical calculus of the ribbon category $\cN_{\mr{fd}}^\str$
of Figures \ref{injection dynamical}--\ref{co-injection dynamical} with \(S= (V)\) and specialized weight \(\lambda\).

\begin{figure}[H]
	\begin{minipage}{0.24\textwidth}
		\centering
		\includegraphics[scale = 0.75]{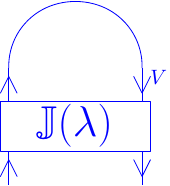}
		\captionof{figure}{}
		\label{dyn eval left boundary}
	\end{minipage}
	\begin{minipage}{0.24\textwidth}
		\centering
		\includegraphics[scale = 0.75]{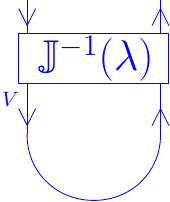}
		\captionof{figure}{}
		\label{dyn co-eval left boundary}
	\end{minipage}
	\begin{minipage}{0.24\textwidth}
		\centering
		\includegraphics[scale = 0.75]{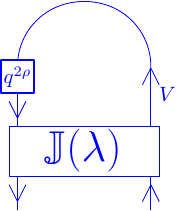}
		\captionof{figure}{}
		\label{dyn eval right boundary}
	\end{minipage}
	\begin{minipage}{0.24\textwidth}
		\centering
		\includegraphics[scale = 0.75]{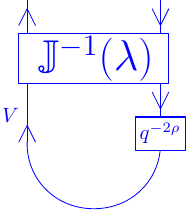}
		\captionof{figure}{}
		\label{dyn co-eval right boundary}
	\end{minipage}
\end{figure}

Finally, for \(S\in\Rep^\str\) and \(\sigma\in\wts(\cF^\str(S))\), let us write \(\mathbb{P}_S[\sigma]\) for the morphism in \(\End_{\cN_{\mr{fd}}^\str}(\ul{S})\) represented by the projection along the \(\hh^\ast\)-grading of \(\cF^\str(\ul{S})\) onto the weight space  \(\cF^\str(\ul{S})[\sigma]\). For ease of notation, we will also write \(\wts(S)\) for \(\wts(\cF^\str(S))\). The coupon in $\textup{Rib}_{\cN_{\mr{fd}}^\str}$ labeled by \(\mathbb{P}_S[\sigma]\) will be depicted by Figure \ref{projection coupon}.

\begin{figure}[H]
	\centering
	\includegraphics[scale=1.1]{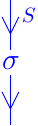}
	\caption{}
	\label{projection coupon}
\end{figure}

\begin{lemma}
	\label{lemma in Section 3.3}
	For any \(V\in\Mfd\), \(\mu\in\hh_{\mathrm{reg}}^\ast\) and \(\sigma\in\wts(V)\), one has
	\begin{figure}[H]
		\begin{minipage}{0.48\textwidth}
			\centering
			\includegraphics[scale=0.75]{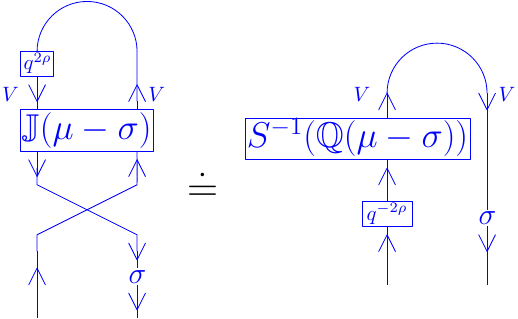}
			\caption{}
			\label{first calculation with J}
		\end{minipage}
		\begin{minipage}{0.48\textwidth}
			\centering
			\includegraphics[scale=0.75]{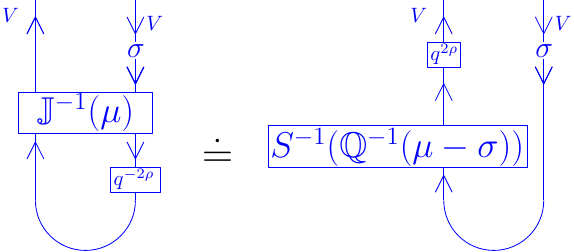}
			\caption{}
			\label{second calculation with J}
		\end{minipage}
	\end{figure}
\noindent	
relative to the graphical calculus for $\cN_{\mr{fd}}^\str$.
\end{lemma}
\begin{proof}
	These are obtained through a direct algebraic calculation in $\cN_{\mr{fd}}$. Figure \ref{first calculation with J} 
	follows immediately from the definition of the dual representation, the property (\ref{action of q^2rho}) of the element \(q^{2\rho}\), and the definition (\ref{Q algebraic def}) of \(\QQ(\mu-\sigma)\). 
	
	Figure \ref{second calculation with J} is a bit more involved. Let us write \(\mathbb{J}^{-1}(\mu) = \sum_{\beta\in Q}A_\beta\otimes B_\beta\), with \(A_\beta\in U_q[\beta]\) and \(B_\beta\in U_q[-\beta]\), and note that Proposition \ref{prop Q inverse} implies that
	\begin{equation}
	\label{identity for Q inverse}
	\QQ^{-1}(\mu-\sigma)\big\vert_{V[\sigma]} = \sum_{\beta\in Q} B_{\beta}S^{-1}(A_\beta)\big\vert_{V[\sigma]}.
	\end{equation}
	Then the left-hand side of Figure \ref{second calculation with J} is mapped by $\cF_{\mathcal{N}_{\mr{fd}}^\str}^{\mr{RT}}$ to the morphism in \(\Hom_{\cN_{\mr{fd}}^\str}(\emptyset,\ul{V^\ast}\tens \ul{V})\) represented by \((\id_{\ul{V^\ast}}\otimes P_V[\sigma])\circ A\), where \(A\) is the morphism in \(\Hom_{\cN_{\mr{fd}}}(\mathbb{1},\ul{V^\ast}\otimes \ul{V})\) sending \(1\in\CC\) to
	\begin{equation}
	\label{first description for iota}
	\sum_{w\in \mathcal{B}_V} \sum_{\beta \in Q} A_\beta\cdot w^\ast\otimes B_\beta q^{-2\rho}\cdot w,
	\end{equation}
	where \(\mathcal{B}_V\) is a homogeneous basis for \(\ul{V}\). The definition of the dual representation then readily asserts that (\ref{first description for iota}) equals
	\[
	\sum_{w,w'\in \mathcal{B}_V} \sum_{\beta \in Q}{w'}^\ast(B_\beta q^{-2\rho}\cdot w)A_\beta\cdot w^\ast\otimes w'=
	\sum_{w'\in\mathcal{B}_V}\sum_{\beta\in Q}A_\beta q^{2\rho}S^{-1}(B_\beta)\cdot {w'}^\ast\otimes w'.
	\]
	Writing $\mathcal{B}_V[\sigma]\subseteq\mathcal{B}_V$ for the basis elements in $\mathcal{B}_V$ of weight $\sigma$, we conclude from (\ref{action of q^2rho}) and (\ref{identity for Q inverse})
	that
	\begin{equation*}
	\begin{split}
	(\textup{id}_{\ul{V}^*}\otimes P_V[\sigma])(A(1))&=\sum_{w\in\mathcal{B}_V[\sigma]}\sum_{\beta\in Q}S(B_\beta S^{-1}(A_\beta))q^{2\rho}\cdot w^\ast\otimes w\\
	&=\sum_{w\in\mathcal{B}_V[\sigma]}S(\mathbb{Q}^{-1}(\mu-\sigma))q^{2\rho}\cdot w^\ast\otimes w.
	\end{split}
	\end{equation*}
	We conclude that the $\mathcal{F}_{\cN_{\mr{fd}}^{\str}}^{\mr{RT}}$-image of the
	left-hand side of Figure \ref{second calculation with J} is the morphism in \(\Hom_{\cN_{\mr{fd}}^\str}(\emptyset,\ul{V^\ast}\tens \ul{V})\) represented by
	\[
	(\pi_{V^\ast}(S(\QQ^{-1}(\mu-\sigma))q^{2\rho})\otimes P_V[\sigma])\circ \widehat{\iota}_{\ul{V}},
	\]
	which equals the $\mathcal{F}_{\cN_{\mr{fd}}^\str}^{\mr{RT}}$-image of the diagram depicted by the right-hand side of Figure \ref{second calculation with J}.
\end{proof}

\subsection{The $q$-KZB type  equations}
\label{Subsection Twisted trace functions}

In view of \eqref{Y in terms of T}, applying weighted cyclic boundary conditions to \(M_\mu\) and highest-weight-to-highest-weight components to \(M_\lambda\) turns \(\mathcal{T}_{\bs{S}}^{\bs{w},\bs{g}}(\lambda,\mu)\) (see \eqref{T_S def})
into the spin component \(\mathcal{Y}_{\bs{S}}^{\bs{w},\bs{g}}(\lambda,\mu,\xi)\) of the weighted trace function for the intertwiner 
$\Phi_{\mu; S_1,S_2,S_3}^{w_1,w_2,w_3}$. 
In this subsection we graphically add these boundary conditions to Figure \ref{bulk q-KZB diagram} and derive from it 
 a dual $q$-KZB type equation for $\mathcal{Z}_{\bs{S}}^{\bs{w},\bs{g}}(\lambda,\mu)$ (see \eqref{Zdef}) using the graphical calculus for $\cN_{\mr{adm}}^\str$.
We will need the following lemma, which graphically represents the cyclic property of the trace.
\begin{lemma}
	\label{lemma cyclicity}
	For any \(\mu\in\hh_{\mathrm{reg}}^\ast\), \(\xi\in\hh^\ast\) with \(\Re(\xi)\) deep in the negative Weyl chamber, \(S_1, S_2\in \Rep^\str\) and \(w_i\in\cF^\str(S_i)[\nu_i]\) with \(i = 1,2\), such that \(\nu_1+\nu_2 = 0\), we have 
	\begin{figure}[H]
		\centering
		\includegraphics[scale=0.72]{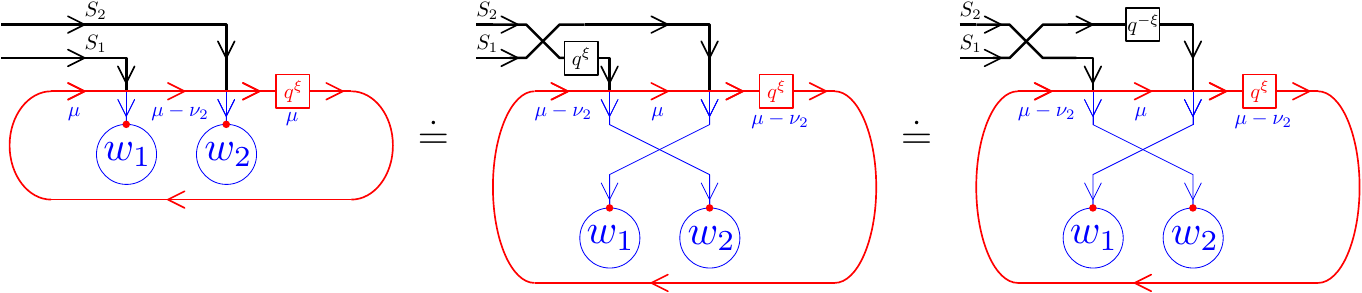}
		\caption{}
		\label{cyclic visual}
	\end{figure}
	relative to the graphical calculus for \(\cN_{\mr{adm}}^\str\) \textup{(}so the black strands are in $\mathbb{B}_{\cN_{\mr{adm}}^\str}$\textup{)}.
\end{lemma}
\begin{proof}
	Let us define
	\[
	\phi_1:=\cF^\str(\Phi_{\mu-\nu_2;S_1}^{w_1}), \qquad \phi_2 := \cF^\str(\Phi_{\mu;S_2}^{w_2})
	\]
	and set
	\begin{align*}
	\phi := (\phi_1\otimes\id_{\cF^\str(S_2)})\phi_2, \qquad \til{\phi} = (\phi_2\otimes\id_{\cF^\str(S_1)})\phi_1.
	\end{align*}
	Then the fact that \(q^\xi\) is group-like and the cyclic property of the trace assert that
	\begin{align*}
	\Tr_{M_\mu}\left(\phi\circ\pi_{\mu}(q^\xi) \right) =\ & \Tr_{M_\mu}\left( ((\phi_1\circ\pi_{\mu-\nu_2}(q^\xi))\otimes\pi_{\cF^\str(S_2)}(q^\xi)
	)\circ\phi_2 \right)\\
	=\ & P_{\cF^\str(\ul{S_2}),\cF^\str(\ul{S_1})}(\pi_{\cF^\str(S_2)}(q^\xi)\otimes\id_{\cF^\str(\ul{S_1})})\Tr_{M_{\mu-\nu_2}}(\widetilde{\phi}\circ\pi_{\mu-\nu_2}(q^\xi)).
	\end{align*}
	If we then set \(\Phi:= \Phi_{\mu;S_1,S_2}^{w_1,w_2}\) and \(\til{\Phi}:=\Phi_{\mu-\nu_2;S_2,S_1}^{w_2,w_1}\), then with the notations (\ref{notation with curly H}) for the weighted traces and (\ref{sigma def}) for the bracket-moving isomorphisms \(\sigma\), this equality asserts that
	\[
	\mathcal{H}_\mu^{\Phi}(q^\xi) = \sigma_{\ul{S_1},\ul{S_2}}\, P_{\cF^\str(\ul{S_2}),\cF^\str(\ul{S_1})}(\pi_{\cF^\str(S_2)}(q^\xi)\otimes\id_{\cF^\str(\ul{S_1})})\sigma^{-1}_{\ul{S_2},\ul{S_1}}\mathcal{H}_{\mu-\nu_2}^{\widetilde{\Phi}}(q^\xi).
	\]
	The lift of this equality in \(\Hom_{\cN_{\mr{fd}}}(\CC_0,\cF^\str(\ul{S_1}\tens \ul{S_2}))\) to an equality in \(\Hom_{\cN_{\mr{fd}}^\str}(\CC_0,\ul{S_1}\tens \ul{S_2}))\) proves the first equality in Figure \ref{cyclic visual}. The second equality in Figure \ref{cyclic visual} follows upon noting that \(H_{\mu-\nu_2}^{\widetilde{\Phi}}(q^\xi)\) is of weight 0.
\end{proof}

In the remainder of the subsection we fix $S\in\Rep^\str$ and choose a decomposition $S=S_1\tens S_2\tens S_3$ of $S\in\Rep^\str$, as well as vectors $w_i\in\cF^\str(S_i)[\nu_i]$ 
and $g_i\in\cF^\str(S_i^*)[\nu_i']$ with $\nu_1+\nu_2+\nu_3=0=\nu_1'+\nu_2'+\nu_3'$ (cf. Subsection \ref{Subsection Formal weighted trace functions}). 

Taking the highest-weight-to-highest-weight component in the upper Verma strand in Figure \ref{bulk q-KZB diagram} will allow us to resolve the crossings in the upper Verma strand 
using Lemma \ref{lemma 3.10}, whereas taking the weighted trace in the lower Verma strand allows us to cyclically permute the vertex operators in the lower Verma strand using Lemma \ref{lemma cyclicity}. It leads to the following result.
\begin{lemma}
	\label{lemma boundary qKZB with general xi}
	Under the same conventions as in Proposition \ref{prop bulk q-KZB} we have, for $\Psi=\Psi_{\lambda; S_3^*,S_2^*,S_1^*}^{g_3,g_2,g_1}$ and for \(\xi\in\hh^\ast\) with \(\Re(\xi)\) deep in the negative Weyl chamber, 
	\begin{figure}[H]
		\centering
		\includegraphics[scale=0.65]{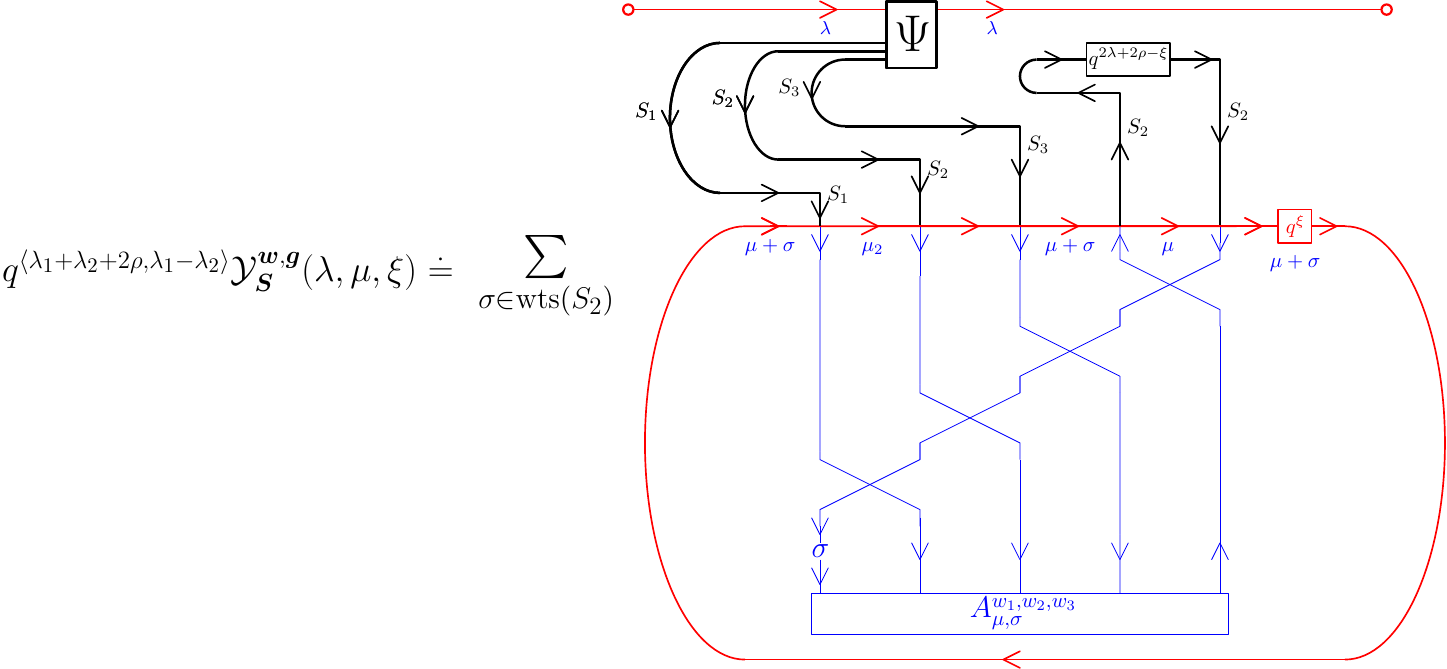}
		\caption{}
		\label{boundary qKZB A}
	\end{figure}
\noindent
relative to the graphical calculus for \(\cN_{\mr{adm}}^\str\) and \(\cN_{\mr{fd}}^\str\) \textup{(}the graphical calculus for $\cN_{\mr{fd}}^\str$ applies to the part of the diagram involving solely black or blue strands\textup{)}. Here $A_{\mu;\sigma}^{w_1,w_2,w_3}$ is the morphism in $\cN_{\mr{fd}}^\str$ such that 
	\begin{figure}[H]
	\centering
		\includegraphics[scale=0.75]{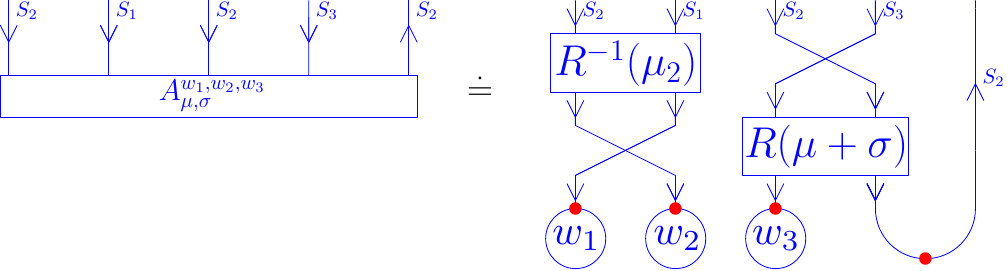}
	\caption{}
	\label{Adef}
	\end{figure}
\end{lemma}
\begin{proof}
	In view of (\ref{Y in terms of T}),
	inserting in Figure \ref{bulk q-KZB diagram} highest-weight-to-highest-weight boundary conditions in the upper Verma strand and cyclic boundary conditions weighted by \(q^\xi\) in the lower Verma strand shows that 
	\begin{figure}[H]
		\centering
		\includegraphics[scale=0.7]{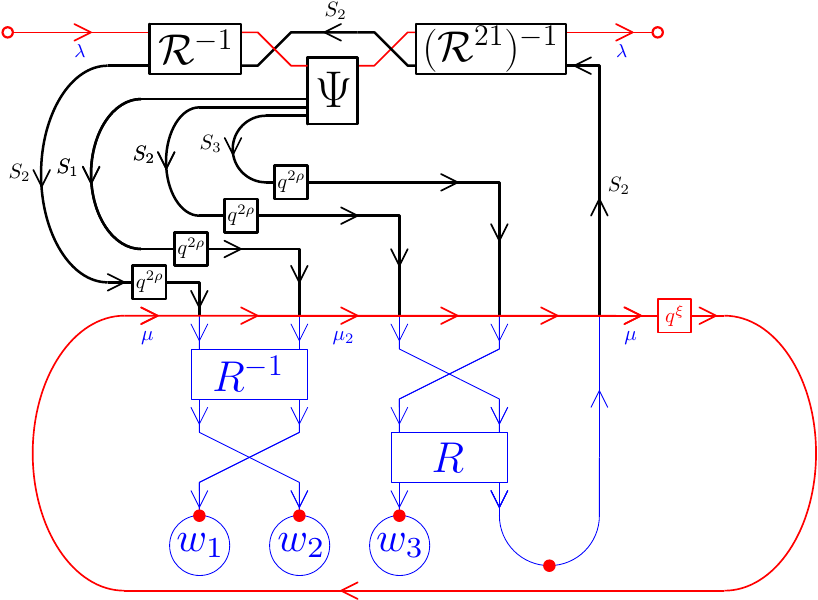}
		\caption{}
		\label{boundary qKZB AA}
	\end{figure}
\noindent
is a graphical representation of $q^{\langle \lambda_1+\lambda_2+2\rho,\lambda_1-\lambda_2\rangle}\mathcal{Y}_{\bs{S}}^{\bs{w},\bs{g}}(\lambda,\mu,\xi)$.
Here we used \cite[Section 3.3]{DeClercq&Reshetikhin&Stokman-2022} and Subsections \ref{GcSection} \& \ref{Subsection Boundary conditions} to express the upper part of the diagram in terms of the standard graphical calculus for $\cN_{\textup{adm}}^\str$ and $\cN_{\textup{fd}}^\str$.
By Lemma \ref{lemma 3.10} we can replace the coupons in Figure \ref{boundary qKZB AA} labeled by \(\cR^{-1}_{S_2^\ast,M_{\lambda}}\) and \((\cR^{21})^{-1}_{S_2^\ast,M_{\lambda}}\) by a coupon labeled by \(\pi_{S_2^\ast}(q^{-\lambda})\), which can then be transported along the cap labeled by \(S_2\). Moreover, one can remove the coupons labeled by \(q^{2\rho}\) on the three middle black strands, since \(\langle \Psi\rangle\) has weight 0. If we then apply \(\mathbb{P}_{S_2}[\sigma]\) on the leftmost blue strand and sum over all \(\sigma \in\wts(S_2)\), Figure \ref{boundary qKZB AA} may be replaced by
	\begin{center}
		\includegraphics[scale=0.7]{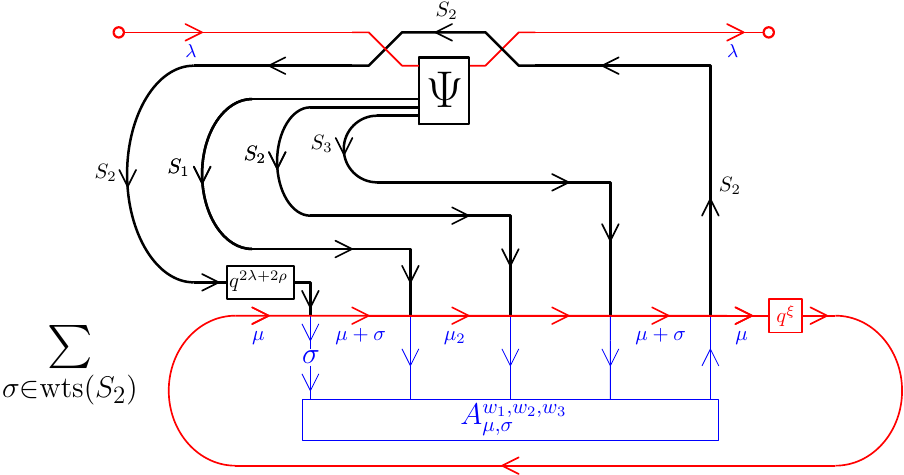}
	\end{center}
\noindent
In the process of coloring the vertical regions below the red strand with weights, note that imposing the projection operator \(\mathbb{P}_{S_2}[\sigma]\) on the leftmost blue strand forces the vertical region right of this strand to be colored by \(\mu+\sigma\). This then implies that the vertical region left of the rightmost blue strand carries the same color \(\mu+\sigma\), because the total weight on the three middle blue strands is zero since \(\Psi\) is a morphism \(M_\lambda\to S^\ast\tens M_\lambda\). 
By Lemma \ref{lemma cyclicity} the diagram above may then be replaced by
	\begin{center}
		\includegraphics[scale=0.7]{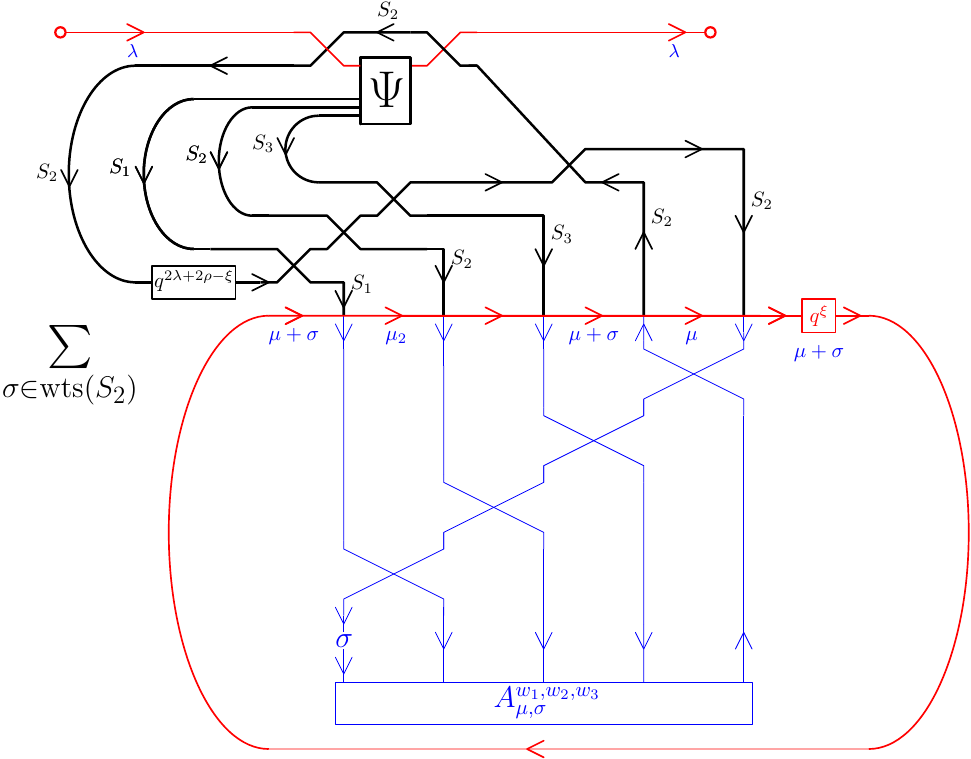}
	\end{center}
	Upon contracting the loop labeled by \(S_2\), this reduces to the right-hand side of Figure \ref{boundary qKZB A}.
\end{proof}

Note that the right hand side of Figure \ref{boundary qKZB A} does not yet graphically represent a sum of spin components of weighted traces of intertwiners, because of the weighted cap in the black strand labeled by $S_2$. To push this weighted cap below the Verma strand using Proposition \ref{pushdiagram} we need to take 
\begin{equation}
\label{choice of xi}
\xi = 2\lambda+2\rho,
\end{equation}
in which case the weighted cap is represented by $\ul{e_{S_2}}$, with $e_{S_2}$ the evaluation morphism in $\Rep^\str$.
With this choice of \(\xi\), the result of Lemma \ref{lemma boundary qKZB with general xi} can be refined to a dual $q$-KZB type equation for the spin components 
\begin{equation}
\label{set Z equal to}
\mathcal{Z}_{\bs{S}}^{\bs{w},\bs{g}}(\lambda,\mu):=\mathcal{Y}_{\bs{S}}^{\bs{w},\bs{g}}(\lambda,\mu,2\lambda+2\rho)
\end{equation}
of the weighted trace functions, with \(\Re(\lambda)\) lying deep in the negative Weyl chamber. 
\begin{proposition}
	\label{prop boundary qKZB}
	Under the same conventions as in Proposition \ref{prop bulk q-KZB} we have, for \(\lambda\in\hh^\ast\) with \(\Re(\lambda)\) deep in the negative Weyl chamber and $\Psi=\Psi_{\lambda; S_3^*,S_2^*,S_1^*}^{g_3,g_2,g_1}$,
	\begin{figure}[H]
		\centering
		\includegraphics[scale=0.75]{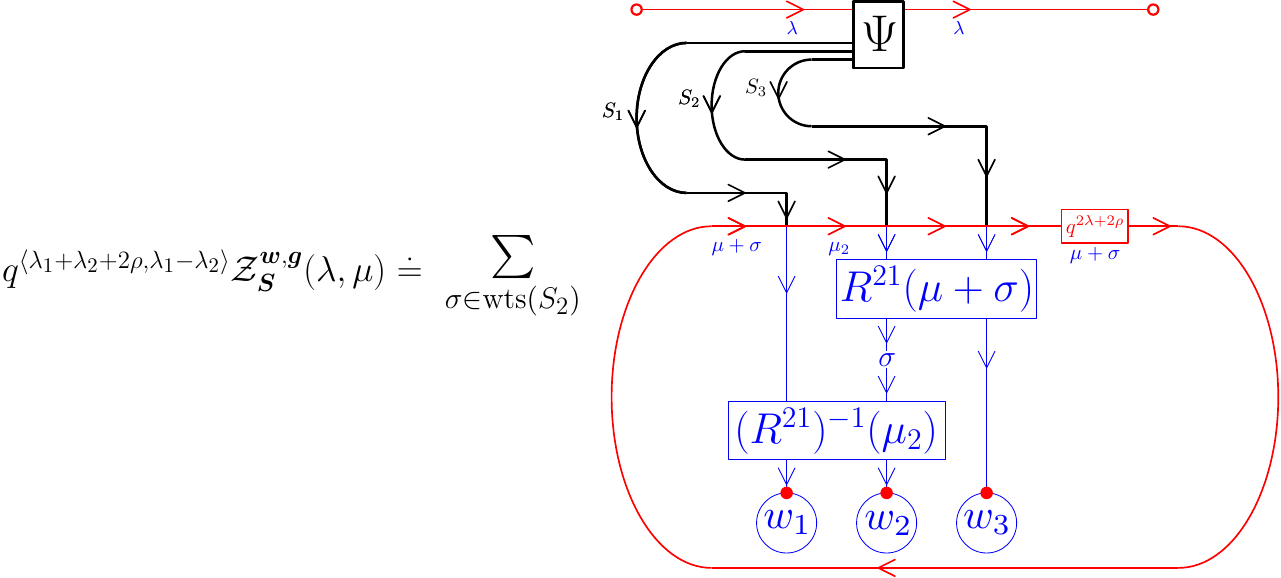}
		\caption{}
		\label{boundary qKZB G}
	\end{figure}
\noindent 
relative to the graphical calculus for \(\Nadm^\str\) and \(\mathcal{N}_{\mr{fd}}^\str\).
\end{proposition}
\begin{proof}
	By Proposition \ref{pushdiagram} we have
	\begin{figure}[H]
	\centering
		\includegraphics[scale=0.9]{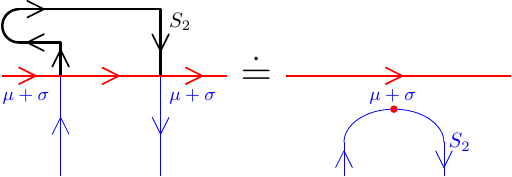}
		\caption{}
		\label{bhhelp}
	\end{figure}
This can be used to accordingly manipulate the strand colored by \(\ul{S_2}\) in Figure \ref{boundary qKZB A}. More precisely, with the specialization of \(\xi\) as in (\ref{choice of xi}), the remaining cap colored by \(\ul{S_2}\) in Figure \ref{boundary qKZB A} can be identified with the cap colored by \(S_2\) in the graphical calculus for \(\Mfd^\str\), and hence we can pull it through the red strand labeled by \(M_{\mu+\sigma}\) using Figure \ref{bhhelp}.
	
We conclude that by Lemma \ref{lemma boundary qKZB with general xi} and by the above observation, 
$q^{\langle\lambda_1+\lambda_2+2\rho,\lambda_1-\lambda_2\rangle}\mathcal{Z}_{\bs{S}}^{\bs{w},\bs{g}}(\lambda,\mu)$ 
is graphically represented by
	\begin{figure}[H]
		\centering
		\includegraphics[scale=0.75]{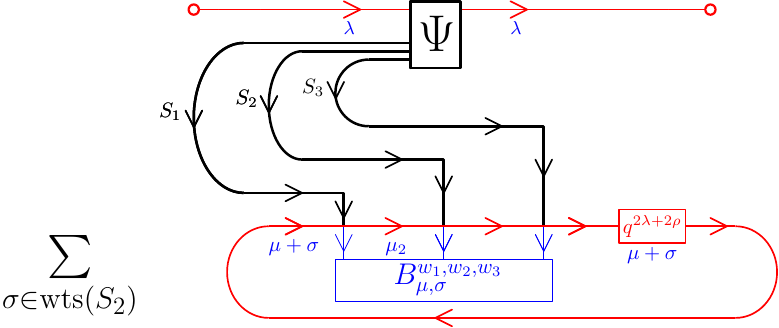}
		\caption{}
		\label{boundary qKZB I}
	\end{figure}
	\noindent
	with $B_{\mu;\sigma}^{w_1,w_2,w_3}$ the morphism in $\cN_{\mr{fd}}^\str$ such that
	\begin{figure}[H]
		\centering
		\includegraphics[scale=0.75]{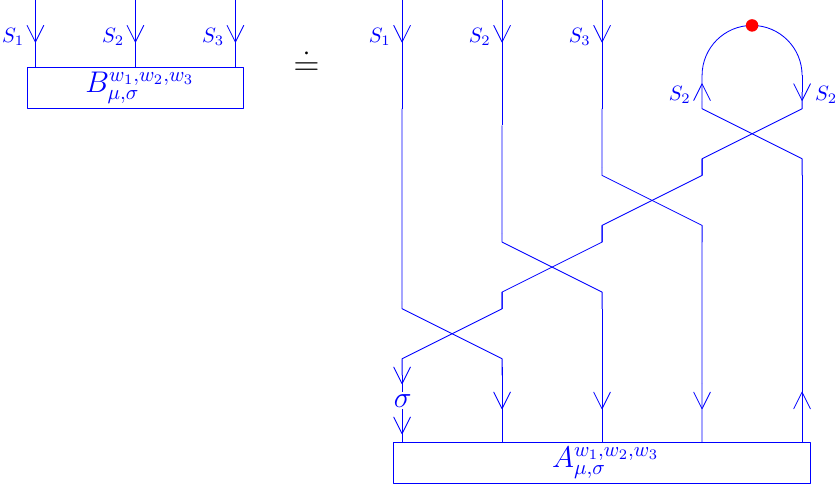}
		\caption{}
		\label{boundary qKZB J}
	\end{figure}
	\noindent 
	We now substitute Figure \ref{Adef} in Figure \ref{boundary qKZB J} and use the graphical calculus for $\cN_{\mr{fd}}^\str$ to simplify. It follows that
	 the $\mathcal{F}_{\cN_{\mr{fd}}^\str}^{\textup{RT}}$-images of the right-hand side of Figure \ref{boundary qKZB J} and of
	\begin{center}
		\includegraphics[scale=0.75]{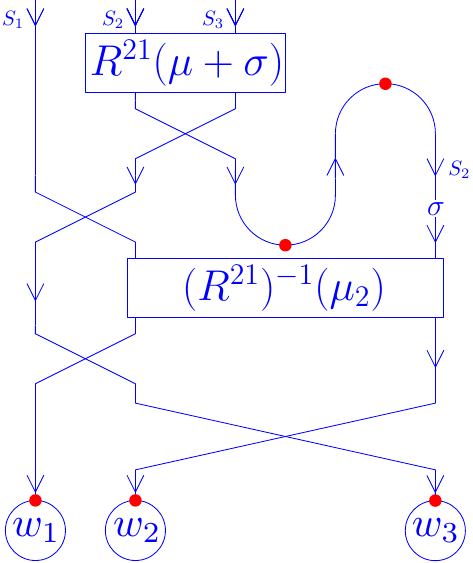}
	\end{center}
	\noindent
	coincide. Using Figure \ref{duality A} and simplifying further, 
	 Figure \ref{boundary qKZB I} can be transformed into the right-hand side of Figure \ref{boundary qKZB G}.
	 This concludes the proof. 
\end{proof}
We now give the reformulation of Proposition \ref{prop boundary qKZB} as a concrete equation for \(\mathcal{Z}_{\bs{S}}^{\bs{w},\bs{g}}(\lambda,\mu)\). 

If \(\mathcal{K}\in\End_{\mathcal{N}_{\textup{fd}}}(\cF^\str(\ul{S}))\) and \(\cD\in\End_{\cN_{\textup{fd}}}(\cF^\str(\ul{S^\ast}))\) and
\[
\mathcal{K}(w_1\otimes w_2\otimes w_3) = \sum_m u_1^m\otimes u_2^m\otimes u_3^m, \qquad \cD(g_3\otimes g_2\otimes g_1) = \sum_n \omega_3^n\otimes \omega_2^n\otimes \omega_1^n
\]
are expansions with homogeneous vectors \(u_i^m\in\cF^\str(\ul{S_i})\) and \(\omega_i^n\in\cF^\str(\ul{S_i^\ast})\), then we write
\[ 
\mathcal{Z}_{\bs{S}}^{\mathcal{K}(\bs{w}),\mathcal{D}(\bs{g})}(\lambda,\mu):=
\sum_{m,n}\mathcal{Z}_{\bs{S}}^{(u_1^m,u_2^m,u_3^m),(\omega_3^n,\omega_2^n,\omega_1^n)}(\lambda,\mu).
\]
This is well defined since $\mathcal{K}$ and $\mathcal{D}$ are $\mathfrak{h}$-linear, so they restrict to endomorphisms of $\cF^\str(\ul{S})[0]$ and $\cF^\str(\ul{S^*})[0]$, respectively.

Recall from \cite[Definition 1.6 and Section 3.3]{DeClercq&Reshetikhin&Stokman-2022} the notations
\begin{equation}\label{thetakappa}
\theta(\lambda):=\lambda+\rho-\frac12\sum_{i=1}^r x_i^2\in U(\hh), \qquad \kappa:=q^{\sum_{i=1}^rx_i\otimes x_i}\in \mathcal{U}^{(2)}.
\end{equation}
\begin{corollary}
	\label{thm boundary qKZB}
	Consider arbitrary \(\lambda,\mu\in\hh_{\mathrm{reg}}^\ast\) with \(\Re(\lambda)\) deep in the negative Weyl chamber, \(S = S_1\tens S_2\tens S_3\in \Rep^\str\), \(w_i\in\cF^\str(S_i)[\nu_i]\) and \(g_i\in \cF^\str(S_i^\ast)[\nu_i']\) for \(i\in\{1,2,3\}\), with \(\nu_1+\nu_2+\nu_3 = \nu_1'+\nu_2'+\nu_3' = 0\). Then 
	\begin{equation}\label{qKZBfromgraphics}
	\mathcal{Z}_{\bs{S}}^{\bs{w},\bs{g}}(\lambda,\mu) = \sum_{\sigma\in\wts(S_2)}\mathcal{Z}_{\bs{S}}^{\mathcal{K}_{\mu,\sigma}(\bs{w}), \mathcal{D}_\lambda(\bs{g})}(\lambda,\mu+\sigma)
	\end{equation}
	with \(\mathcal{K}_{\mu,\sigma}\in\End_{\cN_{\textup{fd}}}(\cF^\str(\ul{S}))\) and \(\mathcal{D}_\lambda\in\End_{\cN_{\textup{fd}}}(\cF^\str(\ul{S^\ast}))\) defined by
	\begin{align}
	\label{u vectors def}
	\mathcal{K}_{\mu,\sigma}:=\ & \mathscr{J}_{\boldsymbol{S}}(\mu+\sigma)\, \Rdyn^{21}_{S_2,S_3}(\mu+\sigma)\,\mathbb{P}_{S_2}[\sigma]\,
	\Rdyn^{21}_{S_1,S_2}(\mu-\mh_{S_3})^{-1}\mathscr{J}_{\boldsymbol{S}}(\mu+\sigma)^{-1}, \\
	\label{the funny Js}
	\mathscr{J}_{\boldsymbol{S}}(\mu+\sigma) :=\ & j_{S_1}(\mu+\sigma-\mh_{S_2\tens S_3})\otimes j_{S_2}(\mu+\sigma-\mh_{S_3}) \otimes j_{S_3}(\mu+\sigma), \\
	\label{the funny D}
	\mathcal{D}_\lambda :=\ & \kappa^{-2}_{S_3^\ast,S_2^\ast}(q^{2\theta(\lambda)})_{S_2^\ast}.
	\end{align}
	\end{corollary}
\begin{proof}
	Proposition \ref{prop boundary qKZB} asserts that 
	\begin{figure}[H]
		\centering
		\includegraphics[scale=0.75]{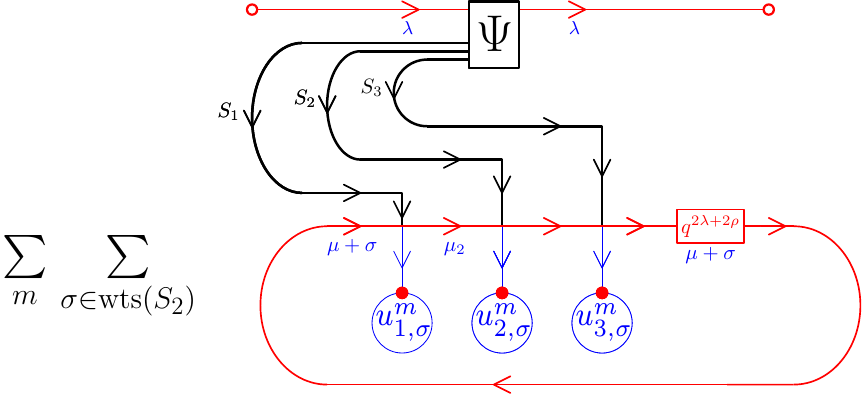}
		\caption{}
		\label{D6}
	\end{figure}
	\noindent 
	is a graphical representation of \(q^{\langle\lambda_1+\lambda_2+2\rho,\lambda_1-\lambda_2\rangle}\,\mathcal{Z}_{\bs{S}}^{\bs{w},\bs{g}}(\lambda,\mu)\),
	where \(u_{i,\sigma}^m\in \cF^\str(S_i)\) are homogeneous vectors such that
	\[
	\sum_m u_{1,\sigma}^m\otimes u_{2,\sigma}^m\otimes u_{3,\sigma}^m = \mathcal{K}_{\mu,\sigma}(w_1\otimes w_2\otimes w_3).
	\]
	Observe that Figure \ref{D6} represents
	\[
	\sum_{m}\sum_{\sigma\in\wts(S_2)}\mathcal{Z}_{\bs{S}}^{(u_{1,\sigma}^m,u_{2,\sigma}^m,u_{3,\sigma}^m), \bs{g}}(\lambda,\mu+\sigma),
	\]
	such that we obtain the equation
	\[
	q^{\langle\lambda_1+\lambda_2+2\rho,\lambda_1-\lambda_2\rangle}\,\mathcal{Z}_{\bs{S}}^{\bs{w},\bs{g}}(\lambda,\mu) = \sum_{\sigma\in\wts(S_2)} \mathcal{Z}_{\bs{S}}^{\mathcal{K}_{\mu,\sigma}(\bs{w}),\bs{g}}(\lambda,\mu+\sigma),
	\]
	with \(\lambda_j\) as defined in (\ref{lambda_j def}). This can be rewritten as
	\[
	\mathcal{Z}_{\bs{S}}^{\bs{w},\bs{g}}(\lambda,\mu) = q^{\langle 2\lambda+2\rho-2\nu'_3-\nu'_2,\nu'_2\rangle} \sum_{\sigma\in\wts(S_2)} \mathcal{Z}_{\bs{S}}^{\mathcal{K}_{\mu,\sigma}(\bs{w}),\bs{g}}(\lambda,\mu+\sigma),
	\]
	and clearly \(q^{\langle 2\lambda+2\rho-2\nu'_3-\nu'_2,\nu'_2\rangle}(g_3\otimes g_2\otimes g_1) = \cD_\lambda(g_3\otimes g_2\otimes g_1)\).
\end{proof}
Formula \eqref{qKZBfromgraphics} already manifests the structure of the dual \(q\)-KZB equation for weighted trace functions of intertwiners from \cite{Etingof&Varchenko-2000}. The comparison to the standard form of the dual \(q\)-KZB equations will be postponed to Section \ref{Section Etingof-Varchenko normalization}.

\section{Graphical derivation of dual coordinate MR equations}
\label{Section dual MR}

In this section, we use the graphical calculus to derive a dual coordinate Macdonald-Ruijsenaars (MR) type equations for the spin components \(\mathcal{Z}_{\bs{S}}^{\bs{w},\bs{g}}(\lambda,\mu)\) (see \eqref{Zdef}) of the weighted trace functions. A subset of these equations will yield the dual MR equations for weighted trace functions from \cite{Etingof&Varchenko-2000}, as we shall show in 
Subsection \ref{Subsection Etingof-Varchenko normalization for MR}.

For the dual coordinate MR type equations it will suffice to work with the\ \textquotedblleft twofold\textquotedblright\ trace functions \(\mathcal{Z}_{\bs{S}}^{\bs{w},\bs{g}}(\lambda,\mu)\) with \(\bs{w} = (w_1,w_2)\) and \(\bs{g} = (g_2,g_1)\), obtained from (\ref{Y def}) upon setting \(S_3 = \emptyset\). 
Our approach will follow the same lines as Section \ref{Section dual q-KZB}. We will graphically derive operator versions of the equations for \(\mathcal{T}_{\bs{S}}^{\bs{w},\bs{g}}(\lambda,\mu)\in\End_{\cM_{\mr{adm}}^\str}(M_\mu\tens M_{\lambda})\).
Consequently, we will graphically apply weighted cyclic and highest-weight-to-highest-weight boundary conditions to the Verma strands to derive the desired equations for
$\mathcal{Z}_{\bs{S}}^{\bs{w},\bs{g}}(\lambda,\mu)$.
\subsection{Operator dual MR equations}
\label{Subsection Topological operator dual MR}

For $W\in\Rep$ consider the central element $\Omega_W\in Z(U_q)$ defined by 
\begin{equation}
\label{Omega_W def}
\Omega_W = (\id_{U_q(\g)}\otimes \Tr_W)(\cR^{21}\cR(1\otimes q^{2\rho})),
\end{equation}
see e.g.\ \cite{Drinfeld-1990, Reshetikhin-1990}. Let \(\chi_W\) be the formal character of \(W\), defined by 
\[
\chi_W(q^\xi) := \Tr_W(\pi_W(q^{\xi})) = \sum_{\sigma\in\wts(W)}\,q^{\langle\xi,\sigma\rangle}\dim(W[\sigma])\qquad (\xi\in\mathfrak{h}^*).
\]
We now have the following direct observation.
\begin{lemma}
	\label{lemma Omega tilde}
	For any \(\lambda\in\hh^*\) and \(W\in\Rep\) we have 
	\begin{figure}[H]
		\centering
		\includegraphics[scale=0.75]{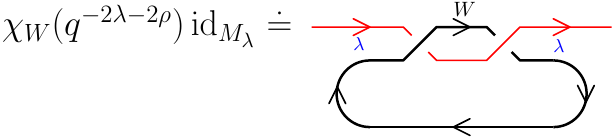}
		\caption{}
		\label{Omega tilde A}
	\end{figure}
	\noindent
	relative to the graphical calculus for $\mathcal{M}_{\textup{adm}}^\str$ and $\Rep^\str$.
\end{lemma}
\begin{proof}
	For any \(A\in \End_{\cM_{\mr{adm}}^\str}(M_\lambda\tens W)\), the graphical calculus of \(\cM_{\mr{adm}}\) asserts that 
	\begin{center}
		\includegraphics[scale=0.75]{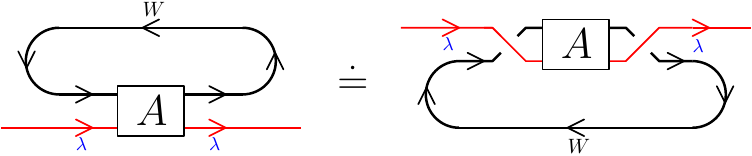}
	\end{center}
	Upon choosing \(A:= c^{-1}_{M_\lambda,W}c^{-1}_{W,M_\lambda}\),	we obtain
	\begin{figure}[H]
		\centering
		\includegraphics[scale=0.75]{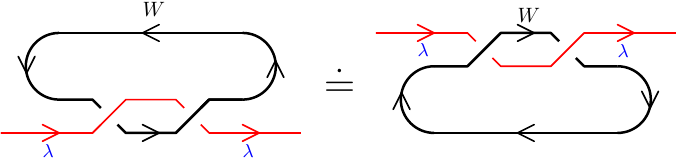}
		\caption{}
		\label{Omega tilde B}
	\end{figure}
	The left hand side of Figure \ref{Omega tilde B} is mapped to \(\pi_\lambda(\widetilde{\Omega}_W)\) by $\cF_{\mathcal{M}_{\textup{adm}}^\str}^{\textup{br}}$, with
	\begin{equation}\label{tildeOmega}
	\widetilde{\Omega}_W:= (\id_{U_q(\g)}\otimes\Tr_W)\bigl(\cR^{-1}(\cR^{21})^{-1}(1\otimes q^{2\rho})\bigr).
	\end{equation}
	Since \((\cR^{21})^{-1}\) is a universal $R$-matrix for \(U_q(\g)\), we have
	$\widetilde{\Omega}_W\in Z(U_q)$, hence $\pi_\lambda(\widetilde{\Omega}_W)$ is a scalar multiple of $\textup{id}_{M_\lambda}$. It suffices to show that the scalar
	is equal to $\chi_W(q^{-2\lambda-2\rho})$.
	
	Applying $\mathcal{F}_{\mathcal{M}_{\textup{adm}}^\str}^{\textup{br}}$ to Figure \ref{Omega tilde B} gives the identity
	\begin{equation}\label{hhh}
	\pi_\lambda(\widetilde{\Omega}_W)=(\id_{U_q(\g)}\otimes\Tr_W)\bigl((\cR^{21})^{-1}\cR^{-1}(1\otimes q^{-2\rho})\bigr).
	\end{equation}
Upon noting that \(\cR = \kappa\ol{\cR}\), where by \cite[Section 1.4]{DeClercq&Reshetikhin&Stokman-2022} the quasi R-matrix \(\ol{\cR}\) has its first tensor components in \(U^+\) and its second tensor components in \(U^-\), we find from \eqref{hhh} that
	\[
	\pi_\lambda(\widetilde{\Omega}_W)(\mathbf{m}_\lambda) 
	=(\pi_\lambda\otimes\Tr_W)\bigl((\cR^{21})^{-1}(1\otimes q^{-\lambda-2\rho}) \bigr)(\mathbf{m}_\lambda),
	\]
	compare with the proof of Lemma \ref{lemma 3.10}. The only component of \((\cR^{21})^{-1}\) that contributes to the trace over \(W\) is the component in \(U_q(\g)[0]\otimes U_q(\g)[0]\). Consequently
	\[
	\pi_\lambda(\widetilde{\Omega}_W)(\mathbf{m}_\lambda) = \Tr_W(\pi_W(q^{-2\lambda-2\rho}))\mathbf{m}_\lambda = \chi_W(q^{-2\lambda-2\rho})\,\mathbf{m}_\lambda,
	\]
	which gives the desired result.
\end{proof}

The analog of Lemma \ref{lemma before topological qKZB} in the current setting is as follows.
\begin{lemma}
	\label{lemma before topological MR}
	Let \(W\in\Rep\), \(\lambda_0,\lambda_1,\lambda_2\in\hh_{\mathrm{reg}}^\ast\), \(S_j\in\Rep^\str\) and \(\Psi_j\in\Hom_{\cM_{\mr{adm}}^\str}(M_{\lambda_j},S_j^\ast\tens M_{\lambda_{j-1}})\) for \(j= 1,2\). Let us write \(S:= S_1\tens S_2\) and set
	\begin{equation}
	\label{Psi def MR}
	\Psi:=(\id_{S_2^\ast}\tens\Psi_1)\Psi_2\in\Hom_{\cM_{\mr{adm}}^\str}(M_{\lambda_2}, S^\ast\tens M_{\lambda_0}).
	\end{equation}
	Then we have
	\begin{figure}[H]
		\centering
		\includegraphics[scale=0.75]{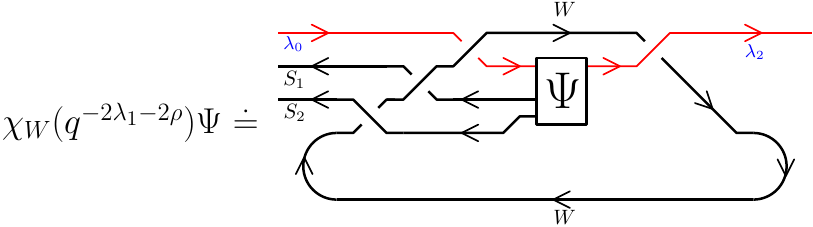}
		\caption{}
		\label{before topological MR}
	\end{figure}
	in the graphical calculus relative to $\cM_{\textup{adm}}^\str$ and $\Rep^\str$.
\end{lemma}
\begin{proof}
	By Lemma \ref{lemma Omega tilde}, the left-hand side of Figure \ref{before topological MR} can be graphically represented by
	\begin{center}
		\includegraphics[scale=0.75]{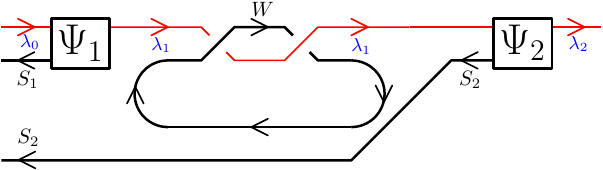}
	\end{center}
	Apply the graphical calculus for \(\Rep^\str\) to pull the strand labeled by $W$ underneath the strand labeled by $S_2$, to obtain
	\begin{center}
		\includegraphics[scale=0.75]{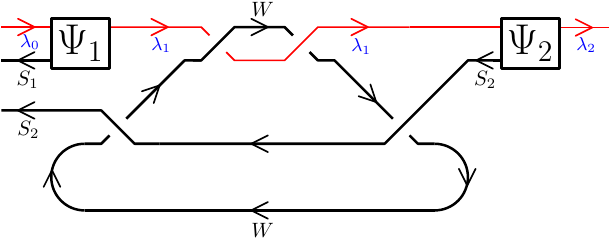}
	\end{center}
	Upon pulling the strand labeled by \(W\) over the coupon labeled by \(\Psi_1\) and under the coupon labeled by \(\Psi_2\), which is allowed in the graphical calculus for
	$\mathcal{M}_{\textup{adm}}$, we obtain
	\begin{center}
		\includegraphics[scale=0.75]{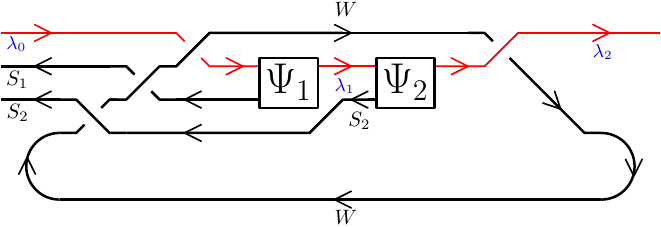}
	\end{center}
	which yields the right-hand side of Figure \ref{before topological MR} by the definition of \(\Psi\). 
\end{proof}

In the remainder of the subsection we fix 
\begin{equation}\label{S2}
S=S_1\tens S_2
\end{equation}
with $S_j\in\Rep^\str$, and we choose weight vectors $w_j\in\cF^\str(S_j)[\nu_j]$,  $g_j\in\cF^\str(S_j^\ast)[\nu_j']$
for $j=1,2$ with weights satisfying \(\nu_1+\nu_2 = 0 = \nu_1'+\nu_2'\). 
We furthermore fix \(\lambda,\mu\in\hh_{\mathrm{reg}}^\ast\) and set
\begin{equation}
\label{lambda_j MR def}
\lambda_0=\lambda_2:=\lambda, \quad \lambda_1= \lambda-\nu_2', \qquad \mu_1:=\mu-\nu_2.
\end{equation}
Define 
\begin{equation}\label{Psi2}
\Psi_{\lambda;S_2^*,S_1^*}^{g_2,g_1}:=(\textup{id}_{S_2^*}\tens\Psi_{\lambda_1;S_1^*}^{g_1})\Psi_{\lambda_2;S_2^*}^{g_2}\in\textup{Hom}_{\mathcal{M}_{\textup{adm}}^\str}\bigl(M_\lambda,S^*\tens M_\lambda\bigr)
\end{equation}
with $\Psi_{\lambda_j;S_j^*}^{g_j}$ as in (\ref{choice of Psi_i}).
Denote $\bs{S}=(S_1,S_2)$, $\bs{w}=(w_1,w_2)$ and $\bs{g}=(g_1,g_2)$, and set
\begin{equation}
\label{Y def MR}
\begin{split}
\mathcal{Y}_{\bs{S}}^{\bs{w},\bs{g}}(\lambda,\mu, \xi):=&\widehat{e}_{\ul{S}}\Bigl(\mathcal{H}_\mu^{\Phi_{\mu;S_1,S_2}^{w_1,w_2}}(q^\xi)\tens\langle\Psi_{\lambda;S_2^*,S_1^*}^{g_2,g_1}\rangle\Bigr)\\
=&\bigl(\textup{Tr}_{M_\mu}\otimes\langle\cdot\rangle\bigr)\bigl(\mathcal{T}_{\bs{S}}^{\bs{w},\bs{g}}(\lambda,\mu)\circ(\pi_\mu(q^\xi)\tens\textup{id}_{M_\lambda})\bigr)
\end{split}
\end{equation}
with
\[
\mathcal{T}_{\bs{S}}^{\bs{w},\bs{g}}(\lambda,\mu)= (\id_{M_{\mu}}\tens\widetilde{e}_S\tens\id_{M_{\lambda}})\left(\Phi_{\mu;S_1,S_2}^{w_1,w_2}\tens\Psi_{\lambda;S_2^*,S_1^*}^{g_2,g_1}\right).
\]
Note that these definitions are the special cases of \eqref{Y in terms of T} and \eqref{T_S def}
when $S_3=\emptyset$.

The following result describes the analog of Proposition \ref{prop bulk q-KZB} in the current context. 
\begin{proposition}
	\label{prop bulk MR}
	For $\Psi=\Psi_{\lambda;S_2^*,S_1^*}^{g_2,g_1}$ and \(W\in\Rep\), we have 
		\begin{figure}[H]
		\centering
		\includegraphics[scale=0.75]{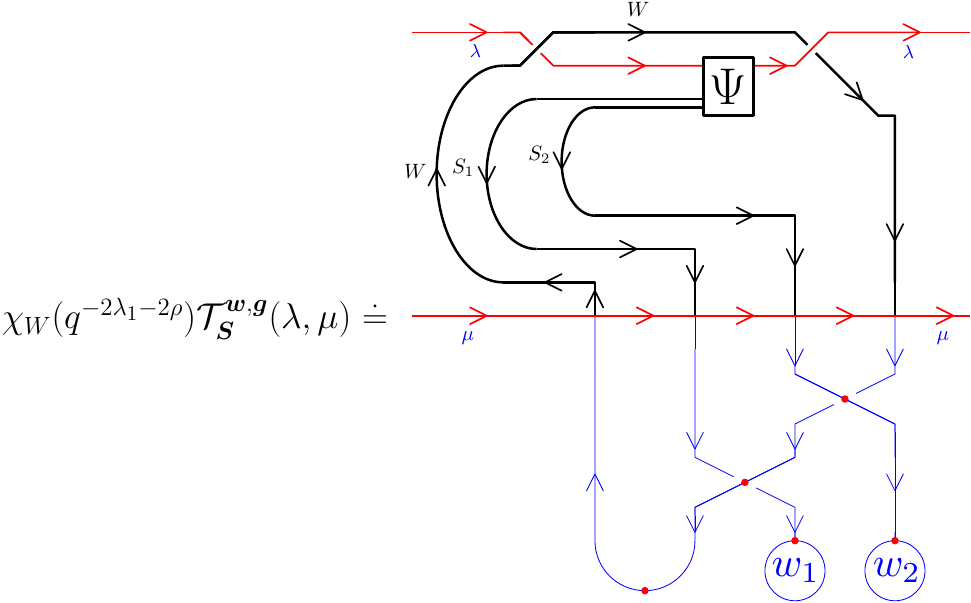}
		\caption{}
		\label{bulk MR diagram}
	\end{figure}
	\noindent
	relative to the graphical calculi for $\mathcal{M}_{\textup{adm}}^\str$ \textup{(}on black and red strands\textup{)}, $\Rep^\str$ \textup{(}on solely black strands\textup{)}
	and $\mathcal{N}_{\textup{fd}}^\str$ \textup{(}on blue strands\textup{)}. 
\end{proposition}
\begin{proof}
	Lemma \ref{lemma before topological MR} asserts that the left-hand side of Figure \ref{bulk MR diagram} can be graphically represented by
	\begin{center}
		\includegraphics[scale=0.75]{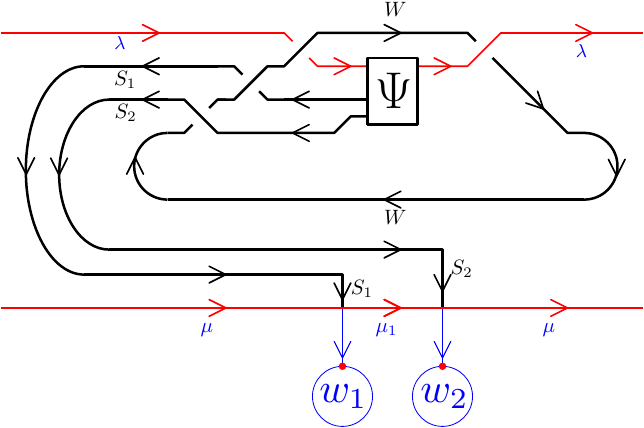}
	\end{center}
	Upon pushing the braidings through the caps representing \(\widetilde{e}_S\) and \(e_W\), 
	which are allowed moves in the graphical calculus for \(\Mfd^\str\), this becomes
	\begin{center}
		\includegraphics[scale=0.75]{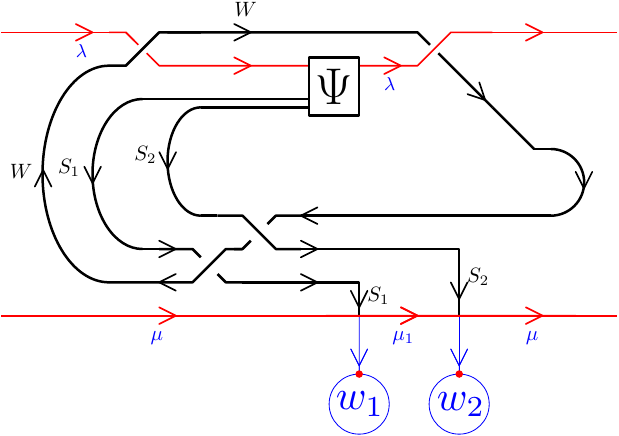}
	\end{center}
	If we now push first the cup labeled by \(W\) (representing \(\widetilde{\iota}_W\)) and then the two braidings through the red strand by invoking Proposition \ref{pushdiagram}, this transforms into
	\begin{figure}[H]
		\centering
		\includegraphics[scale=0.75]{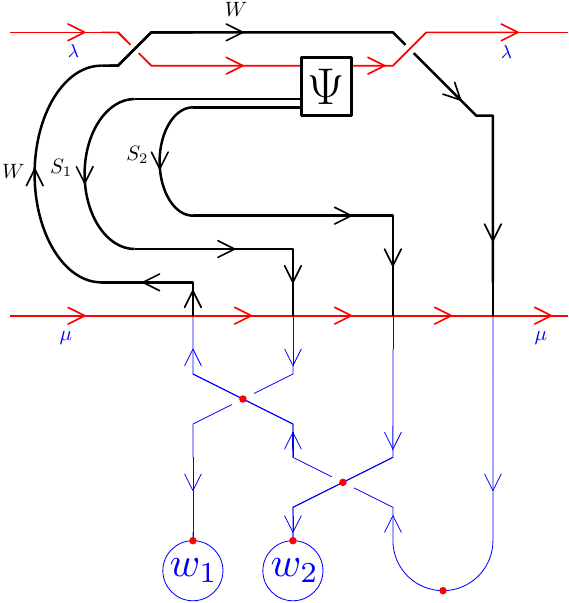}
		\caption{}
		\label{alternative bulk MR diagram}
	\end{figure}
	The claim now follows from application of Figures \ref{braided evaluation C}--\ref{braided evaluation D}.
\end{proof}

\subsection{The dual MR type equations}
\label{Subsection dual MR proof}

In analogy to Section \ref{Subsection Twisted trace functions}, we will now apply weighted cyclic boundary conditions and highest-weight-to-highest-weight components to the lower respectively upper Verma strands in both sides of Figure \ref{bulk MR diagram}. This will give rise to a difference equation of Macdonald-Ruijsenaars type for the spin 
components  \(\mathcal{Z}_{\bs{S}}^{\bs{w},\bs{g}}(\lambda,\mu)\)  (see \eqref{Zdef}) of weighted trace functions. We keep the conventions and notations as in the previous subsection (see in particular \eqref{S2}--\eqref{Y def MR}). 

\begin{lemma}
	\label{lemma boundary MR with general xi}
	For any \(W\in\Rep\) and \(\xi\in\hh^\ast\) with \(\Re(\xi)\) deep in the negative Weyl chamber, we have for $\Psi=\Psi_{\lambda;S_2^*,S_1^*}^{g_2,g_1}$,
	\begin{figure}[H]
		\centering
		\includegraphics[scale=0.7]{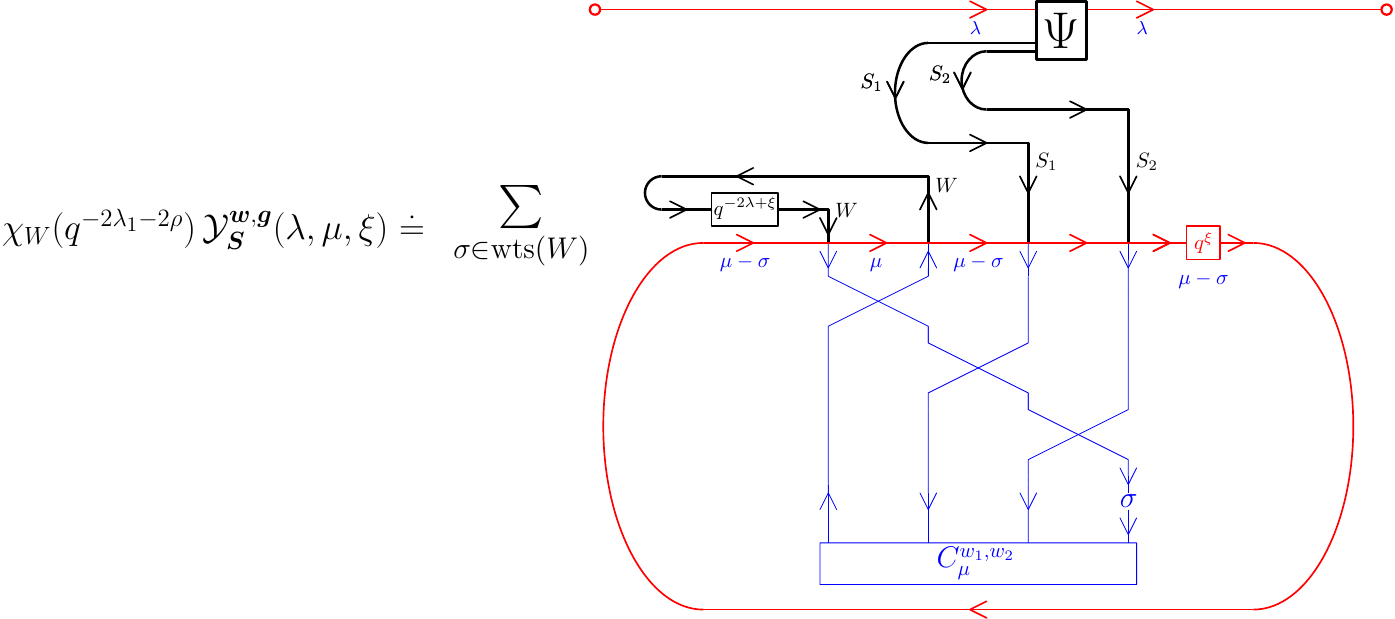}
		\caption{}
		\label{boundary MR A}
	\end{figure}
	\noindent
	relative to the graphical calculus for $\mathcal{N}_{\textup{adm}}^\str$ and $\mathcal{N}_{\textup{fd}}^\str$. Here $C_{\mu}^{w_1,w_2}$ is the morphism in $\mathcal{N}_{\textup{fd}}^\str$ such that 
	\begin{center}
		\includegraphics[scale=0.75]{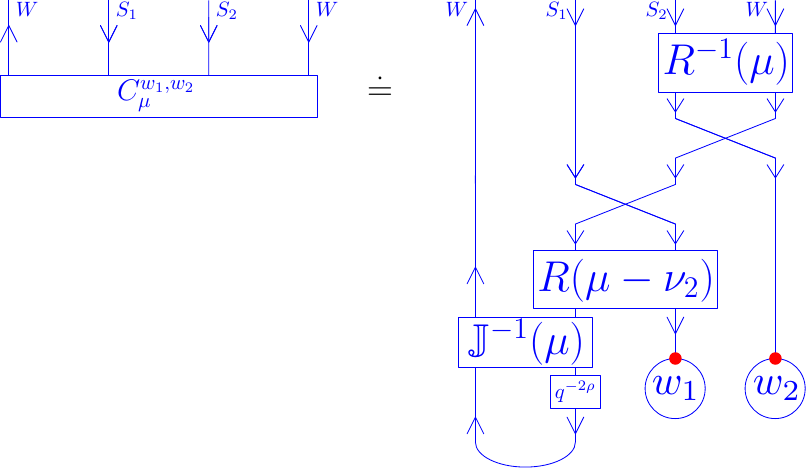}
	\end{center}
	\end{lemma}
\begin{proof}
	By Proposition \ref{prop bulk MR} and Figures \ref{braided evaluation C} \& \ref{braided evaluation D}, the left-hand side of Figure \ref{boundary MR A} is graphically represented by
	\begin{center}
		\includegraphics[scale=0.75]{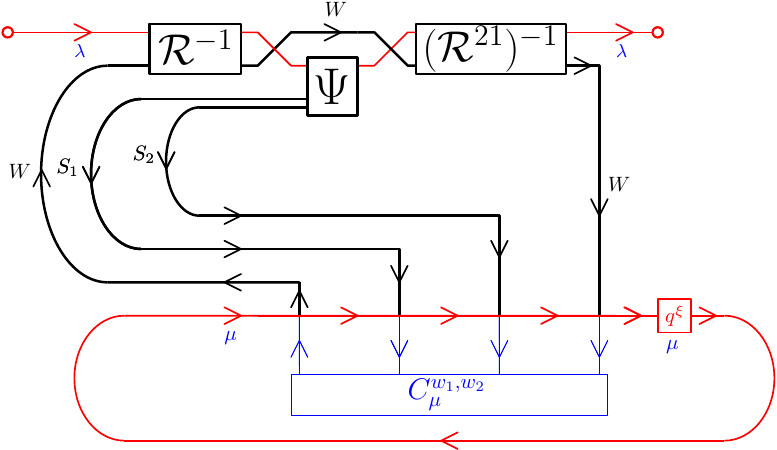}
	\end{center}
	Here we have again applied the reformulations in terms of the standard graphical calculus for \(\cN_{\mr{fd}}^\str\) and \(\cN_{\mr{adm}}^\str\) as discussed in \cite[Section 3.3]{DeClercq&Reshetikhin&Stokman-2022} and Section \ref{Subsection Boundary conditions}. In particular, the caps colored in black in the diagram above should be viewed within the graphical calculus for \(\cN_{\mr{fd}}^\str\), and we have immediately invoked the fact that \(\mathcal{H}_\mu^{\Phi_{\mu;S_1,S_2}^{w_1,w_2}}(q^\xi)\) is of weight 0 to remove the coupons labeled by \(q^{2\rho}\) on \(S_1\) and \(S_2\) resulting from this graphical reformulation. Note also that we have replaced Figure \ref{co-injection dynamical} with \(S = (W)\) and specialized to \(\lambda\), by Figure \ref{dyn co-eval right boundary}, which is justified by the discussion in Subsection \ref{Subsection Boundary conditions}. 
	
	As in the proof of Lemma \ref{lemma boundary qKZB with general xi}, we can use Lemma \ref{lemma 3.10}, apply \(\mathbb{P}_W[\sigma]\) and sum over all \(\sigma\in\wts(W)\) to replace this diagram by
	\begin{center}
		\includegraphics[scale=0.75]{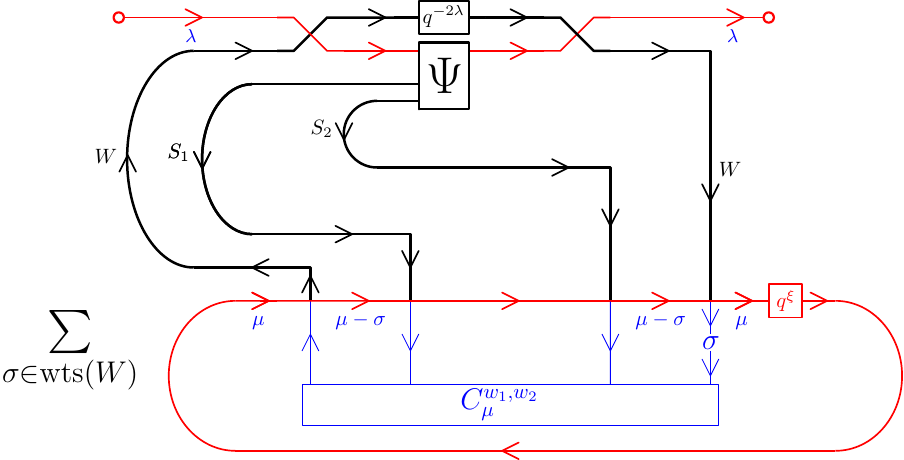}
	\end{center}
	Note that requiring the weight in the rightmost blue strand to yield \(\sigma\) forces both the vertical regions below the red strand left of the rightmost and right of the leftmost blue strand, to be colored by the weight \(\mu-\sigma\), by the same argument as in the proof of Lemma \ref{lemma boundary qKZB with general xi}. 
	
	By Lemma \ref{lemma cyclicity} the diagram above can be turned into
	\begin{center}
		\includegraphics[scale=0.75]{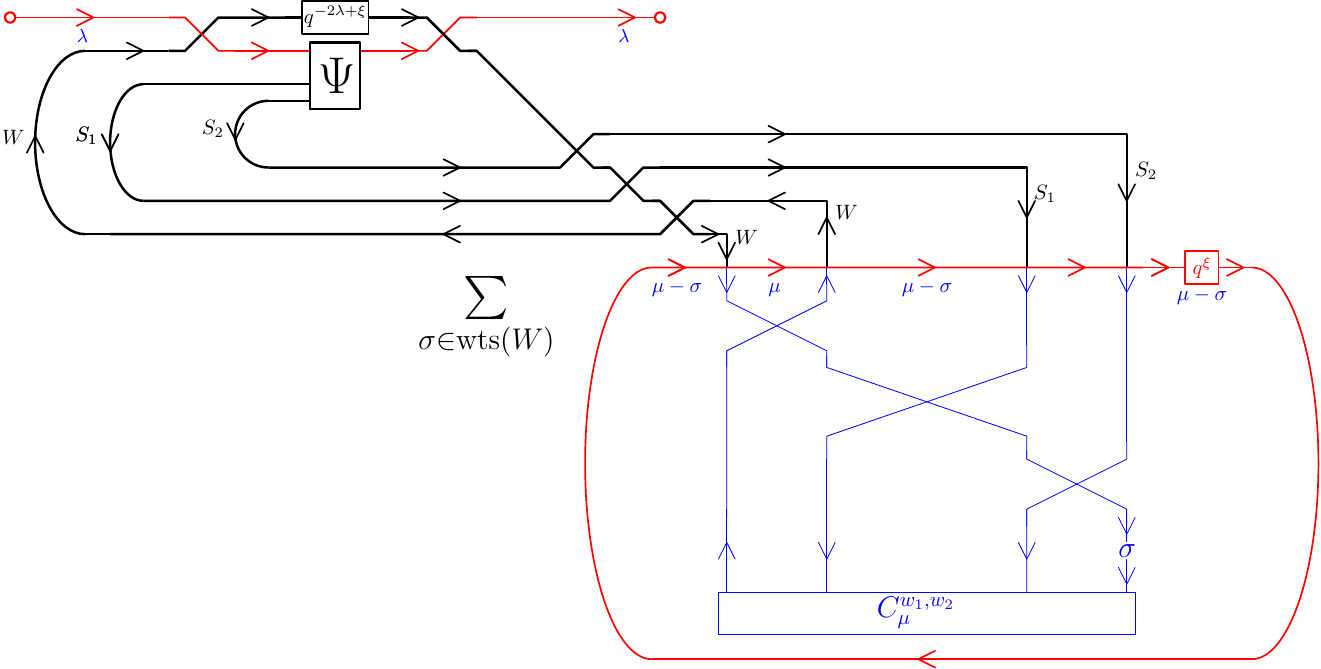}
	\end{center}
	Finally we contract the loop labeled by \(W\) to arrive at the right-hand side of Figure \ref{boundary MR A}.
\end{proof}

The final step will once more be to choose the twisting parameter \(\xi\) in such a way that the remaining cap colored by \(\ul{W}\) in Figure \ref{boundary MR A} can be pulled through the Verma strand via Proposition \ref{pushdiagram}. This requires the identification of this cap with the cap colored by \(W\) in the graphical calculus for \(\Mfd^\str\), which dictates to choose \(\xi = 2\lambda+2\rho\), in agreement with (\ref{choice of xi}). Let us once more write 
\begin{equation}\label{Zspec}
\mathcal{Z}_{\boldsymbol{S}}^{\bs{w},\bs{g}}(\lambda,\mu)=\mathcal{Y}_{\boldsymbol{S}}^{\bs{w},\bs{g}}(\lambda,\mu,2\lambda+2\rho), 
\end{equation}
as in (\ref{set Z equal to}). We then obtain the following result.
\begin{proposition}
	\label{prop also for MR}
	For any \(W\in\Rep\) and \(\lambda\in\hh_{\mathrm{reg}}^\ast\) such that \(\Re(\lambda)\) lies deep in the negative Weyl chamber, we have 
	for $\Psi=\Psi_{\lambda;S_2^*,S_1^*}^{g_2,g_1}$,
	\begin{figure}[H]
		\centering
		\includegraphics[scale=0.75]{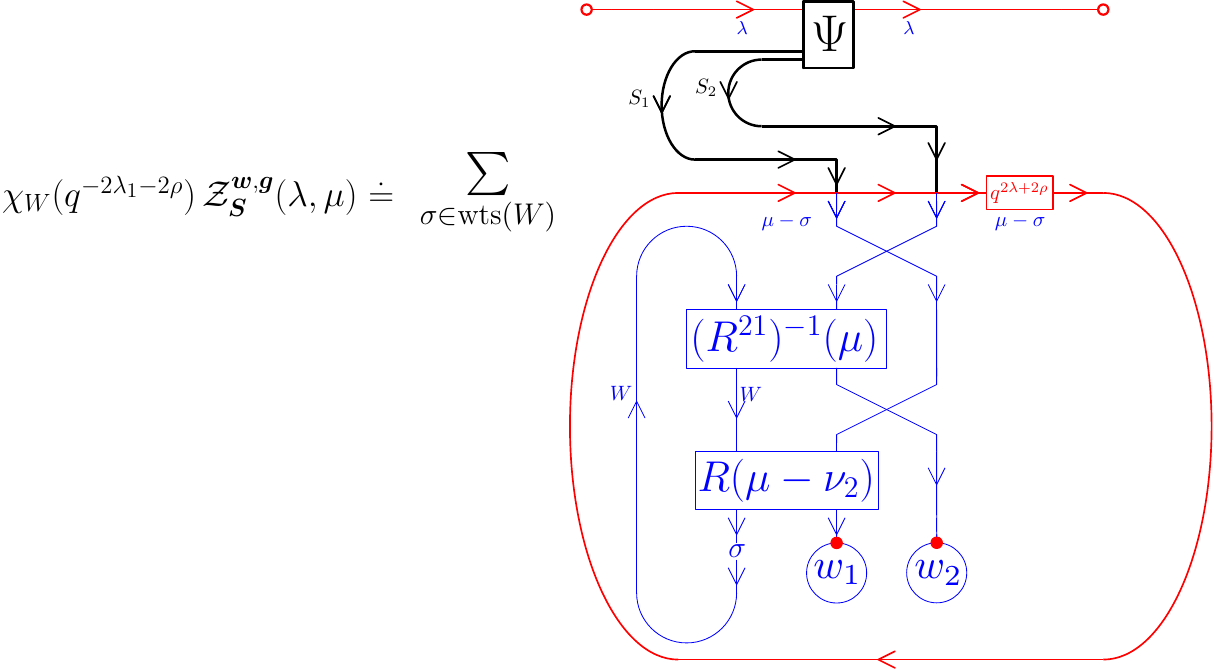}
		\caption{}
		\label{D7}
	\end{figure}
\end{proposition}
\begin{proof}
	By Proposition \ref{pushdiagram} we have 
	\begin{center}
		\includegraphics[scale=0.9]{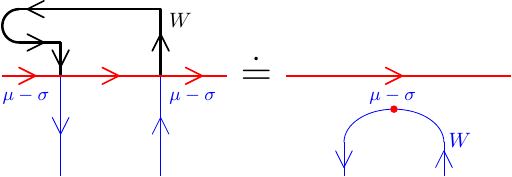}
	\end{center}
	As in the proof of Proposition \ref{prop boundary qKZB}, we may use this to manipulate the remaining cap colored by \(\ul{W}\) in Figure \ref{boundary MR A}, by pulling it through the red strand labeled by \(M_{\mu-\sigma}\). This is only allowed since with the present specialization \(\xi = 2\lambda+2\rho\), the latter cap is colored by the morphism \(\widehat{e}_{\ul{W}}\circ(\pi_W(q^{2\rho})\tens\id_{\ul{W^\ast}}) \) in \(\cN_{\mr{fd}}^\str\). This morphism algebraically coincides with \(\til{\cF^{\mr{frgt}}}(\til{e}_{W})\), and hence this cap can be identified with the cap with the same orientation and colored by \(W\) in the graphical calculus for \(\Mfd^\str\).
	
	In combination with the result of Lemma \ref{lemma boundary MR with general xi} and Figure \ref{dyn eval right boundary}, we conclude that the left-hand side of Figure \ref{D7} is graphically represented by
	\begin{center}
		\includegraphics[scale=0.75]{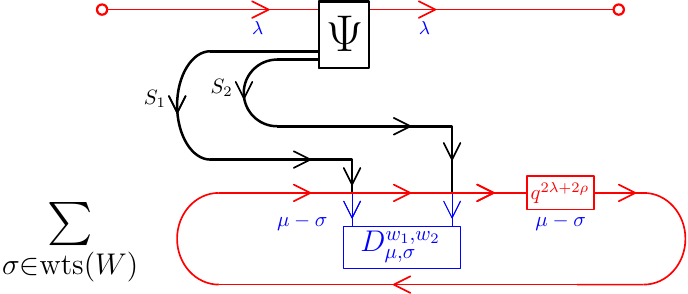}
	\end{center}
	\noindent
	with $D_{\mu\sigma}^{w_1,w_2}$ the morphism in $\mathcal{N}_{\textup{fd}}^\str$ such that
	\begin{figure}[H]
		\centering
		\includegraphics[scale=0.75]{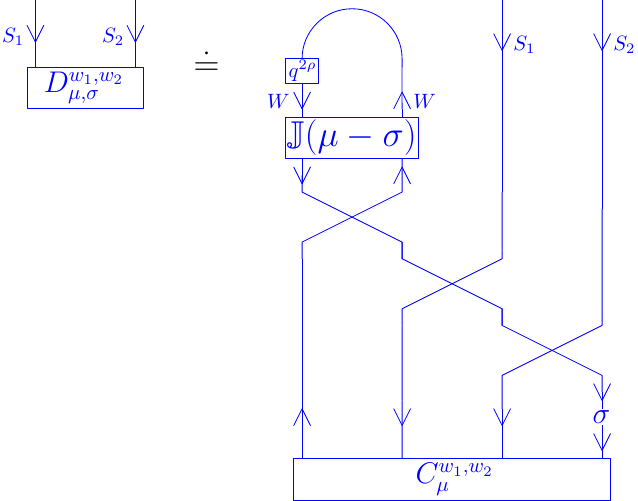}
		\caption{}
		\label{boundary MR H}
	\end{figure}
	Using the graphical calculus for $\mathcal{N}_{\textup{fd}}^\str$ and the results from Subsection \ref{Subsection Boundary conditions}, the 
	$\mathcal{F}_{\mathcal{N}_{\textup{fd}}^\str}^{\textup{RT}}$-image of the right-hand side of Figure \ref{boundary MR H} and of 
	\begin{center}
		\includegraphics[scale=0.75]{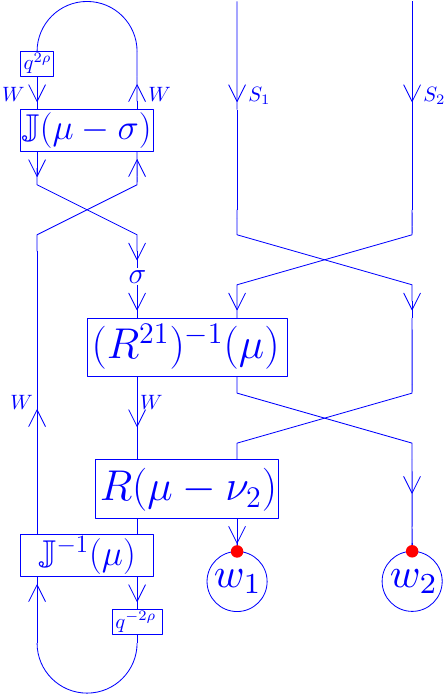}
	\end{center}
	\noindent
	coincide.
	Note that one can insert an additional coupon colored by the projection morphism \(\mathbb{P}_W[\sigma]\) below the coupon labeled by \(\Rdyn(\mu-\nu_2)\), since \(\mathbb{J}(\mu-\sigma)\) and \(\mathbb{J}^{-1}(\mu)\) are of zero weight. Consequently, by Lemma \ref{lemma in Section 3.3}, the $\mathcal{F}_{\mathcal{N}_{\textup{fd}}^\str}^{\textup{RT}}$-image of the right-hand side of Figure \ref{boundary MR H} and of 
	\begin{center}
		\includegraphics[scale=0.75]{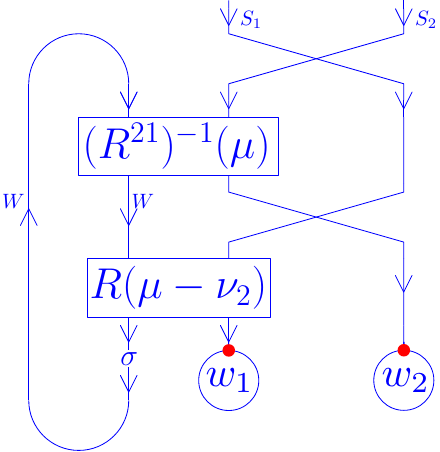}
	\end{center}
\noindent 
coincide. This proves the result.
\end{proof}

We give the reformulation of Proposition \ref{prop also for MR} as a concrete equation for $\mathcal{Z}_{\bs{S}}^{\bs{w},\bs{g}}(\lambda,\mu)$. In Section \ref{Subsection Etingof-Varchenko normalization for MR} we will show how this equation gives rise to the dual Macdonald-Ruijsenaars (MR) equations for weighted trace functions of $k$-point quantum vertex operators from \cite{Etingof&Varchenko-2000}, as well as to new families of dual MR equations. 

\begin{corollary}
	\label{thm MR for Y}
	Consider arbitrary \(\lambda,\mu\in\hh_{\mathrm{reg}}^\ast\) with \(\Re(\lambda)\) deep in the negative Weyl chamber, \(W\in\Mfd\), \(S = S_1\tens S_2\in \Rep^\str\), \(w_j\in\cF^\str(S_j)[\nu_j]\) and \(g_j\in \cF^\str(S_j^\ast)[\nu_j]\), for \(j\in\{1,2\}\), with \(\nu_1+\nu_2 = 0=\nu_1'+\nu_2'\). Then 
	\begin{equation}
	\label{equation MR with general xi}
	\mathcal{Z}_{\bs{S}}^{\bs{w},\mathcal{D}_{\lambda,W}^\vee(\bs{g})}(\lambda,\mu) = \sum_{\sigma\in\wts(W)}\mathcal{Z}_{\bs{S}}^{\mathcal{K}_{\mu,\sigma,W}^\vee(\bs{w}), \bs{g}}(\lambda,\mu-\sigma),
	\end{equation}
	with \(\mathcal{K}_{\mu,\sigma,W}^\vee\in \End_{\cN_{\textup{fd}}}(\cF^\str(\ul{S}))\) and \(\mathcal{D}_{\lambda,W}^\vee\in\End_{\cN_{\textup{fd}}}(\cF^\str(\ul{S^\ast}))\) defined by
	\begin{align}
	\label{u vectors def MR}
	\mathcal{K}_{\mu,\sigma,W}^\vee :=\ & \mathscr{J}^\vee_{\boldsymbol{S}}(\mu-\sigma)\, \Tr_{W[\sigma]}\left(\Rdyn^{21}_{W,S_2}(\mu)^{-1}\Rdyn_{W,S_1}(\mu-\mh_{S_2}) \right)\mathscr{J}^\vee_{\boldsymbol{S}}(\mu-\sigma)^{-1}, \\
	\label{the funny J check}
	\mathscr{J}^\vee_{\boldsymbol{S}}(\mu-\sigma) :=\ & j_{S_1}(\mu-\sigma-\mh_{S_2})\otimes j_{S_2}(\mu-\sigma), \\
	\label{the funny D check}
	\mathcal{D}_{\lambda,W}^\vee:=\ & \Tr_W\left((q^{-2\lambda-2\rho})_W\,\kappa^2_{W,S_2^\ast}\right).
	\end{align}
\end{corollary}
\begin{proof}
	Proposition \ref{prop also for MR} asserts that 
	\begin{figure}[H]
		\centering
		\includegraphics[scale=0.65]{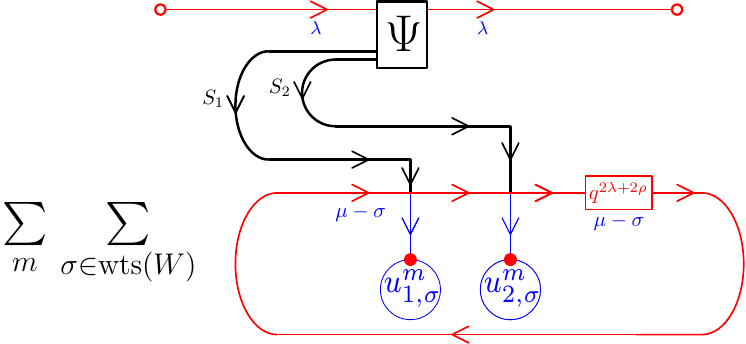}
		\caption{}
		\label{D8}
	\end{figure}
	\noindent 
	for $\Psi=\Psi_{\lambda;S_2^*,S_1^*}^{g_2,g_1}$ is a graphical representation of \(\chi_W(q^{-2\lambda_1-2\rho})\,\mathcal{Z}_{\bs{S}}^{\bs{w},\bs{g}}(\lambda,\mu)\),
	where the vectors \(u_{i,\sigma}^m\in\cF^\str(S_i)\) are homogeneous such that
	\[
	\sum_m u_{1,\sigma}^m\otimes u_{2,\sigma}^m = \mathcal{K}_{\mu,\sigma,W}^\vee(w_1\otimes w_2).
	\]
	Observe that Figure \ref{D8} represents
	\[
	\sum_m \sum_{\sigma\in\wts(W)}\mathcal{Z}_{\boldsymbol{S}}^{(u_{1,\sigma}^m, u_{2,\sigma}^m),\bs{g}}(\lambda,\mu-\sigma),
	\]
	such that we obtain the equation
	\begin{equation}
	\label{almost a conclusion}
	\chi_W(q^{-2\lambda_1-2\rho})\,\mathcal{Z}_{\bs{S}}^{\bs{w},\bs{g}}(\lambda,\mu) = \sum_{\sigma\in\wts(W)}\mathcal{Z}_{\bs{S}}^{\mathcal{K}_{\mu,\sigma,W}^\vee(\bs{w}), \bs{g}}(\lambda,\mu-\sigma).
	\end{equation}
	The definition of the character \(\chi_W\), in combination with the expression (\ref{lambda_j MR def}) for \(\lambda_1\), asserts that
	\[
	\chi_W(q^{-2\lambda_1-2\rho})(g_2\otimes g_1) = \mathcal{D}_{\lambda,W}^\vee(g_2\otimes g_1),
	\]
	such that the left-hand side of (\ref{almost a conclusion}) equals \(\mathcal{Z}_{\bs{S}}^{\bs{w},\mathcal{D}_{\lambda,W}^\vee(\bs{g})}(\lambda,\mu)\).
\end{proof}

\section{The $q$-KZB and coordinate MR equations}
\label{Section Etingof-Varchenko normalization}

In this section we show that the equations for spin components of weighted traces of intertwiners that we obtained using graphical calculus (see Corollary \ref{thm boundary qKZB} and Corollary \ref{thm MR for Y}) entail the dual $q$-KZB and dual MR equations for weighted traces of intertwiners, as derived in  \cite{Etingof&Varchenko-2000} by intricate algebraic computations. 

The dual $q$-KZB equations are a system of $k$ difference equations for weighted traces of $k$-point quantum vertex operators, cf. \cite{Etingof&Varchenko-2000}. The $i^{\textup{th}}$ equation ($1\leq i\leq k$) will follow from Corollary \ref{thm boundary qKZB} by choosing the decomposition $S=S_1\tens S_2\tens S_3$
of $S=(V_1,\ldots,V_k)\in\Rep^\str$ with 
\[
S_1 = (V_1,\dots,V_{i-1}), \quad S_2 = (V_i), \quad S_3 = (V_{i+1},\dots,V_k).
\]
The system of dual MR equations for weighted traces of $k$-point quantum vertex operators from \cite{Etingof&Varchenko-2000} is a family of difference equations naturally parametrised by $W\in\Rep$. In this section we derive from Corollary \ref{thm MR for Y} $k+1$ such systems of dual $q$-MR equations. The $i^{\textup{th}}$
system ($0\leq i\leq k$) corresponds to choosing in Corollary \ref{thm MR for Y} the decomposition $S=S_1\tens S_2$ of $S=(V_1,\ldots,V_k)\in\Rep^\str$ with
\[
S_1 = (V_1,\dots,V_i), \quad S_2 = (V_{i+1},\dots,V_k).
\]
The original system of dual MR equations from \cite{Etingof&Varchenko-2000} then corresponds to the special case that $i=k$. In Subsection \ref{Subsection Etingof-Varchenko normalization} we discuss the dual $q$-KZB equations, in Subsection \ref{Subsection Etingof-Varchenko normalization for MR} the dual MR equations.

\subsection{The dual \(q\)-KZB equations}
\label{Subsection Etingof-Varchenko normalization}
To relate the equations for the spin components $\mathcal{Z}_{\bs{S}}^{\bs{w},\bs{g}}(\lambda,\mu)$ of weighted traces of intertwiners obtained in Corollary \ref{thm boundary qKZB} to the dual $q$-KZB equations from \cite[Thm. 1.4]{Etingof&Varchenko-2000}, we need to express the parametrisation of the spin components in terms of explicit linear operators acting on $\cF^\str(S^*)[0]$. This boils down to explicitly computing the expectation values of the dual intertwiners
 \(\Psi_{\lambda;S_3^\ast,S_2^\ast,S_1^\ast}^{g_3,g_2,g_1}\) (see (\ref{bold Psi def})--(\ref{choice of Psi_i})).

Recall the notational conventions introduced at the start of Section \ref{Section dual q-KZB}. In particular we have $\lambda\in\mathfrak{h}_{\textup{reg}}^*$,  $S=S_1\tens S_2\tens S_3\in\Rep^\str$ and $g_j\in\cF^\str(S_j^*)[\nu_j']$ for $j=1,2,3$ with the weight $\nu_j'$ such that
$\nu_1'+\nu_2'+\nu_3'=0$. 
\begin{lemma}
	\label{lemma expectation value of Psi}
	The expectation value \(\langle \Psi_{\lambda;S_3^\ast,S_2^\ast,S_1^\ast}^{g_3,g_2,g_1}\rangle\in \Hom_{\cN_{\mr{fd}}^\str}(\CC_\lambda,\ul{S^\ast}\tens \CC_\lambda)\) is represented by the vector
	\[
	j_{S^\ast}(\lambda)\,\Rdyn_{S_2^\ast,S_1^\ast}(\lambda)^{-1}\Rdyn_{S_3^\ast,S_1^\ast}(\lambda-\mathrm{h}_{S_2^\ast})^{-1}\Rdyn_{S_3^\ast,S_2^\ast}(\lambda)^{-1}\,\left(j_{S_3^\ast}(\lambda_3)^{-1}g_3\otimes j_{S_2^\ast}(\lambda_2)^{-1}g_2\otimes j_{S_1^\ast}(\lambda_1)^{-1}g_1\right)
	\]
	in \(\cF^\str(\ul{S^\ast})[0]\),
	with \(\lambda_j\) as defined in (\ref{lambda_j def}).
\end{lemma}
\begin{proof}
	By definition, \(\Psi_{\lambda;S_3^\ast,S_2^\ast,S_1^\ast}^{g_3,g_2,g_1}\) is graphically represented by
	\begin{center}
		\includegraphics[scale=0.67]{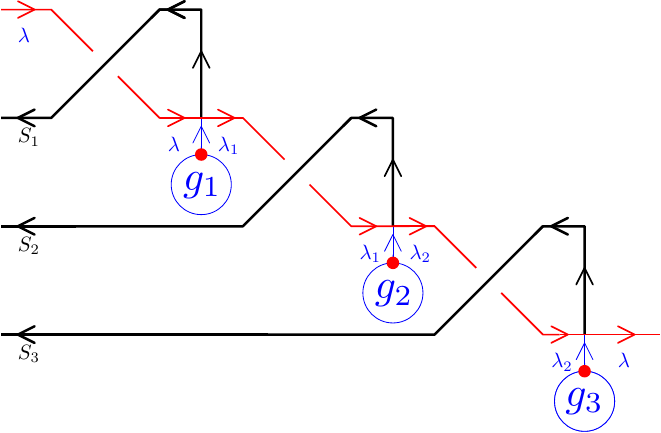}
	\end{center}
	Upon pulling the strands labeled by \(S_2\) and \(S_3\) over the coupons colored by $\Psi_{\lambda_1;S_1^*}^{g_1}$ and $\Psi_{\lambda_2;S_2^*}^{g_2}$ using the graphical calculus for $\mathcal{M}_{\textup{adm}}^\str$, we obtain
	\begin{center}
		\includegraphics[scale=0.75]{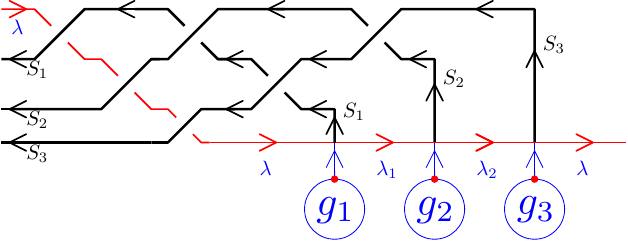}
	\end{center}
as graphical representation of \(\Psi_{\lambda;S_3^\ast,S_2^\ast,S_1^\ast}^{g_3,g_2,g_1}\).
After pushing the braidings on the black strands through the red strand using Proposition \ref{pushdiagram}, and using Figure \ref{evaluation in S}, we obtain
	\begin{center}
		\includegraphics[scale=0.75]{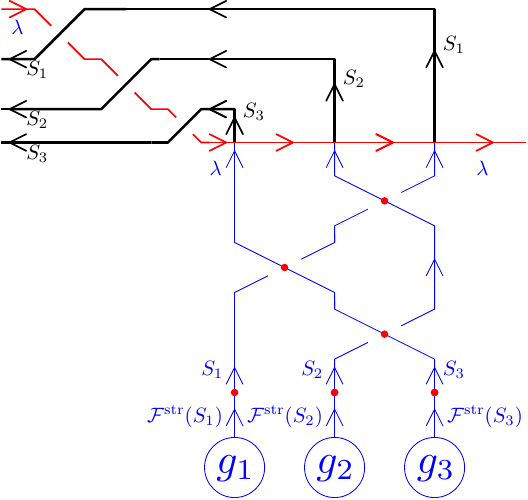}
	\end{center}
	as graphical representation. Lemma \ref{lemma 3.10} then shows that
	\begin{figure}[H]
		\centering
		\includegraphics[scale=0.75]{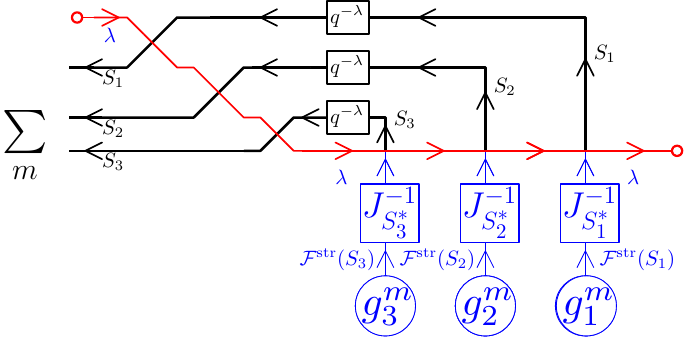}
		\caption{}
		\label{Psi D}
	\end{figure}
	\noindent 
	is a graphical representation of the expectation value \(\langle\Psi_{\lambda;S_3^\ast,S_2^\ast,S_1^\ast}^{g_3,g_2,g_1}\rangle\),
	where the blue coupons are colored by non-dynamical fusion operators, and where \(g_i^m\in\cF^\str(S_i^\ast)\) are homogeneous vectors such that
	\begin{equation*}
	\begin{split}
	&\sum_m g_3^m\otimes g_2^m\otimes g_1^m =\\
	&\,\,=\Rdyn_{S_2^\ast,S_1^\ast}(\lambda)^{-1}\Rdyn_{S_3^\ast,S_1^\ast}(\lambda-\mathrm{h}_{S_2^\ast})^{-1}\Rdyn_{S_3^\ast,S_2^\ast}(\lambda)^{-1}\left(j_{S_3^\ast}(\lambda_3)^{-1}g_3\otimes j_{S_2^\ast}(\lambda_2)^{-1}g_2\otimes j_{S_1^\ast}(\lambda_1)^{-1}g_1\right).
	\end{split}
	\end{equation*}
	Note that Figure \ref{Psi D} graphically represents
	\[
	\pi_{\cF^\str(S^\ast)}(q^{-\lambda})j_{S^\ast}(\lambda)\left(\sum\nolimits_{m}g_3^m\otimes g_2^m\otimes g_1^m\right).
	\]
	The result now follows from the fact that the factor \(\pi_{\cF^\str(S^\ast)}(q^{-\lambda})\) can be omitted from the above expression, since it acts on a zero weight space.
\end{proof}

Fix $i\in\{1,\ldots,k\}$ and fix $\lambda,\mu\in\mathfrak{h}^*_{\mathrm{reg}}$ with $\Re(\lambda)$ deep in the negative Weyl chamber. We now first specialise Corollary \ref{thm boundary qKZB} to obtain the $i^{\textup{th}}$ dual $q$-KZB equation for certain spin components of weighted traces of $k$-point quantum vertex operators. For this we specialize $\bs{S}, \bs{w}$ and $\bs{g}$ in the definition \eqref{T_S def}--\eqref{Zdef} of $\mathcal{Z}_{\bs{S}}^{\bs{w},\bs{g}}(\lambda,\mu)$ in the following way.
 
In the remainder of this subsection we take $S=(V_1,\ldots,V_k)\in\Rep^\str$ and decompose it as $S=S_1\tens S_2\tens S_3$ with
\begin{equation}\label{idecomp}
S_1 = (V_1,\dots,V_{i-1}), \quad S_2 = (V_i), \quad S_3 = (V_{i+1},\dots,V_k).
\end{equation}
Here $S_1$ (resp. $S_3$) should be read as $\emptyset$ if $i=1$ (resp. $i=k$). Note that the resulting $3$-tuple $\mathbf{S}=(S_1,S_1,S_3)$ depends on $i$, while $S=S_1\tens S_2\tens S_3$ does not.
We parametrise the $3$-tuples of (co)vectors $\bs{w}=(w_1,w_2,w_3)$ and $\bs{g}=(g_1,g_2,g_3)$ as follows.

Let $\xi_j\in\textup{wts}(V_j)$ and $\xi_j'\in\textup{wts}(V_j^\ast)$ such that 
\[
\xi_1+\dots+\xi_k = 0 = \xi_1'+\dots+\xi_k',
\]
and set
\begin{equation*}
\begin{split}
\nu_1&:=\xi_1+\cdots+\xi_{i-1},\qquad \nu_2:=\xi_i,\qquad \nu_3:=\xi_{i+1}+\cdots+\xi_k,\\
\nu_1'&:=\xi_1'+\cdots+\xi_{i-1}',\qquad \nu_2':=\xi_i',\qquad \nu_3':=\xi_{i+1}'+\cdots+\xi_k'.
\end{split}
\end{equation*}
These weights give rise to the shifted weights $\lambda_j$ and $\mu_j$ ($j=0,1,2$) as defined by \eqref{lambda_j def} and \eqref{mu_j def},
\begin{equation}\label{lambdamushifts}
\begin{split}
\lambda_1&=\lambda-\xi_i'-\cdots-\xi_k',\qquad \lambda_2=\lambda-\xi_{i+1}'-\cdots-\xi_k',\\
\mu_1&=\mu-\xi_i-\cdots-\xi_k,\qquad\,\mu_2=\mu-\xi_{i+1}-\cdots-\xi_k.
\end{split}
\end{equation}
The corresponding $3$-tuples $\bs{w}=(w_1,w_2,w_3)$ and $\bs{g}=(g_1,g_2,g_3)$ with vectors $w_j\in\cF^\str(S_j)[\nu_j]$ and covectors
$g_j\in\cF^\str(S_j^*)[\nu_j']$ are
\begin{equation*}
\begin{split}
w_1&:= j_{S_1}(\mu_{1})(v_1\otimes\dots\otimes v_{i-1}), \qquad w_2 := v_i, \qquad w_3 := j_{S_3}(\mu)(v_{i+1}\otimes\dots\otimes v_k),\\
g_1 &:= j_{S_1^*}(\lambda_{1})(f_{i-1}\otimes\dots\otimes f_1), \qquad g_2 := f_i, \qquad g_3 := j_{S_3^*}(\lambda)(f_k\otimes\dots\otimes f_{i+1})
\end{split}
\end{equation*} 
with \(v_j\in V_j[\xi_j]\) and \(f_j\in V_j^\ast[\xi'_j]\).

\begin{definition}\label{Thetai}
We write
\begin{equation}
\label{Theta def}
\Theta_{S,i}^{v_1,\dots,v_k\vert f_k,\dots, f_1}(\lambda,\mu):= \widehat{e}_{\ul{S}}\left(\Tr_{M_\mu}\bigl(\Phi_\mu^{v_1,\dots,v_k}\circ\pi_\mu(q^{2\lambda+2\rho})\bigr)\otimes\mathbb{B}_{\lambda,i}(f_k\otimes\dots\otimes f_1) \right)
\end{equation}
with the endomorphism
\(\mathbb{B}_{\lambda,i}\in\End_{\cN_{\textup{fd}}}(\cF^\str(\ul{S^\ast}))\) defined by
\begin{equation}
\label{def B_lambda,i}
\mathbb{B}_{\lambda,i}:=j_{S^\ast}(\lambda)\Rdyn_{V_i^\ast,(V_{i-1}^\ast,\dots,V_1^\ast)}(\lambda)^{-1}\Rdyn_{(V_k^\ast,\dots,V_{i+1}^\ast),(V_{i-1}^\ast,\dots,V_1^\ast)}(\lambda-\mathrm{h}_{V_i^\ast})^{-1}\Rdyn_{(V_k^\ast,\dots,V_{i+1}^\ast),V_i^\ast}(\lambda)^{-1}.
\end{equation}
\end{definition}
In \eqref{Theta def} we have written $\widehat{e}_{\underline{S}}$ for $\cF^\str(\widehat{e}_{\underline{S}})$ for the map underlying the right evaluation morphism of $\cN_{\textup{fd}}^\str$, viewed as map $\cF^\str(\ul{S})\otimes\cF^\str(\ul{S^*})\rightarrow\mathbb{C}_0$. In other words,
\begin{equation}\label{hateN}
\widehat{e}_{\ul{S}}(u_1\otimes\cdots\otimes u_k,h_k\otimes\cdots\otimes h_1)=h_1(u_1)\cdots h_k(u_k) \qquad (u_i\in V_i,\, h_i\in V_i^*).
\end{equation}

Note that $\Theta_{S,i}^{v_1,\dots,v_k\vert f_k,\dots, f_1}(\lambda,\mu)$ is a spin component of the $q^{2\lambda+2\rho}$-weighted trace of the $k$-point quantum vertex operator $\Phi_\mu^{v_1,\dots,v_k}$. The linear operator $\mathbb{B}_{\lambda,i}$ describes how the spin component is deviating from the standard parametrisation of the spin components.
\begin{remark}
Note that the $i$-dependence of $\Theta_{S,i}^{v_1,\dots,v_k\vert f_k,\dots, f_1}(\lambda,\mu)$ is captured by the linear operator $\mathbb{B}_{\lambda,i}\vert_{\cF^\str(\ul{S^*})[0]}$ on $\cF^\str(\ul{S^*})[0]$. 
\end{remark}

With the current choices of $\bs{S}$, $\bs{w}$ and $\bs{g}$ the spin components $\mathcal{Z}_{\bs{S}}^{\bs{w},\bs{g}}(\lambda,\mu)$ of the weighted traces of intertwiners occurring in the equations in Corollary \ref{thm boundary qKZB} reduce to $\Theta_{S,i}^{v_1,\dots,v_k\vert f_k,\dots, f_1}(\lambda,\mu)$.
\begin{lemma}
	\label{prop Theta is Z}
	We have
	\[
	\mathcal{Z}_{\bs{S}}^{\bs{w},\bs{g}}(\lambda,\mu)=\Theta_{S,i}^{v_1,\dots,v_k\vert f_k,\dots, f_1}(\lambda,\mu).
	\]
\end{lemma}
\begin{proof}
	By \eqref{t functions def}, \eqref{Y def} and \eqref{Zdef} we have
	\[
	\mathcal{Z}_{\bs{S}}^{\bs{w},\bs{g}}(\lambda,\mu) = \widehat{e}_{\ul{S}}\left(H_\mu^{\Phi_{\mu;S_1,S_2,S_3}^{w_1,w_2,w_3}}(q^{2\lambda+2\rho})\otimes\langle\Psi\rangle \right)
	\]
	with \(\Psi:=\Psi_{\lambda;S_3^\ast,S_2^\ast,S_1^\ast}^{g_3,g_2,g_1}\) given by (\ref{bold Psi def})--(\ref{choice of Psi_i}). Our choice for \(w_i\), together with equation (\ref{large intertwiner}), asserts that 
	\(
	\Phi_{\mu;S_1,S_2,S_3}^{w_1,w_2,w_3} = \Phi_\mu^{v_1,\dots,v_k}
	\), and hence by (\ref{trace formally}), 
	\[
	H_\mu^{\Phi_{\mu;S_1,S_2,S_3}^{w_1,w_2,w_3}}(q^{2\lambda+2\rho})=\Tr_{M_\mu}(\Phi_\mu^{v_1,\dots,v_k}\circ\pi_\mu(q^{2\lambda+2\rho})).
	\]
	With the present choice of weight vectors \(g_j\), Lemma \ref{lemma expectation value of Psi} asserts that
	\(
	\langle\Psi\rangle = \mathbb{B}_{\lambda,i}(f_k\otimes\dots\otimes f_1)
	\). This concludes the proof.
\end{proof}

Given two endomorphisms \(A\in \End_{\cN_{\mr{fd}}}(\cF^\str(\ul{S}))\) and \(B\in \End_{\cN_{\mr{fd}}}(\cF^\str(\ul{S^\ast}))\), we set
\[
\Theta_{S,i}^{A(v_1,\dots,v_k)\vert B(f_k,\dots,f_1)}:=
\sum_{m,n}\Theta_{S,i}^{w_1^m, \dots,w_k^m\vert g_k^n,\dots,g_1^n}
\]
where \(w_j^m\in V_j\) and \(g_j^n\in V_j^\ast\) are homogeneous vectors such that 
\[
A(v_1\otimes\dots\otimes v_k) = \sum_m w_1^m\otimes\dots\otimes w_k^m, \qquad B(f_k\otimes\dots\otimes f_1) = \sum_n g_k^n\otimes\dots\otimes g_1^n.
\]
This is well defined since $A$ and $B$ are $\mathfrak{h}$-linear. 

With the current choices of $\bs{S}$, $\bs{w}$ and $\bs{g}$, the graphically derived dual $q$-KZB type equation in
 Corollary \ref{thm boundary qKZB} gives rise to the following equation for $\Theta_{S,i}^{v_1,\ldots,v_k\vert f_k,\ldots,f_1}(\lambda,\mu)$.
\begin{corollary}
	\label{cor qKZB for Theta}
We have
	\begin{equation}
	\label{q-KZB for Theta}
	\Theta_{S,i}^{v_1,\dots,v_k\vert f_k,\dots, f_1}(\lambda,\mu) = \sum_{\sigma\in\wts(V_i)}\Theta_{S,i}^{\mathcal{K}_{\mu,\sigma}^{(i)}(v_1,\dots,v_k)\vert \mathcal{D}_\lambda^{(i)}(f_k,\dots,f_1)}(\lambda,\mu+\sigma)
	\end{equation}
	with \(\mathcal{K}_{\mu,\sigma}^{(i)}\in\End_{\cN_{\textup{fd}}}(\cF^\str(\ul{S}))\) and \(\mathcal{D}_\lambda^{(i)}\in\End_{\cN_{\textup{fd}}}(\cF^\str(\ul{S^\ast}))\) defined by
	\begin{align}
	\label{K_mu,sigma,i def}
	\begin{split}
	\mathcal{K}_{\mu,\sigma}^{(i)} :=\ & \Rdyn^{21}_{V_i, V_{i+1}}(\mu+\sigma-\mh_{(V_{i+2},\dots,V_k)})\cdots \Rdyn^{21}_{V_i, V_k}(\mu+\sigma)\,\mathbb{P}_{V_i}[\sigma]\\&\circ\Rdyn^{21}_{V_1,V_i}(\mu-\mh_{(V_2,\dots,V_{i-1}, V_{i+1},\dots,V_k)})^{-1}\cdots \Rdyn^{21}_{V_{i-1},V_i}(\mu-\mh_{(V_{i+1},\dots,V_k)})^{-1},
	\end{split}\\
	\label{D_lambda,i def}
	\mathcal{D}_\lambda^{(i)} :=\ & \kappa^{-2}_{V_k^\ast,V_i^\ast}\dots\kappa^{-2}_{V_{i+1}^\ast,V_i^\ast}\,(q^{2\theta(\lambda)})_{V_i^\ast}.
	\end{align}
\end{corollary}
\begin{proof}
	First observe that by Proposition \ref{prop hexagon for dynamical R-matrix} we have
	\[
	\mathcal{K}_{\mu,\sigma}^{(i)} = 
	\Rdyn^{21}_{(V_i), (V_{i+1},\dots,V_k)}(\mu+\sigma)\,\mathbb{P}_{V_i}[\sigma]\,\Rdyn^{21}_{(V_1,\dots,V_{i-1}), (V_i)}(\mu-\mh_{(V_{i+1},\dots,V_k)})^{-1}
	\]
	and that by definition of the comultiplication \(\Delta\) we can write
	\[
	\mathcal{D}_\lambda^{(i)} = 
	\kappa^{-2}_{(V_k^\ast,\dots,V_{i+1}^\ast),(V_i^\ast)}\,(q^{2\theta(\lambda)})_{V_i^\ast}.
	\]
	By Lemma \ref{prop Theta is Z}, Corollary \ref{thm boundary qKZB} gives
	\begin{equation}\label{almostthere}
	\Theta_{S,i}^{v_1,\dots,v_k\vert f_k,\dots, f_1}(\lambda,\mu) = \sum_{m,n}\sum_{\sigma\in\wts(V_i)} \mathcal{Z}_{\boldsymbol{S}}^{(\upsilon_{1,\sigma}^m, \upsilon_{2,\sigma}^m, \upsilon_{3,\sigma}^m), (\varpi_{3}^n, \varpi_{2}^n, \varpi_{1}^n)}(\lambda,\mu+\sigma),
	\end{equation}
	where \(\upsilon_{j,\sigma}^m\in\cF^\str(S_j)\) and \(\varpi_{j}^n\in\cF^\str(S_j^\ast)\) are homogeneous vectors such that
	\begin{equation}\label{the funny varpi}
	\begin{split}
	\sum_m \upsilon_{1,\sigma}^m\otimes \upsilon_{2,\sigma}^m\otimes \upsilon_{3,\sigma}^m&=\Bigl\{
	\mathscr{J}_{\boldsymbol{S}}(\mu+\sigma)\,\mathcal{K}_{\mu,\sigma}^{(i)}\,\mathscr{J}_{\boldsymbol{S}}(\mu+\sigma)^{-1}\\
	\circ\bigl( &j_{(V_1,\dots,V_{i-1})}(\mu_{1}+\sigma)\otimes\id_{V_i}\otimes  j_{(V_{i+1},\dots,V_{k})}(\mu+\sigma)\bigr)\Bigr\}(v_1\otimes\dots\otimes v_k),\\
	\sum_n \varpi_{3}^n\otimes \varpi_{2}^n\otimes \varpi_{1}^n&=\mathcal{D}_\lambda^{(i)}\left(j_{(V_k^\ast,\dots,V_{i+1}^\ast)}(\lambda) \otimes\id_{V_i^\ast}\otimes j_{(V_{i-1}^\ast,\dots,V_1^\ast)}(\lambda_{1}) \right)(f_k\otimes\dots\otimes f_1).
	\end{split}
	\end{equation}
	By the definition (\ref{the funny Js}) of \(\mathscr{J}_{\boldsymbol{S}}(\mu+\sigma)\), the first equation in \eqref{the funny varpi} can be rewritten as
	\[
	\sum_m \upsilon_{1,\sigma}^m\otimes \upsilon_{2,\sigma}^m\otimes \upsilon_{3,\sigma}^m = \left( j_{(V_1,\dots,V_{i-1})}(\mu+\sigma-\mh_{(V_i,\dots,V_k)})\otimes\id_{V_i}\otimes  j_{(V_{i+1},\dots,V_{k})}(\mu+\sigma)\right)\mathcal{K}_{\mu,\sigma}^{(i)}(v_1\otimes\dots\otimes v_k).
	\]
	Since dynamical fusion operators are weight preserving, the second equation 
	in (\ref{the funny varpi}) can be rewritten as
	\[
	\sum_n \varpi_{3}^n\otimes \varpi_{2}^n\otimes \varpi_{1}^n= \left(j_{(V_k^\ast,\dots,V_{i+1}^\ast)}(\lambda) \otimes\id_{V_i^\ast}\otimes j_{(V_{i-1}^\ast,\dots,V_1^\ast)}(\lambda_{1}) \right)\mathcal{D}_\lambda^{(i)}(f_k\otimes\dots\otimes f_1).
	\]
	Substituting in \eqref{almostthere}, the result follows from Lemma \ref{prop Theta is Z}.
\end{proof}

We will now replace the $i$-dependent operator $\mathbb{B}_{\lambda,i}\vert_{\cF^\str(\ul{S^*})[0]}$ in the definition of the spin components $\Theta_{S,i}^{v_1,\ldots,v_k\vert f_k,\ldots,f_1}(\lambda,\mu)$ by an appropriate $i$-independent linear operator without changing the equation \eqref{q-KZB for Theta}. Returning to the graphical context, replacing 
$\mathbb{B}_{\lambda,i}\vert_{\cF^\str(\ul{S^*})[0]}$ amounts to replacing the color $\Psi=\Psi_{\lambda;S_3^*,S_2*,s_1^*}^{g_3,g_2,g_1}$ in the graphical representation of $\mathcal{Z}_{\bs{S}}^{\bs{w},\bs{g}}(\lambda,\mu)$
by another dual intertwiner $\Psi\in\textup{Hom}_{\mathcal{M}_{\textup{adm}}^\str}(M_\lambda,S^*\tens M_\lambda)$. 
By a direct inspection of the graphical proof of Corollary \ref{thm boundary qKZB}, it follows that \eqref{qKZBfromgraphics} will hold true for the class of intertwiners
$\Psi\in\textup{Hom}_{\mathcal{M}_{\textup{adm}}^\str}(M_\lambda,S^*\tens M_\lambda)$ satisfying Figure \ref{SS1} with highest-weight to highest-weight conditions imposed on the Verma line (note that Lemma \ref{lemma before topological qKZB} already describes a subclass of such intertwiners, including
$\Psi_{\lambda,S_3^*,S_2^*,S_1^*}^{g_3,g_2,g_1}$). These intertwiners can be characterized as follows.
\begin{lemma}\label{generalizelemma}
With the current conventions, in particular, for $S=(V_1,\ldots,V_k)\in\Rep^\str$ with \textup{(}$i$-dependent\textup{)} decomposition $S=S_1\tens S_2\tens S_3$ given by \eqref{idecomp}, the intertwiner $\Psi\in\textup{Hom}_{\mathcal{M}_{\textup{adm}}^\str}(M_\lambda,S^*\tens M_\lambda)$ satisfies
\begin{figure}[H]
		\centering
		\includegraphics[scale=0.75]{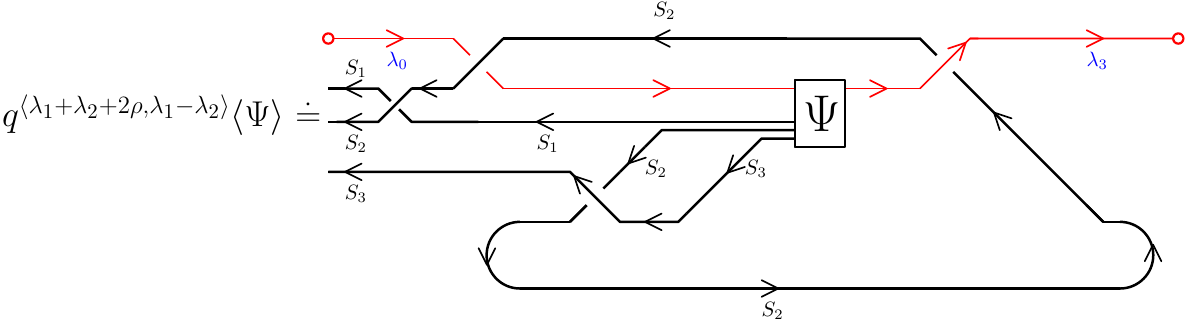}
	\end{figure}
if and only if
\begin{equation}\label{requiredformula}
q^{\langle \lambda_1+\lambda_2+2\rho,\lambda_1-\lambda_2\rangle}\langle\Psi\rangle=\mathcal{R}_{V_i^*,(V_{i-1}^*,\ldots,V_1^*)}\Upsilon_\lambda^{(i)}
\mathcal{R}_{(V_k^*,\ldots,V_{i+1}^*),V_i^*}\langle\Psi\rangle,
\end{equation}
where
\[
\Upsilon^{(i)}_\lambda:=\kappa_{V_i^\ast,V_1^\ast}\dots\kappa_{V_{i}^\ast,V_{i-1}^\ast}\,(q^{-2\theta(\lambda)})_{V_i^\ast}\,\kappa_{V_{i+1}^\ast,V_{i}^\ast}\dots\kappa_{V_k^\ast,V_i^\ast}.
\]
\end{lemma}
\begin{proof}
This follows using graphical calculus in $\mathcal{N}_{\textup{fd}}^\str$ and Lemma \ref{lemma 3.10} (cf. the proof of \cite[Prop. 2.11]{DeClercq&Reshetikhin&Stokman-2022}).
\end{proof}
We now obtain the following generalisation of Corollary \ref{cor qKZB for Theta}. Recall the definition of the operator 
$\mathcal{D}_\lambda^{(i)}\in\textup{End}_{\cN_{\textup{fd}}}\bigl(\cF^\str(\ul{S^*})\bigr)$ (see (\ref{D_lambda,i def})), which acts by a multiplicative scalar on any tensor product $f_k\otimes\cdots\otimes f_1$
of weight vectors $f_j\in V_j^*[\nu_j']$.
\begin{proposition}\label{almostEVdualqKZB}
With the above assumptions, 
let
$\widetilde{\mathbb{B}}_\lambda\in\textup{End}_{\mathbb{C}}(\cF^\str(\ul{S^*})[0])$ be a linear operator satisfying
\begin{equation}\label{ioper}
\widetilde{\mathbb{B}}_\lambda=\mathcal{R}_{V_i^*,(V_{i-1}^*,\ldots,V_1^*)}^{-1}\circ\Upsilon_\lambda^{(i)}\circ\mathcal{R}_{(V_k^*,\ldots,V_{i+1}^*),V_i^*}\circ\widetilde{\mathbb{B}}_\lambda\circ
\mathcal{D}_\lambda^{(i)}
\end{equation}
and consider the associate spin component
\begin{equation}\label{tildeTheta}
\widetilde{\Theta}_S^{v_1,\ldots,v_k\vert f_k,\ldots,f_1}(\lambda,\mu):=\widehat{e}_{\ul{S}}\bigl(\textup{Tr}_{M_\mu}(\Phi_\mu^{v_1,\ldots,v_k}\circ\pi_\mu(q^{2\lambda+2\rho}))
\otimes\widetilde{\mathbb{B}}_\lambda(f_k\otimes\cdots\otimes f_1)\bigr)
\end{equation}
of the $q^{2\lambda+2\rho}$-weighted trace of the $k$-point quantum vertex operator $\Phi_\mu^{v_1,\ldots,v_k}$.
Then
\begin{equation}
	\label{q-KZB for Theta gen}
	\widetilde{\Theta}_{S}^{v_1,\dots,v_k\vert f_k,\dots, f_1}(\lambda,\mu) = \sum_{\sigma\in\wts(V_i)}\widetilde{\Theta}_{S}^{\mathcal{K}_{\mu,\sigma}^{(i)}(v_1,\dots,v_k)\vert \mathcal{D}_\lambda^{(i)}(f_k,\dots,f_1)}(\lambda,\mu+\sigma).
	\end{equation}
\end{proposition}
\begin{proof}
Let $\Psi\in\textup{Hom}_{\mathcal{M}_{\textup{adm}}^\str}(M_\lambda,S^*\tens M_\lambda)$ be the unique intertwiner such that 
\[
\langle\Psi\rangle=\widetilde{\mathbb{B}}_\lambda(f_k\otimes\cdots\otimes f_1),
\]
so that
\[
\widetilde{\Theta}_S^{v_1,\ldots,v_k\vert f_k,\ldots,f_1}(\lambda,\mu)=\widehat{e}_{\ul{S}}\bigl(\textup{Tr}_{M_\mu}(\Phi_\mu^{v_1,\ldots,v_k}\circ\pi_\mu(q^{2\lambda+2\rho}))
\otimes\langle\Psi\rangle\bigr).
\]
It can be graphically represented by Figure \ref{Ydiagram} with $\xi=2\lambda+2\rho$. Note that 
$\widetilde{\Theta}_S^{v_1,\ldots,v_k\vert f_k,\ldots,f_1}(\lambda,\mu)$ is equal to 
$\Theta_{S,i}^{v_1,\ldots,v_k\vert f_k,\ldots,f_1}(\lambda,\mu)$ 
when $\widetilde{\mathbb{B}}_\lambda=\mathbb{B}_{\lambda,i}\vert_{\cF^\str(\ul{S^*})[0]}$. In this case the corresponding intertwiner is $\Psi=\Psi_{\lambda;S_3^*,S_2^*,S_1^*}^{g_3,g_2,g_1}$.

We now return to the graphical proof of the dual $q$-KZB type equation in Corollary \ref{thm boundary qKZB}, which led to equation \eqref{q-KZB for Theta} for the spin component
\(\Theta_{S,i}^{v_1,\dots,v_k\vert f_k,\dots,f_1}(\lambda,\mu)\) of the $k$-point quantum vertex operator $\Phi_\mu^{v_1,\ldots,v_k}$ (see \eqref{Theta def}).
The starting point of the graphical derivation was Lemma \ref{lemma before topological qKZB} applied to $\Psi=\Psi_{\mu;S_3^*,S_2^*,S_1^*}^{g_3,g_2,g_1}$. 
Looking at the proof of Lemma \ref{lemma boundary qKZB with general xi}, which is the step that boundary conditions are imposed, it is clear that the derivation of the dual $q$-KZB type of equation for $\mathcal{Z}_{\bs{S}}^{\bs{w},\bs{g}}(\lambda,\mu)$ in Corollary \ref{thm boundary qKZB} is also valid when $\Psi=\Psi_{\mu;S_3^*,S_2^*,S_1^*}^{g_3,g_2,g_1}$ is replaced by any intertwiner $\Psi\in\textup{Hom}_{\mathcal{M}_{\textup{adm}}^\str}(M_\lambda,S^*\tens M_\lambda)$ satisfying Figure \ref{SS1} with highest-weight to highest-weight conditions imposed on the Verma line. By Lemma \ref{generalizelemma}, this holds true if and only if $\langle\Psi\rangle$ satisfies equation \eqref{requiredformula}.

By the definition of the shifted weights $\lambda_j$ (see \eqref{lambdamushifts}), we have 
\[
\langle\lambda_1+\lambda_2+2\rho,\lambda_1-\lambda_2\rangle=-\langle2\lambda+2\rho-\xi_i'-2\xi_{i+1}'\cdots-2\xi_k',\xi'_i\rangle,
\]
and hence equation \eqref{ioper} for $\widetilde{\mathbb{B}}_\lambda$ implies that if the expectation value of 
$\Psi$  is of the form
$\langle\Psi\rangle=\widetilde{\mathbb{B}}_\lambda(f_k\otimes\cdots\otimes f_1)$, then $\langle\Psi\rangle$ satisfies \eqref{requiredformula}. This concludes the proof of the proposition.
\end{proof}
For an $\mathfrak{h}$-linear operator $\mathbb{A}\in\textup{End}_{\cN_{\textup{fd}}}(\mathcal{F}^\str(\ul{S}))$, we write $\mathbb{A}^T\in\textup{End}_{\cN_{\textup{fd}}}(\mathcal{F}(\ul{S^*}))$
for its transpose relative to $\widehat{e}_{\ul{S}}$ (see \eqref{hateN}),
\begin{equation}\label{transposedef}
\widehat{e}_{\ul{S}}(\mathbb{A}v,h)=\widehat{e}_{\ul{S}}(v,\mathbb{A}^Th)\qquad \forall\, (v,h)\in\mathcal{F}^\str(\ul{S})\times\mathcal{F}^\str(\ul{S^*}).
\end{equation}
Note that $(q^h_{V_i})^T=q^{-h}_{V_i^*}$ for $h\in\mathfrak{h}$, and
\[
\bigl(q^{\theta(-\lambda-2\rho)}_{V_i}\bigr)^T=q^{\theta(\lambda)}_{V_i^*}.
\]
Since
\begin{equation}\label{hattilde}
\widehat{e}_{\ul{S}}(v,h)=\widetilde{e}_S(v,h)\qquad \forall (v,h)\in\cF^\str(\ul{S})[0]\times\cF^\str(\ul{S^*})[0],
\end{equation}
so that $\mathbb{A}^T\vert_{\cF^\str(\ul{S})[0]}$ may as well be computed by taking the transpose of $A$ relative to 
$\widetilde{e}_{\ul{S}}$ and then restricting to $\cF^\str(\ul{S})[0]$. This is of use when the operator is given in terms of the action of the quantum group. For example, with $S=(V_1,\ldots,V_k)\in\Rep^\str$, 
\[
(\mathcal{R}_{V_j,V_{j+1}})^T\vert_{\cF^\str(\ul{S^*})[0]}=\mathcal{R}^{21}_{V_{j+1}^*,V_j^*}\vert_{\cF^\str(\ul{S^*})[0]}
\]
since $(S\otimes S)(\mathcal{R})=\mathcal{R}$. 

We will also be needing the opposite transpose; for $\mathbb{B}\in\textup{End}_{\cN_{\textup{fd}}}(\mathcal{F}(\ul{S^*}))$ we will write
$\mathbb{B}^{T*}\in\textup{End}_{\cN_{\textup{fd}}}(\mathcal{F}(\ul{S}))$ for the operator such that 
\begin{equation}\label{dualtranspose}
\widehat{e}_{\ul{S}}(v,\mathbb{B}h)=\widehat{e}_{\ul{S}}(\mathbb{B}^{T*}v,h)\qquad \forall\, (v,h)\in\mathcal{F}^\str(\ul{S})\times\mathcal{F}^\str(\ul{S^*}).
\end{equation}

We now have the following lemma.

\begin{lemma}\label{qKZfusion}
The linear operator $\widetilde{\mathbb{B}}_\lambda:=(j_S(-\lambda-2\rho)^{-1})^T\vert_{\cF^\str(\ul{S^*})[0]}$ on $\mathcal{F}^\str(\ul{S^*})[0]$ satisfies equation \eqref{ioper}
for all $i=1,\ldots,k$.
\end{lemma}
\begin{proof}
By the observations in the paragraph preceding the lemma it suffices to show that $\widehat{\mathbb{B}}_\lambda:=j_S^{-1}(-\lambda-2\rho)\vert_{\cF^\str(\ul{S})[0]}$
satisfies 
\begin{equation}
	\label{equations for A tilde}
	\widehat{\mathbb{B}}_\lambda =\widetilde{\mathcal{D}}_\lambda^{(i)}\circ \widehat{\mathbb{B}}_\lambda \circ \cR^{21}_{V_i, (V_{i+1},\dots,V_k)}\,\widetilde{\Upsilon}_\lambda^{(i)}\,(\cR^{21})^{-1}_{(V_1,\dots,V_{i-1}),V_i}
	\end{equation}
	for $i=1,\ldots,k$,
	with \(\widetilde{\mathcal{D}}_\lambda^{(i)}, \widetilde{\Upsilon}_\lambda^{(i)}\in\End_\CC(\cF^\str(\ul{S})[0])\) defined by
	\begin{align*}
	\widetilde{\mathcal{D}}_\lambda^{(i)}:=\ & (q^{2\theta(-\lambda-2\rho)})_{V_i}\kappa^{-2}_{V_i,V_{i+1}}\dots\kappa^{-2}_{V_i,V_k}, \\
	\widetilde{\Upsilon}_\lambda^{(i)}:=\ & \kappa_{V_i,V_{i+1}}\dots\kappa_{V_i,V_k}\,(q^{-2\theta(-\lambda-2\rho)})_{V_i}\,\kappa_{V_1,V_i}\dots\kappa_{V_{i-1}, V_i}.
	\end{align*}
	The equations \eqref{equations for A tilde} for $\widehat{\mathbb{B}}_\lambda$ are equivalent to the $q$-KZ equations for fusion operators, see \cite[Theorem 3.13]{DeClercq&Reshetikhin&Stokman-2022}.
\end{proof}
Recall that $S=(V_1,\ldots,V_k)\in\Rep^\str$.
\begin{definition}\label{definitionTT}
For $v_j\in V_j[\xi_j]$,
$f_j\in V_j^*[\xi_j']$ with weights $\xi_j,\xi_j'$ such that $\sum_j\xi_j=0=\sum_j\xi_j'$, define 
\[
\TT_S^{v_1,\dots,v_k\vert f_k,\dots,f_1}(\lambda,\mu) :=\delta(q^{2\lambda+2\rho})\,\widehat{e}_{\ul{S}}\left(j_S(-\lambda-2\rho)^{-1}\Tr_{M_\mu}(\Phi_\mu^{v_1,\dots,v_k}\circ\pi_\mu(q^{2\lambda+2\rho}))\otimes(f_k\otimes\dots\otimes f_1) \right)
\]
for $\lambda,\mu\in\mathfrak{h}_{\textup{reg}}^*$ with $\Re(\lambda)$ deep in the negative Weyl chamber,
with
	\begin{equation}\label{Wd}
	\delta(q^{2\lambda+2\rho}):= q^{2\langle \lambda+\rho,\rho\rangle}\prod_{\alpha\in\Phi^+}(1-q^{-2\langle\lambda+\rho,\alpha\rangle})
	\end{equation}
the Weyl denominator.
\end{definition}
We follow here \cite{Etingof&Varchenko-2000} by adding the Weyl denominator to the definition of $\TT_S^{v_1,\dots,v_k\vert f_k,\dots,f_1}(\lambda,\mu) $ as a convenient normalisation factor. We furthermore define
\begin{equation}\label{universalTT}
\begin{split}
\mathbb{T}^{f_k,\ldots,f_1}_S(\lambda,\mu)&:=\sum_{v_j\in\mathcal{B}_{V_j}: \sum_j\textup{wts}(v_j)=0}\mathbb{T}^{v_1,\ldots,v_k\vert f_k,\ldots,f_1}(\lambda,\mu)(v_k^*\otimes \cdots\otimes v_1^*),\\
\mathbb{T}_S(\lambda,\mu)&:=\sum_{u_j\in\mathcal{B}_{V_j}: \sum_j\textup{wts}(u_j)=0}(u_1\otimes\cdots\otimes u_k)\otimes\mathbb{T}_S^{u_k^*,\ldots,u_1^*}(\lambda,\mu).
\end{split}
\end{equation}
which take values in $\cF^\str(\ul{S^*})[0]$ and $\cF^\str(\ul{S})[0]\otimes\cF^\str(\ul{S^*})[0]$, respectively. Recall here that $b^*\in V_j^*$ for $b\in\mathcal{B}_{V_j}$ is the linear functional such that $b^*(b')=\delta_{b,b'}$ for all $b'\in\mathcal{B}_{V_j}$. 
The meaning of these functions
can be best understood in terms of the generating function $\Phi^S_\mu$ of the $k$-point quantum vertex operators
$M_\mu\rightarrow M_\mu\otimes\cF^\str(S)$, which is defined by
\begin{equation}\label{generatingPhimu}
\Phi_\mu^S:=\sum_{v_j\in\mathcal{B}_{V_j}: \sum_j\textup{wts}(v_j)=0}\Phi_\mu^{v_1,\ldots,v_k}\otimes v_k^*\otimes\cdots\otimes v_1^*.
\end{equation}
It may be viewed as morphism $\ul{M_\mu}\rightarrow\ul{M_\mu}\otimes\cF^\str(\ul{S})\otimes\cF^\str(\ul{S^*})[0]$ in $\mathcal{N}_{\textup{adm}}$. For fixed $\lambda\in\mathfrak{h}_{\textup{reg}}^*$ with $\Re(\lambda)$ deep in the negative Weyl chamber, $\mathbb{T}_S(\lambda,\mu)$ then becomes the normalised $q^{2\lambda+2\rho}$-weighted trace function
of $\Phi_\mu^S$,
\begin{equation}\label{TTPhi}
\mathbb{T}_S(\lambda,\mu)=
\delta(q^{2\lambda+2\rho})j_S(-\lambda-2\rho)^{-1}\textup{Tr}_{M_\mu}(\Phi_\mu^S\circ\pi_\mu(q^{2\lambda+2\rho})),
\end{equation}
and $\mathbb{T}_S^{f_k,\ldots,f_1}(\lambda,\mu)$ is the spin-component
\begin{equation}\label{TTf}
\mathbb{T}_S^{f_k,\ldots,f_1}(\lambda,\mu)=\widehat{e}_{\ul{S}}\bigl(\mathbb{T}_S(\lambda,\mu),f_k\otimes\cdots\otimes f_1\bigr)
\end{equation}
of the $q^{2\lambda+2\rho}$-weighted trace of $\Phi_\mu^S$. Here the first tensor component of $\mathbb{T}_S(\lambda,\mu)\in\cF^\str(\ul{S})[0]\otimes\cF^\str(\ul{S^*})[0]$ is paired with the dual vector $f_k\otimes\cdots\otimes f_1\in\cF^\str(\ul{S^*})[0]$. 
Pairing the second tensor component of $\mathbb{T}_S(\lambda,\mu)$ with $v_1\otimes\cdots\otimes v_k\in\cF^\str(\ul{S})[0]$ yields the normalised $q^{2\lambda+2\rho}$-weighted trace of $\Phi_\mu^{v_1,\ldots,v_k}$,
\begin{equation}\label{TTv}
\begin{split}
\mathbb{T}_S^{v_1,\ldots,v_k}(\lambda,\mu):=&\widehat{e}_{\ul{S}}\bigl(v_1\otimes\cdots\otimes v_k,\mathbb{T}_S(\lambda,\mu)\bigr)\\
=&\delta(q^{2\lambda+2\rho})j_S(-\lambda-2\rho)^{-1}\textup{Tr}_{M_\mu}(\Phi_\mu^{v_1,\ldots,v_k}\circ\pi_\mu(q^{2\lambda+2\rho})).
\end{split}
\end{equation}
\begin{remark}
The universal weighted trace function $\varphi_{V_1,\ldots,V_k}(\lambda,\mu)$ defined in \cite[(1.8)]{Etingof&Varchenko-2000} is equal to $\mathbb{T}_S(\lambda-\rho,\mu)$.
\end{remark}
 We now come to the main result of this subsection.
  We will see shortly that it is equivalent to the dual $q$-KZB equations \cite[Thm. 1.4]{Etingof&Varchenko-2000} for weighted trace functions.
\begin{theorem}\label{dualqKZB}
With the above notations and assumptions,
\begin{equation}\label{qKZBcoord}
\TT_{S}^{\,v_1,\dots,v_k\vert f_k,\dots, f_1}(\lambda,\mu) = \sum_{\sigma\in\wts(V_i)}\TT_{S}^{\,\mathcal{K}_{\mu,\sigma}^{(i)}(v_1,\dots,v_k)\vert \mathcal{D}_\lambda^{(i)}(f_k,\dots,f_1)}(\lambda,\mu+\sigma)
\end{equation}
for $i=1,\ldots,k$.
\end{theorem}
\begin{proof}
Note that
\[
\TT_S^{v_1,\dots,v_k\vert f_k,\dots,f_1}(\lambda,\mu)=\widehat{e}_{\ul{S}}\left(\Tr_{M_\mu}(\Phi_\mu^{v_1,\dots,v_k}\circ\pi_\mu(q^{2\lambda+2\rho}))\otimes\mathbb{Y}_{\lambda,S^*}(f_k\otimes\dots\otimes f_1) \right)
\]
with $\mathbb{Y}_{\lambda,S^*}:=\delta(q^{2\lambda+2\rho})(j_S(-\lambda-2\rho)^{-1})^T$.
By Lemma \ref{qKZfusion}, the linear operator $\mathbb{Y}_{\lambda,S^*}$ satisfies the equations
\eqref{ioper} for $i=1,\ldots,k$. 
 The result now immediately follows
 from Proposition \ref{almostEVdualqKZB}.
\end{proof}
We have the following reformulations of Theorem \ref{dualqKZB}.
\begin{corollary}\label{dualqKZBcor}
With the above notations \textup{(}in particular, $f_j\in V_j^*[\xi_j']$ with $\sum_j\xi_j'=0$\textup{)}, we have
\begin{enumerate}
\item[\textup{(i)}] For $i=1,\ldots,k$,
\[
\mathbb{T}_S(\lambda,\mu)=\sum_{\sigma\in\textup{wts}(V_i)}\bigl(\bigl(\mathcal{D}_\lambda^{(i)}\bigr)^{T*}\otimes\bigl(\mathcal{K}_{\mu,\sigma}^{(i)}\bigr)^{T}\bigr)\mathbb{T}_S(\lambda,\mu+\sigma).
\]
\item[\textup{(ii)}] For $i=1,\ldots,k$,
\[
\sum_{\sigma\in\textup{wts}(V_i)}\bigl(\mathcal{K}_{\mu,\sigma}^{(i)}\bigr)^T\,\mathbb{T}_S^{f_k,\ldots,f_1}(\lambda,\mu+\sigma)=
q^{\langle \xi_i'+2\xi_{i+1}'+\cdots+2\xi_k'-2\lambda-2\rho,\xi_i'\rangle}\mathbb{T}_S^{f_k,\ldots,f_1}(\lambda,\mu).
\]
\end{enumerate}
\end{corollary}
\begin{proof}
(i) By \eqref{universalTT} we have
\begin{equation}\label{universalTTwrittenout}
\mathbb{T}_S(\lambda,\mu):=\sum\mathbb{T}_S^{v_1,\ldots,v_k\vert u_k^*,\ldots,u_1^*}(\lambda,\mu)(u_1\otimes\cdots\otimes u_k)\otimes (v_k^*\otimes\cdots\otimes v_1^*),
\end{equation}
where the sum is over $u_j, v_j\in\mathcal{B}_{V_j}$ such that $ \sum_j\textup{wt}(u_j)=0=\sum_j\textup{wt}(v_j)$. Now substitute \eqref{qKZBcoord} into the right hand side of this expression, then the result follows from the definition of the transpose.\\
(ii) This follows in a similar way from the definition of $\mathbb{T}_S^{f_k,\ldots,f_1}(\lambda,\mu)$ (see the first line of \eqref{universalTT}), formula \eqref{qKZBcoord},
and the 
observation that 
\[
\mathcal{D}_\lambda^{(i)}(f_k\otimes\cdots\otimes f_1)=q^{\langle 2\lambda+2\rho-\xi_i'-2\xi_{i+1}'\cdots-2\xi_k',\xi_i'\rangle}f_k\otimes\cdots\otimes f_1.
\]
\end{proof}

\begin{remark}\label{Dicheck}
For $S=(V_1,\ldots,V_k)\in\Rep^\str$ and $i=1,\ldots,k$ the $\mathfrak{h}$-linear operator
\begin{equation}\label{Dicheckformula}
\mathbb{D}_{\lambda,i}^{\vee,S}:=\bigl(\mathcal{D}_\lambda^{(i)}\bigr)^{T*}=\bigl(q^{2\theta(-\lambda-2\rho)}\bigr)_{V_i}\kappa_{V_i,V_{i+1}}^{-2}\cdots\kappa_{V_i,V_k}^{-2}
\end{equation}
on $\cF^\str(\ul{S})$
coincides with the operator $D_i^\vee$ from \cite[(1.16)]{Etingof&Varchenko-2000}.
\end{remark}

Corollary \ref{dualqKZBcor}(ii) shows that for any $\lambda\in\mathfrak{h}_{\textup{reg}}^*$ with $\Re(\lambda)$ deep in the negative Weyl chamber and any choice of weight vectors $f_j\in V_j^*[\xi_j']$ with $\sum_j\xi_j'=0$,
the corresponding spin component $\mathbb{T}_S^{f_k,\ldots,f_1}(\lambda,\mu)\in\cF^\str(\ul{S^*})[0]$ of the $q^{2\lambda+2\rho}$-weighted trace of $\Phi_\mu^S$ is a joint eigenfunction of the 
difference operators
\[
(\widetilde{\mathbb{L}}_i^{\vee,S}f)(\mu):=\sum_{\sigma\in\textup{wts}(V_i)}\bigl(\mathcal{K}_{\mu,\sigma}^{(i)}\bigr)^Tf(\mu+\sigma)\qquad (i=1,\ldots,k)
\]
on $\mathfrak{h}_{\textup{reg}}^*$, where $f$ is a $\cF^\str(\ul{S^*})[0]$-valued function on $\mathfrak{h}_{\textup{reg}}^*$. The $\widetilde{\mathbb{L}}_i^{\vee,S}$ are gauged versions of the dual $q$-KZB operators $K_i^\vee$ arising in the dual $q$-KZB equations from \cite[Thm. 1.4]{Etingof&Varchenko-2000}. The gauge arises from the fact that 
we need to renormalise $\mathbb{T}_S^{f_k,\ldots,f_1}(\lambda,\mu)$ and
$\mathbb{T}_S(\lambda,\mu)$ further by a $\mu$-dependent factor to 
reach the normalisation of the weighted trace functions as in \cite{Etingof&Varchenko-2000}. The normalisation factor  \(\mathbb{X}_{\mu,S^\ast}\in\End_{\cN_{\textup{fd}}}(\cF^\str(\ul{S^\ast}))\) is
	\begin{equation}\label{doubleX}
	\mathbb{X}_{\mu,S^\ast}:=\QQ_{V_k^\ast}(\mu)^{-1}\otimes\QQ_{V_{k-1}^\ast}(\mu+\mh_{V_{k}^\ast})^{-1}\otimes\cdots\otimes\QQ_{V_1^\ast}(\mu+\mh_{(V_k^\ast,\ldots,V_2^\ast)})^{-1},
	\end{equation}
and the resulting normalised weighted trace functions are
\begin{equation}\label{Fnorm}
\begin{split}
\mathbb{F}_S(\lambda,\mu)&:=\bigl(\textup{id}_{\cF^\str(\ul{S})[0]}\otimes\mathbb{X}_{\mu,S^*}\bigr)\mathbb{T}_S(\lambda,\mu),\\
\mathbb{F}^{f_k,\ldots,f_1}_S(\lambda,\mu)&:=\widehat{e}_{\ul{S}}\bigl(\mathbb{F}_S(\lambda,\mu),f_k\otimes\cdots\otimes f_1\bigr)=
\mathbb{X}_{\mu,S^*}\mathbb{T}_S^{f_k,\ldots,f_1}(\lambda,\mu).
\end{split}
\end{equation}
\begin{remark}\label{Fremark}\hfill
\begin{enumerate}
\item
The normalised universal trace function $F_{V_1,\ldots,V_k}(\lambda,\mu)$ defined in \cite[(1.9)]{Etingof&Varchenko-2000} is equal to 
$\mathbb{F}_S(\lambda-\rho,-\mu-\rho)$ \textup{(}be aware here of Remark \ref{notconv}\textup{)}.
\item The normalisation by $\mathbb{X}_{\mu,S^\ast}$ turns the weighted trace function $\mathbb{F}_S(\lambda,\mu)$ into a function admitting an important symmetry with respect to interchanging $\lambda$ and $\mu$, see \cite[Thm. 1.5]{Etingof&Varchenko-2000}. 
\end{enumerate}
\end{remark}
The equations in Corollary \ref{dualqKZBcor} give rise to difference equations for $\mathbb{F}_S(\lambda,\mu)$ in $\mu$, involving the following gauged versions of $\widetilde{\mathbb{L}}_i^{\vee,S}$. 
\begin{definition}\label{LiS}
For $i=1,\ldots,k$ we call the difference operators
\begin{equation}
\bigl(\mathbb{L}_i^{\vee,S}f\bigr)(\mu):=\sum_{\sigma\in\textup{wts}(V_i)}\mathcal{K}_{\mu,\sigma}^{(i),\vee}f(\mu+\sigma)
\end{equation}
for functions $f: \mathfrak{h}_{\textup{reg}}^*\rightarrow\cF^\str(\ul{S^*})[0]$ the dual $q$-KZB operators, with $\mathcal{K}_{\mu,\sigma}^{(i),\vee}$ the $\mathfrak{h}$-linear operator
\begin{equation*}
\begin{split}
\mathcal{K}_{\mu,\sigma}^{(i),\vee}:=&
R_{V_i^*V_{i-1}^*}\bigl(\mu-h_{(V_{i-2}^*,\ldots,V_1^*)}\bigr)^{-1}\cdots R_{V_i^*V_1^*}\bigl(\mu\bigr)^{-1}\mathbb{P}_{V_i^*}[-\sigma]\\
\circ &R_{V_k^*V_i^*}\bigl(\mu+\sigma-h_{(V_{k-1}^*,\ldots,V_{i+1}^*,V_{i-1}^*,\ldots,V_1^*)}\bigr)\cdots R_{V_{i+1}^*V_i^*}\bigl(\mu+\sigma-h_{(V_{i-1}^*,\ldots,V_1^*)}\bigr)
\end{split}
\end{equation*}
on $\cF^\str(\ul{S^*})$.
\end{definition}
The difference equations for $\mathbb{F}_S(\lambda,\mu)$ resulting from Corollary \ref{dualqKZBcor} now read as follows.
\begin{corollary}\label{corKKK}\textup{(}\cite[Thm. 1.4]{Etingof&Varchenko-2000}\textup{)}.
With the above notations, we have for $i=1,\ldots,k$,
\begin{equation}\label{qKZBnew}
\bigl(\mathbb{D}_{\lambda,i}^{\vee,S}\otimes\mathbb{L}_i^{\vee,S}\bigr)\mathbb{F}_S(\lambda,\cdot)=\mathbb{F}_S(\lambda,\cdot),
\end{equation}
which is equivalent to the system of difference equations
\begin{equation}
\mathbb{L}_i^{\vee,S}\mathbb{F}_S^{f_k,\ldots,f_1}(\lambda,\cdot)=q^{\langle\xi_i'+2\xi_{i+1}'+\cdots+2\xi_k'-2\lambda-2\rho,\xi_i'\rangle}
\mathbb{F}_S^{f_k,\ldots,f_1}(\lambda,\cdot)
\end{equation}
for  $f_j\in V_j^*[\xi_j']$  with $\sum_j\xi_j'=0$.
\end{corollary}
\begin{proof}
By Remark \ref{Dicheck} the difference equations for $\mathbb{T}_S(\lambda,\mu)$ and $\mathbb{T}_S^{f_k,\ldots,f_1}(\lambda,\mu)$ from Corollary \ref{dualqKZBcor} imply
\begin{equation*}
\begin{split}
&\mathbb{F}_S(\lambda,\mu)=\sum_{\sigma\in\textup{wts}(V_i)}\Bigl(\mathbb{D}_{\lambda,i}^{\vee,S}\otimes\mathbb{X}_{\mu,S^*}(\mathcal{K}_{\mu,\sigma}^{(i)})^T
\mathbb{X}_{\mu+\sigma,S^*}^{-1}\Bigr)\mathbb{F}_S(\lambda,\mu+\sigma),\\
&\sum_{\sigma\in\textup{wts}(V_i)}\bigl(\mathbb{X}_{\mu,S^*}(\mathcal{K}_{\mu,\sigma}^{(i)})^T\mathbb{X}_{\mu+\sigma,S^*}^{-1}\bigr)\mathbb{F}_S^{f_k,\ldots,f_1}(\lambda,\mu)
=q^{\langle\xi_i'+2\xi_{i+1}'+\cdots+2\xi_k'-2\lambda-2\rho,\xi_i'\rangle}\mathbb{F}^{f_k,\ldots,f_1}_S(\lambda,\mu)
\end{split}
\end{equation*}
for the (spin-components of the) renormalised weighted trace function $\mathbb{F}_S(\lambda,\mu)$. So it remains to show that
\begin{equation}\label{toX}
\mathbb{X}_{\mu,S^*}(\mathcal{K}_{\mu,\sigma}^{(i)})^T\mathbb{X}_{\mu+\sigma,S^*}^{-1}\vert_{\cF^\str(\ul{S^*})[0]}=\mathcal{K}_{\mu,\sigma}^{(i),\vee}\vert_{\cF^\str(\ul{S^*})[0]}.
\end{equation}
By \eqref{K_mu,sigma,i def} we have
\begin{equation*}
\begin{split}
\bigl(\mathcal{K}_{\mu,\sigma}^{(i)}\bigr)^T=&\bigl(R^{21}_{V_{i-1}V_i}(\mu-h_{(V_{i+1},\ldots,V_k)})^{-1}\bigr)^T\cdots
\bigl(R^{21}_{V_1V_i}(\mu-h_{(V_2,\ldots,V_{i-1},V_{i+1},\ldots,V_k)})^{-1}\bigr)^T\\
&\circ \mathbb{P}_{V_i^*}[-\sigma]R^{21}_{V_iV_k}(\mu+\sigma)^T\cdots R^{21}_{V_iV_{i+1}}(\mu+\sigma-h_{(V_{i+2},\ldots,V_k)})^T.
\end{split}
\end{equation*}
To compute the effect of dynamically conjugating $\bigl(\mathcal{K}_{\mu,\sigma}^{(i)}\bigr)^T$ by the tensor product of $\mathbb{Q}_{V_j^*}$-operators, we use for any $V,W\in\Rep$ the formula
\begin{equation}\label{transposeT}
R_{VW}(\mu)^T=\bigl(\mathbb{Q}_{W^*}(\mu)\otimes\mathbb{Q}_{V^*}(\mu+h_{W^*})\bigr)
R^{21}_{W^*V^*}(\mu+h_{(W^*,V^*)})\bigl(\mathbb{Q}_{W^*}(\mu+h_{V^*})^{-1}\otimes\mathbb{Q}_{V^*}(\mu)^{-1}\bigr)
\end{equation}
in $\textup{End}_{\mathcal{N}_{\textup{fd}}}\bigl(W^*\otimes V^*\bigr)$. Formula \eqref{transposeT} can be directly inferred from \cite[(3.12)]{Etingof&Varchenko-2000}, taking into account  Remark \ref{smfRemark}(2) and Remark \ref{notconv} (we invite the reader to give a graphical proof using the enriched graphical calculus for $\mathcal{N}_{\textup{fd}}$ from
Subsections \ref{sectiongi} \& \ref{Subsection Q(lambda)}). This formula implies for $i<j$,
\begin{equation*}
\begin{split}
R_{V_iV_j}^{21}&(\mu+\sigma-h_{(V_{j+1},\ldots,V_k)})^T\bigl(\mathbb{Q}_{V_j^*}(\mu+\sigma+h_{(V_k^*,\ldots,V_{j+1}^*)})\otimes
\mathbb{Q}_{V_i^*}(\mu+\sigma+h_{(V_k,\ldots,V_j^*)})\bigr)=\\
=&\bigl(\mathbb{Q}_{V_j^*}(\mu+\sigma+h_{(V_k^*,\ldots,V_{j+1}^*,V_i^*)})\otimes\mathbb{Q}_{V_i^*}(\mu+\sigma+h_{(V_k^*,\ldots,V_{j+1}^*)})\bigr)
R_{V_j^*V_i^*}(\mu+\sigma+h_{(V_k,\ldots,V_{j}^*,V_i^*)})
\end{split}
\end{equation*}
and for $r<i$,
\begin{equation*}
\begin{split}
\bigl(R^{21}_{V_rV_i}&(\mu-h_{(V_{r+1},\ldots,V_{i-1},V_{i+1},\ldots,V_k)})^{-1}\bigr)^T\\
&\circ
\bigl(\mathbb{Q}_{V_i^*}(\mu+h_{(V_k^*,\ldots,V_{i+1}^*,V_{i-1}^*,\ldots,V_r^*)})\otimes\mathbb{Q}_{V_r^*}(\mu+h_{(V_k^*,\ldots,V_{i+1}^*,V_{i-1}^*,\ldots,V_{r+1}^*)})\bigr)\\
&=\bigl(\mathbb{Q}_{V_i^*}(\mu+h_{(V_k^*,\ldots,V_{i+1}^*,V_{i-1}^*,\ldots,V_{r+1}^*)})\otimes\mathbb{Q}_{V_r^*}(\mu+h_{(V_k^*,\ldots,V_{r+1}^*)})\bigr)
R_{V_i^*V_r^*}(\mu+h_{(V_k^*,\ldots,V_r^*)})^{-1}.
\end{split}
\end{equation*}
Using in addition that
\[
\mathbb{P}_{V_i^*}[-\sigma]\mathbb{Q}_{V_j^*}(\mu+\sigma+h_{V_i^*})=\mathbb{Q}_{V_j^*}(\mu)\mathbb{P}_{V_i^*}[-\sigma]
\]
for all $j=1,\ldots,k$, which follows from the fact that $\mathbb{Q}_{V_j^*}(\mu)$ is $\mathfrak{h}$-linear,
a straightforward computation now establishes 
\begin{equation}\label{toXh}
\mathbb{X}_{\mu,S^*}(\mathcal{K}_{\mu,\sigma}^{(i)})^T\mathbb{X}_{\mu+\sigma,S^*}^{-1}=\mathcal{K}_{\mu+h_{(V_k^*,\ldots,V_1^*)},\sigma}^{(i),\vee},
\end{equation}
which in turn gives \eqref{toX}.
\end{proof}

\begin{remark}\label{relremqKZB}
By Remark \ref{Fremark}(1), the dual $q$-KZB equations \eqref{qKZBnew} are equivalent to
\begin{equation}\label{qKZBn2}
\bigl(\mathbb{D}_{\lambda-\rho,i}^{S,\vee}\otimes \mathbb{L}^{\vee,S}_{\textup{EV},i})F_S(\lambda,\cdot)=F_S(\lambda,\cdot)
\end{equation}
for the Etingof-Varchenko normalized universal weighted trace function $F_S(\lambda,\mu)$, where 
\begin{equation}\label{qKZBnew2}
\bigl(\mathbb{L}^{\vee,S}_{\textup{EV},i}f\bigr)(\mu):=\sum_{\sigma\in\textup{wts}(V_i^*)}\mathcal{K}_{-\mu-\rho,-\sigma}^{(i),\vee}f(\mu+\sigma).
\end{equation}
By Remarks \ref{notconv} \& \ref{Fremark}(1), the operator $\mathbb{R}(\lambda)$ in \cite{Etingof&Varchenko-2000} equals $R^{21}(-\lambda-\rho)$,
which shows that $\mathbb{L}_{\textup{EV},i}^{\vee,S}$ is the operator $K_i^\vee$ 
from \cite[(1.17)]{Etingof&Varchenko-2000}. Combined with Remark \ref{Dicheck}, we conclude that \eqref{qKZBn2} are the dual $q$-KZB equations 
from \cite[Thm. 1.4]{Etingof&Varchenko-2000}.
\end{remark}

\subsection{The dual coordinate Macdonald-Ruijsenaars equations}
\label{Subsection Etingof-Varchenko normalization for MR}
In this subsection we use Corollary \ref{thm MR for Y} with $S\in\Rep^\str$ of the form
\begin{equation}\label{SMRsplit}
S=S_1\tens S_2,\qquad S_1=(V_1,\ldots,V_i),\qquad S_2=(V_{i+1},\ldots,V_k)
\end{equation}
for $i=0,\ldots,k$ to derive dual coordinate Macdonald-Ruijsenaars (MR) type equations for the weighted trace functions. We will show that the equations for $i=0$ reduce to the dual MR equations from \cite[Thm. 1.2]{Etingof&Varchenko-2000}. 

Let \(\lambda,\mu\in\hh_{\mathrm{reg}}^\ast\) with \(\Re(\lambda)\) deep in the negative Weyl chamber,  \(v_j\in V_j[\xi_j]\) and \(f_j\in V_j^\ast[\xi'_j]\) such that \(\sum_j\xi_j
= 0 = \sum_j\xi_j'\) and set
\begin{equation}
\label{Xi def}
\Xi_{S,i}^{v_1,\dots,v_k\vert f_k,\dots, f_1}(\lambda,\mu):= \widehat{e}_{\ul{S}}\left(\Tr_{M_\mu}(\Phi_\mu^{v_1,\dots,v_k}\circ\pi_\mu(q^{2\lambda+2\rho}))\otimes\mathbb{A}_{\lambda,i}(f_k\otimes\dots\otimes f_1) \right),
\end{equation}
with $\mathbb{A}_{\lambda,i}\in\textup{End}_{\cN_{\textup{fd}}}\bigl(\cF^\str(\ul{S^*})\bigr)$ the linear operator
\[
\mathbb{A}_{\lambda,i}:=j_{S^\ast}(\lambda)\Rdyn_{(V_k^\ast,\dots,V_{i+1}^\ast),(V_{i}^\ast,\dots,V_1^\ast)}(\lambda)^{-1}.
\]
We can now derive from Corollary \ref{thm MR for Y} the following equations for $\Xi_{S,i}^{v_1,\dots,v_k\vert f_k,\dots, f_1}(\lambda,\mu)$. The proof is analogous to the proof of Corollary \ref{cor qKZB for Theta}. We will sketch the main steps of the proof for the sake of completeness.
\begin{proposition}
	\label{prop MR for Theta}
	For any \(W\in\Rep\) and \(i=0,\dots,k\), we have the equation
	\begin{equation}
	\label{MR for Theta}
	\Xi_{S,i}^{v_1,\dots,v_k\vert \mathcal{D}_{\lambda,W}^{\vee,(i)}(f_k,\dots,f_1)}(\lambda,\mu) = \sum_{\sigma\in\wts(W)}\Xi_{S,i}^{\mathcal{K}_{\mu,\sigma,W}^{\vee,(i)}(v_1,\dots,v_k)\vert f_k,\dots,f_1}(\lambda,\mu-\sigma),
	\end{equation}
	with \(\mathcal{K}_{\mu,\sigma,W}^{\vee,(i)}\in\End_{\cN_{\textup{fd}}}(\cF^\str(\ul{S})) \) and \(\mathcal{D}_{\lambda,W}^{\vee,(i)} \in\End_{\cN_{\textup{fd}}}(\cF^\str(\ul{S^\ast}))\) defined by
	\begin{align}
	\label{K_mu,sigma,i vee def}
	\begin{split}
	\mathcal{K}_{\mu,\sigma,W}^{\vee,(i)} :=\ & \Tr_{W[\sigma]}\left(\Rdyn^{21}_{W,V_k}(\mu)^{-1}\Rdyn^{21}_{W,V_{k-1}}(\mu-\mh_{V_k})^{-1}\dots \Rdyn^{21}_{W,V_{i+1}}(\mu-\mh_{(V_{i+2,\dots,V_k})})^{-1}\right.\\& 
	\qquad\qquad\qquad\circ\Rdyn_{W,V_i}(\mu-\mh_{(V_{i+1},\dots,V_k)})\dots \Rdyn_{W,V_1}(\mu-\mh_{(V_{2},\dots,V_k)}) \Big),
	\end{split} \\
	\label{D_lambda,i vee def}
	\mathcal{D}_{\lambda,W}^{\vee,(i)} :=\ & \Tr_W\left((q^{-2\lambda-2\rho})_W\,\kappa^{2}_{W,V_k^\ast}\dots\kappa^{2}_{W,V_{i+1}^\ast}\right).
	\end{align}
\end{proposition}
\begin{proof}
Set 
\begin{equation*}
\begin{split}
\nu_1&:=\xi_1+\cdots+\xi_i,\qquad \nu_2:=\xi_{i+1}+\cdots+\xi_k,\\
\nu_1'&:=\xi_1'+\cdots+\xi_i',\qquad \nu_2':=\xi_{i+1}'+\cdots+\xi_k',
\end{split}
\end{equation*}
so that $\nu_1+\nu_2=0=\nu_1'+\nu_2'$. Define vectors $w_j\in\cF^\str(S_j)[\nu_j]$ and covectors $g_j\in\cF^\str(S_j^*)[\nu_j']$ for $j=1,2$ by
\begin{align}\label{choice of w, g}
\begin{split}
	w_1 = j_{(V_1,\dots,V_i)}(\mu-\nu_2)(v_1\otimes\dots\otimes v_i), \qquad w_2 = j_{(V_{i+1},\dots,V_k)}(\mu)(v_{i+1}\otimes\dots\otimes v_k),\\
	g_1 = j_{(V_i^\ast,\dots,V_1^\ast)}(\lambda-\nu'_2)(f_i\otimes\dots\otimes f_1), \qquad g_2 = j_{(V_k^\ast,\dots,V_{i+1}^\ast)}(\lambda)(f_k\otimes\dots\otimes f_{i+1}).
\end{split}
\end{align}
Then 
\[
\Phi_{\mu;S_1,S_2}^{w_1,w_2} = \Phi_\mu^{v_1,\dots,v_k}
\]
by (\ref{large intertwiner}), and for the dual intertwiner $\Psi_{\lambda;S_2^*,S_1^*}^{g_2,g_1}$ 
(see \eqref{Psi2}) we have
\[
\langle \Psi_{\lambda;S_2^*,S_1^*}^{g_2,g_1}\rangle=\mathbb{A}_{\lambda,i}(f_k\otimes\cdots\otimes f_1)
\]
by Lemma \ref{prop Theta is Z} (it is the special case that $S_3=\emptyset$). Hence
	\begin{equation}
	\label{Xi is also Z}
	\Xi_{S,i}^{v_1,\dots,v_k\vert f_k,\dots, f_1}(\lambda,\mu) = \mathcal{Z}_{\bs{S}}^{\bs{w},\bs{g}}(\lambda,\mu)
	\end{equation}
	by \eqref{Y def MR} and \eqref{Zspec}, where $\bs{S}=(S_1,S_2)$, $\bs{w}=(w_1,w_2)$ and $\bs{g}=(g_1,g_2)$.
	
	Note also that by Proposition \ref{prop hexagon for dynamical R-matrix} we have
	\[
	\mathcal{K}_{\mu,\sigma,W}^{\vee,(i)} = \Tr_{W[\sigma]}\left(\Rdyn^{21}_{W,(V_{i+1},\dots,V_k)}(\mu)^{-1}\Rdyn_{W,(V_1,\dots,V_i)}(\mu-\mh_{(V_{i+1},\dots,V_k)}) \right),
	\]
	and moreover
	\begin{equation}\label{Dal}
	\mathcal{D}_{\lambda,W}^{\vee,(i)} = 
	\Tr_W\left((q^{-2\lambda-2\rho})_W\,\kappa^2_{W,(V_k^\ast,\dots,V_{i+1}^\ast)} \right).
	\end{equation}
	Consequently, the result of Corollary \ref{thm MR for Y} can be reformulated as
	\[
	\sum_{n} \mathcal{Z}_{\boldsymbol{S}}^{\bs{w},(\varpi_{2}^n, \varpi_{1}^n)}(\lambda,\mu) = \sum_{m}\sum_{\sigma\in\wts(W)} \mathcal{Z}_{\boldsymbol{S}}^{(\upsilon_{1,\sigma}^m, \upsilon_{2,\sigma}^m), \bs{g}}(\lambda,\mu-\sigma),
	\]
	where \(\upsilon_{j,\sigma}^m\in\cF^\str(S_j)\) and \(\varpi_{j}^n\in\cF^\str(S_j^\ast)\) are homogeneous vectors such that 
	\begin{equation}
	\label{the funny upsilon}
	\sum_m \upsilon_{1,\sigma}^m\otimes \upsilon_{2,\sigma}^m=\mathscr{J}^\vee_{\boldsymbol{S}}(\mu-\sigma)\,\mathcal{K}_{\mu,\sigma,W}^{\vee,(i)}
	(v_1\otimes\dots\otimes v_k),
	\end{equation}
	and
	\begin{equation}
	\label{the funny varpi again}
	\begin{split}
	\sum_n \varpi_{2}^n\otimes \varpi_{1}^n&= \mathcal{D}_{\lambda,W}^{\vee,(i)}\bigl(j_{(V_k^\ast,\dots,V_{i+1}^\ast)}(\lambda)\otimes j_{(V_{i}^\ast,\dots,V_1^\ast)}(\lambda-\nu_2') \bigr)(f_k\otimes\dots\otimes f_1)\\
	&=\bigl(j_{(V_k^\ast,\dots,V_{i+1}^\ast)}(\lambda)\otimes j_{(V_{i}^\ast,\dots,V_1^\ast)}(\lambda-\nu_2') \bigr)\mathcal{D}_{\lambda,W}^{\vee,(i)}
	(f_k\otimes\cdots\otimes f_1)
	\end{split}
	\end{equation}
	where the last equality in \eqref{the funny varpi again} follows from \eqref{Dal} and the $\mathfrak{h}$-linearity of the dynamical fusion operators.
	The result now immediately follows from the identity (\ref{Xi is also Z}).
\end{proof}

The next step is proving that the normalized weighted trace functions \(\TT_S^{v_1,\dots,v_k\vert f_k,\dots, f_1}(\lambda,\mu)\) satisfy the same equations (\ref{MR for Theta}) of dual MR type. The reasoning goes along the same lines as in Section \ref{Subsection Etingof-Varchenko normalization} for the dual $q$-KZB equations. By definition we have
\begin{equation}
\label{expression for Xi}
\Xi_{S,i}^{v_1,\dots,v_k\vert f_k,\dots,f_1}(\lambda,\mu) = \widehat{e}_{\ul{S}}\left(\Tr_{M_\mu}(\Phi_\mu^{v_1,\dots,v_k}\circ\pi_\mu(q^{2\lambda+2\rho}))\tens\langle\Psi\rangle\right),
\end{equation}
where \(\Psi\in \Hom_{\cM_{\mr{adm}}^\str}(M_\lambda,S^\ast\tens M_\lambda)\) is such that \(\langle\Psi\rangle = \mathbb{A}_{\lambda,i}(f_k\otimes\dots\otimes f_1)\).  We now consider for which other intertwiners \(\Psi\in\Hom_{\cM_{\mr{adm}}^\str}(M_\lambda,S^\ast\tens M_\lambda)\) the resulting expression \eqref{expression for Xi} satisfies
equation \eqref{MR for Theta}.  It turns out that $\Psi$ should satisfy the equation obtained from Figure \ref{before topological MR}
  upon taking highest-weight-to-highest-weight components (with the appropriate choices for the $\lambda_j$). 
\begin{lemma}\label{lemma equation for Psi}
Let \(W\in\Rep\), \(\lambda_0,\lambda_1,\lambda_2\in\hh_{\mathrm{reg}}^\ast\) and $S=S_1\tens S_2\in\Rep^\str$.
The intertwiner
$\Psi\in\textup{Hom}_{\mathcal{M}_{\textup{adm}}^\str}\bigl(M_{\lambda_2},S^\ast\tens M_{\lambda_0}\bigr)$ satisfies
\begin{figure}[H]
		\centering
		\includegraphics[scale=0.75]{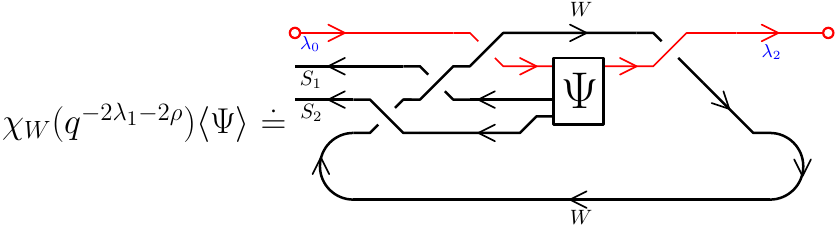}
\end{figure}
if and only if 
\begin{equation}
	\label{eq for Psi}
	\chi_W(q^{-2\lambda_1-2\rho})\langle\Psi\rangle = \Tr_W\left(\cR^{21}_{W,S_2^\ast}\,\cR^{-1}_{W,S_1^\ast}\,(q^{-\lambda_0-\lambda_2-2\rho})_W\right)\langle\Psi\rangle
	\end{equation}
in $\cF^\str(\ul{S^*})[\lambda_2-\lambda_0]$.
\end{lemma}
\begin{proof}
	In the graphical calculus for \(\Nadm^\str\), Figure NEW ONE is equivalent to \(\chi_W(q^{-2\lambda_1-2\rho})\langle\Psi\rangle\) being represented by
	\begin{center}
		\includegraphics[scale=0.75]{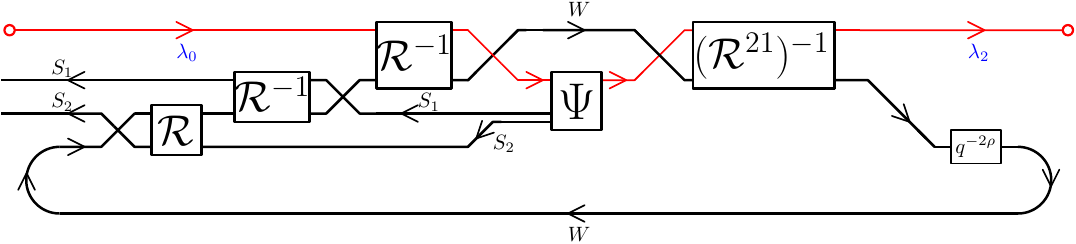}
	\end{center}
	By Lemma \ref{lemma 3.10}, this becomes
	\begin{center}
		\includegraphics[scale=0.75]{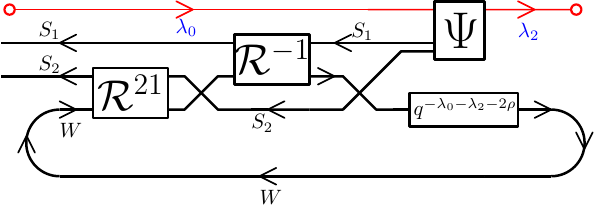}
	\end{center}
	which represents the right-hand side of (\ref{eq for Psi}).
\end{proof}
\noindent
Note that Lemma \ref{lemma before topological MR} gives a large class of dual quantum vertex operators $\Psi$ that satisfy \eqref{eq for Psi}.

\begin{corollary}\label{generalcaseMR}
Suppose that $\widetilde{\mathbb{A}}_{\lambda}\in\textup{End}_{\mathbb{C}}\bigl(\cF^\str(\ul{S^*})[0]\bigr)$ satisfies 
\begin{equation}
\label{equations for L}
\widetilde{\mathbb{A}}_\lambda\circ\mathcal{D}_{\lambda,W}^{\vee,(i)} = \Tr_W\left(\cR^{21}_{W,(V_k^\ast,\dots,V_{i+1}^\ast)}\cR^{-1}_{W,(V_i^\ast,\dots,V_1^\ast)}\,(q^{-2\lambda-2\rho})_W\right)\circ\widetilde{\mathbb{A}}_\lambda
\end{equation}
for all \(i = 0,\dots,k\) and \(W\in\Rep\). Then
\begin{equation}
\label{expression for tilde Xi}
\widetilde{\Xi}_{S}^{v_1,\dots,v_k\vert f_k,\dots,f_1}(\lambda,\mu) := \widehat{e}_{\ul{S}}\left(\Tr_{M_\mu}(\Phi_\mu^{v_1,\dots,v_k}\circ\pi_\mu(q^{2\lambda+2\rho}))\tens\widetilde{\mathbb{A}}_\lambda(f_k\otimes\cdots\otimes f_1)\right)
\end{equation}
satisfies 
\begin{equation}
	\label{MR for tilde Theta}
	\widetilde{\Xi}_{S}^{v_1,\dots,v_k\vert \mathcal{D}_{\lambda,W}^{\vee,(i)}(f_k,\dots,f_1)}(\lambda,\mu) = \sum_{\sigma\in\wts(W)}
	\widetilde{\Xi}_{S}^{\mathcal{K}_{\mu,\sigma,W}^{\vee,(i)}(v_1,\dots,v_k)\vert f_k,\dots,f_1}(\lambda,\mu-\sigma)
	\end{equation}
	for $i=0,\ldots,k$ and $W\in\Rep$.
\end{corollary}
\begin{proof}

Revisiting the proof of Corollary \ref{thm MR for Y} and taking Lemma \ref{lemma equation for Psi} into account, it follows that formula \eqref{equation MR with general xi} 
for $\mathcal{Z}_{\bs{S}}^{\bs{w},\bs{g}}(\lambda,\mu)$ will also hold true when the intertwiner $\Psi=\Psi_{\lambda;S_2^*,S_1^*}^{g_2,g_1}$ in the definition of $\mathcal{Z}_{\bs{S}}^{\bs{w},\bs{g}}(\lambda,\mu)$
(see \eqref{Zspec} and \eqref{Y def MR}) is replaced by any intertwiner $\Psi\in\textup{Hom}_{\mathcal{M}_{\textup{adm}}^\str}\bigl(M_\lambda,S^*\tens M_\lambda\bigr)$ with expectation value satisfying \eqref{eq for Psi} for 
\begin{equation}\label{lambdaspecial}
(\lambda_0,\lambda_1,\lambda_2)=(\lambda,\lambda-\xi_{i+1}'\cdots-\xi_{k}',\lambda).
\end{equation}
It remains to note that this holds true if the expectation value of $\Psi\in\textup{Hom}_{\mathcal{M}_{\textup{adm}}^\str}(M_\lambda,S^*\tens M_\lambda)$ is of the form
$\langle\Psi\rangle=\widetilde{\mathbb{A}}_\lambda(f_k\otimes\cdots\otimes f_1)$ with $\widetilde{\mathbb{A}}_\lambda$ satisfying \eqref{equations for L}.
\end{proof}
The next goal is to show that $\widetilde{\mathbb{A}}_\lambda=(j_S(-\lambda-2\rho)^{-1})^T\vert_{\cF^\str(\ul{S^*})[0]}$ satisfies \eqref{equations for L}. For this we require the following lemma.
\begin{lemma}
	\label{lemma before inverse MR}
	Let \(W\in\Rep\), \(\lambda_0,\lambda_1,\lambda_2\in\hh_{\mathrm{reg}}^\ast\) and $S=S_1\tens S_2$. Define 
	\[
	\Phi:= (\Phi_1\tens\id_{S_2})\Phi_2 \in \Hom_{\cM_{\mr{adm}}^\str}(M_{\lambda_2}, M_{\lambda_0}\tens S)
	\]
	with \(\Phi_i\in\Hom_{\cM_{\mr{adm}}^\str}(M_{\lambda_i},M_{\lambda_{i-1}}\tens S_i)\) for \(i = 1,2\). 
	Then we have
	\begin{equation}
	\label{eq for Phi}
	\chi_W(q^{2\lambda_1+2\rho})\langle\Phi\rangle = \Tr_W\left((\cR^{21}_{W,S_2})^{-1}\cR_{W,S_1}\,(q^{\lambda_0+\lambda_2+2\rho})_W\right)\langle\Phi\rangle.
	\end{equation}
\end{lemma}
\begin{proof}
This follows by deriving the analog of Lemma \ref{lemma before topological MR} for quantum vertex operators $\Phi$ instead of dual quantum vertex operators $\Psi$, and subsequently taking highest-weight-to-highest-weight components (cf. Lemma \ref{lemma equation for Psi}).
\end{proof}
This leads to Macdonald-Ruijsenaars' type equations for the \(k\)-point dynamical fusion operator. 
\begin{corollary}
	\label{lemma inverse MR for j}
	Let \(\lambda\in\hh_{\mathrm{reg}}^\ast\), \(W\in\Rep\) and \(S = (V_1,\dots,V_k)\in\Rep^\str\). Then for any \(i = 0,\dots,k\) we have
	\begin{equation}\label{dMRj}
	\begin{split}
	j_S(\lambda)&\circ \Tr_W\left((q^{2\lambda+2\rho})_W\,\kappa^{-2}_{W,V_{i+1}}\dots\kappa^{-2}_{W,V_k} \right) =\\ 
	=\ & \Tr_W\left((\cR^{21}_{W,(V_{i+1},\dots,V_k)})^{-1}\cR_{W,(V_1,\dots,V_i)}\,(q^{2\lambda+2\rho})_W\,\kappa^{-1}_{W,V_1}\dots\kappa^{-1}_{W,V_k} \right) \circ j_S(\lambda)
	\end{split}
	\end{equation}
in $\textup{End}_{\cN_{\mr{fd}}}(\cF^\str(\ul{S}))$.
\end{corollary}
\begin{proof}
It is easy to check that all operators in \eqref{dMRj} are $\mathfrak{h}$-linear.
	Let \(v_j\in V_j[\xi_j]\) and take \(S_1 = (V_1,\dots,V_i)\), \(S_2 = (V_{i+1},\dots,V_k)\), \(\lambda_0 = \lambda-\xi_1-\dots-\xi_k\), \(\lambda_1 = \lambda-\xi_{i+1}-\dots-\xi_k\), \(\lambda_2 = \lambda\), \(\Phi_1=\Phi_{\lambda_1}^{v_1,\dots,v_i}\) and \(\Phi_2=\Phi_\lambda^{v_{i+1},\dots,v_k}\). Then Lemma \ref{lemma before inverse MR} asserts that
	\begin{equation*}
	\begin{split}
	\chi_W(q^{2\lambda+2\rho-2\xi_{i+1}-\dots-2\xi_k})&\langle\Phi_\lambda^{v_1,\dots,v_k}\rangle=\\ 
	=\ & \Tr_W\left((\cR^{21}_{W,(V_{i+1},\dots,V_k)})^{-1}\cR_{W,(V_1,\dots,V_i)}\,(q^{2\lambda+2\rho-\xi_1-\dots-\xi_k})_W\right)\langle\Phi_\lambda^{v_1,\dots,v_k}\rangle.
	\end{split}
	\end{equation*}
	This leads to the desired formula \eqref{dMRj} in view of the definition of \(\kappa\) and the character \(\chi_W\), and upon noting that \(\langle\Phi_\lambda^{v_1,\dots,v_k}\rangle = j_S(\lambda)(v_1\otimes\dots\otimes v_k)\).
\end{proof}
\begin{proposition}\label{propdMRj}
The linear operator $\widetilde{\mathbb{A}}_\lambda:=(j_S(-\lambda-2\rho)^{-1})^T\vert_{\cF^\str(\ul{S^*})[0]}$ satisfies 
\eqref{equations for L} for all $W\in\Rep$ and all $i=0,\ldots,k$.
\end{proposition}
\begin{proof}
Note that the linear operator
\[
L:=\Tr_W\left((\cR^{21}_{W,(V_{i+1},\dots,V_k)})^{-1}\cR_{W,(V_1,\dots,V_i)}\,(q^{-2\lambda-2\rho})_W \right) 
\]
on $\cF^\str(S)$ is $\mathfrak{h}$-linear. We claim that $L^T\vert_{\cF^\str(\ul{S^*})[0]}$ (see \eqref{transposedef}) is the restriction of the $\mathfrak{h}$-linear operator
\begin{equation}\label{transposebig}
\Tr_W\left(\cR^{21}_{W,(V_k^\ast,\dots,V_{i+1}^\ast)}\cR^{-1}_{W,(V_i^\ast,\dots,V_1^\ast)}\,(q^{-2\lambda-2\rho})_W\right)
\end{equation}
on $\cF^\str(S^*)$ to $\cF^\str(\ul{S^*})[0]$. To verify this, note that by \eqref{hattilde} $L^T\vert_{\cF^\str(\ul{S^*})[0]}$ is the restriction to $\cF^\str(\ul{S^*})[0]$ of
the transpose of $L$ relative to the pairing 
$\widetilde{e}_S: \cF^\str(S)\otimes\cF^\str(S^*)\rightarrow\mathbb{C}$. To compute the transpose of $L$ relative to $\widetilde{e}_S$, first note that
\begin{equation*}
(\cR^{21}_{W,(V_{i+1},\dots,V_k)})^{-1}\cR_{W,(V_1,\dots,V_i)}=(\cR^{21}_{W,V_k})^{-1}\cdots(\cR^{21}_{W,V_{i+1}})^{-1}\cR_{W,V_i}\cdots\cR_{W,V_1},
\end{equation*}
which follows from  
\begin{equation}\label{DeltaR}
(\textup{id}\otimes\Delta)(\mathcal{R})=\mathcal{R}^{13}\mathcal{R}^{12},\qquad (\Delta\otimes\textup{id})(\mathcal{R})=\mathcal{R}^{13}\mathcal{R}^{23}.
\end{equation} Using in addition that $(\textup{id}\otimes S^{-1})(\mathcal{R})=\mathcal{R}^{-1}=(S\otimes\textup{id})(\mathcal{R})$ and \eqref{action of q^2rho}, the transpose of $L$ relative to $\widetilde{e}_S$ is 
\begin{equation}\label{transposebig2}
(q^{2\rho})_{(V_k^*,\ldots,V_1^*)}\Tr_W\left(\cR^{21}_{W,(V_k^\ast,\dots,V_{i+1}^\ast)}\cR^{-1}_{W,(V_i^\ast,\dots,V_1^\ast)}\,(q^{-2\lambda-2\rho})_W\right)
(q^{-2\rho})_{(V_k^*,\ldots,V_1^*)}.
\end{equation}
Here we have also used the fact that 
\begin{equation*}
\cR^{21}_{W,V_k^*}\cdots\cR^{21}_{W,V_{i+1}^*}\cR_{W,V_i^*}^{-1}\cdots\cR_{W,V_1^*}^{-1}=\cR^{21}_{W,(V_k^\ast,\dots,V_{i+1}^\ast)}\cR^{-1}_{W,(V_i^\ast,\dots,V_1^\ast)},
\end{equation*}
which follows from \eqref{DeltaR}. Upon restricting \eqref{transposebig2} to $\cF^\str(\ul{S^*})[0]$, the action of $q^{\pm 2\rho}$ drops out, hence $L^T\vert_{\cF^\str(\ul{S^*})[0]}$
is the restriction of \eqref{transposebig} to $\cF^\str(\ul{S^*})[0]$.

After this observation, it suffices to prove that $\widehat{\mathbb{A}}_\lambda:=j_S^{-1}(-\lambda-2\rho)\vert_{\cF^\str(\ul{S})[0]}$ satisfies
\begin{equation}
	\label{equations for K}
	\widehat{\mathbb{A}}_\lambda \circ \Tr_W\left((\cR^{21}_{W,(V_{i+1},\dots,V_k)})^{-1}\cR_{W,(V_1,\dots,V_i)}\,(q^{-2\lambda-2\rho})_W \right) 	
	=\widetilde{\mathcal{D}}_{\lambda,W}^{\vee,(i)}\circ \widehat{\mathbb{A}}_\lambda
	\end{equation}
	as linear operators on $\cF^\str(\ul{S})[0]$, with \(\widetilde{\mathcal{D}}_{\lambda,W}^{\vee,(i)} \in \End_{\cN_{\textup{fd}}}(\cF^\str(\ul{S}))\) defined by
	\begin{equation}
	\label{D tilde def}
	\widetilde{\mathcal{D}}_{\lambda,W}^{\vee,(i)}:= \Tr_W\left((q^{-2\lambda-2\rho})_W\,\kappa^{-2}_{W,V_{i+1}}\dots\kappa^{-2}_{W,V_{k}} \right)
	\end{equation}
	(here we used that the operators $\kappa_{W,V_j}$ for $j=1,\ldots,k$ pairwise commute).
This is an immediate consequence of Corollary \ref{lemma inverse MR for j}, upon noting that \(\kappa^{-1}_{W,V_1}\dots\kappa^{-1}_{W,V_k}\) acts as the identity on \(\ul{W}\otimes\cF^\str(\ul{S})[0]\).
\end{proof}

This leads to the following equations for the normalized spin components  \(\TT_S^{v_1,\dots,v_k\vert f_k,\dots,f_1}(\lambda,\mu)\) of the weighted trace function, defined in Definition \ref{definitionTT} (we will see later at the end of this subsection that the equations for $i=0$ are equivalent to the dual MR equations in \cite[Thm. 1.2]{Etingof&Varchenko-2000}).

\begin{theorem}
	\label{thm dual MR}
Let \(\lambda,\mu\in\hh_{\mathrm{reg}}^\ast\) with \(\Re(\lambda)\) deep in the negative Weyl chamber. Let \(S = (V_1,\dots,V_k)\in\Rep^\str\), \(v_j\in V_j[\xi_j]\) and \(f_j\in V_j^\ast[\xi_j']\) such that \(\xi_1+\dots+\xi_k = 0 = \xi_1'+\dots+\xi_k' \). 

	For any \(i = 0,\dots,k\) and \(W\in\Rep\) we have
	\[
	\TT_S^{v_1,\dots,v_k\vert \mathcal{D}_{\lambda,W}^{\vee,(i)}(f_k,\dots,f_1)}(\lambda,\mu) = \sum_{\sigma\in\wts(W)} \TT_S^{\mathcal{K}_{\mu,\sigma,W}^{\vee,(i)}(v_1,\dots,v_k)\vert f_k,\dots,f_1}(\lambda,\mu-\sigma)
	\]
	with \(\mathcal{K}_{\mu,\sigma,W}^{\vee,(i)}\) and \(\mathcal{D}_{\lambda,W}^{\vee,(i)}\) defined by (\ref{K_mu,sigma,i vee def})--(\ref{D_lambda,i vee def}).
\end{theorem}
\begin{proof}
This follows immediately from Definition \ref{definitionTT}, Corollary \ref{generalcaseMR} and Proposition \ref{propdMRj}.
\end{proof}
Similarly as we did for the dual $q$-KZB equations in the previous subsection, we reformulate the equations from Theorem \ref{thm dual MR} as difference equations for the enormalised $k$-point weighted trace functions $\mathbb{F}_S(\lambda,\mu)$ and $\mathbb{F}_S^{f_k,\ldots,f_1}(\lambda,\mu)$. As a first step we have the following immediate corollary of Theorem \ref{thm dual MR} (cf. Corollary \ref{dualqKZBcor} in the context of the dual $q$-KZB equations).
\begin{corollary}\label{dualqMRcor}
With the notations and assumptions as in Theorem \ref{thm dual MR}, we have
\begin{enumerate}
\item For $i=0,\ldots,k$,
\[
\sum_{\sigma\in\textup{wts}(W)}\Bigl(\textup{id}_{\cF^\str(\ul{S})[0]}\otimes\bigl(\mathcal{K}_{\mu,\sigma,W}^{\vee,(i)}\bigr)^T\Bigr)\mathbb{T}_S(\lambda,\mu-\sigma)=\Bigl(\bigl(\mathcal{D}_{\lambda,W}^{\vee,(i)}\bigr)^{T\ast}\otimes\textup{id}_{\cF^\str(\ul{S}^*)[0]}\Bigr)\mathbb{T}_S(\lambda,\mu).
\]
\item For $i=0,\ldots,k$,
\[
\sum_{\sigma\in\textup{wts}(W)}\bigl(\mathcal{K}_{\mu,\sigma,W}^{\vee,(i)}\bigr)^T\mathbb{T}_S^{f_k,\ldots,f_1}(\lambda,\mu-\sigma)=
\chi_W\bigl(q^{-2(\lambda-\xi_{i+1}^\prime\cdots-\xi_k^\prime)-2\rho}\bigr)\mathbb{T}_S^{f_k,\ldots,f_1}(\lambda,\mu).
\]
\end{enumerate}
\end{corollary}
Note that for $i=0,\ldots,k$ we have
\begin{equation}\label{DiWcheckformula}
\mathbb{D}_{\lambda,W,i}^{\vee,S}:=\bigl(\mathcal{D}_{\lambda,W}^{\vee,(i)}\bigr)^{T*}=\textup{Tr}_W\bigl((q^{-2\lambda-2\rho}\bigr)_W\kappa_{WV_{i+1}}^{-2}\cdots\kappa_{WV_k}^{-2}\bigr).
\end{equation}
The equations in Corollary \ref{dualqMRcor} now give rise to equations for $\mathbb{F}_S(\lambda,\mu)$ involving $\mathbb{D}_{\lambda,W,i}^{\vee,S}$ and the following difference operators in $\mu$. Note that $R_{W^*V_j^*}^{21}(\mu+h_{(W^*,V_j^*)})$ is invertible since $R_{W^*V_j^*}^{21}(\mu)$ is $\mathfrak{h}$-linear.

\begin{definition}\label{idualMRoper}
Let $W\in\Rep$ and $S=(V_1,\ldots,V_k)\in\Rep^\str$. For $i=0,\ldots,k$ we call the difference operators
\begin{equation*}
\begin{split}
\bigl(\mathbb{L}_{W,i}^{\vee,S}f\bigr)(\mu):=&\sum_{\sigma\in\textup{wts}(W)}\textup{Tr}_{W^*[-\sigma]}\Bigl(
R_{W^*V_1^*}(\mu+h_{(W^*,V_k^*,\ldots,V_1^*)})\cdots R_{W^*V_i^*}(\mu+h_{(W^*,V_k^*,\ldots,V_i^*)})\\
&\qquad\qquad\quad R_{W^*V_{i+1}^*}^{21}(\mu+h_{(W^*,V_k^*,\ldots,V_{i+1}^*)})^{-1}\cdots
R_{W^*V_k^*}^{21}(\mu+h_{(W^*,V_k^*)})^{-1}\Bigr)f(\mu-\sigma)
\end{split}
\end{equation*}
for functions $f: \mathfrak{h}_{\textup{reg}}^*\rightarrow\cF^\str(\ul{S^*})[0]$ the coordinate dual Macdonald-Ruijsenaars \textup{(}MR\textup{)} operators associated to $W$ and $S$.
\end{definition}
The formula for the coordinate dual MR operator $\mathbb{L}_{W,i}^{\vee,S}$ can be compressed as follows.
\begin{remark}\label{Lcompact}
By Proposition \ref{prop hexagon for dynamical R-matrix}(1) we have
\begin{equation*}
\begin{split}
(\mathbb{L}_{W,i}^{\vee,S}f)(\mu)=\sum_{\sigma\in\textup{wts}(W)}\textup{Tr}_{W^*[-\sigma]}\Bigl(&R_{W^*,(V_i^*,\ldots,V_1^*)}(\mu+h_{(W^*,V_k^*,\ldots,V_1^*)})\\
&\quad
R^{21}_{W^*,(V_k^*,\ldots,V_{i+1}^*)}(\mu+h_{(W^*,V_k^*,\ldots,V_{i+1}^*)})^{-1}\Bigr)f(\mu-\sigma)
\end{split}
\end{equation*}
for $f: \mathfrak{h}_{\textup{reg}}^*\rightarrow\cF^\str(\ul{S^*})[0]$. For $i=0$ and $i=k$ this reduces to
\begin{equation}\label{case0k}
\begin{split}
(\mathbb{L}_{W,0}^{\vee,S}f)(\mu)&=\sum_{\sigma\in\textup{wts}(W)}\textup{Tr}_{W^*[-\sigma]}\Bigl(R^{21}_{W^*,(V_k^*,\ldots,V_{1}^*)}(\mu-\sigma)^{-1}\Bigr)f(\mu-\sigma),\\
(\mathbb{L}_{W,k}^{\vee,S}f)(\mu)&=\sum_{\sigma\in\textup{wts}(W)}\textup{Tr}_{W^*[-\sigma]}\Bigl(R_{W^*,(V_k^*,\ldots,V_1^*)}(\mu-\sigma)\Bigr)f(\mu-\sigma).
\end{split}
\end{equation}
\end{remark}
The difference equations resulting from Corollary \ref{dualqMRcor} now read as follows.
\begin{corollary}[Dual MR equations]\label{dualqMRthm}
Let \(\lambda\in\hh_{\mathrm{reg}}^\ast\) with \(\Re(\lambda)\) deep in the negative Weyl chamber. Let \(S = (V_1,\dots,V_k)\in\Rep^\str\) and let \(f_j\in V_j^\ast[\xi_j']\) such that \(\xi_1'+\dots+\xi_k'=0\). Let $W\in\Rep$.
\begin{enumerate}
\item For $i=0,\ldots,k$,
\begin{equation}\label{dualMReqnuniversal}
\bigl(\textup{id}_{\cF^\str(\ul{S})[0]}\otimes\mathbb{L}_{W,i}^{\vee,S}\bigr)\mathbb{F}_S(\lambda,\cdot)=
\bigl(\mathbb{D}_{\lambda,W,i}^{\vee,S}\otimes\textup{id}_{\cF^\str(\ul{S^*})[0]}\bigr)\mathbb{F}_S(\lambda,\cdot).
\end{equation}
\item For $i=0,\ldots,k$,
\begin{equation}\label{dualMReqn}
\mathbb{L}_{W,i}^{\vee,S}\mathbb{F}_S^{f_k,\ldots,f_1}(\lambda,\cdot)=\chi_W\bigl(q^{-2(\lambda-\xi_{i+1}^\prime-\cdots-\xi_k^\prime)-2\rho}\bigr)\mathbb{F}_S^{f_k,\ldots,f_1}(\lambda,\cdot).
\end{equation}
\end{enumerate}
\end{corollary}
\begin{proof}
By the definition of the renormalised weighted trace function $\mathbb{F}_S(\lambda,\mu)$ (see \eqref{Fnorm}) and Corollary \ref{dualqMRcor} it suffices to prove that
\begin{equation}\label{todddo}
\begin{split}
&\mathbb{X}_{\mu,S^*}\bigl(\mathcal{K}_{\mu,\sigma,W}^{\vee,(i)}\bigr)^T\,\mathbb{X}_{\mu-\sigma,S^*}^{-1}\,\vert_{\cF^\str(\ul{S^*})[0]}=\\
&\quad=\textup{Tr}_{W^*[-\sigma]}\Bigl(
R_{W^*V_1^*}(\mu+h_{(W^*,V_k^*,\ldots,V_1^*)})\cdots R_{W^*V_i^*}(\mu+h_{(W^*,V_k^*,\ldots,V_i^*)})\\
&\qquad\qquad\quad R_{W^*V_{i+1}^*}^{21}(\mu+h_{(W^*,V_k^*,\ldots,V_{i+1}^*)})^{-1}\cdots
R_{W^*V_k^*}^{21}(\mu+h_{(W^*,V_k^*)})^{-1}\Bigr)\vert_{\cF^{\str}(\ul{S^*})[0]}.
\end{split}
\end{equation}
Rewrite the transpose of the linear operator $\mathcal{K}_{\mu,\sigma,W}^{\vee,(i)}$ (see \eqref{K_mu,sigma,i vee def}) by also taking the transpose over the auxiliary space $W$, which in particular results in replacing the trace over $W[\sigma]$ by a trace over $W^*[-\sigma]$. We then obtain
\begin{equation}\label{KTred}
\begin{split}
\bigl(\mathcal{K}_{\mu,\sigma,W}^{\vee,(i)}\bigr)^T&=\textup{Tr}_{W^*[-\sigma]}\Bigl(
R_{WV_1}(\mu+h_{(V_k^*,\ldots,V_2^*)})^T\cdots R_{WV_i}(\mu+h_{(V_k^*,\ldots,V_{i+1}^*)})^T\\
&\qquad\qquad\qquad\qquad R^{21}_{WV_{i+1}}(\mu+h_{(V_k^*,\ldots,V_{i+2})})^{-1,T}\cdots R^{21}_{WV_k}(\mu)^{-1,T}\Bigr)
\end{split}
\end{equation}
where $R_{UV}(\mu)^T\in\textup{End}_{\cN_{\textup{fd}}}(\ul{V^*}\otimes\ul{U^*})$ is the transpose of $R_{UV}(\mu)$ (see \eqref{transposedef}).
By \cite[(3.12)]{Etingof&Varchenko-2000} we have
\begin{equation*}
\begin{split}
R_{UV}(\mu)^T&=\bigl(\mathbb{Q}_{V^*}(\mu)\otimes\mathbb{Q}_{U^*}(\mu+h_{V^*})\bigr)
R^{21}_{V^*U^*}(\mu+h_{(V^*,U^*)})\bigl(\mathbb{Q}_{V^*}(\mu+h_{U^*})^{-1}\otimes\mathbb{Q}_{U^*}(\mu)^{-1}\bigr),\\
R^{21}_{UV}(\mu)^{-1,T}&=\bigl(\mathbb{Q}_{V^*}(\mu)\otimes\mathbb{Q}_{U^*}(\mu+h_{V^*})\bigr)
R_{V^*U^*}(\mu+h_{(V^*,U^*)})^{-1}
\bigl(\mathbb{Q}_{V^*}(\mu+h_{U^*})^{-1}\otimes\mathbb{Q}_{U^*}(\mu)^{-1}\bigr),
\end{split}
\end{equation*}
and substituting into \eqref{KTred} gives, after a careful computation using appropriate weight considerations
\begin{equation*}
\begin{split}
\bigl(\mathcal{K}_{\mu,\sigma,W}^{\vee,(i)}\bigr)^T&=\mathbb{X}_{\mu,S}^{-1}\textup{Tr}_{W^*[-\sigma]}\Bigl(
\mathbb{Q}_{W^*}(\mu+h_{(V_k^*,\ldots,V_1^*)})R_{V_1^*W^*}^{21}(\mu+h_{(V_k^*,\ldots,V_1^*,W^*)})\cdots\\
&\qquad\qquad\qquad\quad\cdots R_{V_iW^*}^{21}(\mu+h_{(V_k^*,\ldots,V_i^*,W^*)})R_{V_{i+1}^*W^*}(\mu+h_{(V_k,\ldots,V_{i+1}^*,W^*)})
\cdots\\
&\qquad\qquad\qquad\qquad\qquad\cdots R_{V_k^*W^*}(\mu+h_{(V_k^*,W^*)})\mathbb{Q}_{W^*}(\mu)^{-1}\Bigr)\mathbb{X}_{\mu-\sigma,S}^{-1}.
\end{split}
\end{equation*}
Restricting the latter identity to $\cF^\str(\ul{S^*})[0]$, and using weight considerations and the cyclicity of the trace, we get \eqref{todddo}.
\end{proof}
We now show that for $i=0$, the equations \eqref{dualMReqn} reduce to the dual Macdonald-Ruijsenaars equations from \cite[Thm. 1.2]{Etingof&Varchenko-2000}. First note that for $i=0$, the equation \eqref{dualMReqnuniversal} reduces to the eigenvalue equation
\begin{equation}\label{eee}
\bigl(\textup{id}_{\cF^\str(\ul{S})[0]}\otimes\mathbb{L}_{W,0}^{\vee,S}\bigr)\mathbb{F}_S(\lambda,\cdot)=\chi_W\bigl(q^{-2\lambda-2\rho}\bigr)\mathbb{F}_S(\lambda,\cdot).
\end{equation}
Secondly, the trace over $W^*[-\sigma]$ in the definition of $\mathbb{L}_{W}^{\vee,S}$ can be rewritten as trace over $W[\sigma]$:
\begin{lemma}\label{Ltodualtrace}
For $S=(V_1,\ldots,V_k)$ and $f: \mathfrak{h}_{\textup{reg}}^*\rightarrow\cF^\str(\ul{S^*})[0]$ we have
\[
\bigl(\mathbb{L}_{W,0}^{\vee,S}f\bigr)(\mu)=\sum_{\sigma\in\textup{wts}(W)}\textup{Tr}_{W[\sigma]}\Bigl(R_{WV_k^*}^{21}(\mu-h_{(V_{k-1}^*,\ldots,V_1^*)})\cdots
R_{WV_2^*}^{21}(\mu-h_{V_1^*})R_{WV_1^*}(\mu)\Bigr)f(\mu-\sigma).
\]
\end{lemma}
\begin{proof}
By \eqref{case0k} we have
\begin{equation}\label{startproof}
\bigl(\mathbb{L}_{W,0}^{\vee,S}f\bigr)(\mu)=\sum_{\sigma\in\textup{wts}(W)}\textup{Tr}_{W[\sigma]}\Bigl(R_{W^*,S^*}^{21}(\mu-\sigma)^{-1,T_1*}\Bigr)f(\mu-\sigma)
\end{equation}
with $T_1*$ the transpose in the first component relative to $\widehat{e}_{\ul{W}}$.  
The identity in $\mathcal{N}_{\textup{fd}}$ underlying Figure \ref{braided evaluation B}, with the colors $(S,T)$ in Figure \ref{braided evaluation B} taken to be $(S^*,(W))$,
is
\begin{equation*}
\begin{split}
\bigl(\textup{id}_{\cF^\str(\ul{S^*})}\otimes &(e_W\circ j_{(W^*,W)}(\lambda))\bigr)\circ \bigl((R_{S^*,W^*}^{-1}(\lambda-h_W)P_{W^*,\cF^\str(\ul{S^*})})\otimes\textup{id}_{\ul{W}}\bigr)\\
=&\bigl((e_W\circ j_{(W^*,W)}(\lambda-h_{\ul{S^*}}))\otimes\textup{id}_{\cF^\str(\ul{S^*})}\bigr)\circ\bigl(\textup{id}_{\ul{W^*}}\otimes (P_{\cF^\str(\ul{S^*}),W}R_{S^*,W}(\lambda))\bigr)
\end{split}
\end{equation*}
for $\lambda\in\mathfrak{h}_{\textup{reg}}^*$, viewed as identity between $\mathfrak{h}$-linear maps $\ul{W^*}\otimes\cF^\str(\ul{S^*})\otimes\ul{W}\rightarrow\cF^\str(\ul{S^*})$. Using
\[
e_W\circ j_{(W^*,W)}(\lambda)=e_W\circ(\textup{id}_{\ul{W^*}}\otimes 
\mathbb{Q}^\prime(\lambda)_W)
\]
with $\mathbb{Q}^\prime(\lambda):=S(\mathbb{Q}(\lambda))$ (cf. Subsection \ref{Subsection Q(lambda)}), the above identity implies that for $w\in W[\sigma]$ and $g\in\cF^\str(\ul{S^*})$,
\begin{equation}\label{Qinv2}
R_{W^*,S^*}^{21}(\lambda-\sigma)^{-1,T_1*}(w\otimes g)=\bigl(\mathbb{Q}^\prime(\lambda-h_{S^*}))_W\otimes\textup{id}_{\cF^\str(\ul{S^*})}\bigr)R_{W,S^*}^{21}(\lambda)
\bigl(\mathbb{Q}^\prime(\lambda)_W^{-1}w\otimes g\bigr).
\end{equation}
Substituting into \eqref{startproof}, we get
\[
\bigl(\mathbb{L}_{W,0}^{\vee,S}f\bigr)(\mu)=\sum_{\sigma\in\textup{wts}(W)}\textup{Tr}_{W[\sigma]}\Bigl(\mathbb{Q}^\prime(\mu-h_{\ul{S^*}})_WR_{W,S^*}^{21}(\mu)\mathbb{Q}^\prime(\mu)_W^{-1}\Bigr)f(\mu-\sigma).
\]
The dynamical shift in $\mathbb{Q}^\prime$
may be removed since $\mathbb{Q}^\prime(\mu)_W$ preserves $W[\sigma]$, $R^{21}_{W,S^*}(\mu)$ is $\mathfrak{h}$-linear, and $f$ takes value in the zero weight space
$\cF^\str(\ul{S^*})[0]$ of $\cF^\str(\ul{S^*})$. The cyclicity of the trace then gives 
\[
\bigl(\mathbb{L}_{W,0}^{\vee,S}f\bigr)(\mu)=\sum_{\sigma\in\textup{wts}(W)}\textup{Tr}_{W[\sigma]}\Bigl(R_{W,S^*}^{21}(\mu)\Bigr)f(\mu-\sigma).
\]
The result now follows by expanding $R_{W,S^*}$ as product of dynamical $R$-matrices using Proposition  \ref{prop hexagon for dynamical R-matrix}(1). 
\end{proof}

Using the alternative expression of the dual MR operator $\mathbb{L}_{W,0}^{\vee,S}$ from Lemma \ref{Ltodualtrace} it now follows directly that \eqref{eee} reduces to the Etingof-Varchenko \cite[Thm. 1.2]{Etingof&Varchenko-2000} dual MR equations.

\subsection{The $q$-KZB and the coordinate MR equations for weighted trace functions}
\label{Section Symmetry}
In \cite[Thm.1.5]{Etingof&Varchenko-2000} Etingof and Varchenko proved the symmetry property 
\begin{equation}\label{symmF}
\mathbb{F}_S(\lambda,\mu)=P_{\cF^\str(\ul{S^*})[0],\cF^\str(\ul{S})[0]}\,\mathbb{F}_{S^*}(-\mu-2\rho,-\lambda-2\rho)
\end{equation}
for the quantum $k$-point weighted trace function $\mathbb{F}_S(\lambda,\mu)$, where $S\in\Rep^\str$ and $\lambda,\mu\in\mathfrak{h}_{\textup{reg}}^*$ are such that both sides of \eqref{symmF} are well defined. Here we identified $\cF^\str(\ul{S})[0]$ with $\cF^\str(\ul{S^{**}})[0]$ by restriction of the $U_q$-linear isomorphism 
\[
\mathcal{I}_{V_1}\otimes\cdots\otimes\mathcal{I}_{V_k}: \cF^\str(\ul{S})\overset{\sim}{\longrightarrow}\cF^\str(\ul{S^{**}})
\]
where for $V\in\Rep$,
\[
\mathcal{I}_V: V\overset{\sim}{\longrightarrow} V^{**},\qquad \mathcal{I}_V(v)f:=f(q^{2\rho}v)\qquad\quad (v\in V,\, f\in V^*)
\]
(note that $\mathcal{I}_V := (\widetilde{e}_V\otimes\id_{V^{\ast\ast}})(\id_V\otimes\iota_{V^\ast})$). The dual $q$-KZB equations and the dual coordinate MR equations for $\mathbb{F}_S(\lambda,\mu)$ (see Corollary \ref{corKKK} and Corollary \ref{dualqMRthm} respectively) are difference equations in $\mu$. Following \cite{Etingof&Varchenko-2000}, the symmetry can be used to derive analogous equations in $\lambda$. We will give the resulting explicit equations in this subsection. We start with the $q$-KZB equations.
\begin{definition}\label{LiSSS}
For $S=(V_1,\ldots,V_k)\in\Rep^\str$ and $i=1,\ldots,k$ we call the difference operators
\begin{equation}
\bigl(\mathbb{L}_i^Sf\bigr)(\lambda):=\sum_{\sigma\in\textup{wts}(V_i)}\mathcal{K}_{\lambda,\sigma}^{(i)}f(\lambda+\sigma)
\end{equation}
for functions $f: \mathfrak{h}_{\textup{reg}}^*\rightarrow\cF^\str(\ul{S^*})[0]$ the $q$-KZB operators, with $\mathcal{K}_{\lambda,\sigma}^{(i)}$ the $\mathfrak{h}$-linear operator
\begin{equation*}
\begin{split}
\mathcal{K}_{\lambda,\sigma}^{(i)}:=&
R_{V_iV_{i+1}}\bigl(-\lambda-2\rho-h_{(V_{i+2},\ldots,V_k)}\bigr)^{-1}\cdots R_{V_iV_k}\bigl(-\lambda-2\rho\bigr)^{-1}\mathbb{P}_{V_i}[\sigma]\\
&R_{V_1V_i}\bigl(-\lambda-2\rho-\sigma-h_{(V_2,\ldots,V_{i-1},V_{i+1},\ldots,V_k)}\bigr)\cdots R_{V_{i-1}V_i}\bigl(-\lambda-2\rho-\sigma-h_{(V_{i+1},\ldots,V_k)}\bigr)
\end{split}
\end{equation*}
on $\cF^\str(\ul{S})$.
\end{definition}
Write for $i=1,\ldots,k$,
\begin{equation}\label{DmuiS}
\mathbb{D}_{\mu,i}^S:=\bigl(q^{\theta(\mu)}\bigr)_{V_i^*}\kappa_{V_i^*V_{i-1}^*}^{-2}\cdots\kappa_{V_i^*V_1^*}^{-2}\in\textup{End}_{\mathcal{N}_{\textup{fd}}}\bigl(\cF^\str(\ul{S^*})\bigr)
\end{equation}
(see \eqref{thetakappa} for the definitions of $\theta(\mu)$ and $\kappa$).
The difference equations for $\mathbb{F}_S(\lambda,\mu)$ resulting from the symmetry \eqref{symmF} and the dual $q$-KZB equations (Corollary \ref{dualqKZBcor}) now read as follows.
\begin{theorem}\label{qKZBTHM}\textup{(}\cite[Thm. 1.3]{Etingof&Varchenko-2000}\textup{)}.
With the above notations, $\mathbb{F}_S(\lambda,\mu)$ satisfies the $q$-KZB equations
\begin{equation}\label{qKZBneww}
\bigl(\mathbb{L}_i^S\otimes \mathbb{D}_{\mu,i}^{S}\bigr)\mathbb{F}_S(\cdot,\mu)=\mathbb{F}_S(\cdot,\mu)
\end{equation}
for $i=1,\ldots,k$, and the same holds true for $\mathbb{T}_S(\lambda,\mu)$ \textup{(}see \eqref{universalTT}\textup{)}. The $q$-KZB equations for $\mathbb{T}_S(\lambda,\mu)$ are equivalent to the system of difference equations
\begin{equation}\label{coordinateqKZBneww}
\mathbb{L}_i^S\mathbb{T}_S^{v_1,\ldots,v_k}(\cdot,\mu)=q^{\langle\nu_i+2(\nu_{i-1}+\cdots+\nu_1)+2\mu+2\rho,\nu_i\rangle}
\mathbb{T}_S^{v_1,\ldots,v_k}(\cdot,\mu)
\end{equation}
for  $i=1,\ldots,k$ and $v_j\in V_j[\nu_j]$ with $\sum_j\nu_j=0$, with $\mathbb{T}_S^{v_1,\ldots,v_k}(\lambda,\mu)$ given by \eqref{TTv}.

\end{theorem}
\begin{proof}
For convenience to the reader we sketch the proof. Apply the symmetry property \eqref{symmF} to both sides of the $i^{\textup{th}}$ dual $q$-KZB equation \eqref{qKZBnew} and replace $S$ by $S^*$ and $i$ by $k-i+1$. After a change of variables and a direct computation using the explicit expression for the dual coordinate $q$-KZB operator 
 (see Definition \ref{LiS}), this yields \eqref{qKZBneww}. 

For the proof of \eqref{coordinateqKZBneww}, first note that
\[
\mathbb{F}_S(\lambda,\mu)=\sum \mathbb{T}^{\mathbb{X}_{\mu,S^*}^{T*}(v_1,\ldots,v_k)}(\lambda,\mu)\otimes v_k^*\otimes\cdots\otimes v_1^*
\]
with the sum over $v_j\in\mathcal{B}_{V_j}$ such that $\sum_j\nu_j=0$, where $\nu_j=\textup{wt}(v_j)$. Then \eqref{qKZBneww} implies
\begin{equation}\label{coordqKZBNORM}
\mathbb{L}_{i}^S\mathbb{T}_S^{\mathbb{X}_{\mu,S^*}^{T*}(v_1,\ldots,v_k)}(\cdot,\mu)=q^{\langle\nu_i+2(\nu_{i-1}+\cdots+\nu_1)+2\mu+2\rho,\nu_i\rangle}\mathbb{T}_S^{\mathbb{X}_{\mu,S^*}^{T*}(v_1,\ldots,v_k)}(\cdot,\mu).
\end{equation}
But $\mathbb{X}_{\mu,S^*}^{T^*}$ is an $\mathfrak{h}$-linear automorphism of $\cF^\str(\ul{S})$ that preserves $V_1[\nu_1]\otimes\cdots\otimes V_k[\nu_k]$ for all weights $\nu_i$, hence 
\eqref{coordqKZBNORM} implies \eqref{coordinateqKZBneww}. Then \eqref{coordinateqKZBneww} implies \eqref{qKZBneww} with $\mathbb{F}_S(\lambda,\mu)$ replaced by
$\mathbb{T}_S(\lambda,\mu)$.
\end{proof}
Next we derive coordinate MR equations for $\mathbb{F}_S(\lambda,\mu)$ using the symmetry \eqref{symmF} and the dual coordinate MR equations 
\eqref{dualMReqnuniversal}.
\begin{definition}[Coordinate MR operators]\label{coordMRDEF}
Let $S=(V_1,\ldots,V_k)\in\Rep^\str$ and $W\in\Rep$. For $i=0,\ldots,k$ we call the difference operators
\begin{equation*}
\begin{split}
\bigl(\mathbb{L}_{W,i}^Sf\bigr)(\lambda):=&\sum_{\sigma\in\textup{wts}(W)}\textup{Tr}_{W^*[-\sigma]}\Bigl(R_{W^*V_k}(-\lambda-2\rho+h_{(W^*,V_1,\ldots,V_k)})
\cdots\\
\cdots&R_{W^*V_{i+1}}(-\lambda-2\rho+h_{(W^*,V_1,\ldots,V_{i+1})})R_{W^*V_{i}}^{21}(-\lambda-2\rho+h_{(W^*,V_1,\ldots,V_{i})})^{-1}\cdots\\
&\qquad\qquad\qquad\qquad\cdots R_{W^*V_1}^{21}(-\lambda-2\rho+h_{(W^*,V_1)})^{-1}\Bigr)f(\lambda+\sigma)
\end{split}
\end{equation*}
for functions $f\in\mathfrak{h}_{\textup{reg}}^*\rightarrow\cF^\str(\ul{S})[0]$ the coordinate Macdonald-Ruijsenaars \textup{(}MR\textup{)} operators associated to $W$ and $S$.
\end{definition}
Write for $i=0,\ldots,k$, set
 \begin{equation}\label{DmuWiS}
 \mathbb{D}_{\mu,W,i}^{S}:=\textup{Tr}_W\Bigl(\bigl(q^{2\mu+2\rho}\bigr)_W\kappa_{WV_{i}^*}^{-2}\cdots\kappa_{WV_1^*}^{-2}\Bigr)\in\textup{End}_{\mathcal{N}_{\textup{fd}}}(\cF^\str(\ul{S})).
 \end{equation}
 The difference equations for $\mathbb{F}_S(\lambda,\mu)$ resulting from the symmetry \eqref{symmF} and the coordinate dual MR equations \eqref{dualMReqnuniversal} now read
 as follows.
 
\begin{theorem}[Coordinate MR equations]\label{coordMRthm}
With the above notations, $\mathbb{F}_S(\lambda,\mu)$ satisfies the coordinate MR equations
\begin{equation}\label{MRi}
\bigl(\mathbb{L}_{W,i}^S\otimes\textup{id}_{\cF^\str(\ul{S^*})[0]}\bigr)\mathbb{F}_S(\cdot,\mu)=
\bigl(\textup{id}_{\cF^\str(\ul{S})[0]}\otimes \mathbb{D}_{\mu,W,i}^{S}\bigr)\mathbb{F}_S(\cdot,\mu)
\end{equation}
for $i=0,\ldots,k$, and the same holds true for
$\mathbb{T}_S(\lambda,\mu)$. The coordinate MR equations for $\mathbb{T}_S(\lambda,\mu)$ are equivalent to the system of difference equations
\begin{equation}\label{coordMRi}
\mathbb{L}_{W,i}^S\mathbb{T}_S^{v_1,\ldots,v_k}(\cdot,\mu)=\chi_W\bigl(q^{2(\mu+\nu_1+\cdots+\nu_{i})+2\rho}\bigr)\mathbb{T}_S^{v_1,\ldots,v_k}(\cdot,\mu),
\end{equation}
for $v_j\in V_j[\nu_j]$ \textup{(}$j=1,\ldots,k$\textup{)} and $i=0,\ldots,k$,
with $\mathbb{T}_S^{v_1,\ldots,v_k}(\lambda,\mu)$ given by \eqref{TTv}.
\end{theorem}
\begin{proof}
The proof is analogous to the proof of Theorem \ref{qKZBTHM}. 

Apply the symmetry property \eqref{symmF} to both sides of the $i^{\textup{th}}$ dual MR equation \eqref{dualMReqnuniversal} and replace $S$ by $S^*$ and $i$ by $k-i$. After a change of variables and a direct computation using the explicit expression for the dual coordinate MR operator (see Definition \ref{idualMRoper}), this yields \eqref{MRi}. 
Then \eqref{MRi} implies
\begin{equation}\label{coordMRiNORM}
\mathbb{L}_{W,i}^S\mathbb{T}_S^{\mathbb{X}_{\mu,S^*}^{T*}(v_1,\ldots,v_k)}(\cdot,\mu)=\chi_W\bigl(q^{2(\mu+\nu_1+\cdots+\nu_{i})+2\rho}\bigr)\mathbb{T}_S^{\mathbb{X}_{\mu,S^*}^{T*}(v_1,\ldots,v_k)}(\cdot,\mu)
\end{equation}
for $v_j\in\mathcal{B}_{V_j}$ of weight $\nu_j$ such that $\sum_j\nu_j=0$. As in the proof of Theorem \ref{qKZBTHM} this implies 
\eqref{coordMRi}, as well as \eqref{MRi} with $\mathbb{F}_S(\lambda,\mu)$ replaced by $\mathbb{T}_S(\lambda,\mu)$.
\end{proof}
\begin{remark}
For $i=k$ the coordinate MR equation \eqref{MRi} reduces to
\begin{equation}\label{ah}
\bigl(\mathbb{L}_{W,k}^S\otimes\textup{id}_{\cF^\str(\ul{S^*})[0]}\bigr)\mathbb{F}_S(\cdot,\mu)=\chi_W(q^{2\mu+2\rho})
\mathbb{F}_S(\cdot,\mu).
\end{equation}
Furthermore, by Lemma \ref{Ltodualtrace} and the proof of Theorem \ref{coordMRthm}, $L_{k,W}^S$ admits the alternative expression
\[
(\mathbb{L}_{W,k}^Sf)(\lambda)=\sum_{\sigma\in\textup{wts}(W)}\textup{Tr}_{W[\sigma]}\Bigl(R_{WV_1}^{21}(-\lambda-2\rho-h_{(V_2,\ldots,V_k)})\cdots
R_{WV_k}^{21}(-\lambda-2\rho)\Bigr)f(\lambda+\sigma)
\]
for $f: \mathfrak{h}_{\textup{ref}}^*\rightarrow\cF^\str(\ul{S})[0]$. Using these two observations it follows directly that the $k^{\textup{th}}$ coordinate MR equations \eqref{ah} reduce to the Etingof-Varchenko \cite[Thm. 1.1]{Etingof&Varchenko-2000}
Macdonald-Ruijsenaars equations.
\end{remark}

\section{Appendix}\label{App}
In \cite[formula (2.38)]{Etingof&Varchenko-2000}, Etingof and Varchenko obtain an expression for  \(\Delta(\QQ(\lambda))\)
by an algebraic proof that invokes the results of \cite{Etingof&Varchenko-1999}. We provide an alternative purely graphical proof.

\begin{proposition}
	\label{prop EV equation 2.38}
	For any \(\lambda\in\hh_{\mathrm{reg}}^\ast\) and \(V,W\in \Rep\), we have
	\[
	\Delta(\QQ(\lambda))\,\mathbb{J}(\lambda+\nu+\nu')\big\vert_{V[\nu]\otimes W[\nu']} = 
	(S\otimes S)(\mathbb{J}^{21}(\lambda)^{-1})\left(\QQ(\lambda)\otimes\QQ(\lambda+\nu) \right)\big\vert_{V[\nu]\otimes W[\nu']}.
	\]
	\end{proposition}
\begin{proof}
	It suffices to show that
	\[
	u\left(\QQ_{V\otimes W}(\lambda)\circ j_{(V,W)}(\lambda+\nu+\nu')(v\otimes w)\right) = (j^{-1}_{(W^\ast,V^\ast)}(\lambda)u)\left(\QQ_V(\lambda)(v)\otimes \QQ_W(\lambda+\nu)(w) \right)
	\]
	for any \(v\in V[\nu]\), \(w\in W[\nu']\) and \(u\in (V\otimes W)^\ast\), which follows from a direct computation using the graphical calculus for \(\cN_{\mr{fd}}^\str\) and Figures \ref{natural proof A}--\ref{natural proof B}.
	\begin{figure}[H]
		\begin{minipage}{0.48\textwidth}
			\centering
			\includegraphics[scale = 0.9]{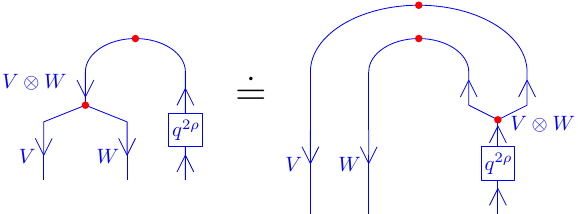}
			\captionof{figure}{}
			\label{natural proof A}
		\end{minipage}\quad
		\begin{minipage}{0.48\textwidth}
			\centering
			\includegraphics[scale = 0.9]{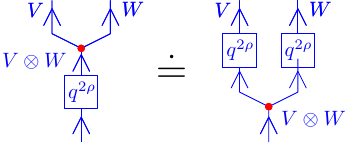}
			\captionof{figure}{}
			\label{natural proof B}
		\end{minipage}
	\end{figure}
	Figure \ref{natural proof A} follows by pushing the unzipping of a cap in the graphical calculus of $\Rep^\str$, as displayed in \cite[Figure 32]{DeClercq&Reshetikhin&Stokman-2022}, to the graphical calculus of $\Ndyn^{\,\str}$ using Proposition \ref{pushdiagram}. Figure \ref{natural proof B}
	 follows from the \(\hh\)-linearity of the colors of the occurring coupons.
\end{proof}

Figure \ref{braided evaluation B} was used in Lemma \ref{Ltodualtrace}. The closely related
Figure \ref{braided evaluation A} translates to the following expression for \(\Rdyn_{VW}(\lambda)^{-1}\).
\begin{corollary}
	For any \(\lambda \in\hh_{\mathrm{reg}}^\ast\) and \(V,W\in\Rep\), we have
	\begin{equation}
	\label{inverse R A}
	\Rdyn_{VW}(\lambda)^{-1}\big\vert_{V[\nu]\otimes W[\nu']} = 
	\left(S(\QQ^{-1}(\lambda-\mh^{(2)}))\otimes 1\right)\,(S\otimes\id)R(\lambda-\nu)\,\left(S(\QQ(\lambda))\otimes 1 \right)\big\vert_{V[\nu]\otimes W[\nu']}
	\end{equation}
	with the natural interpretation of the right hand side.
\end{corollary}
\begin{proof}
	Consider Figure \ref{braided evaluation A}, with \(S = (W)\) and \(T = (V)\), and with the rightmost vertical region on both sides 
	colored by the weight \(\lambda\). The left-hand side is mapped by $\mathcal{F}^{\textup{RT}}_{\mathcal{N}_{\textup{fd}}}$ to the morphism
	\[
	\left(\id_{\ul{W}}\otimes (e_{\ul{V}}\circ j_{(V^\ast,V)}(\lambda))\right)\circ\left((P_{\ul{V^\ast},\ul{W}}\circ \Rdyn_{V^\ast,W}(\lambda-\mh^{(3)}))\otimes\id_{\ul{V}} \right)
	\]
	in \(\Hom_{\cN_{\mr{fd}}}(V^\ast\otimes W\otimes V, W)\),
	which equals 
	\[
	(e_{\ul{V}}\otimes\id_{\ul{W}})\left(\id_{\ul{V^\ast}}\otimes\left(\mathbb{A}_{V,W}(\lambda)\circ P_{\ul{W},\ul{V}}\right) \right)
	\]
	with
	\[
	\mathbb{A}_{V,W}(\lambda) = (\pi_V\otimes\pi_W)\left((S\otimes \id)R(\lambda-\mh^{(1)})\left(S(\QQ(\lambda))\otimes 1\right)\right)
	\in\textup{End}_{\mathcal{N}_{\textup{fd}}}(\ul{V}\otimes\ul{W}).
	\]
	Here we have used (\ref{dyn eval expr}), (\ref{Q algebraic def}) and the definition of the dual representation. 
	The right-hand side is mapped by $\mathcal{F}^{\textup{RT}}_{\mathcal{N}_{\textup{fd}}}$ to the morphism
	\[
	(e_{\ul{V}}\otimes\id_{\ul{W}})\left(\id_{\ul{V^\ast}}\otimes\left(\mathbb{B}_{V,W}(\lambda)\circ P_{\ul{W},\ul{V}}\right) \right)
	\]
	with
	\[
	\mathbb{B}_{V,W}(\lambda) = (\pi_V\otimes\pi_W)\left(\left(S(\QQ(\lambda-\mh^{(2)}))\otimes 1\right)R(\lambda)^{-1}\right).
	\]
	Hence we have \(\mathbb{A}_{V,W}(\lambda) = \mathbb{B}_{V,W}(\lambda)\), leading to (\ref{inverse R A}). 
\end{proof}


\end{document}